%% file: QUCP.tex
\documentclass[10pt]{article}
\usepackage[english,activeacute]{babel}
\usepackage{epsfig}

\usepackage[latin1]{inputenc}
\usepackage[T1]{fontenc}
\usepackage{stmaryrd}
\usepackage{amsmath,amsfonts,amssymb,mathrsfs,amsthm}
\usepackage{xcolor}
\usepackage{dsfont}
\usepackage{graphicx}
\usepackage[mathscr]{eucal}
\usepackage{makeidx}
\usepackage{verbatim}
\usepackage{graphics,graphicx}
\usepackage{textcomp}
\usepackage{float}
\input mymacros.tex

\def\e{{\varepsilon}}

\newcommand\bna{\begin{eqnarray*}} 
\newcommand\ena{\end{eqnarray*}}

\newcommand\bnan{\begin{eqnarray}} 
\newcommand\enan{\end{eqnarray}}

\newcommand\bnp{\begin{proof}} 
\newcommand\enp{\end{proof}}

\newcommand\bneq{\begin{eqnarray*}\left\lbrace \begin{array}{rcl}}
\newcommand\eneq{\end{array} \right.\end{eqnarray*}}
\newcommand\bneqn{\begin{eqnarray}\left\lbrace \begin{array}{rcl}}
\newcommand\eneqn{\end{array} \right.\end{eqnarray}}

\newcommand\nor[2]{\left\|#1\right\|_{#2}}

\newcommand\grando[1]{\mathcal{O}(#1)}

\newcommand\chit{\widetilde{\chi}}

\numberwithin{equation}{section}

\newtheorem{hypo}{Assumption}[section]

\author{Camille Laurent\footnote{CNRS UMR 7598 and UPMC Univ Paris 06, Laboratoire Jacques-Louis Lions, F-75005, Paris, France, email: laurent@ann.jussieu.fr} and Matthieu L\'eautaud\footnote{Universit\'e Paris Diderot, Institut de Math\'ematiques de Jussieu-Paris Rive Gauche, UMR 7586, B\^atiment Sophie Germain, 75205 Paris Cedex 13 France, email: leautaud@math.univ-paris-diderot.fr}
}

\begin{document}
\title{Quantitative unique continuation for operators with partially analytic coefficients. Application to approximate control for waves.}
\maketitle

\textbf{Abstract.} In this article, we first prove quantitative estimates associated to the unique continuation theorems for operators with partially analytic coefficients of Tataru~\cite{Tataru:95, Tataru:99}, Robbiano-Zuily~\cite{RZ:98} and H\"ormander~\cite{Hor:97}. We provide local stability estimates that can be propagated, leading to global ones. 
 
Then, we specify the previous results to the wave operator on a Riemannian manifold $\M$ with boundary. For this operator, we also prove  Carleman estimates and local quantitative unique continuation from and up to the boundary $\d \M$. This allows us to obtain a global stability estimate from any open set $\Gamma$ of $\M$ or $\d \M$, with the optimal time and dependence on the observation. 
 
This provides the cost of approximate controllability: for any $T>2 \sup_{x \in \M}(\dist(x,\Gamma))$, we can drive any data of $H^1_0 \times L^2$ in time $T$ to an $\eps$-neighborhood of zero in $L^2 \times H^{-1}$, with a control located in $\Gamma$, at cost $e^{C/\eps}$.

We also obtain similar results for the Schr\"odinger equation.

\tableofcontents

\section{Introduction and main results}
In this article, we are interested in the quantification of {\em global unique continuation} results of the following form: given a differential operator $P$ on an open set $\Omega \subset \R^n$, and given a small subset $U$ of $\Omega$, having
\bnan
\label{e:UCP-general}
Pu = 0  \text{ in  } \Omega, \quad u|_{U}=0 \Longrightarrow u= 0 \text{ on } \Omega .
\enan
More generally, in cases where~\eqref{e:UCP-general} is known to hold, we are interested in proving a quantitative version of
\bna
Pu \text{ small in  } \Omega, \quad u\text{ small in }U \Longrightarrow u \text{ small in } \Omega.
\ena
A more tractable problem than~\eqref{e:UCP-general} is the so called {\em local unique continuation} problem: given $x^0 \in \R^n$ and $S$ an oriented local hypersurface containing $x^0$, do we have the following implication:
\bnan
\label{e:UCP-general-local}
\text{There is a neighborhood } \Omega  \text{ of }x^0 ,\text{ such that } Pu = 0  \text{ in } \Omega , u|_{ \Omega \cap S^-}  = 0
\Longrightarrow x^0 \notin \supp(u) .
\enan
It turns out that proving~\eqref{e:UCP-general-local} for a suitable class of hypersurface (with regards to the operator $P$) is in general a key step in the proof of properties of the type~\eqref{e:UCP-general}. The first general unique continuation result of the form~\eqref{e:UCP-general-local} is the Holmgren Theorem, stating that, for operators with analytic coefficients, unique continuation holds across any noncharacteristic hypersurface $S$. This local unique continuation result enjoys a global version proved by John~\cite{John:49}, where uniqueness is propagated through a family of noncharateristic hypersurfaces. 

When focusing on operators with (only) smooth coefficients, the most general results was proved by H\"ormander~\cite{Hoermander:63},~\cite[Chapter~XXVIII]{Hoermander:V4}. Uniqueness across a hypersurface holds assuming a strict pseudoconvexity condition (see e.g. Definition~\ref{def: pseudoconvex-surface} below). 
This result uses as a key tools Carleman estimates, which were introduced in~\cite{Carleman:39} and developed at first for elliptic operators in~\cite{Calderon:58}. We also refer to \cite{Zuily:83} for a general presentation of these problems.

\bigskip
A particular motivation arises both from geoseismics~\cite {Symes:83} and control theory~\cite{Lio:88, Lio2:88}: in these contexts, one is interested in recovering the data/energy of a wave from the observation on a small part of the domain along a time interval. As well, unique continuation results for waves have been useful tools to solve inverse problems, for instance using the boundary control method~\cite{Belishev:87} (see also the review article~\cite{Belishev:07} and the book~\cite{KKL:01}). 

More precisely, consider the wave operator $P = \d_t^2 - \Delta_g$ on $\Omega = (-T,T) \times \M$, where $(\M , g)$ is a Riemannian manifold (with or without boundary) and $\Delta_g$ the associated (negative) Laplace-Beltrami operator.  A central question raised by the above applications is that of global unique continuation from sets of the form $(-T,T) \times \omega$, 
where $\omega \subset \M$ (resp. $\omega \subset \d\M$) is an observation region.

In this setting and in the context of control theory, the unique continuation property~\eqref{e:UCP-general} is equivalent to approximate controllability (from $(-T,T) \times \omega$); and an associated quantitative estimate (as proved in the present paper) is equivalent of estimating the cost of approximate controls.

If $\M$ is analytic (and connected), the Holmgren theorem applies, which together with the argument of John~\cite{John:49}, allows to prove unique continuation from $(-T,T) \times \omega$ for any nonempty open set $\omega$ as soon as $T>  \mathcal{L} (\M ,\omega)$, where, for $E \subset \M$, we have set 
\bnan
\label{e:def-L}
\mathcal{L}(\M ,E):= \sup_{x_1 \in \M}\big( \inf_{x_0 \in E} \dist(x_0 , x_1) \big)  , \quad \dist(x_0 , \x_1) = \inf_{\gamma \in C^0([0,1] ; \M), \gamma(0) = x_0 , \gamma(1) = x_1}\length(\gamma).
\enan
Due to finite speed of propagation, it is also not hard to prove that unique continuation from $(-T,T) \times \omega$ does not hold if $T <\mathcal{L}(\M ,\omega)$, so that the result is sharp.

Removing the analyticity condition on $\M$ has lead to a considerable difficulty, since H\"ormander general uniqueness result does not apply in this setting: time-like surfaces, as $\{x_1 = 0\}$, do not satisfy the pseudoconvexity assumption for the wave operator. The local unique continuation can even fail when adding some smooth lower order terms to the wave operator, as proved by Alinhac-Baouendi~\cite{AB:79,Alinhac:83,AB:95}.

This uniqueness problem in the $C^\infty$ setting was first solved by Rauch-Taylor~\cite{RT:72} and Lerner~\cite{Lerner:88} in the case $T= \infty$, and $\M = \R^d$ (under different assumptions at infinity). Then, a result of Robbiano~\cite{Robbiano:91} shows that it holds in any domain $\M$ for $T$ sufficiently large. H\"ormander~\cite{Hormander:92} improved this result down to $T> \sqrt{\frac{27}{23}} \mathcal{L}(\M ,\omega)$. That these two results fail to hold in time $\mathcal{L}$ translates the fact that the local uniqueness results of these two authors are not valid across any noncharacteristic surface.

\bigskip
The proof of local uniqueness results across any noncharacteristic surface for $\d_t^2 - \Delta_g$ was reached by Tataru in~\cite{Tataru:95}, leading to the global unique continuation result in optimal time $T> \mathcal{L}(\M ,\omega)$. The result of Tataru was not restricted to the wave operator:
he considered operators with coefficients that are analytic in part of the variables, interpolating between the Holmgren theorem and the H\"ormander theorem. Technical assumptions of this article were successively removed by Robbiano-Zuily~\cite{RZ:98}, H\"ormander~\cite{Hor:97} and Tataru~\cite{Tataru:99},  leading to a very general local unique continuation result for operators with partially analytic coefficients (containing as particular cases both Holmgren and H\"ormander theorems).

\bigskip
Concerning quantitative estimates of unique continuation, when~\eqref{e:UCP-general} holds, one may expect to have an estimate of the form 
\bnan
\label{e:QUCP-phi}
\nor{u}{\Omega} \leq  \varphi\big(\nor{u}{U},\nor{Pu}{\tilde\Omega} ,\nor{u}{\tilde\Omega}\big) , \quad \text{with } \varphi (a,b,c) \to 0 \text{ when } (a,b) \to 0 \text{ with } c \text{ bounded},
\enan
where $U \subset \Omega \subset \tilde \Omega$ are nonempty, and for appropriate norms. In this context, much less seems to be known.
Two additional difficulties arise: one needs first to quantify the local unique continuation property~\eqref{e:UCP-general-local},  and then to ``propagate'' the local estimates obtained towards a global one.

In the setting of the Holmgren theorem, local estimates of unique continuation of the form~\eqref{e:QUCP-phi} were proved by John~\cite{John:60}: they are of H\"older type, i.e. $\varphi (a,b,c) = (a+b)^\delta c^{1-\delta}$, in the case $P$ is elliptic, and of logarithmic type, i.e. $\varphi (a,b,c) = c \left(\log(1+\frac{c}{a+b})\right)^{-1}$, in the general case.

In the situation of the H\"ormander theorem, it is proved by Bahouri ~\cite{Bahouri:87} that H\"older stability always holds locally. Such local estimates were propagated, leading to global ones (in the case of elliptic operators $P$ of order two, even with low regularity assumptions) by Lebeau and Robbiano~\cite{Robbiano:95, LR:95}. They can also be improved to $\varphi (a,b,c) = a+b$ if boundary conditions are added~\cite{Robbiano:95, LR:95}.

The global problem for the wave operator in the analytic setting was tackled by Lebeau in~\cite{Leb:Analytic}. For $\Omega = \tilde\Omega = (-T,T) \times \M$ and $U= (-T,T)\times \omega$ with $\omega\subset \M$ (or more precisely $\Gamma\subset \partial\M$), he proved that the stability estimate~\eqref{e:QUCP-phi} with $\varphi (a,b,c) = c \left(\log(1+\frac{c}{a+b})\right)^{-1}$ holds for any $T > \mathcal{L}(\M ,\omega)$.
He also proved that this inequality is optimal if there exists a ray of geometric optic that does not intersect $(-T,T)\times \omega$ (and only has transverse intersection with $\d\M$). Under this assumption the (stronger) geometric control estimate (i.e.~\eqref{e:QUCP-phi} with $\varphi (a,b,c) = a+b$) of the Bardos-Lebeau-Rauch-Taylor Theorem~\cite{RT:74,BLR:92} is not satisfied.
When considering the $C^\infty$ situation for this problem, the first result is due to Robbiano~\cite{Robbiano:95}, who proved the result for $T$ suffiently large with $\varphi (a,b,c) =  c \left(\log(1+\frac{c}{a+b})\right)^{-\frac12}$. The result was improved by Phung \cite{Phung:10} to $\varphi (a,b,c) =  c \left(\log(1+\frac{c}{a+b})\right)^{-(1-\e)}$ (still in large time).
In his unpublished lecture notes~\cite{Tatarunotes}, Tataru proposes a strategy to obtain estimates of the form~\eqref{e:QUCP-phi} with $\varphi_\eps = c \left(\log(1+\frac{c}{a+b})\right)^{-(1-\eps)}$ in the general context of the uniqueness theorem for operators with partially analytic coefficients.

\bigskip
In this article, we develop a systematic approach both to quantify the local uniqueness Theorem of Tataru, Robbiano-Zuily and H\"ormander, and to propagate the quantitative local uniqueness results towards a global one (with optimal dependence $\varphi = c \left(\log(1+\frac{c}{a+b})\right)^{-1}$). When doing so, we face both difficulties of producing quantitative and global estimates.
 Then, we specify the previous results to the wave operator on $\M$. For this operator, we also prove appropriate Carleman estimates and local quantitative unique continuation results from and up to the boundary $\d \M$. This allows us to obtain a global stability estimate from any open set of $\M$ or $\d \M$, with the optimal time ($T>\mathcal{L}(\M,\omega)$) and dependence on the observation. This generalizes the result of Lebeau~\cite{Leb:Analytic} to non-analytic manifolds, and provides the cost of approximate controllability. We also treat the case of the Schr\"odinger operator.

\bigskip
In the present introduction, we first discuss the case of the wave and Schr\"odinger equations: in this particular setting, the results are simpler to state and more precise. Moreover, in this context, we are able to deal with the boundary value problem as well. Second, we state the general quantitative uniqueness result for operators with partially analytic coefficients in the setting of Tataru~\cite{Tataru:95, Tataru:99}, Robbiano-Zuily~\cite{RZ:98} and H\"ormander~\cite{Hor:97} (used in the proof for the wave equation). 
\subsection{The wave and Schr\"odinger equations}
\label{sec:intro-wave-schrod}
In this section, we describe the motivating applications of our main result, i.e. to the wave equation. In this very particular setting, we are also able to tackle the boundary value problem. We also state an analogous result for the Schr\"odinger equation.

\begin{theorem}[Quantitative unique continuation for waves]
\label{thmobserwaveintro}
Let $\M$ be a compact Riemannian manifold with (or without) boundary. For any nonempty open subset $\omega$ of $\M$ and any $T> 2\mathcal{L}(\M ,\omega)$, there exist $C, \kappa ,\mu_0>0$ such that for any $(u_0,u_1)\in H^1_0(\M)\times L^2(\M)$ and associated solution $u$ of 
\bneqn
\label{e:free-wave}
\partial_t^2u-\Delta_g u= 0& & \text{ in } [0,T] \times \M , \\
u_{\left|\partial \M\right.}=0 & &  \text{ in } [0,T] \times \d \M ,\\
(u,\partial_tu)_{\left|t=0\right.}=(u_0,u_1)& & \text{ in }\M ,
\eneqn
we have, for any $\mu\geq \mu_0$,
\bna
\nor{(u_0,u_1)}{L^2\times H^{-1}}\leq C e^{\kappa \mu}\nor{u}{L^2((0,T); H^1(\omega))}+\frac{1}{\mu}\nor{(u_0,u_1)}{H^1\times L^2}.
\ena
If $\partial\M\neq \emptyset$ and $\Gamma$ is a non empty open subset of $\partial \M$, for any $T>2\mathcal{L}(\M ,\Gamma)$, there exist $C, \kappa,\mu_0 >0$ such that for any $(u_0,u_1)\in H^1_0(\M)\times L^2(\M)$ and associated solution $u$ of~\eqref{e:free-wave}, we have 
\bna
\nor{(u_0,u_1)}{L^2\times H^{-1}}\leq C e^{\kappa \mu}\nor{\partial_{\nu}u}{L^2((0,T)\times \Gamma)} +\frac{1}{\mu}\nor{(u_0,u_1)}{H^1\times L^2}.
\ena
\end{theorem}
Theorem \ref{thmobserwaveintro} remains valid if $\Delta_g$ is perturbated by lower order terms that are analytic in time. In the special case where they are time independent, the constants in the previous estimates may be chosen uniformly with respect to these perturbations (in the appropiate norms). We refer to Theorem \ref{thmobserwave} for a  precise statement. This result can also be formulated in the following way, closer to the formulation~\eqref{e:QUCP-phi} (see Lemma~\ref{lmfonction}).
We only give the boundary observation case.

\begin{corollary}
\label{corlogstab}
Assume $\partial\M\neq \emptyset$ and $\Gamma$ is a non empty open subset of $\partial \M$. Then, for any $T>2\mathcal{L}(\M ,\Gamma)$, there exists $C >0$ such that for any $(u_0,u_1)\in H^1_0(\M)\times L^2(\M)$ and associated solution $u$ of~\eqref{e:free-wave}, we have 
\bna
\nor{(u_0,u_1)}{L^2\times H^{-1}}\leq C \frac{\nor{(u_0,u_1)}{H^1\times L^2}}{\log\left(\frac{\nor{(u_0,u_1)}{H^1\times L^2}}{\nor{\partial_{\nu}u}{L^2(]0,T[\times \Gamma)}}+1\right)} ,
\ena
\bna
\nor{(u_0,u_1)}{H^1\times L^2}\leq C e^{C\Lambda}\nor{\partial_{\nu}u}{L^2(]0,T[\times \Gamma)} , \quad \text{ with }\Lambda=\frac{\nor{(u_0,u_1)}{H^1\times L^2}}{\nor{(u_0,u_1)}{L^2\times H^{-1}}} .
\ena
\end{corollary}

In the previous estimate, $\Lambda$ has to be considered as the typical frequency of the initial data. So, the estimate states a cost of observability of the order of an exponential of the typical frequency.

As proved by Lebeau \cite{Leb:Analytic} in the analytic case, this exponential dependence is sharp in the general case.

As a consequence of the previous Theorem, we can obtain some approximate controllability results as follows. For the sake of brevity, we only state the case of a boundary control.
\begin{theorem}[Cost of boundary approximate control]
For any $T> 2 \mathcal{L}(\M ,\Gamma)$, there exist $C,c>0$ such that for any $\eps >0$ and any $(u_0 ,u_1) \in H^1_0(\M) \times L^2(\M)$, there exists $g \in L^2((0,T) \times \Gamma)$ with 
$$
\|g\|_{L^2((0,T) \times \Gamma)} \leq C e^{\frac{c}{\eps}} \nor{(u_0 ,u_1)}{H^1_0(\M) \times L^2(\M)} ,
$$
such that the solution of 
\bneq
(\d_t^2 - \Delta) u  =  0 && \text{ in } (0,T) \times \M , \\
u_{| \d \M} = \mathds{1}_\Gamma g && \text{ in } (0,T) \times \d \M ,\\
(u ,\d_t u)_{|t=0} = (u_0 , u_1) ,&& \text{ in }\M ,
\eneq
satisfies $\nor{(u ,\d_t u)_{|t=T}}{L^2(\M) \times H^{-1}(\M)} \leq \eps\nor{(u_0 ,u_1)}{H^1_0(\M) \times L^2(\M)}$. 
\end{theorem}
That this result is a consequence of Theorem \ref{thmobserwaveintro} is proved in \cite[Proof of Theorem~2, Section~3]{Robbiano:95}. The solution of the nonhomogeneous boundary value problem are defined in the sense of transposition, see \cite{Lio:88}.

\bigskip
We also obtain similar results for the Schr\"odinger equation. We only state here the counterpart of Theorem~\ref{thmobserwaveintro} in this setting.
\begin{theorem}
\label{thmobserschrodintro}
Let $\M$ be a compact Riemannian manifold with (or without) boundary. For any nonempty open subset $\omega$ of $\M$ and any $T> 0$, there exist $C ,\kappa , \mu_0 >0$ such that for any $u_0 \in H^2\cap H^1_0$ and associated solution $u$ of 
\bneqn
\label{e:schrod-free}
i\partial_t u+\Delta_g u= 0 &&\text{ in } (0,T) \times \M , \\
u_{\left|\partial \M\right.}=0   && \text{ in } (0,T) \times \d \M ,\\
u(0)=u_0&& \text{ in }\M ,
\eneqn
we have, for any $\mu\geq \mu_0$,
\bna
\nor{u_0}{L^2}\leq C e^{\kappa \mu}\nor{u}{L^2((-T,T); H^1(\omega))}+\frac{1}{\mu}\nor{u_0}{H^2}.
\ena
If $\partial\M\neq \emptyset$ and $\Gamma$ is a non empty open subset of $\partial \M$, then for any $T> 0$, there exist $C ,\kappa , \mu_0 >0$ such that for any $u_0 \in H^2\cap H^1_0$ and associated solution $u$ of~\eqref{e:schrod-free}, we have 
\bna
\nor{u_0}{L^2}\leq C e^{\kappa \mu}\nor{\partial_{\nu}u}{L^2((-T,T)\times \Gamma)}+\frac{1}{\mu}\nor{u_0}{H^2}.
\ena
\end{theorem}
As well, this result still holds with some lower order perturbations, analytic in $t$, see Theorem \ref{thmobserschrod} for a more precise statement. 

Note that some related results have already been proven in the internal case by Phung \cite{Phung:01} with $e^{\kappa \mu}$ replaced by $e^{\kappa \mu^2}$.
\subsection{Quantitative unique continuation for operators with partially analytic coefficients} 

Let us now turn to the general stability result and present the class of partial differential operators we deal with. We consider domains $\Omega \subset \R^n=\R^{n_a} \times \R^{n_b}$, where $n_a + n_b= n$. We denote by $x=(x_a,x_b)$ the global variables and $\xi=(\xi_a,\xi_b)$ the associated dual variables. The variables $x_a$ will denote the set of variables in which the considered operator is analytic.

We recall that, given a bounded domain $\Omega \subset \R^n=\R^{n_a} \times \R^{n_b}$, a smooth function $f : \Omega \to \R$ is analytic with respect to $x_a$ if, for any point $x^0= (x_a^0 , x_b^0) \in \Omega$, $f$ is equal to its partial Taylor expansion at $x_a^0$ with respect to the variable $x_a$ in a neighborhood of $x^0$ in $\Omega$.
Such a function extends as a holomorphic function in the variable $x_a$ in $B(x_a^0 , \eps ) + i B(0, \eps) \times B(x_b^0,\eps)$ for some $\eps >0$.

\medskip
The folowing definition is due to Tataru~\cite[Definition~2.2]{Tataru:99}.
\begin{definition}[analytically principally normal operator]
\label{def:anal principal normal}
Let $P$ be a partial differential operator on an open set $\Omega\subset \R^{n_a} \times \R^{n_b}$ of order $m \in \N^*$  with smooth coefficients and principal symbol $p(x_a, x_b, \xi_a , \xi_b)$. We say that $P$ is an analytically principally normal operator in $\{\xi_a = 0\}$ inside $\Omega$ if the coefficients of $P$ are real-analytic in the variable $x_a$ and for any $x^0\in \Omega$ there exist $\Omega_a\subset \R^{n_a}$, $\Omega_b\subset \R^{n_b}$, such that $x^0\in \Omega_a\times \Omega_b$, $\Omega_a\times \Omega_b\subset \Omega$ and there exists a complex neighborhood $\Omega^\C_a$ of $\Omega_a$ in $\C^{n_a}$ and a constant $C>0$ such that for all $z_a , \tilde{z}_a \in \Omega^\C_a$ and all $(x_b, \xi_b) \in \Omega_b \times \R^{n_b} , \xi_b \neq 0$, we have
\bnan
\label{cond1}
\left|  \left\{ p(z_a , \cdot , 0, \cdot), p(\tilde{z}_a , \cdot , 0, \cdot) \right\} (x_b, \xi_b)\right| + \left|  \left\{ \overline{p(z_a , \cdot , 0, \cdot)}, p(\tilde{z}_a , \cdot , 0, \cdot) \right\} (x_b, \xi_b)\right|  \leq C |p(z_a , x_b , 0, \xi_b)| |\xi_b|^{m-1} , 
\enan
\bnan
\label{cond2}
\left| \d_{z_a}p(z_a , x_b , 0, \xi_b) \right|  \leq C |p(z_a , x_b , 0, \xi_b)|.
\enan
\end{definition}

Note that in this definition, the Poisson brackets are taken only with respect to the $(x_b, \xi_b)$ variables. Yet, the combination of the two conditions \eqref{cond1} and \eqref{cond2} implies that such operators are in particular principally normal in $\{\xi_a = 0\}$ in the usual sense (see~\cite{RZ:98}, \cite{Hor:97} or \cite[Definition~2.1]{Tataru:99}), that is
\bnan
\label{principalnormalclassique}
\left|  \left\{ \overline{p}, p \right\} (x_a , x_b , 0, \xi_b)\right|  \leq C |p(x_a , x_b , 0, \xi_b)| |\xi_b|^{m-1},
\enan
where this time, $ \left\{ \overline{p}, p \right\}$ is computed with respect to all the variables.

Two interesting cases of operators $P$ being analytically principally normal in $\{\xi_a = 0\}$, considered in~\cite{RZ:98} and~\cite{Hor:97}, are operators with analytic coefficients in $x_a$ satisfying one of the following two assumptions: 
\begin{itemize}
\item[(E)] transversal ellipticity: $p(x_a,x_b, 0, \xi_b) \geq c |\xi_b|^m$ for $(x_a,x_b)\in \Omega, \xi_b \in \R^{n_b}$;
\item[(H)] principal normality and invariance with respect to the null bicharacteristic flow in $\{\xi_a = 0\}$: $$\left|  \left\{ \overline{p}, p \right\} (x_a , x_b , 0, \xi_b)\right|  \leq C |p(x_a , x_b , 0, \xi_b)| |\xi_b|^{m-1} \quad \text{and} \quad \d_{x_a}p(x_a , x_b , 0, \xi_b) = 0.$$
\end{itemize}
We now formulate the definition of strongly pseudoconvex surfaces for an operator $P$, see \cite[Definition~28.3.1]{Hoermander:V4}, \cite[Definitions~2.3 and 2.4]{Tataru:99} and~\cite[Section~1.2]{Tatarunotes}.

\begin{definition}[Strongly pseudoconvex oriented surface]
\label{def: pseudoconvex-surface}
Let $\Omega \subset \R^n$,  $\Gamma$ be a closed conic subset of $T^*\Omega$, and let $P$ be principally normal in $\Gamma$ inside $\Omega$ (in the sense of~\eqref{principalnormalclassique}) with principal symbol $p$. Let $S$ be a $C^2$ oriented hypersurface of $\Omega$ and $x^0 \in S \cap \Omega$. We say that $S$ is strongly pseudoconvex in $\Gamma$ at $x^0$ for $P$ if there exists $\phi \in C^2(\Omega; \R)$ such that $S = \{\phi = 0\}$, $\nabla \phi (x^0) \neq 0$, satisfying:
\bnan
\label{e:pseudo-surface-1}
\Re \left\{ \overline{p}, \{p ,\phi \} \right\} (x^0 ,\xi) >0  , 
&\text{ if } p(x^0,\xi) =  \{p ,\phi \}  (x^0 ,\xi) = 0 \text{ and } \xi \in \Gamma_{x^0}, \xi \neq 0 ; \\
\label{e:pseudo-surface-2}
\frac{1}{i\tau}\{ \overline{p}_\phi,p_\phi \} (x^0 ,\xi) >0  , 
&\text{ if } p_\phi(x^0,\xi) =  \{p_\phi ,\phi \}  (x^0 ,\xi) = 0 \text{ and } \xi \in \Gamma_{x^0}, \tau >0 ,
\enan
where $p_\phi(x,\xi) = p(x, \xi + i \tau\nabla \phi)$.
\end{definition}
Note that this is a property of the {\em oriented} surface $S$ solely, and not of the defining function $\phi$ (see~\cite{Hoermander:V4}, beginning of Section~28.3). If $\Gamma = T^*\Omega$, it is the usual condition of the H\"ormander Theorem (see~\cite[Section~28.3]{Hoermander:V4}), that is, under which uniqueness holds for $P$ at $x^0$ across the hypersurface $S$, i.e. from $\phi >0$ to $\phi <0$. 

Below, this condition will always be used for $\Gamma = \{\xi_a = 0\}$. In this case, and using the homogeneity of $p$ in $\xi$, Assumption~\eqref{e:pseudo-surface-2} may be rephrased as:
$$
\frac{1}{i}\{ \overline{p}(x, \xi - i \nabla \phi),p (x, \xi + i \nabla \phi) \} (x^0 ,0, \xi_b) >0  , 
\quad \text{ if } p (\zeta) =  \{p  ,\phi \} (\zeta) = 0  , \quad  \xi_b \in \R^{n_b} ,
$$
where $\zeta =  (x^0 , i \nabla_a\phi(x^0) , \xi_b+ i  \nabla_b \phi(x^0) )$. An important feature of this definition is that it is invariant by changes of coordinates.

Note also that in the case $\Gamma = \{\xi_a = 0\}$, the condition \eqref{e:pseudo-surface-1} is the limit as $\tau\rightarrow 0^+$ of \eqref{e:pseudo-surface-2} on the subset $\left\{p_\phi(x^0,\xi) =  \{p_\phi ,\phi \}  (x^0 ,\xi) = 0\right\}\cap \Gamma_{x^0}$, thanks to the principal normality assumption \eqref{principalnormalclassique}, see Remark~\ref{rem:limittautozero} below.

\bigskip
Before stating our main result, let us discuss some cases of operators of particular interest.
\begin{remark}[H\"ormander case]
If $n_a = 0$, there is no analytic variable. In this case, Definition~\ref{def:anal principal normal} coincides with the definition of principally normal operators~\cite[Chapter~XXVIII]{Hoermander:V4} and Definition~\ref{def: pseudoconvex-surface} with $\Gamma = T^*\Omega$ that of strictly pseudoconvex functions. The unique continuation result under consideration is the classical H\"ormander theorem~\cite[Chapter~XXVIII]{Hoermander:V4}.
\end{remark}

\begin{remark}[Holmgren case]
\label{rkHolmgren}
If $n_a=n$, that is the operator is analytic in all the variables, we have $x_a=x, \xi_a = \xi$, and hence $\Gamma = \Omega \times \{\xi_a = 0\}= \Omega \times \{\xi= 0\}$. In this situation, conditions \eqref{cond1}, \eqref{cond2} are empty since all the terms vanish.

Next, concerning the conditions on the surface $\{\phi = 0\}$, notice that \eqref{e:pseudo-surface-1} is also empty since $\Gamma_{x^0} \cap \left\{\xi\neq 0\right\}=\emptyset$. For \eqref{e:pseudo-surface-2}, if $\xi\in \Gamma_{x^0}$, that is $\xi=0$, we have $p_\phi(x^0 ,\xi) =p(x^0, i\tau \nabla \phi)=(i\tau)^m p(x^0,  \nabla \phi)$: any noncharacteristic surface is a strongly pseudoconvex oriented surface. 

Note that, in the case $n_a = n$, the results presented here hold under the condition:
$$
p (x^0 , \nabla \phi(x^0) ) =  \{p  ,\phi \} (x^0 , \nabla \phi(x^0) ) = 0 \Longrightarrow \frac{1}{i}\{ \overline{p}(x, \xi - i \nabla \phi),p (x, \xi + i \nabla \phi) \} (x^0 ,0 ) >0  ,
$$
which is weaker than the noncharactericity condition $p (x^0 , \nabla \phi(x^0) ) \neq 0$ of the Holmgren theorem.
\end{remark}

\begin{remark}[Wave type and Schr\"odinger type operators]
\label{rknoncaractwave}
Let us now consider the case of operators $P$ of principal symbol of the form $p_2(x,\xi) = Q_x(\xi)$, where $Q_x$ is a smooth family of real quadratic forms, such that $Q_x(0, \xi_b)$ is definite on $\R^{n_b}$. This is the case of the wave operator or Schr\"odinger type operators. First, condition (E) is fulfilled thanks to the positiveness of $Q_x(0, \xi_b)$. Then, Assumption~\eqref{e:pseudo-surface-1} holds (uniformly with respect to $x \in \Omega$) according to the definiteness of $Q_x((0, \xi_b))$. It is indeed empty since $p_2(x,(0, \xi_b))$ does not vanish for $\xi_b \neq 0$.
Moreover, we have $\{p_2 , \phi\}(x,\xi) = 2 \tilde{Q}_x(\xi , \nabla \phi)$, where  $\tilde{Q}_x$ is the polar form of $Q_x$, and 
$$\{p_2 , \phi\}(x,\xi+i\nabla \phi) = 2 \tilde{Q}_x(\xi , \nabla \phi) + 2i Q_x(\nabla \phi).$$
As a consequence, $\Im\{p_2 , \phi\}(x,\xi+i\nabla \phi) = 2Q_x(\nabla \phi)$ so that \eqref{e:pseudo-surface-2} is also empty (and thus satisfied) for any noncharacteristic hypersurface. 

In conclusion, for real quadratic forms which are definite on $\R^{n_b}$ at $\xi_a = 0$, any noncharacteristic hypersurface is strongly pseudoconvex in the sense of Definition~\ref{def: pseudoconvex-surface}. In the case $n_a = 1$, this includes the following operators of particular interest:
\begin{itemize}
\item $P = D_{x_a}^2 - \sum_{i,j=1}^{n-1}\alpha_{ij}(x)D_{x_b^j}D_{x_b^i}$ (wave operator) with $p = \xi_a^2 - \sum_{i,j=1}^{n-1}\alpha_{ij}(x)\xi_b^j \xi_b^i$;
\item $P = D_{x_a} - \sum_{i,j=1}^{n-1}\alpha_{ij}(x)D_{x_b^j}D_{x_b^i}$ (Schr\"odinger operator) with $p =  - \sum_{i,j=1}^{n-1}\alpha_{ij}(x)\xi_b^j \xi_b^i$.
\end{itemize}
where the quadratic form with coefficients $\alpha_{i,j}$ is positive definite.
\end{remark}

\bigskip
We are now prepared to formulate our main Theorem in the general framework.
We first describe the geometric context and then state the Theorem.

\bigskip

\textbf{Geometric setting:} (see Figure~\ref{f:geom-setting})
We first fix two splittings of $\R^n$ as $\R^n=\R_{x'}^{n-1}\times \R_{x_n}$ and $\R^n=\R^{n_a}_{x_a}\times \R^{n_b}_{x_b}$, possibly in two different basis.
We let $D$ be a bounded open subset of $\R^{n-1}$ with smooth boundary and $G = G(x',\eps) \in C^1(\overline{D} \times [0,1+\eta) )$, for some $\eta>0$, such that  
\begin{itemize}
\item For all $\eps \in (0,1]$, we have $\{x' \in \R^{n-1} , G(x' , \eps) \geq 0\} = \overline{D}$;
\item for all $x' \in D$, the function $\eps \mapsto G(x' , \eps)$ is strictly increasing;
\item for all $\eps \in (0,1]$, we have $\{x' \in \R^{n-1} , G(x' , \eps) = 0\} = \partial D$.
\end{itemize}
We set $G(x', 0) = 0$, $S_0 = \overline{D} \times \{0\}$ and,  for $\eps \in (0,1]$, 
\bna
&S_\eps = \{(x',x_n) \in \R^{n} , x_n \geq 0 \text{ and } G(x' , \eps) =x_n \} =( \overline{D} \times \R )\cap\{(x',x_n) \in \R^{n} , G(x' , \eps) =x_n \} ;\\
&K = \{x \in \R^n ,0\leq  x_n \leq G(x', 1)\}.
\ena

\begin{figure}[h!]
  \begin{center}
    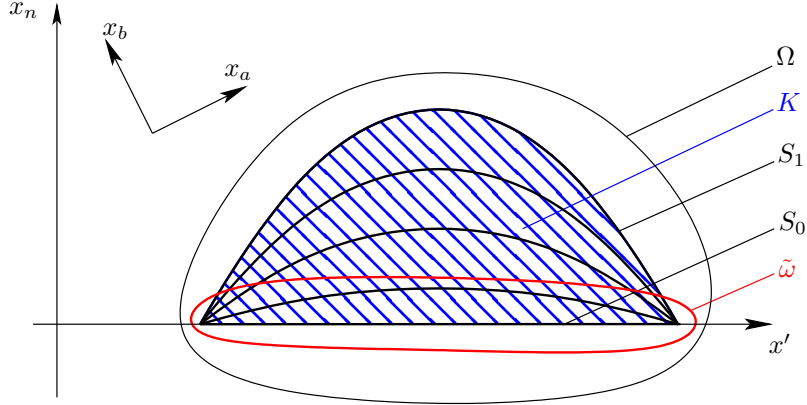 
    \caption{Geometric setting of Theorem~\ref{thmsemiglobal}}
    \label{f:geom-setting}
 \end{center}
\end{figure}

\begin{theorem}
\label{thmsemiglobal}
In the above geometric setting, we moreover let $\Omega$ be a neighborhood of $K$, and $P$ be a differential operator of order $m$, analytically principally normal operator on $\Omega$ in $\{\xi_a = 0\}$. 

Assume also that, for any $\e\in [0,1+\eta)$, the oriented surfaces $S_{\e}=\left\{\phi_\eps=0\right\}$ with $\phi_\eps(x',x_n): =  G(x' , \eps) - x_n $ are strictly pseudoconvex in $\left\{\xi_a=0\right\}$ for $P$ on the whole $S_\eps$, in the sense of Definition \ref{def: pseudoconvex-surface}.

Then, for any open neighborhood $\tilde{\omega} \subset \Omega$ of $S_0$, there exists a neighborhood $U$ of $K$, and constants $\kappa ,C ,\mu_0 >0$ such that for all $\mu\geq \mu_0$ and $u\in C^{\infty}_0(\R^n)$,  we have
$$
\nor{ u}{L^2(U)}\leq C e^{\kappa \mu}\left(\nor{ u}{H^{m-1}_b(\tilde{\omega})} + \nor{Pu}{L^2(\Omega)}\right)+\frac{C}{\mu^{m-1}}\nor{u}{H^{m-1}(\Omega)} ,
$$
where we have denoted $\nor{u}{H^{m-1}_b(\tilde{\omega})}=\sum_{|\beta| \leq m-1}\nor{ D_b^\beta u}{L^2(\tilde\omega)}$.

If $n_a=n$ (Holmgren case), we get also for some $\widetilde{\varphi}\in C^{\infty}_0(\tilde{\omega})$ and for any $s\in \R$, the existence of $\kappa ,C ,\mu_0 >0$ such that for all $\mu\geq \mu_0$ and $u\in C^{\infty}_0(\R^n)$,  we have
$$
\nor{ u}{L^2(U)}\leq C e^{\kappa \mu}\left(\nor{\widetilde{\varphi} u}{H^{-s}(\R^n)} + \nor{Pu}{L^2(\Omega)}\right)+\frac{C}{\mu^{m-1}}\nor{u}{H^{m-1}(\Omega)}.
$$
If $n_a=0$ (H\"ormander case), there is $c, \kappa ,C ,\mu_0 >0$ such that for all $\mu\geq \mu_0$ and $u\in C^{\infty}_0(\R^n)$,  we have

$$
\nor{ u}{H^{m-1}(U)}\leq C e^{\kappa \mu}\left(\nor{ u}{H^{m-1}(\tilde{\omega})} + \nor{Pu}{L^2(\Omega)}\right)+Ce^{-c\mu}\nor{u}{H^{m-1}(\Omega)}
$$
\end{theorem}

Note that in the first two cases, we obtain a result of the type~\eqref{e:QUCP-phi} with a logarithmic function $\varphi$, whereas in the framework of the H\"ormander theorem, we obtain the stronger H\"older-type dependence:
$$
\nor{ u}{H^{m-1}(U)}\leq C \left(\nor{ u}{H^{m-1}(\tilde{\omega})} + \nor{Pu}{L^2(\Omega)}\right)^{\delta}\nor{u}{H^{m-1}(\Omega)}^{1-\delta}
$$
for some $\delta\in (0,1)$.

The formulation of the above result using a foliation by hypersurfaces is inspired by that of~\cite[Theorem p.~224]{John:49} in the context of the Holmgren theorem. The statement describing the hypersurfaces by graph could look rigid. We will give later in Theorem \ref{thmsemiglobaldepchgt} a slight variant where the partial analyticity and the foliation by graphs can be described in different coordinates (i.e. the linear change of coordinates between the two different splittings $\R^n=\R_{x'}^{n-1}\times \R_{x_n}$ and $\R^n=\R^{n_a}\times \R^{n_b}$ may be replaced by a diffeomorphism). We chose not to present this more general result here for the sake of the exposition. Most of global Theorems for the wave and Schr\"odinger equations on a manifold are proved in that setting, after some suitable change of coordinates.

\subsection{Idea of the proof}

As already mentioned, unique continuation theorems (e.g. the H\"ormander theorem) are often proved with Carleman estimates. Such inequalities are already quantitative, and hence furnish a good starting point towards local quantitative unique continuation results. This strategy has already been followed in~\cite{Robbiano:95, LR:95} in the case of elliptic operators, see also \cite{Bahouri:87}.
Starting from the Carleman inequality, the idea is to apply the estimates to some function $\chi(x) u$ where $\chi$ is a well chosen cutoff function. The exponential weight $e^{\tau \psi(x)}$ (where $\psi$ is an appropriate weight function) in the Carleman estimate naturally leads to some inequality of the form
\bnan
\label{estimRobclass}
\nor{u}{V_2}\leq e^{\kappa \mu} \big(\nor{u}{V_1} + \nor{Pu}{V_3} \big)+e^{-\kappa \mu}\nor{u}{V_3},
\enan
uniformly for $\mu\geq \mu_0$ and for some small open sets $V_1\subset V_2 \subset V_3$ depending on the local geometry. Optimizing in $\mu$ (see~\cite{Robbiano:95} or~\cite[Lemma~5.2]{LeLe:09}) this can then be written as an interpolation estimate 
\bna
\nor{u}{V_2}\leq \big(\nor{u}{V_1} + \nor{Pu}{V_3} \big)^{\delta}\nor{u}{V_3}^{1-\delta},
\ena 
for some $\delta\in (0,1)$. The interest of these interpolation estimates is that they can be easily iterated, leading to some global ones. It ends up with some H\"older type dependence, i.e.~\eqref{e:QUCP-phi} with $\varphi = (a+b)^\delta c^{1-\delta}$. We refer for instance to the survey article \cite{LeLe:09} for a description of these estimates in the elliptic case, with application to spectral estimates and control results for the heat equation.

\bigskip
Yet, in the context of the unique continuation theorem for partially analytic operators, the Carleman estimates proved in~\cite{Tataru:95, RZ:98, Hor:97, Tataru:99} contain a "microlocal" weight of the form $e^{-\frac{\e}{2\tau}|D_a|^2}e^{\tau\psi(x)}$. As for usual Carleman estimates, the term $e^{\tau\psi(x)}$ (loosely speaking) gives some strength to the set where $\psi$ is positive, but the additional term $e^{-\frac{\e}{2\tau}|D_a|^2}$ localizes in the low frequencies in the variable $x_a$. In this context, the proof of unique continuation proceeds with a (qualitative) complex analytic argument (maximum principle). 
This additional argument in the proof of unique continuation also requires to be quantified.  
As in~\cite{Robbiano:95}, this procedure naturally leads to local logarithmic (instead of H\"older) stability estimates. The main issue one then has to face  when quantifying unique continuation is that such estimate cannot be iterated (or would yield dependence estimates of the type~\eqref{e:QUCP-phi} with a function $\varphi$ being a composition of as many $\log$ as steps needed in the iteration).

\bigskip
One idea to overcome this difficulty, proposed by Tataru in his unpublished notes \cite{Tatarunotes}, was to propagate some low frequency estimates of the form
\bna
\left\{\begin{array}{rcl}
\nor{u}{H^{m-1}}&=&1\\
\nor{m\left(\frac{D_a}{\mu}\right)\sigma(\frac{x}{R})Pu}{L^2}&\leq& e^{-\mu^{\alpha }}
\end{array}\right. \Longrightarrow \nor{m\left(\frac{D_a}{\tau}\right)\sigma(\frac{x}{r})u}{H^{m-1}}\leq e^{-\tau }, \quad \forall \tau<c\ \mu^{\alpha}
\ena
and for all $u$ supported in $\left\{\phi<\phi(x_0)\right\} $, for some apropriate compactly supported cutoff functions $\sigma$ and $m(\xi)$ in the Gevrey class $1/\alpha$, $\alpha <1$, and for some $r<R$. This kind of estimates can be propagated and led to some global stability estimates of the form~\eqref{e:QUCP-phi} with $\varphi_\eps = c \left(\log(1+\frac{c}{a+b})\right)^{-(1-\eps)}$.

The loss $1-\e$ in the power of $\log$ is due to the use of functions Gevrey $\alpha$ with compact support. The optimal case $\alpha=1$ would correspond to analytic functions. Yet, analytic functions cannot have compact support, which is a key ingredient in the usual application of Carleman estimates.

Let us now explain our strategy to solve this problem.
\subsubsection{Obtaining local information at low frequency}
Part of the proof of the present paper is inspired by this idea of propagating only low frequency (in the analytic variable $x_a$) estimates. However, we replace the Gevrey cutoff functions by some analytic ``almost'' localized functions of the form $\chi_{\lambda}:=e^{-\frac{|D_a|^2}{\lambda}} \chi$ where $\chi$ is smooth with the expected compact support. It turns out that the right choice of $\lambda$ is $\lambda=C\mu$ where $\mu$ is the frequency where we want to measure our solution. That such functions are not compactly supported makes the commutator estimates much more intricate and requires a careful study of the dependence with respect the  regularisation parameter $\lambda$, the local frequency $\mu$ and the parameter $\tau$ in the Carleman estimate. All estimates are carried out up to an exponentially small remainder (in terms of these parameters). 

Following this procedure, the local estimates we prove (which we are in addition able to propagate) are some generalization of \eqref{estimRobclass}, but only with regards to the low frequencies (in the analytic variable $x_a$). In a neighborhood of a point $x^0$, they are of the form
\bnan
\label{estimpropintro}
\nor{m_\mu \left(\frac{D_a}{\beta\mu}\right) \chi_{2,\mu} u}{H^{m-1}}
\leq C e^{\kappa \mu}\left( \nor{m_\mu \left(\frac{D_a}{\mu} \right) \chi_{1,\mu} u}{H^{m-1}} + \nor{Pu}{L^2(B(x^0,R))}\right)+Ce^{-\kappa' \mu}\nor{u}{H^{m-1}} ,
\enan
uniformly for $\mu\geq \mu_0$. See the beginning of Section \ref{sectionlocal} for a more precise statement and remarks on this result. Here, $\chi_1$ and $\chi_2$ are some cutoff in the physical space that localize respectively to the place where the information is taken (locally in $\left\{\phi>\rho\right\}$ for some $\rho>0$) and to where it is propagated (a small neighborhood of $x^0$). The Fourier multipliers $m_\mu$ cuts off (analytically) the $\xi_a$ frequencies. All these cutoff functions are used only with their analytic regularization. They never localize exactly. Using such regularized cutoff functions and Fourier multipliers follows the spirit of analytic semiclassical analysis~\cite{Sjostrand:95} (see also~\cite{MartinezBook}). However, we do not make use of that theory and rather construct by hand the appropriate mollifiers, making the proof selfcontained in this respect.

The proof of estimates like \eqref{estimpropintro}, stated more precisely in Theorem \ref{th:alpha-unif} is the object of Section \ref{sectionlocal}. It proceeds in three steps. First, as in the usual proofs of unique continuation results, starting from the hypersurface $\{ \phi = 0\}$, one needs to construct a weight function $\psi$ with both properties 
\begin{itemize}
\item to satisfy the assumptions required to apply the Carleman estimate ($\psi$ should be a strictly pseudoconvex {\em function} in the sense of Definition~\ref{def: pseudoconvex-function});
\item to have level sets appropriately located with respect to those of $\phi$.
\end{itemize}
This corresponds to the so called ``convexification process''.  

Second, we apply as a black box the Carleman estimates of~\cite{Tataru:95, RZ:98, Hor:97, Tataru:99} (or some similar ones that we prove in the presence of boundary) to $\chi u$, where $\chi$ is a particular cutoff function (localizing near the point of interest, and according to levelsets of $\psi$), containing both rough cutoffs and mollified ones. We then need to estimate terms arising from the commutator $e^{-\frac{\e}{2\tau}|D_a|^2}e^{\tau\psi}[P,\chi]$, that are either well localized or have an exponentially small contribution. 

Finally, we need to transfer the information given by the Carleman estimate to some estimate like~\eqref{estimpropintro} on the low frequencies of the function. This is done through a complex analysis argument, the Carleman parameter $\tau$ playing the role of complex variable, as in ~\cite{Tataru:95}. If $\zeta$ is the complex variable, the Carleman estimates corresponds to an estimate on $\zeta = i \tau \in i\R^+$. Combined with {\em a priori} estimates, a Phragm\'en-Lindel\"of type theorem allows to extend this estimate to part of the real domain, where it corresponds to estimating $\nor{m \left(\frac{D_a}{\beta\mu}\right) \chi u}{}$. To obtain estimates that are uniform with respect to the frequency (and regularization) parameter $\mu$, we also need, following~\cite{Tatarunotes}, to use a scaling argument, replacing $\tau$ by $\tau/\mu$.
\subsubsection{Propagating local informations to global ones}
Once the local estimate are proved, we need to iterate them to obtain a global estimate. This is the object of Section \ref{s:semiglobal-estimate}. At first, we define some tools that will allow later in an abstract way to propagate easily our local estimate \eqref{estimpropintro}.  Roughly speaking, \eqref{estimpropintro} says that, for solution of $Pu=0$, some information can be transfered from the support of $\chi_1$ to the support of $\chi_2$. We formalize that with the notion of {\em zone of dependence}. Roughly speaking, we say that on open set $O_2$ depends on $O_1$ if \eqref{estimpropintro} holds for every $\chi_1$ equals to $1$ on $O_1$ and any $\chi_2$ supported in $O_2$. This part allows to make the proof of Theorem \ref{thmsemiglobal} a complete geometric one. Even if quite different in definition, it is close in spirit to the interpolation theory developped in Lebeau \cite{Leb:Analytic} to propagate globally the local information obtained by the Cauchy-Kowaleski theorem.
Moreover, it should adapt to some more general kind of foliation.  
Note that at each step of this propagation argument, we have a loss in the the range of frequency: from an information on frequencies $\leq \mu$, we obtain an information on frequencies $\leq \beta \mu$, with $\beta$ small. This is overcome by the fact that we only have a finite number of steps in this iterative procedure.

Once this propagation result is done, we are left with some information about the low frequency of our solution. Since we have no information about the high frequency part, the only thing to do is to use some trivial bound of the type
\bna
\nor{\left(1-m\left(\frac{D_a}{\mu}\right)\right)u}{L^2}\leq \frac{C}{\mu^{m-1}}\nor{u}{m-1}
\ena
This is actually much worse than the negative exponential that we already had. But it turns out to be the best we can do without any more information.

In section \ref{Application}, we specify our general result to the case of the wave and Schr\"odinger equations. The main task is to construct some noncharacteristic hypersurfaces that allow to be in the situation of Theorem \ref{thmsemiglobal}. This part is quite classical and was already present for instance in \cite{Leb:Analytic}. We recall the argument in the present context.

\subsubsection{Carleman estimates for the Dirichlet boundary value problem}
Finally, to prove the results of Section~\ref{sec:intro-wave-schrod}, it remains to deal with the boundary-value problem. This is the object of Section~\ref{s:Dirichlet-pb-waves}. 
As far as (qualitative) unique continuation is concerned, there is no need to prove quantitative estimates up to the boundary. As a consequence, we need here to carry over the analysis of~\cite{Tataru:95,RZ:98, Hor:97, Tataru:99} at the boundary. In this context, we consider a particular class of operators and a particular boundary condition. We assume that the operator belongs to the class described in Remark \ref{rknoncaractwave} (hence encompassing wave and Schr\"odinger type operators), that is, with symbols of the form $p_2(x,\xi) = Q_x(\xi)$ where $Q_x$ is a smooth family of real quadratic forms. 
We further assume that the analytic variables $x_a$ are tangent to the boundary, and that the functions satisfy {\em Dirichlet boundary conditions}. Recall that this situation is of particular interest for the wave/Schr\"odinger equations, for which $x_a$ is the time variable, which is always tangent to the boundary of cylindrical domains. 
 
The proof of the quantitative unique continuation result up to and from the boundary relies on a Carleman estimate at the boundary for such operators. As such, it interpolates between the ``boundary elliptic Carleman estimates'' of Lebeau and Robbiano~\cite{LR:95}, and the ``partially analytic Carleman estimates'' of Tataru~\cite{Tataru:95} (see also~\cite{RZ:98, Hor:97}). 
 Then, we obtain the counterpart of the local estimate of Theorem~\ref{th:alpha-unif} for this boundary value problem. All local, semiglobal and global results shall then follow as in the boundaryless case. We only need to be careful when performing changes of variables.

\bigskip
\noindent

We wish to thank Daniel~Tataru for having allowed us to use some ideas from his unpublished lecture notes~\cite{Tatarunotes}, and Luc Robbiano for his comments on a preliminary version of the paper.
The first author is partially supported by the Agence Nationale de la Recherche under grant EMAQS ANR-2011-BS01-017-0 and IPROBLEMS ANR-13-JS01-0006.
The second author is partially supported by the Agence Nationale de la Recherche under grant GERASIC ANR-13-BS01-0007-01.

When finalizing this article it came to our attention that another group, Roberta~Bosi, Yaroslav~Kurylev and Matti~Lassas has been working independently on issues related to this paper.

\section{Preliminaries}
The preliminary results presented in this section are mainly used in Section~\ref{sectionlocal} for the local estimate. Some are also used independently in Section~\ref{s:semiglobal-estimate} for the semiglobal estimate. They concern:
\begin{enumerate}
\item The Carleman estimate adapted to operators with partially analytic coefficients, as stated in~\cite{Tataru:95,RZ:98, Hor:97, Tataru:99};
\item The regularization procedure for cutoff functions and Fourier multipliers (which is a key part in the proofs);
\item Some preliminary commutator-type estimates.
\end{enumerate}

\subsection{Notation}
Before this, let us recall basic notation, used all along the article. 

Above and below, $\dist$ stands for the Euclidean distance in $\R^n$, $\R$ or $\R^{n_a}$, or the Riemannian distance on $(\M ,g)$.
For $K \subset \R^n$ (resp. $\R$, resp. $\R^{n_a}$) we define a $d$ neighborhood of $K$ by 
$$
\vois(K,d): = \bigcup_{x \in K} B(x, d) ,
$$
where balls are taken according to the distance $\dist$.
For two open set $U ,U'$, we write $U \Subset U'$ if $\bar U$ is compact and $\bar U \subset U'$.

We denote by $\mathcal{F}$ the Fourier transform in all variables, $\mathcal{F}_a$ in the variables $x_a \in \R^{n_a}$ only. When there is no possible confusion, we shall write $\hat{u} = \mathcal{F}_a(u)$ or $\hat{u} = \mathcal{F}(u)$.

We set  $\left< \xi \right> = (1 + |\xi|^2)^\frac12$, and denote by $\nor{\cdot}{m}$ the classical $H^m$ norm on $\R^n$: $\nor{u}{m} := \| \left< \xi \right>^m \mathcal{F} (u) \|_{L^2(\R^n)}$. Similarly, $$\nor{u}{m,\tau}=\nor{\left( \tau^2+|D|^2\right)^\frac{m}{2} u}{0} = \nor{(\tau^2 + |\xi|^2)^\frac{m}{2}\mathcal{F} (u) }{0}$$ will denote the weighted (semiclassical) $H^m$ norm for $\tau \geq 1$. In the main part of this article, $\tau$ will be a large parameter.
Finally, we use the notation $\nor{\cdot}{H^k \to H^\ell}$ for the operator norm from $H^k(\R^n)$ to $H^\ell(\R^n)$.
\subsection{The Carleman estimate}
Before stating the Carleman estimate used in the main part of paper, we need to introduce the definition of appropriate weight functions $\psi$.

\begin{definition}[Strongly pseudoconvex function]
\label{def: pseudoconvex-function}
Let $P$ be a principally normal operator in $\Omega \subset \R^n$, with principal symbol $p$, let $\psi \in C^2(\Omega; \R)$ and $\Gamma$ be a closed conic subset of $T^*\Omega$. Let $x^0 \in \Omega$. We say that $\psi$ is strongly pseudoconvex in $\Gamma$ at $x^0$ for $P$ if:
\bnan
\Re \left\{ \overline{p}, \{p ,\psi \} \right\} (x^0 ,\xi) >0  , 
&\text{ if } p(x^0,\xi) = 0 \text{ and } \xi \in \Gamma_{x^0}, \xi \neq 0 ; \\
\frac{1}{i\tau}\{ \overline{p}_\psi,p_\psi \} (x^0 ,\xi) >0  , 
&\text{ if } p_\psi(x^0,\xi) = 0 \text{ and } \xi \in \Gamma_{x^0}, \tau >0 ,
\enan
where $p_\psi(x,\xi) = p(x, \xi + i \tau\nabla \psi)$.
\end{definition}
Note that in the case $\Gamma = T^*\Omega$, this property is the usual one for proving a Carleman estimate with the weight function $\psi$. It is classical that a strongly pseudoconvex surface $S$ (in the sense of Definition~\ref{def: pseudoconvex-surface}) is a level surface for some pseudoconvex function (see e.g.~\cite[Proposition~28.3.3]{Hoermander:V4} or~\cite[Theorem~1.5]{Tatarunotes}), and that both definitions are stable with respect to small $C^2$ perturbations. In what follows a more precise link (adapted to our needs) between these two notions shall be made in Section~\ref{s:geom-setting}. 

In this paper (as in~\cite{Tataru:95, RZ:98, Hor:97, Tataru:99}), Definitions~\ref{def: pseudoconvex-surface} and~\ref{def: pseudoconvex-function} shalls alway be used with $\Gamma = \Omega \times \{\xi_a = 0\}$.

\bigskip
For $\e$, $\tau>0$ we define the operator 
\bnan
\label{Qeps}
Q_{\e,\tau}^{\psi} u 
= Q_{\e,\tau}^{\psi}(x,D_a)u
= e^{-\frac{\e}{2\tau}|D_a|^2}(e^{\tau\psi}u)
\enan
introduced in~\cite{Tataru:95}.

The following result is~due to Tataru \cite[Theorem~2]{Tataru:99}. A proof in cases (E) and (H) can be found in~\cite{Hor:97} (see in this reference Equation~(5.15), and the last equation before Section~7, respectively). Some closely related estimates are also proved in \cite[Proposition~4.6]{RZ:98}.

In Section~\ref{s:Dirichlet-pb-waves}, when studying the boundary value problem for wave equations, we include a proof of this result in the case (H) assuming that $P$ has a real principal part, is of order $m=2$, and under the additional assumption that the coefficients of $P$ do not depend on $x_a$.

\begin{theorem}
\label{thmCarleman}
Let $x^0 \in \Omega = \Omega_a \times \Omega_b \subset \R^{n_a} \times \R^{n_b}$ and $P$ be a partial differential operator on $\Omega$ of order $m$. Assume that 
\begin{itemize}
\item $P$ is analytically principally normal operator in $\{\xi_a = 0\}$ inside $\Omega$ (in the sense of Definition~\ref{def:anal principal normal});
\item $\psi$ is a quadratic polynomial in $x=(x_a,x_b)$, strongly pseudoconvex in $\Omega \cap \{\xi_a = 0\}$ at $x^0$ for $P$ (in the sense of Definition~\ref{def: pseudoconvex-function}).
\end{itemize}
Then, there exists $\eps >0$, $\mathsf{R}>0$, $\mathsf{d}>0$, $C>0$, $\tau_0>0$ such that $B(x^0, \mathsf{R}) \subset \Omega$ and for any $\tau >\tau_0$, we have
\bnan
\label{Carleman}
\tau \nor{Q_{\e,\tau}^{\psi}u}{m-1,\tau}^2
\leq C\left(\nor{Q_{\e,\tau}^{\psi}P u}{0}^2
+ {  \nor{e^{\tau(\psi-\mathsf{d})}P u}{0}^2}
+ \nor{e^{\tau(\psi-\mathsf{d})}u}{m-1,\tau}^2\right)
\enan
for any $u\in C^{\infty}_0(B(x^0, \mathsf{R}))$.
\end{theorem}

Note that most Carleman estimates in \cite{Tataru:95, RZ:98, Hor:97, Tataru:99} do not contain the term \\$ \nor{e^{\tau(\psi-\mathsf{d})}P u}{0}^2$ in the right hand-side. Also, this result was stated in some case where pseudoconvexity holds on all $\Omega$. Yet, pseudoconvexity at one points implies the pseudoconvexity in a small neighborhood (see \cite[Lemma 2.5]{Tataru:99}), so it implies the local Carleman estimate for functions supported close to $x^0$.

\subsection{Regularization of cutoff functions and Fourier multipliers}

All along the paper, we shall use several cutoff functions and need to regularize them. Here, we explain the regularization procedure we use,  give some of its basic properties, and define some (appropriately regularized) Fourier multipliers.
\subsubsection{Regularization of functions}
Before describing the regularization operators, let us collect some basic facts about gaussian integrals.
Note first that we have (derive with respect to $z$ or see e.g. \cite[(2.1.7) p17]{Lebedev}), for $z \geq 0$, 
\bna
\int_z^{+ \infty} e^{-s^2} ds = \frac{e^{-z^2}}{\sqrt{\pi}} \int_0^{+ \infty} \frac{e^{-z^2s^2}}{1+s^2} ds  \leq \frac{\sqrt{\pi}}{2}e^{-z^2}.
\ena
As a consequence, we have the following estimates
\bna
\int_r^{+\infty}e^{-\frac{s^2}{t}}ds
\leq \frac{\sqrt{\pi}}{2} \sqrt{t} e^{-\frac{r^2}{t}} ,
\quad \int_{r}^{+ \infty} \langle s \rangle^m e^{-\frac{s^2}{t}} ds \leq C_m \langle r \rangle^m \langle t \rangle^\frac{m+1}{2} e^{-\frac{r^2}{t}}
\quad 
\text{ for all } r\geq 0 , t>0 , m\in \N ,
\ena
where the second estimate is obtained by iterated integration by parts. As a consequence, we also have
\bnan
\label{e:estimgausshorsboul}
\int_{x_a \in \R^{n_a},|x_a|\geq r} e^{-\frac{|x_a|^2}{t}} dx_a 
\leq C_{n_a}\left\langle r\right\rangle^{n_a-1} \langle t \rangle^\frac{n_a}{2}e^{-\frac{r^2}{t}} 
\text{ for all } r\geq 0 , t>0 .
\enan
Moreover, we have for any measurable set $E \subset \R^{n_a}$, any $x_a \in \R^{n_a}$, and any $t>0$,
\bna
\int_{E} e^{-\frac{1}{t} |x_a - y_a|^2}~dy_a  \leq \int_{\R^{n_a}} e^{-\frac{1}{t} |x_a - y_a|^2}~dy_a  = (\pi t)^{\frac{n_a}{2}} . 
\ena
In addition, according to~\eqref{e:estimgausshorsboul}, there exists $C_{n_a}>0$ such that for any closed set $E \subset \R^{n_a}$, any $x_a \notin E$, and any $t>0$, we have
\bna
\int_{E} e^{-\frac{1}{t} |x_a - y_a|^2}~dy_a  \leq \int_{B(x_a , \dist(x_a , E))^c} e^{-\frac{1}{t} |x_a - y_a|^2}~dy_a  \leq C_{n_a}\left\langle \dist(x_a , E) \right\rangle^{n_a-1} \langle t \rangle^\frac{n_a}{2}e^{-\frac{\dist(x_a , E)^2}{t}},
\ena
Hence there exists $C_{n_a}>0$ such that for any closed set $E \subset \R^{n_a}$, any $x_a \in \R^{n_a}$, and any $t>0$, we have
\bnan
\label{e:estimgaussdistE}
\int_{E} e^{-\frac{1}{t} |x_a - y_a|^2}~dy_a  \leq
 C_{n_a}\left\langle \dist(x_a , E) \right\rangle^{n_a-1} \langle t \rangle^\frac{n_a}{2}e^{-\frac{\dist(x_a , E)^2}{t}}.
\enan

\bigskip
We are now prepared to define the appropriate regularization process, used all along the article.
We shall use the notation $f_{\lambda}$ to denote
\begin{itemize}
\item $f_{\lambda}:= e^{-\frac{|D|^2}{\lambda}} f$ for a function $f \in  L^{\infty}(\R)$;
\item or (more often used) \bna
f_{\lambda}:=e^{-\frac{|D_a|^2}{\lambda}} f ,
\ena 
for a function $f \in  L^{\infty}(\R^{n})$, and {\em a fortiori} for $f \in L^{\infty}(\R^{n_a})$.
\end{itemize}
We hope that this use shall not be confusing for the reader. We now discuss in more detail the basic properties of this regularization process in the second case only (the first case can be seen as the particular situation $n_a = 1 , n_b = 0$).

This definition can be rewritten as 
\bna
f_{\lambda}(x_a , x_b)=\left(\frac{\lambda}{4\pi}\right)^{\frac{n_a}{2}}\Big( e^{-\frac{\lambda}{4}|\cdot|^2} *_{\R^{n_a}} f(\cdot , x_b) \Big)(x_a)
 = \left(\frac{\lambda}{4\pi}\right)^{\frac{n_a}{2}}\int_{\R^{n_a}} f\left(y_a, x_b\right)e^{-\frac{\lambda}{4} |x_a- y_a|^2}~dy_a .
\ena

Note that similar smoothing of functions are used systematically when working with analytic microlocal  analysis, see~\cite{Sjostrand:95} or~\cite{MartinezBook}. In this context, it is related to the Fourier-Bros-Iagolnitzer transform. In applications to unique continution, it has been used in 
\cite{RT:72,Lerner:88,Robbiano:91,Hormander:92,Leb:Analytic,Robbiano:95, Tataru:95,RZ:98,Hor:97,Tataru:99}. In particular, the operator $Q_{\e,\tau}^{\psi}$ defined in~\eqref{Qeps} contains such a regularization (the regularizing parameter $\lambda$ being linked to the Carleman large parameter $\tau$).

We will use several times in the proofs that
\bnan
\label{normLismoothing1}
\left\|f_{\lambda}\right\|_{L^{2}(\R^n)} 
 \leq  \|e^{-\frac{|\cdot|^2}{\lambda}}\|_{L^\infty(\R^{n_a})} \|\mathcal{F}_a(f)(\xi_a , x_b)\|_{L^{2}(\R^n)} 
= \left\|f\right\|_{L^{2}(\R^n)}
\enan
and 
\bnan
\label{normLismoothing2}
\left\|f_{\lambda}\right\|_{L^{\infty}}
\leq \left(\frac{\lambda}{4\pi}\right)^{\frac{n_a}{2}} \|e^{-\frac{\lambda}{4}|\cdot|^2}\|_{L^1(\R^{n_a})}
 \left\|f\right\|_{L^{\infty}(\R^n)}
 =   \left\|f\right\|_{L^{\infty}(\R^n)} .
\enan
Notice also that we have
\bna
f \geq 0 \Longrightarrow f_\lambda \geq 0 ,\quad \text{ and hence } \quad f \geq g \Longrightarrow f_\lambda \geq g_\lambda .
\ena
Moreover, the function $f_\lambda$ may be extended as an entire function in the variable $x_a$ by
\bna
f_{\lambda}(z_a , x_b)
 = \left(\frac{\lambda}{4\pi}\right)^{\frac{n_a}{2}}\int_{\R^{n_a}} f\left(y_a, x_b\right)e^{-\frac{\lambda}{4} (z_a- y_a)^2}~dy_a , \quad z_a \in \C^{n_a} , x_b \in \R^{n_b} ,
\ena
(where $\zeta_a^2 = \zeta_a\cdot\zeta_a = |\Re\zeta_a|^2 - |\Im\zeta_a|^2 + 2 i\Re\zeta_a \cdot \Im\zeta_a$ is the real inner product)
with the uniform bound
\bnan
\label{estimalambdacompl}
\left|f_{\lambda}(z_a , x_b)\right|& \leq &\left(\frac{\lambda}{4\pi}\right)^{\frac{n_a}{2}}
\nor{f}{L^\infty}
\int_{y_a\in \supp(f(\cdot , x_b))} \left| e^{-\frac{\lambda}{4} (z_a- y_a)^2}\right|~dy_a \nonumber\\
& \leq &\left(\frac{\lambda}{4\pi}\right)^{\frac{n_a}{2}}
\nor{f}{L^\infty} e^{\frac{\lambda}{4}|\Im(z_a)|^2} 
\int_{y_a\in \supp(f(\cdot , x_b))} e^{-\frac{\lambda}{4} |\Re(z_a)- y_a|^2}~dy_a 
\nonumber\\
& \leq & C \left\langle \lambda\right\rangle^{\frac{n_a}{2}}
\nor{f}{L^\infty} e^{\frac{\lambda}{4}|\Im(z_a)|^2} \nonumber\\
&& \times
\left\langle \dist(\Re(z_a) , \supp(f(\cdot , x_b)) ) \right\rangle^{n_a-1}
e^{-\frac{\lambda}{4} \dist(\Re(z_a) , \supp(f(\cdot , x_b)) )^2} 
\enan
where the last estimate comes from~\eqref{e:estimgaussdistE}. Note that strictly speaking, if $f$ is only in $L^{\infty}(\R^n)$, $\supp(f(\cdot , x_b))$ is not really well defined for every $x_b\in \R^{n_b}$. But $\supp f$ (in the distributional sense of support) is a well defined closed set and we can define for every $x_b\in \R^{n_b}$ the closed set of $\R^{n_a}$, $\left\{x_a\in\R^{n_a}\left|(x_a,x_b)\in \supp f \right.\right\}$ that is $\supp(f(\cdot , x_b))$ for continuous functions. We will not discuss more this subtlety and will continue to write some expressions similar to \eqref{estimalambdacompl}. The estimate then makes sense by taking an element of the class in $L^{\infty}$ that is zero outside of $\supp(f)$ and that is bounded by $\nor{f}{L^{\infty}}$. 

For functions compactly supported in the $x_a$ variable, we have the simpler estimate
\bnan
\label{estimalambdacompl-bis}
\left|f_{\lambda}(z_a , x_b)\right|\leq C \lambda^{\frac{n_a}{2}}
\nor{f}{L^\infty} | \supp(f(\cdot , x_b)) | e^{\frac{\lambda}{4}|\Im(z_a)|^2} 
e^{-\frac{\lambda}{4} \dist(\Re(z_a) , \supp(f(\cdot , x_b)) )^2}  .
\enan
\subsubsection{Fourier multipliers}

Finally, we also need to introduce frequency localization functions, i.e. appropriately smoothed Fourier multipliers. 
Let $m(\xi_a)$ be a smooth radial function (i.e. depending only on $|\xi_a|$), compactly supported (in $|\xi_a| < 1 $) such that $m(\xi_a)=1$ for $|\xi_a| < 3/4$. We shall denote by $M^{\mu}$ the Fourier multiplier $M^{\mu}u= m\left(\frac{D_a}{\mu}\right)u$, that is 
$$
(M^{\mu} u) (x_a , x_b)=  \mathcal{F}_a^{-1} \left( m \left(\frac{\xi_a}{\mu}\right) \mathcal{F}_a (u)(\xi_a , x_b)\right)(x_a) ,
$$
where $\mathcal{F}_a$ denotes the Fourier transform in the variable $x_a$ only. 
Given $\lambda , \mu >0$, we shall denote by $M^{\mu}_{\lambda}$  the Fourier multiplier of symbol $m^{\mu}_{\lambda}(\xi_a)=m_{\lambda}\left(\frac{\xi_a}{\mu}\right)$, i.e. $M^{\mu}_{\lambda}= m_{\lambda}\left(\frac{D_a}{\mu}\right)$ or
$$
(M^{\mu}_{\lambda} u) (x_a , x_b)=  \mathcal{F}_a^{-1} \left( m_{\lambda}\left(\frac{\xi_a}{\mu}\right) \mathcal{F}_a (u)(\xi_a , x_b)\right)(x_a) ,
$$
with, according to the above notation for the subscript $\lambda$, 
\bna
m_{\lambda}(\xi_a) 
 = \left(\frac{\lambda}{4\pi}\right)^{\frac{n_a}{2}}\int_{\R^{n_a}} m \left(\eta_a \right)e^{-\frac{\lambda}{4} |\xi_a- \eta_a|^2}~dy_a .
\ena

Note that in this definition, the symbol is first regularized and then dilated. We hope the notation (with the subscript for the regularization and the exponent for the dilation) will not be confusing for the reader. Note also that these Fourier multipliers only act in the variable $x_a$.

\subsection{Some preliminary estimates}
In this section, we state several technical lemmata of commutator type, needed to prove the main local result Theorem~\ref{th:alpha-unif}. The proofs can certainly be omitted by the hurried reader. The spirit is that all the estimates that we would expect for exact cutoff are true with their analytically regularized version, up to some term exponentially small in term of $\lambda$. So, the important fact in all the estimates is the uniformity with respect to $\lambda$ and $\mu$ as large parameter. 

\subsubsection{Some basic preliminary estimates}
\label{s:prelim-estim}

\begin{lemma}
\label{lmsuppdisjoint}
\begin{enumerate}
\item
For any $d>0$, there exist $C,c>0$ such that for any $f_1, f_2\in L^{\infty}(\R^n)$ such that $\dist(\supp(f_1),\supp(f_2))\geq d$ and all $\lambda\geq 0$, we have 
\bna
\nor{f_{1,\lambda}f_2 }{L^{\infty}} \leq C e^{-c \lambda}\nor{f_1}{L^{\infty}} \nor{f_2}{L^{\infty}} ,
\qquad 
\nor{f_{1,\lambda}f_{2,\lambda}}{L^{\infty}}\leq C e^{-c \lambda}\nor{f_1}{L^{\infty}} \nor{f_2}{L^{\infty}} .
\ena
\item If moreover $f_1, f_2\in C^{\infty}(\R^n)$ have bounded derivatives, then for all $k \in \N$, there exist $C,c>0$ such that for all $\lambda\geq 1$, we have
\bna
\nor{f_{1,\lambda}f_2 }{H^k(\R^n) \to H^k(\R^n) } \leq C e^{-c \lambda} .
\ena
\item 
Let $f_1, f_2\in L^{\infty}(\R^{n_a})$ such that $\dist(\supp(f_1),\supp(f_2))>0$ . Then there exist $C,c>0$ such that for all $\lambda\geq 1$, for all $k \in \N$, for all $\mu \geq 1$,  we have 
\bna
\nor{f_{1,\lambda}(D_a /\mu )f_2 (D_a/\mu)}{H^k(\R^n) \to H^k(\R^n) } \leq C e^{-c \lambda} , \\
\nor{f_{1,\lambda}(D_a /\mu )f_{2,\lambda} (D_a/\mu)}{H^k(\R^n) \to H^k(\R^n) } \leq C e^{-c \lambda} .
\ena
\end{enumerate}
\end{lemma}
\bnp
Let us set $d=\dist(\supp(f_1),\supp(f_2))>0$.
We have
\bna
\left|f_{1,\lambda}(x_a ,x_b)\right|\leq C\lambda^{n_a/2}\|f_1\|_{L^\infty}\int_{y_a\in \supp_{x_a}(f_1(\cdot , x_b))} e^{-\frac{\lambda |y_a-x_a|^2}{4}}~d y_a .
\ena
Moreover, for all $x_b \in \R^{n_b}$ we have 
$$
\dist_{\R^{n_a}} \big(\supp_{x_a}(f_1(\cdot , x_b)) , \supp_{x_a}(f_2(\cdot , x_b)) \big) \geq d ,
$$
so that for all $x = (x_a ,x_b)\in \supp(f_2)$, we have $|y_a-x_a|\geq d$ in the above integral. As a consequence, we obtain, for all $x = (x_a ,x_b)\in \supp(f_2)$,
\bna
\left|f_{1,\lambda}(x_a ,x_b)\right|
& \leq & C\lambda^{n_a/2}\|f_1\|_{L^\infty}\int_{|y_a-x_a|\geq d} e^{-\frac{\lambda |y_a-x_a|^2}{4}}~dy_a \leq C\|f_1\|_{L^\infty}\lambda^{n_a/2}\int_{|y_a| \geq d} e^{-\frac{\lambda |y_a|^2}{4}}~dy_a \\
& \leq & Ce^{-c\lambda}\|f_1\|_{L^\infty}, 
\ena
which provides the first estimate in item 1.

The second estimate is obtained by decomposing 
$$f_{1,\lambda}f_{2,\lambda}=f_{1,\lambda}f_{2,\lambda}\mathds{1}_{\vois(\supp(f_2),d/3)}+ f_{1,\lambda}f_{2,\lambda}\mathds{1}_{\vois(\supp(f_2),d/3)^c} , 
$$
and applying the previous result to the products $f_{1,\lambda}\mathds{1}_{\vois(\supp(f_2),d/3)}$ and $f_{2,\lambda}\mathds{1}_{\vois(\supp(f_2),d/3)^c}$, where all the supports are disjoint as required.

Item 2 is proved by induction on $k \in \N$. For $k=0$, it is precisely the first estimate of item 1. Now assume that it holds for $k-1$ and write $\|f_{1,\lambda}f_2 u\|_{H^k} \leq \|f_{1,\lambda}f_2 u\|_{H^{k-1}} + \|\nabla (f_{1,\lambda}f_2 u)\|_{H^{k-1}}$. It only remains to estimate $\|\nabla (f_{1,\lambda}f_2 u)\|_{H^{k-1}}$: for this, it sufficies to write
$$\nabla (f_{1,\lambda}f_2 u)= (\nabla f_{1})_{\lambda}f_2 u+f_{1,\lambda}\nabla (f_2) u)+f_{1,\lambda}f_2 \nabla (u), 
$$
where all functions have the appropriate support properties to apply the case $k-1$. This finally yields 
$\|\nabla (f_{1,\lambda}f_2 u)\|_{H^{k-1}} \leq Ce^{-c\lambda} \|u\|_{H^{k-1}} + Ce^{-c\lambda} \|\nabla u\|_{H^{k-1}}$ and concludes the proof of item 2.

The proof of item 3 only relies on the fact that for any $k \in \N$
$$
\nor{f_{1,\lambda}(D_a /\mu )f_2 (D_a/\mu)}{H^k(\R^n) \to H^k(\R^n) }
= \nor{f_{1,\lambda}(\xi_a /\mu )f_2 (\xi_a/\mu)}{L^{\infty}}
=\nor{f_{1,\lambda}f_2 }{L^{\infty}},
$$ (and similarly for the other term) and the use of item 1.
\enp
Similarly, we have
\begin{lemma}
\label{l:noyau-chaleur-hormander}
Let $f_2\in C^{\infty}(\R^n)$ with all derivatives bounded, and $d>0$. Then for every $k\in \N$, there exist $C,c>0$ such that for all $f_1 \in H^k(\R^n)$ such that $\dist(\supp(f_1),\supp(f_2))\geq d$ and all $\lambda\geq 0$, we have 
\bna
\nor{f_{1,\lambda}f_2 }{H^k} \leq C e^{-c \lambda}\nor{f_1}{H^k} .
\ena
\end{lemma}
\bnp
We have 
$$
f_{1,\lambda}f_2 (x_a,x_b)=  
\left(\frac{\lambda}{4\pi}\right)^{\frac{n_a}{2}}\int_{\R^{n_a}} f_2(x_a,x_b)f_1\left(y_a, x_b\right)e^{-\frac{\lambda}{4} |x_a- y_a|^2}~dy_a 
$$
so that 
\bna
|f_{1,\lambda}f_2| (x_a,x_b)
& \leq &  
\left(\frac{\lambda}{4\pi}\right)^{\frac{n_a}{2}}\int_{|x_a - y_a|\geq d} |f_2(x_a,x_b)f_1\left(y_a, x_b\right)|e^{-\frac{\lambda}{4} |x_a- y_a|^2}~dy_a \\
& \leq &  \|f_2\|_{L^\infty(\R^n)}
\left(\frac{\lambda}{4\pi}\right)^{\frac{n_a}{2}}\Big( \mathds{1}_{|\cdot|\geq d}e^{-\frac{\lambda}{4}|\cdot|^2} *_{\R^{n_a}} |f_1|(\cdot , x_b) \Big)(x_a) .
\ena
As a consequence, using the Young inequality, we have 
\bna
\|f_{1,\lambda}f_2\|_{L^2} \leq  \|f_2\|_{L^\infty(\R^n)}\left(\frac{\lambda}{4\pi}\right)^{\frac{n_a}{2}}\nor{ \mathds{1}_{|\cdot|\geq d}e^{-\frac{\lambda}{4}|\cdot|^2}}{L^1(\R^{n_a})} \|f_1\|_{L^2(\R^n)} ,
\ena
and, using~\eqref{e:estimgausshorsboul}, we obtain
\bna
\|f_{1,\lambda}f_2\|_{L^2} \leq Ce^{-d \lambda} \|f_2\|_{L^\infty(\R^n)}\|f_1\|_{L^2(\R^n)} ,
\ena
 which implies the result in the case $k=0$. We obtain the case $k>0$ by differentiating and applying the same result (see e.g. the proof of Lemma~\ref{lmsuppdisjoint}).
\enp

\begin{lemma}
\label{lmsuppdisjointchi}
Let $\psi : \R^n \to \R$ be a $C^\infty$ function, $f_1\in C^{\infty}(\R)$ with bounded derivatives and $f_2\in C^{\infty}_0(\R^n)$ such that $\dist(\supp(f_1 \circ \psi),\supp(f_2))>0$ . Then, for all $k \in \N$, there exist $C,c>0$ such that for all $\lambda> 0$, we have 
\bna
\nor{f_{1,\lambda}(\psi)f_2 }{H^k(\R^n) \to H^k(\R^n) } \leq C e^{-c \lambda}
\ena
\end{lemma}
\bnp
We prove the estimate $\nor{f_{1,\lambda}(\psi)f_2 }{L^{\infty}(\R^n)}\leq C e^{-c \lambda}$ which implies the result in the case $k=0$. We obtain the case $k>0$ by differentiating and applying the same result (see e.g. the proof of Lemma~\ref{lmsuppdisjoint}).

Since $f_2\in C^{\infty}_0(\R^n)$, the set $K:=\psi(\supp (f_2))=\left\{\psi(x);x\in \supp (f_2) \right\}$ is a compact set of $\R$. Moreover, the assumption $\dist(\supp(f_1(\psi)),\supp(f_2))>0$ implies that $\dist(\supp(f_1),K)>0$. Indeed, otherwise, we would have $\supp(f_1) \cap \psi(\supp (f_2)) \neq \emptyset$: taking $t$ in this intersection, there would be $x \in \supp(f_2)$ such that $\psi(x) = t \in \supp(f_1)$, i.e. $x \in \supp(f_1(\psi))$, which contradicts the assumption.
Now, note that $x \in \supp(f_2)$ implies that $\psi(x) \in K$, so that we have the pointwise estimate $|f_2|\leq \|f_2\|_{L^\infty} \mathds{1}_{K}\circ \psi$ on $\R^n$. As a consequence, we have 
\bna
\nor{f_{1,\lambda}(\psi)f_2 }{L^{\infty}(\R^n)} \leq C\nor{f_{1,\lambda}(\psi)\mathds{1}_{K}(\psi) }{L^{\infty}(\R^n)} \leq C \nor{f_{1,\lambda}\mathds{1}_{K} }{L^{\infty}(\R)}\leq C e^{-c \lambda} ,
\ena
where we have used Lemma \ref{lmsuppdisjoint} together with $\dist(\supp(f_1),K)>0$.
\enp
\begin{lemma}
\label{lmsupHK}
Let $f_1, f_2 \in C^{\infty}_0(\R^n)$ such that $f_1=1$ in a neighborhood of $\supp(f_2)$. Then for all $k \in \N$ there exist $C,c>0$ such that for all $\lambda>0$, and all $u \in H^k(\R^n)$, we have
\bna
 \nor{f_{2,\lambda} \d^\alpha u }{0} \leq &C \nor{f_{1,\lambda} u }{k}  + C  e^{-c \lambda}\nor{ u }{k} ,& \quad \text{for all $\alpha$ such that $|\alpha|\leq k$}; \\
 \nor{f_{2,\lambda} u }{k} \leq &C \nor{f_{1,\lambda} u }{k}  + C e^{-c \lambda}\nor{ u }{k} .&
\ena
\end{lemma}
\bnp
Let $d=\dist(\supp(f_2),\supp(1-f_1))>0$. 
Thanks to the first item of Lemma \ref{lmsuppdisjoint}, we have
\bna
\nor{f_{2,\lambda} \mathds{1}_{\vois(\supp(f_2),d/3)^c}\d^\alpha u }{0}\leq Ce^{-c\lambda}\nor{u}{k} .
\ena
Concerning the other term, we use again Lemma \ref{lmsuppdisjoint} applied to $\mathds{1}_{\vois(\supp(f_2),d/3)}$ and some $\d^\alpha (1-f_1)$ (using $\d^\alpha (f_{1,\lambda}) =(\d^\alpha f_{1})_{\lambda}$), to obtain 
\bna
\nor{f_{2,\lambda} \mathds{1}_{\vois(\supp(f_2),d/3)}\d^\alpha u }{0}&\leq& \nor{f_{2,\lambda} \mathds{1}_{\vois(\supp(f_2),d/3)}\d^\alpha (f_{1,\lambda}u) }{0}\\
&&  +\nor{f_{2,\lambda} \mathds{1}_{\vois(\supp(f_2),d/3)}\d^\alpha ((1-f_{1,\lambda})u) }{0}\\
&\leq&\nor{f_{2,\lambda}\mathds{1}_{\vois(\supp(f_2),d/3)}\d^\alpha (f_{1,\lambda} u) }{0}+Ce^{-c\lambda}\nor{u}{k} .
\ena
Writing then 
\bna
\nor{f_{2,\lambda} \mathds{1}_{\vois(\supp(f_2),d/3)}\d^\alpha (f_{1,\lambda}u) }{0}\leq C\nor{\d^\alpha (f_{1,\lambda} u )}{0}\leq C\nor{f_{1,\lambda} u }{k}
\ena
concludes the proof of the first estimate of the Lemma.

\medskip
The second inequality follows from noticing that $\d^\alpha (f_{2,\lambda}u)$ is a sum of terms of the form $(\d^\beta f_2)_{\lambda}\d^{\alpha-\beta} u$ for which we can apply the first part of the Lemma.
\enp

\begin{lemma}
\label{lmFL1}
Assume $m_1, m_2\in L^{\infty}(\R^{n_a})$ are bounded by $1$, and satisfy $\dist(\supp (m_1),\supp(m_2))\geq d>0$. Then, there exists $C>0$ such that for all $f\in L^\infty(\R^{n_b};L^{\infty}(\R^{n_a}))$ satisfying $\mathcal{F}_a(f)\in L^\infty(\R^{n_b};L^1(\R^{n_a}))$ and all $\mu , \lambda >0$, we have
\bna
\nor{m_{1,\lambda}(D_a/\mu)f(x)m_{2,\lambda}(D_a/\mu)}{L^2(\R^n)\rightarrow L^2(\R^n)}
\leq \nor{\mathcal{F}_a(f)}{L^\infty_{x_b}L^1(|\xi_a|\geq d \mu/3)}+Ce^{-c\lambda}\nor{\mathcal{F}_a(f)}{L^\infty(\R^{n_b};L^1(\R^{n_a}))}
\ena
and 
\bna
\nor{m_{1,\lambda}(D_a/\mu)f(x)m_{2}(D_a/\mu)}{L^2(\R^n)\rightarrow L^2(\R^n)}
\leq \nor{\mathcal{F}_a(f)}{L^\infty_{x_b}L^1(|\xi_a|\geq d \mu/3)}+Ce^{-c\lambda}\nor{\mathcal{F}_a(f)}{L^\infty(\R^{n_b};L^1(\R^{n_a}))}\ena
\end{lemma}
\bnp
We begin with the first estimate, the second one being simpler to handle.
We denote $m^{\mu}_{j,\lambda}(\xi_a) = m_{j,\lambda}(\xi_a/\mu)$ for $j=1,2$, and, to lighten the notation, set $\hat f = \mathcal{F}_a(f)$ (in the proof only).
We set $f_L = \mathds{1}_{|D_a| \leq d \mu/3}f$ (that is $\widehat{f_L}(\xi)=\mathds{1}_{|\xi_a| \leq d \mu/3}\widehat{f}(\xi) $) and  $f_H= \mathds{1}_{|D_a| \geq d \mu/3}f$.
We first have
\bna
\nor{m^{\mu}_{1,\lambda}(D_a)f_H(x)m^{\mu}_{2,\lambda}(D_a)}{L^2\rightarrow L^2}\leq \nor{f_H}{L^\infty(\R^n)} \leq \nor{\hat{f}_H}{L^\infty(\R^{n_b};L^1(\R^{n_a}))}
\leq \nor{\hat{f}}{L^\infty_{x_b}L^1(|\xi_a|\geq d \mu/3)} .
\ena
Then, it remains to estimate $\nor{m^{\mu}_{1,\lambda}(D_a)f_L(x)m^{\mu}_{2,\lambda}(D_a)}{L^2\rightarrow L^2}$. We work in the Fourier domain: for $u\in L^2(\R^n)$, we have
\bna
\mathcal{F}_a\big(m^{\mu}_{1,\lambda}(D_a)f_L(x)m^{\mu}_{2,\lambda}(D_a) u\big)(\xi_a, x_b)=m^{\mu}_{1,\lambda}(\xi_a)\left[\widehat{f_L}(\xi_a , x_b)*\left[m^{\mu}_{2,\lambda}(\xi_a) \hat{u}(\xi_a,x_b)\right]\right] ,
\ena
where $*$ denotes the convolution in the variable $\xi_a$ only.
Now, we set $\widetilde{m}_1=\mathds{1}_{\vois(\supp(m_1),d/3)}$ and $\widetilde{m}_2=\mathds{1}_{\vois(\supp(m_2),d/3)}$, satisfying $\|\widetilde{m}_j\|_{L^\infty} \leq 1$ and $$\dist(\supp (\widetilde{m}_1),\supp (\widetilde{m}_2))\geq d/3 . $$ We write
\bna
m^{\mu}_{1,\lambda}(\xi_a)\left[\widehat{f_L}(\xi_a, x_b)*\left(m^{\mu}{2,\lambda}(\xi_a) \hat{u}(\xi_a , x_b)\right)\right]&=& Y_1 + Y_2 + Y_3 ,
\ena
with 
\bna
Y_1 &= &\widetilde{m}^{\mu}_{1}m^{\mu}_{1,\lambda}(\xi_a)\left[\widehat{f_L}(\xi_a, x_b)*\left(\widetilde{m}^{\mu}_{2}m^{\mu}_{2,\lambda}(\xi_a) \hat{u}(\xi_a, x_b)\right)\right]\\
Y_2 &= &(1-\widetilde{m}_{1,\mu})m^{\mu}_{1,\lambda}(\xi_a)\left[\widehat{f_L}(\xi_a, x_b)*\left(\widetilde{m}^{\mu}_{2}m^{\mu}_{2,\lambda}(\xi_a) \hat{u}(\xi_a, x_b)\right)\right]\\
Y_3 &= & 
m^{\mu}_{1,\lambda}(\xi_a)\left[\widehat{f_L}(\xi_a, x_b)*\left((1-\widetilde{m}^{\mu}_{2})m^{\mu}_{2,\lambda}(\xi_a) \hat{u}(\xi_a, x_b)\right)\right] .
\ena
The term $Y_1$ vanishes since $\widetilde{m}^{\mu}_{2}m^{\mu}_{2,\lambda}(\xi_a) u(\xi_a, x_b)$ is supported in $\xi_a/\mu  \in \vois(\supp(m_2),d/3)$; hence, using that $\supp(\hat{f_L}) \subset \{|\xi_a|/\mu \leq d/3\}$ the convolution $\left[\widehat{f_L}(\xi_a, x_b)*\left(\widetilde{m}^{\mu}_{2}m_{2,\lambda}^{\mu}(\xi_a) u(\xi_a, x_b)\right)\right]$ is supported in $\xi_a/\mu \in \vois(\supp(m_2),2d/3)$ which does not intersect the support (in $\xi_a/\mu$) of $\widetilde{m}^{\mu}_{1}$ that is $\vois(\supp(m_1),d/3)$.

\medskip
Concerning the term $Y_2$, Lemma \ref{lmsuppdisjoint} implies $\nor{(1-\widetilde{m}^{\mu}_{1})m^{\mu}_{1,\lambda}}{L^{\infty}_{\xi_a}}\leq Ce^{-c\lambda }$. This, together with the Young inequality in the variable $\xi_a$ and the uniform boundedness of $\widetilde{m}^{\mu}_{2}m^{\mu}_{2,\lambda}$, yields
\begin{align*}
& \nor{(1-\widetilde{m}^{\mu}_{1})m^{\mu}_{1,\lambda}(\xi_a)\left[\widehat{f_L}(\xi_a, x_b)*\left(\widetilde{m}^{\mu}_{2}m^{\mu}_{2,\lambda}(\xi_a) \hat{u}(\xi_a, x_b)\right)\right]}{L^2(\R^n)} \\
& \quad  \leq  \nor{(1-\widetilde{m}^{\mu}_{1})m^{\mu}_{1,\lambda}}{L^{\infty}_{\xi_a}}\nor{\hat{f_L}}{L^\infty_{x_b}L^1_{\xi_a}}\nor{\mathcal{F}_a(u)}{L^2(\R^{n_a} \times \R^{n_b})}\\
&  \quad \leq  Ce^{-c\lambda }\nor{\hat{f}}{L^\infty_{x_b}L^1_{\xi_a}} \nor{u}{L^2(\R^n)}.
\end{align*}
The term $Y_3$ is treated similarly and the proof is complete.

\medskip
The second estimate of the Lemma follows the same proof and is actually simpler because the term $(1-\widetilde{m}^{\mu}_{2})m^{\mu}_{2}$ is zero.
\enp
\begin{lemma}
\label{lmFL1x}
Assume $f_1$, $f_2\in L^{\infty}(\R^{n})$ are bounded by $1$, and satisfy $\dist(\supp (f_1),\supp(f_2))\geq d>0$. Then, there exists $C>0$ such that for all $m\in L^{\infty}(\R^{n_a})$ satisfying $\hat{m}\in L^1(\R^{n_a})$ and all $\lambda >0$, we have
\bna
\nor{f_{1,\lambda}(x)m(D_a)f_{2,\lambda}(x)}{L^2(\R^n)\rightarrow L^2(\R^n)}\leq \nor{\hat{m}}{L^1(|\eta_a|\geq d/3)}+Ce^{-c\lambda}\nor{\hat{m}}{L^1(\R^{n_a})}
\ena
and 
\bna
\nor{f_{1,\lambda}(x)m(D_a)f_2(x)}{L^2(\R^n)\rightarrow L^2(\R^n)}\leq \nor{\hat{m}}{L^1(|\eta_a|\geq d /3)}+Ce^{-c\lambda}\nor{\hat{m}}{L^1(\R^{n_a})}.
\ena
\end{lemma}
\bnp
This is essentially the same proof as the previous Lemma except that we have to be careful that the functions $f_i$ depend on all variables, while $m$ only depends on the variable $x_a \in \R^{n_a}$. Again, we set $m_L = \mathds{1}_{|D_a| \leq d /3}m$ (that is $\widehat{m_L}(\eta_a)=\mathds{1}_{|\eta_a| \leq d /3}\widehat{m}(\eta_a) $) and $m_H= \mathds{1}_{|D_a| \geq d /3}m$. We first have
\bna
\nor{f_{1,\lambda}(x)m_H(D_a)f_{2,\lambda}(x)}{L^2(\R^n)\rightarrow L^2(\R^n)}
& \leq &  \nor{m_H(D_a)}{L^2(\R^n)\rightarrow L^2(\R^n)} \\
& \leq & \nor{m_H}{L^\infty(\R^{n_a})} 
\leq \nor{\hat{m}_H}{L^1(\R^{n_a})} 
\leq \nor{\hat{m}}{L^1(|\eta_a|\geq d /3)} .
\ena
Concerning the second term, and denoting $\check{m}_L = \mathcal{F}_a^{-1}(m_L)$, i.e. $\check{m}_L(\eta_a) = \hat{m}_L(-\eta_a)$, we have
\bna
f_{1,\lambda}(x)m_L(D_a)f_{2,\lambda}(x)u=f_{1,\lambda}(x) \ \check{m}_L *_{\R^{n_a}}\big(f_{2,\lambda}(\cdot , x_b)u(\cdot , x_b)\big) .
\ena
We then remark that we can finish the proof as in the previous Lemma: introducing $\tilde{f}_j :=\mathds{1}_{\vois(\supp(f_j),d/3)}$, $j=1,2$, we notice that we have
$$
\supp \Big(\check{m}_L *_{\R^{n_a}}\left[\tilde{f}_2 f_{2,\lambda} u\right]
\Big)
\subset  \vois(\supp(f_2),d/3) +\left\{(x_a,0),|x_a|\leq d/3\right\}
\subset \vois(\supp(f_2),2d/3) .
$$
Moreover, Lemma \ref{lmsuppdisjoint} still yields 
\bna
\nor{(1 - \tilde{f}_j )f_{j,\lambda }}{L^{\infty}(\R^n)}\leq Ce^{-c\lambda} ,\quad j = 1,2 ,
\ena
so that the proof then follows exactly that of Lemma~\ref{lmFL1}. We obtain the second inequality similarly.
\enp
\begin{lemma}
\label{lemma: Gevrey Chambertinf}
Let $k\in \N$ and $f \in C^\infty_0(\R^n)$. 
Then, there exist $C,c$ such that, for any $\lambda, \mu>0$, we have
\bna
\nor{M^{\mu}_{\lambda}f_{\lambda}(1-M^{2\mu}_{\lambda})}{H^k(\R^n) \to H^k(\R^n)}  \leq C e^{-c \frac{\mu^2}{\lambda}} +C e^{-c \lambda} ; \\
\nor{(1-M^{2\mu}_{\lambda})f_{\lambda}M^{\mu}_{\lambda}}{H^k(\R^n) \to H^k(\R^n)}  \leq C e^{-c \frac{\mu^2}{\lambda}} +C e^{-c \lambda} .
\ena
\end{lemma}
\bnp
Note first that $\mathcal{F}_a( \d_{x_a}^{\alpha}\d_{x_b}^{\beta}  f_{\lambda})(\xi_a, x_b)=(i \xi_a)^{\alpha}e^{-\frac{|\xi_a|^2}{\lambda}}\d_{x_b}^{\beta} \mathcal{F}_a(f)(\xi_a, x_b)$. Hence, for $k=0$, the result is a direct consequence of (the first estimate in) Lemma~\ref{lmFL1}. Note that we also use the fact that $(1-m)_{\lambda}=1-m_{\lambda}$.

For $k\geq1$, the proof proceeds by induction, noticing that 
$$
\nabla \left[(1-M^{2\mu}_{\lambda})f_{\lambda}M^{\mu}_{\lambda}u\right]=  (1-M^{2\mu}_{\lambda})(\nabla f)_{\lambda}M^{\mu}_{\lambda}u+ (1-M^{2\mu}_{\lambda})f_{\lambda}M^{\mu}_{\lambda}\nabla u
$$ (see e.g. the proof of Lemma~\ref{lmsuppdisjoint}).
\enp
\begin{lemma}
\label{l:commgevrey2}
Let $f_1$ and $f_2\in C^{\infty}(\R^n)$ bounded as well as all their derivatives, with\\ $\dist(\supp (f_1),\supp(f_2))\geq d>0$. Then for every $k\in \N$, there exist $C,c>0$ such that for all $\mu > 0$ and $\lambda >0$, we have 
$$
\nor{f_{1,\lambda} M^{\mu}_{\lambda} f_{2,\lambda}}{H^k(\R^n) \to H^k(\R^n)} \leq  C e^{-c \frac{\mu^2}{\lambda}} +C e^{-c \lambda}
, \qquad 
\nor{f_{1,\lambda} M^{\mu}_{\lambda} f_{2}}{H^k(\R^n) \to H^k(\R^n)} \leq  C e^{-c \frac{\mu^2}{\lambda}} +C e^{-c \lambda}.
$$ 
\end{lemma}
\bnp
We first prove both estimates for $k=0$, by using Lemma \ref{lmFL1x} with $m$ replaced by $m_b=m_{\lambda}\left(\frac{\cdot}{\mu}\right)$. The Fourier transform of $m_b$ is given by
$$
\hat{m}_b(\eta_a)
= \mu^{n_a} \mathcal{F}_a(m_\lambda)(\mu \eta_a)
=\mu^{n_a}e^{-\frac{|\eta_a|^2\mu^2}{\lambda}}\hat{m}\left(\eta_a\mu\right).
$$
As a consequence, we have 
\bna
 \nor{\hat{m}_b}{L^1(|\eta_a|\geq d/3)} \leq  e^{-\frac{d^2\mu^2}{9\lambda}} \nor{\hat{m}}{L^1(\R^{n_a})}
 \ena 
 and $\nor{\hat{m}_b}{L^1(\R^{n_a})} \leq \nor{\hat{m}}{L^1(\R^{n_a})}$, so that 
\bna
 \nor{\hat{m}_b}{L^1(|\eta_a|\geq d/3)} +Ce^{-c\lambda}\nor{\hat{m}_b}{L^1(\R^{n_a})}\leq  C e^{-\frac{d^2\mu^2}{9\lambda}}+Ce^{-c\lambda}.
\ena
Lemma~\ref{lmFL1x} then yields the sought result in the case $k=0$.

Again, for $k\geq 1$, the result is proved by induction noticing that 
$$
\nabla \left[f_{1,\lambda} M^{\mu}_{\lambda} f_{2,\lambda} u\right]=(\nabla f_{1})_{\lambda} M^{\mu}_{\lambda} f_{2,\lambda} u+ f_{1,\lambda}M^{\mu}_{\lambda} (\nabla f_{2})_{\lambda} u+f_{1,\lambda} M^{\mu}_{\lambda} f_{2,\lambda} \nabla u ,
$$ 
and using that the relative support properties of $\nabla f_i$ are preserved (see e.g. the proof of Lemma~\ref{lmsuppdisjoint}).
\enp

\begin{lemma}
\label{l:commutateurnon}
Let $k\in \N$ and let $f\in C^{\infty}_0(\R^n)$.  
Then there exist $C,c>0$ such that for all $\mu > 0$, $\lambda>0$ and $u \in H^k(\R^n)$, we have
\bnan
\label{chgtordrexi}
\nor{M^{\mu}_{\lambda} f_{\lambda} u }{k} \leq \nor{f_{\lambda} M^{2\mu}_{\lambda} u }{k}  + C \left(e^{-c \frac{\mu^2}{\lambda}} + e^{-c \lambda}\right)\nor{ u }{k} .
\enan
Moreover, for any $f_1 \in C^{\infty}(\R^n)$ bounded as well as all its derivatives, such that $f_1=1$ on a neighborhood of $\supp(f)$, for any $k \in \N$, there exist $C,c>0$ such that for all $\mu > 0$, $\lambda>0$ and $u \in H^k(\R^n)$, we have 
\bnan
\label{chgtordrex}
\nor{f_{\lambda} M^{\mu}_{\lambda} u }{k} \leq C\nor{M^{\mu}_{\lambda} f_{1,\lambda} u }{k}  + C \left(e^{-c \frac{\mu^2}{\lambda}} + e^{-c \lambda}\right)\nor{ u }{k} .
\enan
\end{lemma}
\bnp
We write
\bna
\nor{M^{\mu}_{\lambda} f_{\lambda}  u }{k} \leq \nor{M^{\mu}_{\lambda} f_{\lambda} M^{2\mu}_{\lambda} u }{k}  +\nor{M^{\mu}_{\lambda} f_{\lambda} (1-M^{2\mu}_{\lambda})u }{k} .
\ena
According to Lemma \ref{lemma: Gevrey Chambertinf}, we first have $\nor{M^{\mu}_{\lambda}  f_{\lambda}(1-M^{2\mu}_{\lambda})u }{k}\leq C \left(e^{-c \frac{\mu^2}{\lambda}} + e^{-c \lambda}\right)\nor{ u }{k}$.
The first term is simply estimated by $\nor{M^{\mu}_{\lambda} f_{\lambda} M^{2\mu}_{\lambda} u }{k}\leq \nor{ f_{\lambda}M^{2\mu}_{\lambda}u }{k}$, which proves~\eqref{chgtordrexi}.

Concerning the second part of the Lemma, we write
\bna
\nor{f_{\lambda} M^{\mu}_{\lambda} u }{k} \leq \nor{f_{\lambda} M^{\mu}_{\lambda} f_{1,\lambda} u }{k}+\nor{f_{\lambda} M^{\mu}_{\lambda}  (1-f_1)_{\lambda} u}{k} .
\ena
For the first term, we only have to remark that $\nor{f_{\lambda} M^{\mu}_{\lambda} f_{1,\lambda} u }{k}\leq C\nor{M^{\mu}_{\lambda} f_{1,\lambda} u }{k}$ uniformly in $\lambda$.
Then, using the assumption $\dist (\supp(f) , \supp(1-f_1)) >0$, Lemma \ref{l:commgevrey2} applies and yields
\bna
\nor{f_{\lambda} M^{\mu}_{\lambda}  (1-f_1)_{\lambda} u}{k}\leq C \left(e^{-c \frac{\mu^2}{\lambda}} + e^{-c \lambda}\right)\nor{ u }{k},
\ena
which eventually proves~\eqref{chgtordrex}.
\enp

\begin{lemma}
\label{lemma: Gevrey Chambertinsomme}
Let $k\in \N$ and $f \in C^{\infty}_0(\R^n)$. Assume $\supp (f) \subset \bigcup_{i\in I} U_i$ where $(U_i)_{i \in I}$ is a finite family of bounded open sets. Let $b_j\in C^{\infty}_0(\R^n)$ such that $b_j =1$ on a neighborhood of $\overline{U_i}$. Then, for any $k \in \N$, there exist $C,c>0$ such that for all $\mu > 0$, $\lambda>0$ and $u \in H^k(\R^n)$, we have 
\bna
\nor{M^{\mu}_{\lambda}f_{\lambda} u }{k}  \leq C\sum_{i}\nor{M^{2\mu}_{\lambda} (b_i)_{\lambda} u }{k} +  C \left(e^{-c \frac{\mu^2}{\lambda}} + e^{-c \lambda}\right)\nor{u }{k} .
\ena
\end{lemma}
\bnp
Applying the first item of Lemma \ref{l:commutateurnon} to $f$, we obtain
\bnan
\label{e:techniq-Mf-1}
\nor{M^{\mu}_{\lambda}f_{\lambda} u }{k}  \leq \nor{f_{\lambda}M^{2\mu}_{\lambda} u }{k}+  C \left(e^{-c \frac{\mu^2}{\lambda}} + e^{-c \lambda}\right)\nor{u }{k} .
\enan
Let now $(f_i)_{i \in I}$ be a smooth partition of unity of a neighborhood of $\supp (f)$ such that 
$$
\sum_{i \in I} f_i = 1 \text{ in a neighborhood of }\supp (f) , 
\quad \supp(f_i) \subset U_i ;  
\quad 0\leq f_i \leq 1 .
$$
Note that in particular, $b_i = 1$ in a neighborhood of $\supp(f_i)$. Using the second estimate of Lemma~\ref{lmsupHK}, we have
\bnan
\label{e:techniq-Mf-2}
\nor{f_{\lambda}M^{2\mu}_{\lambda} u }{k}  
\leq C\nor{\sum_i(f_i)_{\lambda}M^{2\mu}_{\lambda} u }{k}
+ C e^{-c\lambda}\nor{M^{2\mu}_{\lambda} u }{k}
\leq C\sum_i \nor{(f_i)_{\lambda}M^{2\mu}_{\lambda} u }{k}
+ C e^{-c\lambda}\nor{ u }{k} .
\enan
Using the second estimate in Lemma \ref{l:commutateurnon}, we then obtain
\bna
\nor{(f_i)_{\lambda}M^{2\mu}_{\lambda} u }{k}\leq C\nor{M^{2\mu}_{\lambda} (b_i)_{\lambda}u }{k}+  C \left(e^{-c \frac{\mu^2}{\lambda}} + e^{-c \lambda}\right)\nor{u }{k} ,
\ena
which, combined with~\eqref{e:techniq-Mf-1} and~\eqref{e:techniq-Mf-2} concludes the proof of the Lemma. 
\enp
 \begin{lemma}
\label{lmboundchipsi}
There exists $C>0$ such that for all $D\in\R$ and $\widetilde{\chi} \in L^\infty(\R)$ such that $\supp(\widetilde{\chi})\subset (-\infty,D]$, for all $\lambda , \tau >0$, we have
\bna
| e^{\tau z}\widetilde{\chi}_{\lambda}(z)| 
\leq C \nor{\widetilde{\chi}}{L^\infty(\R)}\left\langle \lambda\right\rangle^{1/2}e^{\frac{\lambda}{4}|\Im(z)|^2}e^{D\tau}e^{\frac{\tau^2}{\lambda}} , \quad \text{ for all }z \in \C ;
\ena
\bna
\nor{e^{\tau \psi}\widetilde{\chi}_{\lambda}(\psi)}{L^{\infty}(\R^n)}\leq C \nor{\widetilde{\chi}}{L^\infty(\R)}\left\langle \lambda\right\rangle^{1/2}e^{D\tau}e^{\frac{\tau^2}{\lambda}} , \quad \text{ for all } \psi \in C^0(\R^n ; \R) .
\ena
\end{lemma}
\bnp
First, according to~\eqref{estimalambdacompl}, we have the estimate 
$$|\widetilde{\chi}_{\lambda}(z)|\leq C  \nor{\widetilde{\chi}}{L^\infty(\R)} \lambda^{1/2} 
e^{\frac{\lambda}{4}|\Im(z)|^2} e^{- \frac{\lambda}{4} \dist(\Re(z),\supp \widetilde{\chi})^2} , \quad \text{ for all } z \in \C .$$ 
Now, if $\Re(z)\leq D$, we use the bound $|e^{\tau z}|\leq e^{D\tau}$, which yields 
 $|e^{\tau z}\widetilde{\chi}_{\lambda}(z)|\leq \nor{\widetilde{\chi}}{L^\infty(\R)} e^{\frac{\lambda}{4}|\Im(z)|^2} e^{D\tau}$. 

Next, for $\Re(z)\geq D$, we have $\dist(\Re(z),\supp \widetilde{\chi})\geq \Re(z)-D\geq 0$, and 
\bna
|e^{\tau z}\widetilde{\chi}_{\lambda}(z)| & \leq &  e^{\tau \Re(z)}
C  \nor{\widetilde{\chi}}{L^\infty(\R)} \left\langle \lambda\right\rangle^{1/2} e^{\frac{\lambda}{4}|\Im(z)|^2}  e^{- \frac{\lambda}{4} (\Re(z)-D)^2} \\
& \leq & C  \nor{\widetilde{\chi}}{L^\infty(\R)} \left\langle \lambda\right\rangle^{1/2} e^{\frac{\lambda}{4}|\Im(z)|^2} 
\sup_{s \geq D}\left( e^{\tau s} e^{- \frac{\lambda}{4} (D-s)^2} \right).
\ena
Finally, we have 
$$
\sup_{s \geq D}\left( e^{\tau s} e^{- \frac{\lambda}{4} (D-s)^2} \right)
= \sup_{t \geq 0}\left( e^{\tau (D+ t) } e^{- \frac{\lambda}{4}t^2} \right)
=  e^{\tau D} \sup_{t \geq 0}\left(e^{t(\tau  - \frac{\lambda}{4}t)} \right)
= e^{D \tau}e^{\frac{\tau^2}{\lambda}},
 $$ 
 which concludes the proof of the first estimate of the lemma. The second estimate of the lemma follows from the first estimate for $z = s \in \R$ combined with
\bna
\nor{e^{\tau \psi}\widetilde{\chi}_{\lambda}(\psi)}{L^{\infty}(\R^n)}=\nor{e^{\tau s}\widetilde{\chi}_{\lambda}(s)}{L^{\infty}(\R)} .
\ena
\enp
\begin{lemma}
\label{lemma:expAmu}
There exist $C,c$ such that, for any $\eps ,\tau, \lambda, \mu,>0$, for any $k \in \N$, we have
\bna
\nor{e^{-\frac{\e|D_a|^2}{2\tau}}(1-M^{\mu}_{\lambda})}{H^k(\R^n)\to H^k(\R^n)} 
\leq e^{- \frac{\e \mu^2}{8\tau}} +C e^{-c \lambda}.
\ena
\end{lemma}
\bnp
Since the operator $e^{-\frac{\e|D_a|^2}{2\tau}}(1-M^{\mu}_{\lambda})$ is a Fourier multiplier, we are left to estimate\\ $\sup_{\xi_a \in \R^{n_a}} |e^{-\frac{\e|\xi_a|^2}{2\tau}}(1-m_{\lambda}(\frac{\xi_a}{\mu}))|$. Recall that $m \in C^\infty_0(\R^{n_a}; [0,1])$ is a radial function that we identify below with a function $m  = m(s) \in C^\infty_0(\R^+)$, satisfying $\supp(m) \subset [0,1)$ and $m=1$ on $[0,3/4)$.
We distinguish the following two cases:
\begin{itemize}
\item If $|s|\leq \mu/2$, Lemma \ref{lmsuppdisjoint} applied with $f_1=(1-m_{\lambda}(s))$ and $f_2=\mathds{1}_{|s|\leq 1/2}$ implies $ |\mathds{1}_{|s|\leq \mu/2} (1-m_{\lambda}(\frac{s}{\mu}))|\leq Ce^{-c\lambda}$ uniformly with respect to $\lambda , \mu >0$;
\item If $|s|\geq \mu/2$, we simply have $|\mathds{1}_{|s|\geq \mu/2}  e^{-\frac{\e|s|^2}{2\tau}}(1-m_{\lambda}(\frac{s}{\mu}))| \leq e^{-\frac{\e \mu^2}{8\tau}}$.
\end{itemize}
Combining these two estimates concludes the proof of the lemma.
\enp

\subsubsection{Some more involved preliminary estimates}

We will need the estimate of the following Lemma. 
\begin{lemma}
\label{lmdecrFourier}
Let $\psi$ be a smooth real valued function on $\R^n$, which is a quadratic polynomial in the variable $x_a \in \R^{n_a}$, let $R_\sigma>0$, and $\sigma \in C^\infty_c(B_{\R^n}(0, R_\sigma))$. Let $\chi  \in C^\infty_0(\R)$ with $\supp(\chi) \subset (-\infty, 1)$, and  $\tilde\chi \in C^\infty(\R)$ such that $\tilde\chi = 1$ on a neighborhood of $(-\infty ,\frac32 )$,  $\supp(\tilde\chi) \subset (-\infty, 2)$, and set $\chi_\delta (s) := \chi  (s/\delta)$, $\tilde\chi_\delta (s) := \tilde\chi  (s/\delta)$.  
Let $f\in C^\infty_0(\R^{n})$ be {\em real analytic} in the variable $x_a$ in a neighborhood of $\overline{B}_{\R^{n_a}}(0,R_\sigma)$ and define
\bna
\label{e:def-g}
g=e^{\tau\psi}\chi_{\delta,\lambda}(\psi) \tilde\chi_{\delta}(\psi)f \sigma_\lambda \quad  \in C^\infty_0(\R^{n}).
\ena
Then, there exists $c_0, c_1 >0$  such that for all $N \in \N$ and $\beta \in \N^{n_b}$, there exist $C >0$ such that for all $\delta>0$, there is $\eps_0>0$, so that for any $\lambda\geq 1$, $\tau >0$, and $0<\e<\e_0$, we have
\bna
|\d_{x_b}^{\beta} \mathcal{F}_a(g)(\xi_a,x_b)|& \leq & C \left< \xi_a\right>^{-N} (\tau + \delta^{-1} + 1)^{N+ |\beta|}\lambda^{(n_a+1)/2}e^{\de \tau}\\
&& \times \left( e^{\frac{\tau^2}{\lambda}}e^{c_1 \e^2\lambda}e^{-c_0\e|\xi_a|}+ e^{\frac{\tau^2}{\lambda}} e^{-c_0\lambda} 
+  e^{c_1\lambda \e^2}   e^{\de \tau} e^{- c_0 \delta^2 \lambda} 
\right).
\ena
In particular, for all $\delta>0$, $N\in \N$,  $\beta \in \N^{n_b}$, there is $C,c,\e_0>0$, so that for any $\lambda, \tau \geq 1$,  and $0<\e<\e_0$, we have
\bna
|\d_{x_b}^{\beta} \mathcal{F}_a(g)(\xi_a,x_b)|\leq C \left< \xi_a\right>^{-N} \tau^{N+ |\beta|}\lambda^{(n_a+1)/2}e^{\de \tau}e^{\frac{\tau^2}{\lambda}}\left( e^{C \e^2\lambda}e^{-c\e|\xi_a|}
+    e^{\de \tau} e^{-c\lambda} 
\right).
\ena
\end{lemma}

\bnp[Proof of Lemma \ref{lmdecrFourier}]

First, we prove the result for $N=0$ and $\beta =0$ (the other cases shall be obtained by differentiating $g$).

Let us denote by $R_f'>0$ a real number such that $\supp(f)\subset B(0,R_f')$ and $K_b\subset B_{\R^{n_b}}(0,R_f')$ the projection in the variable $x_b$ of the support of $f$. $K_b$ is compact since $f$ has compact support. 
The function $f$ being real analytic in the variable $x_a$ in a neighborhood of the compact set $\overline{B}_{\R^{n_a}}(0,R_\sigma)$, there exists $R_f >0$ such that $f$ can be extended in an analytic way in a neighborhood of $z_a \in \overline{B}_{\R^{n_a}}(0,R_\sigma+R_f) + i  \overline{B}_{\R^{n_a}}(0, R_f)$, uniformly for $x_b\in K_b$. Note that $z_a$ denotes the complex variable associated to $x_a$, and we can also impose that $0 < R_\sigma + R_f < R_f'$.

Notice also that we can extend $\tilde\chi$ by $1$ (hence analytically) on a neighborhood of $(-\infty ,\frac32 ) + i \R$.
Moreover, since $\psi$ is quadratic in $x_a$, there exists $\eps_0 = \e_0(\delta)>0$ such that 
\bnan
\label{e:setsetsetooo}
\left( \psi(\Re(z_a),x_b) \leq \frac43 \delta  < \frac32 \delta , \quad |\Im(z_a)| \leq \e_0 R_f  , \quad x_b \in K_b \right)
\Longrightarrow \Re (\psi(z_a , x_b)) \leq \frac32 \delta , \\
\label{e:setsetset}
\left( \psi(\Re(z_a),x_b) = \frac43 \delta , \quad |\Im(z_a)| \leq \e_0 R_f  , \quad x_b \in K_b \right)
\Longrightarrow  \Re (\psi(z_a , x_b)) \geq \frac54 \delta .
\enan
In particular, $\tilde\chi(\psi(z_a , x_b)) =1$ on $$
\left( \psi(\Re(z_a),x_b) \leq \frac43 \delta  < \frac32 \delta , \quad |\Im(z_a)| \leq \e_0 R_f  , \quad x_b \in K_b \right) .
$$
As a consequence, given $x_b \in \R^{n_b}$, the function $$z_a \mapsto \chi_{\delta,\lambda}(\psi(z_a , x_b)) \tilde\chi_{\delta}(\psi(z_a , x_b))$$ is an analytic function on a neighborhood of $\{x_a \in \R^{n_a} , \psi(x_a , x_b) \leq \frac43 \delta \} + i\overline{B}_{\R^{n_a}}(0, \eps_0 R_f)$. Hence, $z_a \mapsto g(z_a , x_b)$ is holomorphic in a neighborhood of
$$
\mathcal{A}_{x_b}(\eps_0) : = \left( \{\psi(x_a , x_b) \leq \frac43 \delta \} \cap  \overline{B}_{\R^{n_a}}(0,R_\sigma+R_f) \right) + i  \overline{B}_{\R^{n_a}}(0, \eps_0 R_f) .
$$
The plan of the proof is first to estimate $g$ in the complex domain, and then bound its Fourier transform using a complex deformation. We use the analyticity inside of $\mathcal{A}_{x_b}(\eps_0)$ and the smallness elsewhere on the real domain. 

\medskip
\noindent
\textbf{Step 1: uniform estimates on the function $g$.} 
We estimate separately $f \sigma_\lambda$ and $e^{\tau\psi}\chi_{\delta,\lambda}(\psi)\tilde\chi_{\delta}(\psi)$, and then deduce estimates for $g$.

 According to the basic estimate~\eqref{estimalambdacompl-bis} for $\sigma_\lambda$, we have, uniformly for $x_b \in \R^{n_b}$
\bna
|(f \sigma_\lambda )(z_a, x_b)| & \leq &  C\lambda^{n_a/2}e^{\frac{\lambda}{4}|\Im z_a|^2} e^{- \frac{\lambda}{4} \dist(\Re z_a,\supp \sigma ( \cdot , x_b))^2}, 
\quad z_a \in B_{\R^{n_a}}(0,R_\sigma+R_f) +  i  B_{\R^{n_a}}(0, R_f) , 
\ena
where the constant $C$ depends only on $\|f\|_{L^\infty}$ (on the previous complex domain), $\|\sigma\|_{L^\infty}$ and $R_f'$.

In particular, we have for any $\eps \in [0,1]$,
\bnan
|(f \sigma_\lambda )(z_a, x_b)| & \leq & C\lambda^{n_a/2}e^{\frac{\lambda}{4}\e^2R_f^2} , 
\quad z_a \in B_{\R^{n_a}}(0,R_\sigma+R_f) + i  B_{\R^{n_a}}(0, \eps R_f) ,  \quad x_b \in \R^{n_b} .
\label{estimup}
\enan
We now notice that 
\bnan
\label{e:sups-ball}
\dist(x_a,\supp \sigma ( \cdot , x_b)) \geq \dist((x_a,x_b) , B(0,R_\sigma)) \geq |x_a| - R_\sigma , \quad \text{ for }|x_a|\geq R_\sigma+R_f .
\enan
As a first consequence, we have $\dist(x_a,\supp \sigma ( \cdot , x_b)) \geq R_f$ if $|x_a| =R_\sigma+R_f $, so that for any\\ $\eps \in [0,1]$, we obtain, uniformly for $x_b \in \R^{n_b}$
\bnan
|(f \sigma_\lambda )(z_a, x_b)| & \leq & C\lambda^{n_a/2}e^{\frac{\lambda}{4}\e^2R_f^2} e^{-\frac{\lambda}{4}R_f^2}\leq C\lambda^{n_a/2}e^{\frac{\lambda}{4}(\e^2-1)R_f^2},\label{estimside}\\ 
&&\quad |\Im(z_a)| \leq \e R_f , |\Re(z_a)|=R_\sigma+R_f \nonumber
\enan
Using now the estimate~\eqref{estimalambdacompl-bis} for $\sigma_\lambda$ on the real domain together with the boundedness of $f$ and \eqref{e:sups-ball}, we obtain, uniformly for $x_b \in \R^{n_b}$
\bnan
\label{estimbreal}
|(f \sigma_\lambda )(x_a, x_b)| & \leq &  C\lambda^{n_a/2}  e^{- \frac{\lambda}{4} \dist(x_a,\supp \sigma ( \cdot , x_b))^2} \nonumber\\
& \leq & C\lambda^{n_a/2}  e^{- \frac{\lambda}{4} (|x_a|-R_\sigma)^2}, 
\quad x_a \in \R^{n_a} ,\quad |x_a|\geq R_\sigma+R_f .
\enan

We now estimate the term $e^{\tau \psi} \chi_{\delta , \lambda}(\psi) \tilde\chi_{\delta}(\psi)$ in parts of the complex domain. 

First, on the real domain, we have
$$
|e^{\tau s}\chi_{\delta,\lambda}(s)  \tilde\chi_{\delta}(s)|
 \leq e^{2 \delta \tau} |\chi_{\delta,\lambda}(s)  \tilde\chi_{\delta}(s)| 
\leq C \lambda^\frac12 e^{2 \delta \tau} e^{-c \delta^2\lambda } , \quad s \geq \frac43 \delta ,
$$
after having used~\eqref{e:estimgaussdistE}, where $c$ is a numerical constant. As a consequence, we obtain
\bnan
\label{estimpsicompl-reel}
|e^{\tau \psi (x_a,x_b)}\chi_{\delta,\lambda}( \psi (x_a,x_b))  \tilde\chi_{\delta}( \psi (z_a,x_b))| 
\leq C \lambda^\frac12 e^{2 \delta \tau} e^{-c \delta^2\lambda} ,
\quad \text{ if }\psi(x_a,x_b) \geq \frac43 \delta  .
\enan
Next, for $z\in \C$,  using Lemma~\ref{lmboundchipsi}, there is $C>0$ such that for all $\delta \in\R$ and all $\lambda\geq 1$ , $\tau >0$, we have
\bnan
\label{regpsireal}
|e^{\tau z}\chi_{\delta,\lambda}(z)|
& \leq & C\lambda^{1/2}e^{\frac{\lambda}{4}(\Im z)^2}e^{\de \tau}e^{\frac{\tau^2}{\lambda}},
\quad \text{ for all }  z \in \C .
\enan
Using that $\psi$ is a quadratic polynomial in the variable $x_a$, with real coefficients, we have
$$
|\Im(\psi(z_a,x_b)) | \leq C |\Re(z_a)||\Im(z_a)| + C(K_b)  |\Im(z_a)| , \quad (z_a,x_b) \in \C^{n_a} \times K_b ,
$$
where we have used the fact that $K_b$ is compact. As a consequence, there is a constant $C_0 = C_0(\psi,R_\sigma,R_f,K_b) >0$ such that 
$$
|\Im(\psi(z_a,x_b)) | \leq \e C_0 , \quad \text{for }z_a \in B_{\R^{n_a}}(0,R_\sigma+R_f) + i  B_{\R^{n_a}}(0, \eps R_f) , \quad x_b \in K_b .
$$
Hence, using~\eqref{regpsireal}, we obtain, for all $\eps \in (0,\eps_0)$

\bnan
\label{estimpsicompl}
|e^{\tau \psi (z_a,x_b)}\chi_{\delta,\lambda}( \psi (z_a,x_b)) \tilde\chi_{\delta}( \psi (z_a,x_b)) |
&\leq & C\lambda^{1/2}e^{\lambda \frac{C_0^2 \e^2}{4}}e^{\de \tau}e^{\frac{\tau^2}{\lambda}} 
, \quad  x_b \in K_b , z_a \in \mathcal{A}_{x_b}(\eps) . 
\enan
According to~\eqref{estimalambdacompl}, we also have
\bna
\left|\chi_{\delta ,\lambda}(z)\right|
& \leq & C \lambda^{\frac12}
 e^{\frac{\lambda}{4}|\Im(z)|^2} 
e^{-\frac{\lambda}{4} \dist(\Re(z) , \supp(\chi_{\delta} ))^2}
\leq C \lambda^{\frac12}
 e^{\frac{\lambda}{4}|\Im(z)|^2} 
e^{- c\delta^2 \lambda} ,  \quad \text{ on } \Re(z) \geq \frac54 \delta ,
\ena
where $c$ is a numerical constant. Using~\eqref{e:setsetset}, this yields
\bna
\left|\chi_{\delta ,\lambda}(\psi(z_a, x_b))\right|
\leq C \lambda^{\frac12}
 e^{\frac{C_0^2 \eps^2}{4} \lambda} 
e^{- c\delta^2 \lambda} ,  \quad x_b \in K_b , z_a \in \mathcal{A}_{x_b}(\eps) , \psi(\Re(z_a),x_b) = \frac43 \delta ,
\ena
and, with~\eqref{e:setsetsetooo}, this implies
\bnan
\label{e:bordpsi}
|e^{\tau \psi (z_a,x_b)}\chi_{\delta,\lambda}( \psi (z_a,x_b)) \tilde\chi_{\delta}( \psi (z_a,x_b)) |
\leq C\lambda^{1/2}e^{\frac{C_0^2 \e^2}{4}\lambda }e^{\frac32\de \tau} e^{- c\delta^2 \lambda} 
, \nonumber\\
  x_b \in K_b , z_a \in \mathcal{A}_{x_b}(\eps) , \psi(\Re(z_a),x_b) = \frac43 \delta .
\enan

Let us finally gather all estimates obtained on the function $g$.
Multiplying \eqref{estimpsicompl} with \eqref{estimup} and \eqref{estimside}, there is a constant $C_1>0$ independent on $\lambda$, $\mu$, $\tau$, $\delta$, $\e$ such that, for any $\eps \in (0,\eps_0)$, 
\bnan
|g(z_a, x_b)|  & \leq & C\lambda^{(n_a+1)/2} e^{C_1\lambda \e^2}e^{\de \tau}e^{\frac{\tau^2}{\lambda}}, 
\quad x_b \in K_b , z_a \in \mathcal{A}_{x_b}(\eps) ,
\label{estimgrectangle}\\
\label{estimsideg}
|g(z_a, x_b)|  & \leq & C\lambda^{(n_a+1)/2}e^{\lambda (-\frac{R_f^2}{4}+C_1\e^2)}e^{\de \tau}e^{\frac{\tau^2}{\lambda}}, 
\quad  x_b \in K_b , z_a \in \mathcal{A}_{x_b}(\eps) , |\Re(z_a)|=R_\sigma+R_f. 
\enan
Next, multiplying~\eqref{e:bordpsi} and~\eqref{estimup} we also have
\bnan
\label{estimgrectangle-psi}
|g(z_a, x_b)|  & \leq & C\lambda^{(n_a+1)/2} e^{C_1 \e^2\lambda }e^{\frac32\de \tau} e^{- c\delta^2 \lambda} 
, 
\quad  x_b \in K_b , z_a \in \mathcal{A}_{x_b}(\eps) , \psi(\Re(z_a) , x_b) = \frac43 \delta. 
\enan
Combining \eqref{estimbreal} with \eqref{regpsireal}, and rewriting \eqref{estimpsicompl-reel}, 
we also have on the real domain
\bnan
\label{estimgloin}
|g(x_a, x_b)| & \leq &  C\lambda^{(n_a+1)/2}  e^{\de \tau}e^{\frac{\tau^2}{\lambda}}e^{- \frac{\lambda}{4} (|x_a|-R_\sigma)^2}
, \quad  x_a \in \R^{n_a}  ,  |x_a|\geq R_\sigma+R_f , x_b \in \R^{n_b} ,\\
\label{estimgloin-psi}
|g(x_a, x_b)| & \leq &  C  \lambda^\frac12 e^{2 \delta \tau} e^{- c \delta^2\lambda}
, \quad  x_a \in \R^{n_a}  ,  x_b \in \R^{n_b} , \psi(x_a,x_b) \geq \frac43 \delta.
\enan

\medskip
\noindent
\textbf{Step 2: estimating the Fourier transform using a deformation of contour in the complex domain.}
We now want to estimate $\mathcal{F}_a(g)(\xi_a, x_b)$ uniformly with respect to $x_b$. We split the integral as
\bna
\mathcal{F}_a(g)(\xi_a, x_b) =  \int_{\R^{n_a}} e^{ - ix_a \cdot \xi_a}g(x_a, x_b) dx_a 
= I_0 + I_1 + I_2, 
\ena
with $I_j  = I_j(\xi_a , x_b)$ defined by
\bna
I_0  :=  \int_{|x_a|\leq R_\sigma+R_f , \psi(x_a ,x_b) \leq \frac43\delta}
,\quad
I_1 :=  \int_{|x_a|\leq R_\sigma+R_f , \psi(x_a ,x_b) > \frac43\delta}
,\quad
I_2 :=  \int_{|x_a|> R_\sigma+R_f}
\ena
Using \eqref{estimgloin}, we obtain, for all $\delta , \tau >0$ and $\lambda >1$, 
\bnan
\label{e:estim I_2}
|I_2| & \leq & C\lambda^{(n_a+1)/2}  e^{\de \tau}e^{\frac{\tau^2}{\lambda}}\int_{|x_a|\geq R_\sigma+R_f}e^{- \frac{\lambda}{4} (|x_a|-R_\sigma)^2} dx_a \nonumber\\
& \leq & C\lambda^{(n_a+1)/2}  e^{\de \tau}e^{\frac{\tau^2}{\lambda}}\int_{s= R_f}^{+\infty}(s+R_\sigma)^{n_a-1}e^{- \frac{\lambda}{4} s^2}\leq C\lambda^{(n_a+1)/2} e^{\de \tau}e^{\frac{\tau^2}{\lambda}}e^{-\frac{R_f^2}{4} \lambda} .
\enan
Using \eqref{estimgloin-psi}, we obtain, for all $\delta , \tau >0$ and $\lambda >1$, 
\bnan
\label{e:estim I_1}
|I_1| & \leq & C \lambda^\frac12 e^{2 \delta \tau} e^{- c \delta^2\lambda}.
\enan
We now want to estimate the integral $I_0(\xi_a , x_b)$: we write $x_a = x_1 \frac{\xi_a}{|\xi_a|} + x_a'$ for $x_1 =  x_a \cdot\frac{\xi_a}{|\xi_a|} $ and $x_a'$ such that $x_a' \cdot \xi_a = 0$ and make the orthogonal change of coordinates to $(x_1 , x_a')$ (preserving the ball $B_{\R^{n_a}}(0 ,R_\sigma+R_f)$). This yields 
\bna
 I_0(\xi_a , x_b) & = & \int_{B_{\R^{n_a}}(0 ,R_\sigma+R_f) \cap \{\psi(\cdot ,x_b) \geq \frac43\delta \}} e^{- ix_1 | \xi_a|}g(x_1 , x_a') dx_a' dx_1 \\
 & = & \int_{B_{\R^{n_a-1}}(0 ,R_\sigma+R_f)}   \mathcal{I}_{\xi_a, x_b}(x_a') dx_a'  ,  
 \ena
 \bna
\text{with } \quad \mathcal{I}_{\xi_a, x_b}(x_a') 
= \int_{|x_1|^2\leq (R_\sigma+R_f)^2 - |x_a'|^2,  \psi(x_1 , x_a' , x_b) \leq \frac43\delta}e^{- ix_1 | \xi_a| }g(x_1 , x_a') d x_1 ,
 \ena
 so that $$| I_0(\xi_a , x_b) | \leq C \sup_{x_a' \in B_{\R^{n_a-1}}(0 ,R_\sigma+R_f)}| \mathcal{I}_{\xi_a, x_b}(x_a')| .$$
 Hence, it only remains to estimate $|\mathcal{I}_{\xi_a, x_b}(x_a')|$ uniformly.
Now, $g$ being analytic in a neighborhood of $\mathcal{A}_{x_b}(\eps_0)$, and given any $x_a' \in B_{\R^{n_a-1}}(0 ,R_\sigma+R_f)$, the function $z_1 \mapsto e^{- iz_1 | \xi_a|}g(z_1 , x_a')$ is holomorphic in a neighborhood of the set 
$$
|\Re(z_1)|^2\leq (R_\sigma+R_f)^2 - |x_a'|^2,  \quad \psi(\Re(z_1), x_a' , x_b) \leq \frac43\delta , 
 \quad |\Im(z_1)| \leq \eps R_f, 
$$
for $\e \in (0, \e_0)$.

Now, we have
$$
\{ x_1 \in \R, \quad |x_1|^2\leq (R_\sigma+R_f)^2 - |x_a'|^2, \quad  \psi(x_1 , x_a' , x_b) \leq \frac43\delta  \}
= \bigcup_{k\in J} [\alpha_k^1 , \alpha_k^2 ],
$$
where $J = J(x_a' , x_b)$ has $0, 1$ or $2$ elements since $\psi$ is quadratic. Moreover, we have
\bnan
\label{e:either-or}
\text{either } \quad |\alpha_k^i|^2 + |x_a'|^2= (R_\sigma+R_f)^2, \quad \text{ or } \quad \psi(\alpha_k^i , x_a' , x_b) = \frac43\delta
\enan
 for $k \in J$ and $i = 1,2$, together with 
$$
\mathcal{I}_{\xi_a, x_b}(x_a') = \sum_{k\in J} \int_{[\alpha_k^1 , \alpha_k^2 ]} e^{- ix_1 | \xi_a| }g(x_1 , x_a') d x_1.
$$
To estimate $\mathcal{I}_{\xi_a, x_b}(x_a')$, we now make a change of contour in the complex variable $z_1$ as follows:
\bna
 \int_{[\alpha_k^1 , \alpha_k^2 ]} e^{- ix_1 | \xi_a| }g(x_1 , x_a') d x_1 = I_L+I_T+I_R , \quad \text{ with } I_\star = \int_{\gamma_\star} e^{- iz_1 | \xi_a|}g(z_1 , x_a') d z_1, \quad \text{for } \star = L,T,R , 
\ena
and 
\bna
\gamma_L & =&  [\alpha_k^1 , \alpha_k^1 -i  \eps R_f ]  ,\\
\gamma_T & = & [\alpha_k^1 -i  \eps R_f, \alpha_k^2 -i  \eps R_f]  ,\\
\gamma_R & = & [ \alpha_k^2 -i  \eps R_f, \alpha_k^2  ] ,
\ena
are three oriented segments in $\C$ (see Figure~\ref{f:oriented-contours1}). 
\begin{figure}[h!]
  \begin{center}
    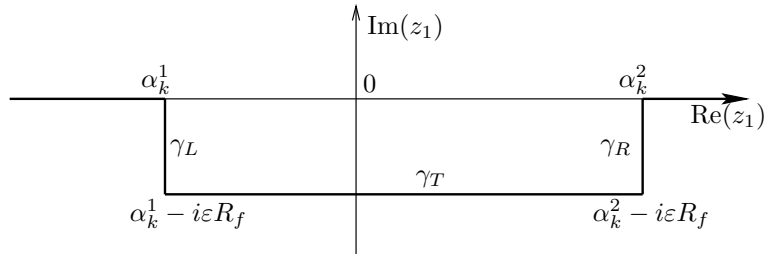 
    \caption{Oriented contours}
    \label{f:oriented-contours1}
 \end{center}
\end{figure}
We have
$$
 |I_\star | \leq \int_{\gamma_\star} e^{\Im(z_1) | \xi_a|} |g(z_1 , x_a')| d z_1, \quad \text{for } \star = L,T,R .
$$
On $\gamma_L$ and $\gamma_R$, using~\eqref{e:either-or} and $\Im(z_1) \leq 0$, we can use either estimate \eqref{estimsideg} or~\eqref{estimgrectangle-psi} and obtain, uniformly in $x_a' , \xi_a ,x_b$, $\delta , \tau >0$, $\lambda >1$, and $\eps \in (0, \eps_0(\delta))$
\bna
|I_L|+|I_R|\leq
C \eps \lambda^{(n_a+1)/2} e^{C_1\lambda \e^2} \left( e^{\de \tau} e^{-\lambda \frac{R_f^2}{4}}e^{\frac{\tau^2}{\lambda}} +   e^{\frac32\de \tau} e^{- c\delta^2 \lambda} \right) .
\ena
On $\gamma_T$, we have $(z_1 , x_a')\in \mathcal{A}_{x_b}(\eps)$ and $\Im(z_1)= -\eps R_f$, and thus using~\eqref{estimgrectangle}, we obtain, uniformly in $x_a' , \xi_a ,x_b$, $\delta , \tau >0$, $\lambda >1$, and $\eps \in (0,\eps_0(\delta))$,
\bna
|I_T|\leq C\lambda^{(n_a+1)/2}  e^{C_1\lambda \e^2}e^{\de \tau}e^{\frac{\tau^2}{\lambda}}e^{-\e R_f|\xi_a|}.
\ena
Combining the estimates on $I_L , I_R, I_T$ now proves that there is $C >0$ such that for any $\xi_a \in \R^{n_a},x_b \in R^{n_b}$, $\delta , \tau >0$, $\lambda >1$, and $\eps \leq \min( \eps_0(\delta) ,\frac{R_f}{2\sqrt{C_1}})$,
$$
|I_0| \leq C\lambda^{(n_a+1)/2}  e^{\de \tau}e^{\frac{\tau^2}{\lambda}}  \left(e^{C_1 \lambda \e^2} e^{-\e R_f|\xi_a|} +  e^{-\frac{R_f^2}{8}\lambda } \right) +  C \lambda^{(n_a+1)/2} e^{C_1\lambda \e^2}   e^{\frac32\de \tau} e^{- c\delta^2 \lambda},
$$
which, in view of Estimate~\eqref{e:estim I_2} and \eqref{e:estim I_1}, implies the result for $N=0$ and $\alpha=0$. 

\medskip
To obtain the result for $N \in \N$ and $\beta \in \N^{n_b}$, we notice that the functions $g_{\alpha, \beta}= \d_{x_a}^{\alpha} \d_{x_b}^{\beta}g$ can be written as a finite sum of terms that have the same form as the one of the assumption of the theorem with some different $f$, $b$ and $\chi_{\delta}$ (with the same support and analyticity properties) and with powers of $\tau^{\alpha'}\delta^{-\beta'}$ for $|\alpha'| + |\beta'| \leq |\alpha| + |\beta|$. The constants in the exponentials do not depend on $\alpha , \beta$ since they are functions of $\psi,R_\sigma,R_f,K_b$ only.
Noting that 
$(i\xi_a)^{\alpha}\d_{x_b}^{\beta} \mathcal{F}_a(g)(\xi_a,x_b) =  \mathcal{F}_a(\d_{x_a}^{\alpha} \d_{x_b}^{\beta}g)(\xi_a,x_b)$ finally concludes the proof of the lemma.
\enp
As a consequence of the previous result, we now have the following lemma.
\begin{lemma}
\label{lemma: analytic Chambertin}
Under the assumptions of  Lemma \ref{lmdecrFourier}, we have the following. For all $k\in \N$, $\delta>0$, there exist $N \in \N$, $C, c_0 ,\e_0>0$, such that for any $\lambda ,\mu,  \tau \geq 1$  and $0<\e<\e_0$, we have
\bna
\nor{M^{\mu/2}_{\lambda} g (1-M^{\mu}_{\lambda})}{H^k(\R^n) \to H^k(\R^n)} \leq C  \tau^{N}\lambda^{(n_a+1)/2}e^{\de \tau}e^{\frac{\tau^2}{\lambda}}\left( e^{C \e^2\lambda}e^{-c_0\e\mu}
+    e^{\de \tau} e^{-c_0\lambda} \right) ,\\
\nor{(1-M^{\mu}_{\lambda}) g M^{\mu/2}_{\lambda}}{H^k(\R^n) \to H^k(\R^n)} 
\leq C  \tau^{N}\lambda^{(n_a+1)/2}e^{\de \tau}e^{\frac{\tau^2}{\lambda}}\left( e^{C \e^2\lambda}e^{-c_0\e\mu}
+    e^{\de \tau} e^{-c_0\lambda}  
\right).
\ena

\end{lemma}
The estimates of this lemma will only be used under the weaker form: for all $c, \delta>0$, $k \in \N$, there exist $c_0 , C, N>0$ such that for any $ \tau , \mu \geq 1$ and $c^{-1}\mu \leq \lambda \leq c \mu$, we have
\bnan
\label{commutweaker}
\nor{M^{\mu/2}_{\lambda} g (1-M^{\mu}_{\lambda})}{H^k(\R^n) \to H^k(\R^n)}  \leq  C\tau^{N} e^{\frac{\tau^2}{\lambda}}e^{2\de \tau}e^{- c_0 \mu}, 
\enan
with the same estimate for the second term. It is obtained by taking $\eps$ sufficiently small in the regime $c^{-1}\mu \leq \lambda \leq c \mu$.
\begin{proof}
The two estimates are proved the same way, so we only prove the first one.
First, Lemma \ref{lmFL1} yields 
\begin{multline}
\nor{M^{\mu/2}_{\lambda} g (1-M^{\mu}_{\lambda})}{H^k(\R^n) \to H^k(\R^n)} \\
\leq  \sum_{|\alpha| + |\beta|\leq k}
 \nor{\xi_a^\alpha \d_{x_b}^\beta \mathcal{F}_a(g)}{L^\infty_{x_b}L^1(|\xi_a|\geq d \mu/3)}+Ce^{-c\lambda}\nor{\xi_a^\alpha \d_{x_b}^\beta \mathcal{F}_a(g)}{L^\infty(\R^{n_b};L^1(\R^{n_a}))} .
\end{multline}
Next, Lemma \ref{lmdecrFourier} with a $N\in \N$ large enough so that $\left< \xi_a \right>^{-(N+k)}$ is integrable on $\R^{n_a}$ yields
\bna
\nor{M^{\mu/2}_{\lambda} g (1-M^{\mu}_{\lambda})}{H^k(\R^n) \to H^k(\R^n)} \leq  
 C  \tau^{N+ k}\lambda^{(n_a+1)/2}e^{\de \tau}e^{\frac{\tau^2}{\lambda}}\left(e^{c_1 \e^2\lambda}e^{-c_0\e \mu}+e^{\de \tau}e^{-c_0\lambda}\right) ,
 \ena
which concludes the proof of the Lemma.
\end{proof}
\section{The local estimate}
\label{sectionlocal}
The aim of this section is to prove the local quantitative uniqueness result, (analytically) localized in frequency in the analytic variables.

In the following, we shall denote by 
\begin{equation}
\label{defb}
\begin{array}{c}
\sigma_R (x) := \sigma(R^{-1}|x-x^0|)\text{ with }
\sigma  \in C^\infty(\R) \text{ such that } \\
 \sigma=1 \text{ in a neighborhood of } ]-\infty,1 ], \text{ and } \sigma=0 \text{ in a neighborhood of }[2,+\infty[.
 \end{array}
\end{equation}
Our main local theorem is the following. See Figure~\ref{f:geom-local} for the geometry of the theorem. An important feature of this local result is that it can be iterated and hence propagated.
\begin{theorem}
\label{th:alpha-unif}
Let $x^0 \in \Omega \subset \R^{n_a} \times \R^{n_b}$ and $P$ be a partial differential operator on $\Omega$ of order $m$. Assume that 
\begin{itemize}
\item $P$ is analytically principally normal operator in $\{\xi_a = 0\}$ inside $\Omega$ (in the sense of Definition~\ref{def:anal principal normal});
\item there is a  function $\phi$ defined in a neighborhood of $x^0$ such that $\phi(x^0)=0$, and $\{\phi = 0\}$ is a $C^2$ strongly pseudoconvex oriented surface in the sense of Definition~\ref{def: pseudoconvex-surface}.
\end{itemize}
Then, there exists $R_0>0$ such that for any $R\in (0,R_0)$, there exist $r, \rho, \tilde{\tau}_0>0$,  for any $\vartheta\in C^{\infty}_0(\R^n)$  such that $\vartheta(x)=1$ on a neighborhood of $\left\{\phi\geq 2\rho\right\}\cap B(x^0,3R)$,
for all $c_1, \kappa>0$ there exist $C, \kappa', \beta_0 >0$ such that for all $\beta \leq \beta_0$, we have
\bna
\nor{M^{\beta\mu}_{c_1\mu} \sigma_{r,c_1\mu} u}{m-1}\leq C e^{\kappa \mu}\left(\nor{M^{\mu}_{c_1\mu} \vartheta_{c_1\mu} u}{m-1} + \nor{Pu}{L^2(B(x^0,4R))}\right)+Ce^{-\kappa' \mu}\nor{u}{m-1} ,
\ena
for all $\mu\geq \frac{\tilde{\tau}_0}{\beta}$ and $u\in C^{\infty}_0(\R^{n})$.
\end{theorem}

\begin{figure}[h!]
  \begin{center}
    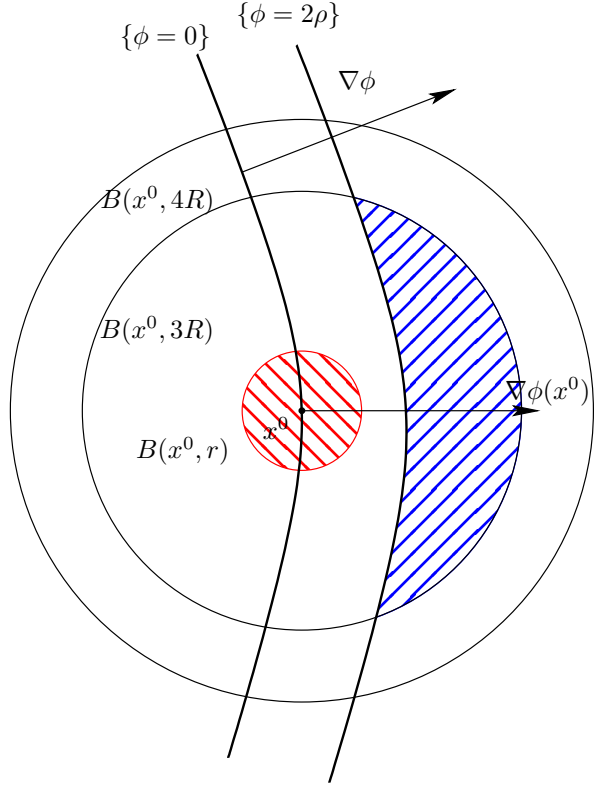 
    \caption{Geometry of the local uniqueness result.}
    Tlue striped region is the observation region (i.e. where $\vartheta =1$). The red striped region is the observed region (i.e. where $\sigma_r = 1$).
    \label{f:geom-local}
 \end{center}
\end{figure}

Note that this local result contains in particular the unique continuation result for operators with partially analytic coefficients~\cite{Tataru:95,RZ:98, Hor:97,Tataru:99} (which it is aimed to quantify). The latter is proved by letting $\mu \to + \infty$ in the estimate (and controlling some error terms), yielding: $Pu = 0 \text{ on }B(x^0,4R)), \ u = 0 \text{ on }\supp(\vartheta) \Longrightarrow u =0 \text{ on }\{ \sigma = 1 \}$.

This theorem allows to systematically quantify this local unique continuation result under partial analyticity conditions (in a way that can be iterated/propagated). As such, it also allows in particular to systematically quantify both the H\"ormander and the Holmgren theorems (again, in a way that can be iterated/propagated). Let us briefly comment on these two extreme situations: $n_a=0$ (H\"ormander case) and $n_a=n$ (Holmgren case).
\begin{remark}
If $n_a=0$, this inequality takes the form:
\bna
\nor{ \sigma_{r} u}{m-1}\leq C \frac{1}{\eps^{\kappa/\kappa'}}\left(\nor{\vartheta u}{m-1} + \nor{Pu}{L^2(B(x^0,4R))}\right)+C \eps \nor{u}{m-1} ,\quad \text{ for all } \eps \leq \eps_0, 
\ena
and hence
\bna
\nor{ \sigma_{r} u}{m-1}\leq C  \left(\nor{\vartheta u}{m-1} + \nor{Pu}{L^2(B(x^0,4R))}\right)^{\delta}\nor{u}{m-1}^{1-\delta} ,\quad \text{ for some } \delta >0, 
\ena
 which is an interpolation inequality of Lebeau-Robbiano type~\cite{LR:95} (see also~\cite{Robbiano:95}), and, as such, propagates well. Here it quantifies the general situation of the H\"ormander theorem (see also~\cite{Bahouri:87}).
 
 \medskip
 If $n_a=n$, we here describe a systematic way to quantify the Holmgren Theorem, which propagates well. See also~\cite{John:60} for a local result and ~\cite{Leb:Analytic} for a global result for waves. 
\end{remark}
 
\begin{remark}
The previous inequality can be written in the following way:

For all $(D, \mu,u) \in \R^+ \times [\frac{\tilde{\tau}_0}{\beta}, + \infty)\times H^{m-1}(\R^{n})$, satisfying 
\bna
\nor{M^{\mu}_{c_1\mu} \vartheta_{c_1\mu} u}{m-1}\leq  e^{-\kappa \mu}D\\
\nor{Pu}{L^2(B(x^0,4R))}\leq e^{-\kappa \mu} D ,
\ena
 we have 
$$
\nor{M^{\beta\mu}_{c_1\mu} \sigma_{r,c_1\mu} u}{m-1}\leq C' e^{-\kappa' \mu}\left(D + \nor{u}{m-1}\right) .
$$
This could certainly be written in the framework of propagation of (semiclassical, partially analytic) microsupport with respect to the variable $x_a$, see~\cite{Sjostrand:95} or \cite[Section 3.2]{MartinezBook}.  If $n_a=n$, it seems related to microlocal proofs of Holmgren theorem and the propagation of the analytic wavefront set (see~\cite{Sjostrand:95}).
\end{remark}

The proof of Theorem \ref{th:alpha-unif} is divided in three steps, given in Sections~\ref{s:geom-setting},~\ref{s:using-carleman}, and~\ref{s:complex-analysis-arg} respectively.

\subsection{Step 1: Geometric setting}
\label{s:geom-setting}
The following lemma is a refined version of~\cite[Lemma~4.1 p514]{RZ:98} or \cite[Lemmata 4.3 and 4.4]{Hor:97}. Its proof essentially follows that of~\cite[Lemma~4.1]{RZ:98}.  We state the geometric part for some balls not necessary euclidian. This will be useful in the case of boundary where some change of variable are used.
\begin{lemma}
\label{l:phi-psi}
Let $P$ be analytically principally normal in $\Omega \subset \R^n$, of order $m$ and principal symbol $p$. Let $\phi \in C^2(\Omega; \R)$ and $S = \{\phi =0\}$ be a $C^2$ oriented hypersurface of $\Omega$. Let $x^0 \in S \cap \Omega$ with $\nabla \phi (x^0) \neq 0$. Assume that $S$ is strongly pseudoconvex in $\Omega \times \{\xi_a = 0\}$ at $x^0$ for $P$ (in the sense of Definition~\ref{def: pseudoconvex-surface}).
Then, there exists $A>0$ such that the function
$$
\psi(x) := (x - x^0) \cdot \nabla \phi (x^0) + A((x - x^0) \cdot \nabla_x \phi (x^0))^2
+ \frac12 \phi'' (x^0)(x - x^0, x - x^0) - \frac{1}{A} |x - x^0|^2
$$
satisfies 
\begin{enumerate}
\item $\psi(x^0) = 0$ , $\nabla_x\psi(x^0) = \nabla_x\phi (x^0)$ ;

\item $\psi$ is strongly pseudoconvex in $\Omega \cap \{\xi_a = 0\}$ at $x^0$ for $P$ (in the sense of Definition~\ref{def: pseudoconvex-function}).

\item Let $N$ be a distance function locally equivalent to the euclidian distance. There exists $R_0>0$, such that for any $R\in (0,R_0)$,  there exists $\eta_0>0$ and for any $0< \eta  < \eta_0$, any $\eta_1, \eta_2>0$ there exist $\rho, r>0$ such that we have
\bnan
\label{eq: psi eta B}
& \Big( \{\phi \leq \rho \} \cap \{\psi \geq -\eta \} \cap  B_N(x^0 ,R) \Big) \subset  B_N(x^0 ,\frac{R}{8}) ,  \\
\label{eq: psi etatilde}
& \Big( \left\{\psi\geq \eta_1 \right\} \cap \overline{B}_N(x^0 ,R) \Big) \subset \left\{\phi > \rho \right\}, \\
\label{eq: psi eta r}
& B_N(x^0, r) \subset \{- \eta_2 <\psi< \eta_2 \} .
\enan
\end{enumerate}
\end{lemma}

Conditions~\eqref{eq: psi eta B}-\eqref{eq: psi etatilde}-\eqref{eq: psi eta r} are illustrated on Figure~\ref{f:local-geom-set}.

\begin{figure}[h!]
  \begin{center}
    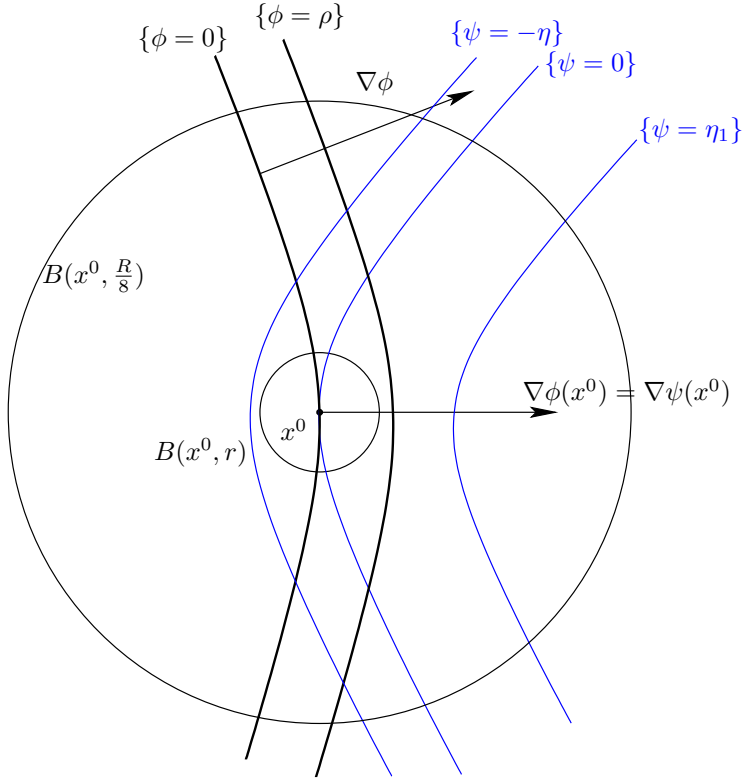 
    \caption{Local geometry of the level sets of the convexified function $\psi$ (in the case $N=$euclidean distance)}
    \label{f:local-geom-set}
 \end{center}
\end{figure}
\bnp
The first item directly follows from the definition of $\psi$  as a second order perturbation of the Taylor expansion of $\phi$ at $x^0$.
\medskip
The proof of the pseudoconvexity in Item 2 is very similar to ~\cite[Lemma 4.1]{RZ:98} or \cite[Lemma 7.4]{Hor:97}. We sketch it for sake of completeness. 

Let us compute $\Re \left\{ \overline{p}, \{p ,\psi \} \right\}$:  we have 
\bna
\Re \left\{ \overline{p}, \{p ,\psi \} \right\}=\Re \left(\frac{\partial^2 p}{\partial \xi\partial x }\left[\frac{\partial \bar{p}}{\partial \xi};\nabla \psi\right]+\psi''_{xx}\left[\frac{\partial \bar{p}}{\partial \xi};\frac{\partial p}{\partial \xi}\right] -\frac{\partial^2p}{\partial \xi^2} \left[\frac{\partial \bar{p}}{\partial x};\nabla \psi\right]\right) .
\ena
Since $\nabla\psi(x^0)=\nabla\psi(x^0)$, we have
\bna
\Re \left\{ \overline{p}, \{p ,\psi \} \right\} (x^0 , \xi)
= \Re \left\{ \overline{p}, \{p ,\phi \} \right\} (x^0 , \xi)
+2A \left|\nabla_x\phi(x^0) \cdot \frac{\partial p}{\partial \xi}(x^0,\xi) \right|^2 - \frac{2}{A}\left|\frac{\partial p}{\partial \xi}(x^0,\xi)\right|^2.
\ena
In this identity, all terms are homogeneous of order $2m-2$ in the variable $\xi$, so it is enough to prove the estimate for $\xi \in \mathbb{S}^{n-1}$
Hence, applying Lemma~\ref{l:fgh-three-fcts} below on the compact set $K =\{ \xi \in \mathbb{S}^{n-1}, \xi_a = 0 , p(x^0, \xi) = 0\}$, together with the first part of the pseudoconvexity assumption yields for $A$ large enough 
\bnan
\Re \left\{ \overline{p}, \{p ,\psi \} \right\} (x^0 ,\xi) >0  , 
&\text{ if } p(x^0,\xi) = 0 \text{ and } \xi_a = 0, \xi_b \neq 0 . \label{e:cond1-pseudocon}  
\enan
For the second estimate, we compute 
\bna
\frac{1}{i}\{ \overline{p}_\phi,p_\phi \} (x ,\xi)& = & \frac{1}{i} \left(\frac{\partial \bar{p}}{\partial \xi}(x,\xi-i\tau\nabla \phi)\frac{\partial p}{\partial x}(x,\xi+i\tau\nabla \phi) +i\tau \phi''_{xx}\left[\frac{\partial \bar{p}}{\partial \xi}(x,\xi-i\tau\nabla \phi);\frac{\partial p}{\partial \xi}(x,\xi+i\tau\nabla \phi)\right]\right)\\
&&-\frac{1}{i} \left(\frac{\partial \bar{p}}{\partial x}(x,\xi-i\tau\nabla \phi)\frac{\partial p}{\partial \xi}(x,\xi+i\tau\nabla \phi) -i\tau \phi''_{xx}\left[\frac{\partial \bar{p}}{\partial \xi}(x,\xi-i\tau\nabla \phi);\frac{\partial p}{\partial \xi}(x,\xi+i\tau\nabla \phi)\right]\right)\\
&=&C_{\tau,\phi,1}(x, \xi)+C_{\tau,\phi,2}(x, \xi) , 
\ena
with 
\bna
C_{\tau,\phi,1}(x, \xi):= \frac{1}{i} \left(\frac{\partial \bar{p}}{\partial \xi}(x,\bar{\zeta})\frac{\partial p}{\partial x}(x,\zeta)-\frac{\partial \bar{p}}{\partial x}(x,\bar{\zeta})\frac{\partial p}{\partial \xi}(x,\zeta)\right), \quad 
C_{\tau,\phi,2}(x, \xi):=  2\tau \phi''_{xx}\left[\frac{\partial \bar{p}}{\partial \xi}(x,\bar{\zeta});\frac{\partial p}{\partial \xi}(x,\zeta)\right] ,
\ena
where we have denoted $\zeta=\xi+i\tau\nabla \phi(x)$. 
But, we notice that for fixed $(x,\xi)$ (and when $\phi$ varies), $C_{\tau,\phi,1}(x, \xi)$ only depends on $\nabla \phi (x)$, while $C_{\tau,\phi,2}(x, \xi)$ is linear in $\phi''_{xx}(x^0)$, once $\nabla \phi(x^0)$ is fixed. So, since $\psi(x^0) = 0$ , $\nabla\psi(x^0) = \nabla\phi (x^0)$, and $\psi''_{xx}(x^0)=\phi''_{xx}(x^0)+2A^t\nabla \phi(x^0)\nabla \phi(x^0)-\frac{2}{A}Id$ we have $C_{\tau,\phi,1}(x^0, \xi) =  C_{\tau,\psi,1}(x^0, \xi)$, i.e.
\bnan
\label{e:Ctaupsi}
 \frac{1}{i}\{ \overline{p}_\psi,p_\psi \} (x^0 ,\xi) = C_{\tau,\phi,1}(x^0, \xi)+4A\tau  \left|\nabla_x\phi(x^0) \cdot \frac{\partial p}{\partial \xi}(x^0,\zeta) \right|^2 - \frac{4\tau}{A}\left|\frac{\partial p}{\partial \xi}(x^0,\zeta)\right|^2 .
\enan
In identity \eqref{e:Ctaupsi}, all terms are homogeneous of order $2m-1$ in the variables $(\tau ,\xi)$, so it is enough to prove the estimate for $(\tau, \xi) \in \mathbb{S}^{n}$, $\tau>0$. We now want this to be positive on the set $\{ (\tau, \xi) \in \mathbb{S}^{n}, \tau>0,  \xi_a = 0 , p_\phi(x^0, \xi) = 0\} = \{ (\tau, \xi) \in \mathbb{S}^{n}, \tau>0,  \xi_a = 0 , p_\psi(x^0, \xi) = 0\}$.

For this, notice first that $\left.\frac{\partial}{\partial \tau }\frac{1}{i}\{ \overline{p}_\phi,p_\phi \}\right|_{\tau=0}=2\Re \left\{ \overline{p}, \{p ,\phi \} \right\}$. Hence, we can write
\bnan
\label{eqncrochet1}
\frac{1}{i}\{ \overline{p}_\phi,p_\phi \}=\frac{1}{i}\{ \overline{p},p \}+2\tau \Re \left\{ \overline{p}, \{p ,\phi \} \right\}+\grando{\tau^2}, \quad \tau \to 0^+,
\enan
with $\grando{\tau^2}$ uniform on $(\tau, \xi) \in \mathbb{S}^{n}$. 

Moreover, by Taylor formula, we have $p_{\phi}=p+i\tau \nabla \phi\cdot\frac{\partial p}{\partial \xi}+\grando{\tau^2}=p+i\tau \left\{p,\phi\right\}+\grando{\tau^2}$, with $\grando{\tau^2}$ uniform on $(\tau, \xi) \in \mathbb{S}^{n}$. Hence, on the compact set $\{ (\tau, \xi) \in \mathbb{S}^{n}, \xi_a = 0 , p_\phi(x^0, \xi) = 0\}$, we have $p=-i\tau \left\{p,\phi\right\} +\grando{\tau^2}$.
But since $P$ is analytically principally normal, \eqref{principalnormalclassique} holds and we have $\{ \overline{p},p \}=\grando{p}$ on the compact set $\{ (\tau, \xi) \in \mathbb{S}^{n}, \xi_a = 0\}$.

In particular, on the set $\{ (\tau, \xi) \in \mathbb{S}^{n}, \xi_a = 0 , p_\phi(x^0, \xi) = 0,\tau \neq 0\}$, we have a constant $C$ so that $\left|\frac{1}{i\tau}\{ \overline{p},p \}\right|\leq C(|\left\{p,\phi\right\}|+|\tau|) $.
Getting back to \eqref{eqncrochet1}, it gives, on this set, the inequality
\bnan
\label{e:limittautozero}
\left|\frac{1}{i\tau}\{ \overline{p}_\phi,p_\phi \}-2 \Re \left\{ \overline{p}, \{p ,\phi \} \right\}\right|\leq C(|\left\{p,\phi\right\}|+|\tau|) .
\enan
Moreover, the first pseudoconvexity assumption~\eqref{e:pseudo-surface-1} and Lemma~\ref{l:fgh-three-fcts} below provide $C_1, C_2>0$ such that, on the set $\left\{\xi_a=0\right\}\cap \left\{|\xi|^2=1\right\} $, we have
\bna
2\Re \left\{ \overline{p}, \{p ,\phi \} \right\}+C_1\left(\left|p\right|^2+|\left\{p,\phi\right\}|^2\right)\geq C_2 .
\ena
This is also true by homogeneity for $|\xi|$ close to $1$ with a different constant. Hence, in the set $\{ (\tau, \xi) \in \mathbb{S}^{n}, \xi_a = 0 , p_\phi(x^0, \xi) = 0,\tau \neq 0\}$, there exist constants $\tilde C, C >0$ such that $|\left\{p,\phi\right\}|\leq \e$ and $|\tau|\leq \e$ imply 
\bna
 \frac{1}{i\tau}\{ \overline{p}_\phi,p_\phi \}\geq C_2- \tilde{C}\left(\left|p\right|^2+|\left\{p,\phi\right\}|^2+|\left\{p,\phi\right\}|+|\tau|\right)\geq C_2-C\e
\ena
where we have used $\left|p\right|\leq C|\tau|\leq C\e $ on this set.

Therefore, there exists $\e , C_3 >0$ such that in $\{ (\tau, \xi) \in \mathbb{S}^{n}, \xi_a = 0 , p_\phi(x^0, \xi) = 0,\tau \neq 0\}$, we have
\bna
|\left\{p,\phi\right\}|\leq \e,\quad  |\tau|\leq \e \Longrightarrow \frac{1}{i\tau}\{ \overline{p}_\phi,p_\phi \}\geq C_3.
\ena
We now extend $\frac{1}{i\tau}\{ \overline{p}_\phi,p_\phi \}$ to the compact set $K_{\e} =\{ (\tau, \xi) \in \mathbb{S}^{n}, \xi_a = 0 , p_\phi(x^0, \xi)= 0,0\leq \tau\leq \e\}$, by giving any positive value when $\tau=0$. We are in position to apply Lemma~\ref{l:fgh-three-fctsbis} with $g=\frac{1}{i\tau}\{ \overline{p}_\phi,p_\phi \}$ (its extension), $f=|\left\{p,\phi\right\}|^2$ and $h=\left|\frac{\partial p}{\partial \xi}(x^0,\zeta)\right|^2 $: This yields $\frac{1}{i\tau}\{ \overline{p}_\psi,p_\psi \} (x^0 ,\xi) >C$ on $K_\eps$. 

The case $\tau\geq \e$ is easier since $\frac{1}{i\tau}\{ \overline{p}_\phi,p_\phi \}$ is continuous. We apply directly Lemma~\ref{l:fgh-three-fcts} using the second pseudoconvexity assumption~\eqref{e:pseudo-surface-2}.

So, at this stage, we have proved, that there exist $C$ so that for $A$ large enough,  $\frac{1}{i\tau}\{ \overline{p}_\psi,p_\psi \} (x^0 ,\xi) >C$ on $\{ (\tau, \xi) \in \mathbb{S}^{n}, \xi_a = 0 , p_\phi(x^0, \xi)= 0,0<\tau\}$. Since, $p_\psi(x^0, \xi)=p_\phi(x^0, \xi)$, this yields
\bnan
\frac{1}{i}\{ \overline{p}_\psi,p_\psi \} (x^0 ,\xi) >0  , 
&\text{ if } p_\psi(x^0,\xi) = 0 \text{ and } \xi_a = 0,  \tau >0 .  \label{e:cond2-pseudocon}
\enan
Combining~\eqref{e:cond1-pseudocon} and~\eqref{e:cond2-pseudocon} implies that $\psi$ is a strongly pseudoconvex function in $\Omega \cap \{\xi_a = 0\}$ at $x^0$ for $P$.

\bigskip
Let us now prove the geometrical part of the lemma, i.e. Item 3. From now on, the parameter $A$ is fixed. To simplify the notation, we set $x^0 = 0$ and assume that $0\leq \rho\leq \eta$.  

Let $C_N$ a positive constant so that $\frac{1}{C_N}N(x,0)\leq |x|\leq C_N N(x,0)$.

Let us first prove \eqref{eq: psi eta B}. 
We have
$$
\frac{1}{A} |x|^2 = - \psi(x)  + x \cdot \nabla \phi (0) + A(x \cdot \nabla \phi (0))^2 + \frac12 \phi'' (0)(x,x) ,
$$
which implies 
$$
\frac{1}{A} |x|^2 \leq \eta  + x  \cdot \nabla\phi (0) + A(x\cdot \nabla \phi (0))^2 + \frac12 \phi'' (0)(x, x ), 
$$
on the set $\{\psi \geq - \eta\}$.
Moreover, the Taylor expansion of $\phi$ yields $x \cdot \nabla \phi (0)+ \frac12 \phi'' (0)(x , x) = \phi(x) + f(x)$ ,with $|f(x)|\leq \epsilon(|x|) |x|^2$, where $\epsilon : \R^+ \to \R^+$ is increasing and $\epsilon(s) \to 0^+$ as $s \to 0^+$. 
For $x \in \{\psi \geq -\eta \}\cap \{\phi \leq \rho \}$, we thus obtain
\begin{equation}
\label{eq:psiphitaylor}
\frac{1}{A} |x|^2 \leq   \eta  + \rho + A(x \cdot \nabla \phi (0))^2 + \epsilon(|x|) |x|^2 
\leq    2 \eta   + A(x \cdot \nabla \phi (0))^2 + \epsilon(|x|) |x|^2 .
\end{equation}
Moreover, for $x \in \{\psi \geq -\eta \} $, the definition of $\psi$ gives
\begin{align*}
x \cdot \nabla \phi (0) &=\psi(x) -  A(x \cdot \nabla \phi (0))^2
- \frac12 \phi'' (0)(x, x ) + \frac{1}{A} |x |^2 \\
& \geq -\eta -  (A C_0^2 + C_0/2  + 0) | x |^2  ,
\end{align*}
for $C_0 = \max(|\nabla \phi (0)| , | \phi'' (0)|)$.
Also, for $x\in  \{\phi \leq \rho \}$, we have
\begin{align*}
x\cdot \nabla \phi (0) &\leq \phi(x) + C_0 /2 | x|^2  
\leq  \rho+C_0 /2 | x |^2 \leq \eta +C_0 /2 | x |^2. 
\end{align*}
Combining the last two inequalities, we obtain for $x \in \{\phi \leq \rho \} \cap \{\psi \geq -\eta \} $,
\begin{align*}
|x \cdot \nabla\phi (0)| \leq \eta + (A C_0^2 + C_0/2) | x|^2  ,
\end{align*}
and hence 
\begin{align*}
|x \cdot \nabla\phi (0)|^2 \leq \eta^2 + 2 \eta(A C_0^2 + C_0/2) | x|^2  + (A C_0^2 + C_0/2)^2 | x|^4 .
\end{align*}
Coming back to~\eqref{eq:psiphitaylor}, this yields for $x \in \{\phi \leq \rho \} \cap \{\psi > -\eta \} $
\begin{align*}
\frac{1}{A} |x|^2 \leq 2\eta  + A\eta^2 + 2 A\eta(A C_0^2 + C_0/2) | x|^2  + A(A C_0^2 + C_0/2)^2 | x|^4 + \epsilon(|x|) |x|^2.
\end{align*}
For $x \in \{\phi \leq \rho \} \cap \{\psi \geq -\eta \} \cap B_N(0,R)$, this yields 
\begin{align*}
\frac{1}{A} |x|^2 \leq  2\eta  + A\eta^2 + 2 A\eta(A C_0^2 + C_0/2) | x|^2  + A(A C_0^2 + C_0/2)^2 (C_N R)^2| x|^2 + \epsilon(C_NR) |x|^2.
\end{align*}
Taking $R \leq R_0$ with $R_0 = R_0(A, C_0)$ sufficiently small such that 
$$
  A(A C_0^2 + C_0/2)^2 (C_N R)^2 + \epsilon(C_NR) < \frac{1}{4A} ,
$$
and $\eta <\eta_0$ sufficiently small such that 
$$
2 A\eta(A C_0^2 + C_0/2) < \frac{1}{4A},
$$
we have by absorption 
\begin{align*}
 |x|^2 \leq 2A( 2\eta  + A\eta^2).
\end{align*}
This gives $N(x,0) < \frac{R}{8}$ as soon as $\eta<\eta_0$ for $\eta_0 = \eta_0(A, C_0, R) $ sufficiently small. This concludes the proof of~\eqref{eq: psi eta B} for the chosen constants and as long as $0\leq \rho\leq \eta$.

\medskip
Let us now prove \eqref{eq: psi etatilde}. Note that performing exactly the same computation as before with $\rho=\eta=0$ and the same $R$, we obtain
\bnan
\label{caseta0}
\{\phi \leq 0 \} \cap \{\psi \geq 0 \} \cap \overline{B}_N(0,R)=\left\{0\right\} .
\enan
Assumme that the compact set $\left\{\psi\geq \eta_1 \right\} \cap \overline{B}_N(0,R)$ is nonempty, otherwise \eqref{eq: psi etatilde} is trivial. The minimum of $\phi$ on that set is reached for some point $x_m$. We have necessary $ \phi(x_m)>0$, otherwise, \eqref{caseta0} implies $x_m=0$, which is impossible since $\eta_1>0$ and $\psi(0)=0$. So, in particular, $x\in \left\{\psi\geq \eta_1 \right\} \cap \overline{B}_N(0 ,R)$ implies $\phi(x)\geq \phi(x_m)>0$. This is \eqref{eq: psi etatilde} with some apropriate $0<\rho< \min (\phi(x_m),\eta)$.

\medskip
Finally, Assertion \eqref{eq: psi eta r} is just a matter of continuity. Since $\psi(0)=0$, there exists $r>0$ such that $N(x,0)\leq r$ implies $|\psi(x)|\leq \eta_2$. 
\enp

\begin{remark}
\label{rem:limittautozero}
Note that the estimate~\eqref{e:limittautozero} implies in particular that $2 \Re \left\{ \overline{p}, \{p ,\phi \} \right\}$ is the limit as $\tau \rightarrow 0$ of $\frac{1}{i\tau}\{ \overline{p}_\phi,p_\phi \}$ on the subset $\{ (\tau, \xi) \in \mathbb{S}^{n}, \xi_a = 0 , p_\phi(x^0, \xi) =\left\{p_\phi,\phi\right\}(x^0, \xi)= 0,\tau \neq 0\}$. However, this is not used directly in the above proof.
\end{remark}
Now, thanks to Lemma \ref{l:phi-psi} and the Carleman estimate of  Theorem \ref{thmCarleman}, we have the following result.

\begin{corollary}
\label{c:geometric-setting}
Let $x^0 \in \Omega = \Omega_a \times \Omega_b \subset \R^{n_a} \times \R^{n_b}$ and $P$ be a partial differential operator on $\Omega$ of order $m$. Assume that 
\begin{itemize}
\item $P$ is analytically principally normal operator in $\{\xi_a = 0\}$ inside $\Omega$ (in the sense of Definition~\ref{def:anal principal normal});
\item there is a  function $\phi$ defined in a neighborhood of $x^0$ such that $\phi(x^0)=0$, and $\{\phi = 0\}$ is a $C^2$ strongly pseudoconvex oriented surface in the sense of Definition~\ref{def: pseudoconvex-surface}.
\end{itemize}
Then, there exists a quadratic polynomial $\psi : \Omega \to \R$, there exists $R_0>0$ such that $B(x^0 ,4 R_0) \subset \Omega$ and for any $R\in (0,R_0]$, there exist  $\eps, \delta ,\rho , r ,\mathsf{d} , \tau_0, C >0$, such that $\delta \leq \frac{\mathsf{d}}{8}$ and 
 \begin{enumerate}

\item The Carleman estimate 
\bnan
\label{Carleman-copie}
\tau \nor{Q_{\e,\tau}^{\psi}u}{m-1,\tau}^2
\leq C\left(\nor{Q_{\e,\tau}^{\psi}P u}{0}^2
+  \nor{e^{\tau(\psi-\mathsf{d})}P u}{0}^2
+ \nor{e^{\tau(\psi-\mathsf{d})}u}{m-1,\tau}^2\right)
\enan
holds for all $\tau \geq \tau_0$ and all $u \in C^\infty_0(B(x^0, 4R))$;
\item we have 
\bnan
\label{geomrhobis}
 & \Big( B(x^0 , 5R/2 ) \setminus  B(x^0 , R /2) \cap \left\{-9\delta\leq \psi\leq 2\delta \right\}\Big) \Subset \left\{\phi > 2\rho \right\}\cap B(x^0 , 3R), \\
& \label{geomrho}
 \left\{\delta/4\leq \psi\leq 2\delta \right\} \cap B(x^0 ,5R/2) \Subset \left\{\phi > 2\rho\right\} \cap B(x^0, 3R), \\
& \label{eq:Brsubsetpsi}
B(x^0,2r) \Subset \{-\delta / 2 \leq \psi \leq \delta /2\}\cap B(x^0, R) .
\enan
\end{enumerate}
\end{corollary}

 \bnp
First, Lemma \ref{l:phi-psi} furnishes the function $\psi$ for some $A$ (large enough in its proof) and $R_0>0$. Once $\psi$ is fixed, Theorem \ref{thmCarleman} yields the Carleman estimate~\eqref{Carleman-copie} for some constants $\mathsf{R}, \mathsf{d}, \tau_0 , \eps ,C$. Then, we take any $R< \min (\mathsf{R}/4,R_0/3)$ and $\delta < \min (\frac{\mathsf{d}}{8}, \eta_0/9)$. Finally, applying the conclusion of Lemma \ref{l:phi-psi} with $\eta=9\delta$, $\eta_1=\delta/4$, $\eta_2=\delta/2$ implies \eqref{geomrhobis}-\eqref{geomrho}-\eqref{eq:Brsubsetpsi}, with eventually some different constants, which concludes the proof of the corollary.
 \enp
\subsection{Step 2: Using the Carleman estimate}
\label{s:using-carleman}
From now on, we let $\Omega, x^0$, $P$ and $\phi$ be fixed as in Corollary~\ref{c:geometric-setting}. The function $\psi$, and constants $R_0$, $R:= R_0$ (that we fix now) and $\delta , \rho , r$ are provided accordingly by Corollary~\ref{c:geometric-setting}, as well as the constants $\mathsf{d} ,\tau_0, C$ of the Carleman estimate~\eqref{Carleman-copie}. We shall moreover assume that there exists $\mathsf{C}>0$ such that 
\bnan
\label{e:lambda-eq-mu}
 \frac{1}{\mathsf{C}}\mu\leq \lambda \leq \mathsf{C} \mu.
\enan
Actually, at the end of the proof, we will take $\lambda=c_1 \mu$, but we believe that to keep the notation $\lambda$ makes the presentation more readable by making a difference between $\mu$ which is the frequency and $\lambda$ which is the  regularization parameter. 
All the constants appearing in the following may depend upon the above ones.

\bigskip
Before going further, we need to introduce some cutoff functions that will be used all along the proof. We first let $\chi(s)$ be a smooth function supported in $(-8, 1)$ such that $\chi(s)=1$ for $s\in [-7, 1/2]$ and set
\begin{equation}
\label{def:chidelta}
\chi_\delta (s) := \chi  (s/\delta) .
\end{equation}
Hence, $\chi_\delta (s)$ is a smooth function supported in $(-8\delta, \delta)$ such that $\chi_\delta(s)=1$ for $s\in [-7\delta, \delta/2]$. 
We also define $\chit$ so that $\chit=1$ on $(-\infty, 3/2)$ and supported in $s\leq 2$, and denote as well $\tilde\chi_\delta (s) := \tilde\chi  (s/\delta)$.
We finally recall that the functions $\sigma_R$ and $\sigma_{2R}$ are defined in~\eqref{defb}.

\bigskip
In this part of the proof, we want to apply the Carleman estimate~\eqref{Carleman-copie} (with weight $\psi$ and constants $\mathsf{d} ,\tau_0, C$  given by Corollary~\ref{c:geometric-setting}) to the functions $\sigma_{2R}\sigma_{R,\lambda}\chit_{\delta}(\psi)\chi_{\delta,\lambda}(\psi)u$ (for any $u \in C^\infty_0(\R^n)$), which is indeed compactly supported in $B(x^0,4R)$ (according to the definition of $\sigma_{2R}$ as in~\eqref{defb}). We first need to estimate the following term
$$
\nor{Q_{\e,\tau}^{\psi}P\sigma_{2R}\sigma_{R,\lambda}\chit_{\delta}(\psi)\chi_{\delta,\lambda}(\psi)u}{0} ,
$$
that will appear in the right handside of the inequality.
Using $\supp(\chi_\delta ) \subset (-\infty ,  \delta)$ with Lemma \ref{lmboundchipsi}, together with~\eqref{e:lambda-eq-mu}, we first have
\bnan
\nor{Q_{\e,\tau}^{\psi}P\sigma_{2R}\sigma_{R,\lambda}\chit_{\delta}(\psi)\chi_{\delta,\lambda}(\psi)u}{0}&\leq& \nor{Q_{\e,\tau}^{\psi}\sigma_{2R}\sigma_{R,\lambda}\chit_{\delta}(\psi)\chi_{\delta,\lambda}(\psi)Pu}{0} \nonumber \\
&& +\nor{Q_{\e,\tau}^{\psi}[\sigma_{2R}\sigma_{R,\lambda}\chit_{\delta}(\psi)\chi_{\delta,\lambda}(\psi),P]u}{0} \nonumber\\
&\leq &C\mu^{1/2}e^{C\frac{\tau^2}{\mu}}e^{\delta \tau }
\nor{Pu}{B(x^0,4R)} \nonumber\\
&&+\nor{Q_{\e,\tau}^{\psi}[\sigma_{2R}\sigma_{R,\lambda}\chit_{\delta}(\psi)\chi_{\delta,\lambda}(\psi),P]u}{0} .
\label{eq:estimP}
\enan
The main task now consists in estimating the term containing the commutator, that we put in the following Lemma.
\begin{lemma}
\label{lemmacommut}
With the previous notations and assumptions, for any $\vartheta\in C^{\infty}_0(\R^n)$  such that $\vartheta(x)=1$ on a neighborhood of $\left\{\phi\geq 2\rho\right\}\cap B(x^0,3R)$, there exist $C>0$, $c>0$ and $N>0$ such that we have the estimate
\bnan
\nor{Q_{\e,\tau}^{\psi}[\sigma_{2R}\sigma_{R,\lambda}\chit_{\delta}(\psi)\chi_{\delta,\lambda}(\psi),P]u}{0}
& \leq &
 C e^{2\delta \tau} \nor{ M^{2\mu}_{\lambda}\vartheta_{\lambda} u}{m-1}   \nonumber \\
 & & +C\mu^{1/2}\tau^N\left(e^{-8\delta \tau}+e^{-\frac{\eps\mu^2}{8\tau}}+e^{-c\mu}e^{\delta \tau}\right)e^{C\frac{\tau^2}{\mu}}e^{\delta \tau }\nor{u}{m-1}
\label{lmeq:estimcommutat}
\enan
for any $u\in C^{\infty}_0(\R^n)$,  $\mu\geq 1$, $\lambda$ such that~\eqref{e:lambda-eq-mu} holds and $\tau \geq 1$. 
\end{lemma}
\bnp
The operator $P$ can be written $P=\sum_{|\alpha|\leq m}p_{\alpha}(x)\partial^{\alpha}$, with $p_{\alpha}$ smooth and analytic in $x_a$ in a neighborhood of $B(x^0,4R) \subset \Omega$. By the Leibniz rule, we have
\begin{multline*}
p_{\alpha}(x)\partial^{\alpha}\left(\sigma_{2R}\sigma_{R,\lambda} \chit_{\delta}(\psi)\chi_{\delta,\lambda}(\psi)u\right)\\
=p_{\alpha}(x)\sum_{ \alpha_1+\alpha_2+\alpha_3+\alpha_4+\alpha_5=\alpha}C_{(\alpha_i)} \partial^{\alpha_1}(\chi_{\delta,\lambda}(\psi))\partial^{\alpha_2}(\sigma_{2R})\partial^{\alpha_3}(\sigma_{R,\lambda}) \partial^{\alpha_4}(\chit_{\delta}(\psi))\partial^{\alpha_5}u.
\end{multline*}
The commutator $[\chit_{\delta}(\psi)\chi_{\delta,\lambda}(\psi)\sigma_{2R}\sigma_{R,\lambda},P]$ consists in all terms in the sum where at least one of the $\alpha_i$ is non zero, for $i=1$, $2$, $3$ or $4$. Hence, we can split it in a sum of  differential operators of order $m-1$ as
 $$[P,\sigma_{2R}\sigma_{R,\lambda} \chit_{\delta}(\psi)\chi_{\delta,\lambda}(\psi)] = B_1 + B_2 + B_3+B_4,$$
where 
\begin{enumerate}
\item $B_1$ contains the terms with $\alpha_1\neq 0 $ and $\alpha_2=\alpha_4=0$;
\item $B_2$ contains some terms with $\alpha_2\neq 0$;
\item $B_3$ contains the terms with $\alpha_3\neq 0$ and $\alpha_1=\alpha_2=\alpha_4=0$;
\item $B_4$ contains some terms with $\alpha_4\neq 0$.
\end{enumerate}
Note that some terms could belong to several categories, and that all terms are supported in $\{ \psi\leq 2 \delta \} \cap B(x^0 ,4 R)$.  More precisely, we have
\begin{enumerate}
\item $B_1$ consists in terms where there is at least one derivative on $\chi_{\delta,\lambda}(\psi)$ and none on $\sigma_{2R}$ and $\chit_{\delta}(\psi)$. According to the definition of $\chi$ and \eqref{def:chidelta}, there are only two possibilities for the localization of a derivative of $\chi_{\delta}$. Since we have $\chi_{\delta,\lambda}'= \frac{1}{\delta}(\chi')_{\delta,\lambda}$, then $\partial^{\alpha_1}(\chi_{\delta,\lambda}(\psi))$ with $\alpha_1\neq 0$ can be decomposed in two categories of terms: we shall use the notation $\chi_{\delta,\lambda}^{-}$ for those terms supported in $[-8\delta,-7\delta]$ and $\chi_{\delta,\lambda}^{+}$ for those supported in $[\delta/2,\delta]$. Hence, the term $B_1$ is a sum of generic terms of the form 
\bna
B_{\pm}= b_\pm(x) \d^{\gamma}
=f \sigma_{2R} \d^{\beta}(\sigma_{R,\lambda}) \chi_{\delta,\lambda}^{\pm}(\psi) \chit_{\delta}(\psi)\d^{\gamma} ,
\ena
where $|\beta|,|\gamma|\leq m-1$, $f\in C^{\infty}_0(\R^n)$ is analytic in $x_a$ in $B(x^0,4R)$, and $\chi_{\delta}^{\pm}$ is a derivative of $\chi_{\delta}$ (with the above convention for the superscript $\pm$). The function $f$ actually contains some terms coming from $p_{\alpha}$ and some derivatives of $\psi$. 
Notice that in the absence of regularization (i.e. the subscript $\lambda$), $B_+$ would be supported in 
$$
 \Big( \left\{\delta/2\leq \psi\leq \delta \right\} \cap B(x^0 ,2R) \Big)\subset \Big( \{\phi > 2\rho \}\cap \{ \psi\leq \delta \} \cap B(x^0 ,2 R) \Big),
 $$
 and $B_-$ in $\left\{-8 \delta\leq \psi\leq -7\delta \right\} \cap B(x^0 ,2R)$.
 
\item $B_2$ consists in terms where there is at least one derivative on $\sigma_{2R}$. Hence, $B_2$ is a sum of generic terms of the form
\bna
\check{B}_{2}= b_2(x) \d^{\gamma}
=\widetilde{b} \d^{\beta}(\sigma_{R,\lambda}) (\chi^{(k)})_{\delta,\lambda}(\psi)\d^{\gamma} ,
\ena
where $k, |\beta|,|\gamma| \leq m-1$, the function $\widetilde{b}$ is smooth supported in $B(x^0 , 4R ) \setminus  B(x^0 , 2R )$ and 
$\widetilde{b}$ contains derivatives of $\sigma_{2R}$, some terms of $p_{\alpha}(x)$, and potentially some derivatives of $\psi$ or $\chit_{\delta}(\psi)$. 
\item $B_3$ consists in terms where there is at least one derivative on $\sigma_{R,\lambda}$ and none on $\chi_{\delta,\lambda}(\psi)$, $\chit_{\delta}(\psi)$ and $ \sigma_{2R}$. Hence, $B_3$ is a sum of generic terms of the form
\bna
\check{B}_{3}= b_3(x) \d^{\gamma}
= f \sigma_{2R} \d^{\beta}(\sigma_{R,\lambda})  \chi_{\delta,\lambda}(\psi)\chit_{\delta}(\psi)\d^{\gamma} ,
\ena
where $f$ is smooth in $(x_a,x_b)$, analytic in $x_a$ in a neighborhood of $B(x^0,4R)$, $|\beta|\geq 1$ and $|\beta|,|\gamma| \leq m-1$.
Notice also that in the absence of regularization (i.e. the subscript $\lambda$), $\check{B}_3$ would be supported in 
 $$
\Big(  \left\{-8\delta\leq \psi\leq \delta \right\} \cap B(x^0 ,2R)\setminus B(x^0 ,R) \Big)
\subset 
\Big(\{\phi > 2\rho \}\cap \{ \psi\leq \delta \} \cap B(x^0 , 2R) \Big).
 $$
\item $B_4$ consists in terms where there is at least one derivative on $\chit_{\delta}(\psi)$. Hence, $B_4$ is a sum of generic terms of the form
\bna
\check{B}_{4}= b_4(x) \d^{\gamma}
=\widetilde{b} \d^{\beta}(\sigma_{R,\lambda}) (\chi^{(k)})_{\delta,\lambda}(\psi)\d^{\gamma}
\ena
where $k, |\beta|,|\gamma| \leq m-1$ and the function $\widetilde{b}$ is smooth supported in $B(x^0 , 4R )\cap \left\{\psi\in [3\delta/2,2\delta]\right\}$ and $\widetilde{b}$ contains derivatives of $\sigma_{2R}$, some terms from $p_{\alpha}(x)$, and some derivatives of $\psi$ or $\chit_{\delta}(\psi)$. 
\end{enumerate}

Now, proving an estimate of the last term in~\eqref{eq:estimP} consists in estimating successively the associated expressions with the generic terms $B_{\pm}$, $\check{B}_{2}$, $\check{B}_{3}$, $\check{B}_{4}$; the final estimate then follows as the LHS of ~\eqref{lmeq:estimcommutat} is bounded by a finite sum of such terms.

\bigskip
\noindent\textbf{Estimating $B_-$.}
Starting with $B_-$, we have, using Lemma \ref{lmboundchipsi} applied to $\chi_{\delta}^-$,
\bnan
\label{estimation B-}
\nor{Q_{\e,\tau}^{\psi}B_- u}{0}\leq \nor{e^{\tau \psi}B_- u}{0}\leq C_\delta \lambda^{1/2}e^{-7 \delta \tau}e^{\frac{\tau^2}{\lambda}} \nor{u}{m-1}\leq C \mu^{1/2}e^{-7 \delta \tau}e^{C\frac{\tau^2}{\mu}} \nor{u}{m-1}.
\enan

\noindent\textbf{Estimating $B_2$.}
Concerning $B_2$, we use Lemma \ref{lmboundchipsi} applied to $\chi^{(k)}_{\delta}$ and Lemma \ref{lmsuppdisjoint} applied to $\widetilde{b}$ and $\d^{\beta}(\sigma_R)$ where $\supp(\widetilde{b}) \cap \supp(\sigma_R) = \emptyset$.
This yields 
\bnan
\label{e:B_2 final}
\nor{Q_{\e,\tau}^{\psi}B_2 u}{0}\leq \nor{e^{\tau \psi}B_2 u}{0}\leq C_\delta \lambda^{1/2}e^{\delta \tau}e^{\frac{\tau^2}{\lambda}}e^{-c\lambda} \nor{u}{m-1}\leq C \mu^{1/2}e^{\delta \tau}e^{C\frac{\tau^2}{\mu}}e^{-c\mu} \nor{u}{m-1}.
\enan

\noindent\textbf{Estimating $B_4$.}
For $B_4$, we use $e^{\tau\psi} \leq e^{2\delta \tau}$ and $\left|(\chi^{(k)})_{\delta,\lambda}(\psi)\right|\leq Ce^{-c\lambda} $ on $\left\{\psi\in [3\delta/2,2\delta]\right\}$ thanks to Lemma \ref{lmsuppdisjoint} applied to $\chi^{(k)}$ and $\mathds{1}_{[3\delta/2,2\delta]}$.
This yields 
\bnan
\label{e:B_4 final}
\nor{Q_{\e,\tau}^{\psi}B_4 u}{0}\leq \nor{e^{\tau \psi}B_4 u}{0}\leq C_\delta e^{2\delta \tau}e^{-c\lambda} \nor{u}{m-1}\leq Ce^{2\delta \tau}e^{-c\mu} \nor{u}{m-1}.
\enan

\noindent\textbf{First estimates on $B_+$ and $B_3$.}
Concerning $B_{\star}$, with $\star=+$ or $\star= 3$, we have
\bna
\nor{Q_{\e,\tau}^{\psi}B_{\star} u}{0}
& = & \nor{e^{-\eps \frac{|D_a|^2}{2\tau}}e^{\tau \psi} B_{\star} u}{0}
\leq \nor{e^{-\eps \frac{|D_a|^2}{2\tau}}M^{\mu}_{\lambda} e^{\tau \psi} B_{\star} u}{0}
+ \nor{e^{-\eps \frac{|D_a|^2}{2\tau}}(1- M^{\mu}_{\lambda}) e^{\tau \psi} B_{\star} u}{0} \nonumber \\
& \leq & \nor{M^{\mu}_{\lambda} e^{\tau \psi} B_{\star} u}{0}
+ C\left(e^{-\frac{\eps\mu^2}{8\tau}}+e^{-c\mu}\right)\nor{ e^{\tau \psi} B_{\star} u}{0} \nonumber \\
& \leq & \nor{M^{\mu}_{\lambda} e^{\tau \psi} B_{\star} u}{0}
+ C\lambda^{1/2}\left(e^{-\frac{\eps\mu^2}{8\tau}}+e^{-c\mu}\right) e^{C\frac{\tau^2}{\mu}} e^{\delta \tau }\nor{u}{m-1} ,
\ena
where the second inequality comes from the application of Lemma \ref{lemma:expAmu} and the third from Lemma \ref{lmboundchipsi}.

Next, concerning the term with $\nor{M^{\mu}_{\lambda} e^{\tau \psi} B_{\star} u}{0}$, we have $B_{\star} = b_\star \d^{\gamma}$ where $\star$ is either $+$ or $3$. So, we can estimate
\bna
\nor{M^{\mu}_{\lambda}  e^{\tau \psi} B_{\star} u}{0}
&\leq& \nor{M^{\mu}_{\lambda}  e^{\tau \psi}b_\star (1 - M^{2\mu}_{\lambda} ) \d^{\gamma} u}{0}
+ \nor{M^{\mu}_{\lambda}  e^{\tau \psi} b_\star M^{2\mu}_{\lambda}\d^{\gamma} u}{0} ,
\ena
where 
$$\nor{M^{\mu}_{\lambda} e^{\tau \psi}b_\star (1 - M^{2\mu}_{\lambda} ) \d^{\gamma} u}{0} \leq C\tau^N e^{C\frac{\tau^2}{\mu}}e^{2\delta \tau - c \mu}\nor{ u}{m-1} , $$
 according to Lemma~\ref{lemma: analytic Chambertin} applied in the specific case of \eqref{commutweaker}. Note that we use that $f\sigma_{2R}=f$ in a neighborhood of $B(x^0,2R) \supset \supp(\sigma_R)$, and $f\sigma_{2R}$ is therefore analytic on a neighborhood of this set. Next we have
\bna
 \nor{M^{\mu}_{\lambda} e^{\tau \psi} b_\star M^{2\mu}_{\lambda}\d^{\gamma} u}{0}
 & \leq & 
  \nor{e^{\tau \psi} b_\star M^{2\mu}_{\lambda}\d^{\gamma} u}{0} .
\ena 
Combining the four above estimates, we now have 
\bnan
\label{estimation tildeB 1}
\nor{Q_{\e,\tau}^{\psi}B_{\star} u}{0}
& \leq &  \nor{e^{\tau \psi} b_\star M^{2\mu}_{\lambda}\d^{\gamma} u}{0}+
C\mu^{1/2}\tau^N\left(e^{-\frac{\eps\mu^2}{8\tau}}+e^{\delta \tau }e^{-c\mu}\right)e^{C\frac{\tau^2}{\mu}}e^{\delta \tau }\nor{u}{m-1} .
\enan
Now, to estimate the first term of the RHS, we will distinguish whether $\star=+$ or $3$, using the geometry of the "almost" location of each $b_{\star}$.

\medskip
\noindent\textbf{Estimating $B_+$.} We have to treat  terms of the form
\bna
B_{+}=b_+\d^{\gamma}=f  \widetilde{\widetilde{b}}_{\lambda} \chi_{\delta,\lambda}^{+}(\psi)\chit_{\delta}(\psi) \d^{\gamma} ,
\ena
where $\widetilde{\widetilde{b}}=\d^{\beta}(\sigma_{R})$, $|\beta|\leq m-1$, is supported in $B(x^0,2R)$ and $f\in C^{\infty}_0(\R^n)$. We decompose $\R^n$ as 
\bna
\R^n&=&O_1\cup O_2 \cup O_3, \quad \text{ with}\\
O_1 &=&  \left\{\psi\notin [\delta/4,2\delta]\right\}\cap B(x^0,5R/2) ,\\
O_2& =&  B(x^0,5R/2)^c ,\\
O_3 &=&   \left\{\psi\in [\delta/4,2\delta]\right\}\cap B(x^0,5R/2) .
\ena
On $O_1$, since $\chi^+_{\delta}$ is supported in $[\delta/2,\delta]$ and using Lemma \ref{lmsuppdisjoint} with $f_2=\mathds{1}_{[\delta/4,2\delta]^c}$, we have $\left|\chi_{\delta,\lambda}^+(\psi)\right| \leq e^{-c\lambda }$. Moreover, we have $e^{\tau\psi}\leq e^{2\delta\tau}$ on the support of $\chit_{\delta}$. Hence, we obtain
\bna
\nor{e^{\tau \psi} b_+ M^{2\mu}_{\lambda}\d^{\gamma} u}{L^2(O_1)}\leq C e^{-c\lambda }e^{2\delta\tau}\nor{u}{m-1}\leq C e^{-c\mu }e^{2\delta\tau}\nor{u}{m-1}.
\ena
On $O_2$, using Lemma \ref{lmsuppdisjoint} with $f_2 = \mathds{1}_{O_2}$ and $f_1 = \widetilde{\widetilde{b}}$ and then Lemma \ref{lmboundchipsi}, we get 
\bna
\nor{e^{\tau \psi} b_+ M^{2\mu}_{\lambda}\d^{\gamma} u}{L^2(O_2)}\leq C \lambda^{1/2}e^{-c\lambda }e^{\delta \tau}e^{\frac{\tau^2}{\lambda}}\nor{u}{m-1}\leq C \mu^{1/2}e^{-c\mu }e^{\delta \tau}e^{C\frac{\tau^2}{\mu}}\nor{u}{m-1}.
\ena

Using \eqref{geomrho}, we can find a smooth cutoff function $ \tilde{\vartheta}$ such that $ \tilde{\vartheta}=1$ on a neighborhood of $O_3$ and supported in $\left\{\phi > 2\rho \right\}\cap B(x^0 , 3R) $. So, for $\lambda$ large enough, we have $\tilde{\vartheta}_{\lambda}\geq 1/2$ on $O_3$. Moreover, we have $|e^{\tau\psi}|\leq e^{2\delta \tau}$ on $O_3$, and thus, we obtain
\bna
\nor{e^{\tau \psi} b_+ M^{2\mu}_{\lambda}\d^{\gamma} u}{L^2(O_3)} 
& \leq & e^{2\delta \tau}\nor{b_+ M^{2\mu}_{\lambda}\d^{\gamma} u}{L^2(O_3)}\leq Ce^{2\delta \tau}\nor{M^{2\mu}_{\lambda}\d^{\gamma} u}{L^2(O_3)} \\
& \leq & C e^{2\delta \tau}\nor{\tilde{\vartheta}_{\lambda} M^{2\mu}_{\lambda}\d^{\gamma} u}{L^2(O_3)}\leq C e^{2\delta \tau}\nor{\tilde{\vartheta}_{\lambda} M^{2\mu}_{\lambda}\d^{\gamma} u}{L^2}.
\ena

Let $\tilde{\tilde{\vartheta}}\in C^{\infty}_0$ such that $\tilde{\tilde{\vartheta}} = 1$ on a neighborhood of $\supp(\tilde{\vartheta})$ and supported in $\left\{\phi > 2\rho \right\}\cap B(x^0 , 3R)$.  This is possible since $\supp \tilde{\vartheta} \subset \left\{\phi > 2\rho \right\}\cap B(x^0 , 3R) $. In particular, since $\vartheta =1$ on $\left\{\phi > 2\rho \right\}\cap B(x^0 , 3R)$ by the assumption, we have $\vartheta =1$ in a neighborhood of $\supp \tilde{\tilde{\vartheta}}$. 
 Then, according to Lemma \ref{lmsupHK} and the properties of $\tilde{\tilde{\vartheta}}$, we have 
\bna
 \nor{\tilde{\vartheta}_{\lambda} M^{2\mu}_{\lambda}\d^{\gamma} u}{L^2}
& \leq &  \nor{\tilde{\tilde{\vartheta}}_{\lambda}  M^{2\mu}_{\lambda}u}{m-1} + e^{- c \lambda }\nor{ u}{m-1},
\ena
and then
\bna
\nor{\tilde{\tilde{\vartheta}}_{\lambda}  M^{2\mu}_{\lambda}u}{m-1}
& \leq &  \nor{ M^{2\mu}_{\lambda}\vartheta_{\lambda} u}{m-1} + Ce^{- c \mu }\nor{ u}{m-1},
\ena
according to Lemma~\ref{l:commutateurnon}.

Combining the previous estimates with \eqref{estimation tildeB 1}, we have obtained
\bnan
\label{estimfinB+}
\nor{Q_{\e,\tau}^{\psi}B_{+} u}{0}
& \leq &  Ce^{2\delta \tau}\nor{ M^{2\mu}_{\lambda}\vartheta_{\lambda} u}{m-1}+
 C\mu^{1/2}\tau^N\left(e^{-\frac{\eps\mu^2}{8\tau}}+e^{-c\mu}e^{\delta \tau}\right)e^{C\frac{\tau^2}{\mu}}e^{\delta \tau }\nor{u}{m-1}
\enan
\medskip
\noindent\textbf{Estimating $B_3$.} We now treat terms of the form
\bna
B_{3}=b_3\d^{\gamma}=f\widetilde{\widetilde{b}}_{\lambda}  \chi_{\delta,\lambda}(\psi)\chit_{\delta}(\psi)\d^{\gamma} , 
\ena
where $\widetilde{\widetilde{b}}=\d^{\beta}(\sigma_{R})$, with $|\beta|\geq 1$, is supported in $B(x^0,2R)\setminus B(x^0,R)$ and $f\in C^{\infty}_0(\R^n)$. We decompose $\R^n$ as 
\bna
\R^n&=&O_1'\cup O_2' \cup O_3', \quad \text{ with} \\
O_1'&=& \left\{\psi\notin [-9\delta,2\delta]\cap \left\{|x-x^0|\in [R/2,5R/2]\right\}\right\} , \\
O_2' &=& \left\{|x-x^0|\notin [R/2,5R/2]\right\} ,\\
O_3' &=& \left\{\psi\in [-9\delta,2\delta]\cap \left\{|x-x^0|\in [R/2,5R/2]\right\}\right\} .
\ena
On $O_1' \cap \supp(\chit_{\delta}(\psi))$, we have $e^{\tau\psi}\left|\chi_{\delta,\lambda}(\psi)\right| \leq e^{-c \lambda }e^{2\delta\tau}$ as a consequence of Lemma~\ref{lmsuppdisjoint} with $f_2=\mathds{1}_{[-9\delta,2\delta]^c}$, since $\chi_{\delta}$ is supported in $[-8 \delta,\delta]$. We thus obtain
\bna
\nor{e^{\tau \psi} b_3 M^{2\mu}_{\lambda}\d^{\gamma} u}{L^2(O_1')}\leq C e^{-c \lambda }e^{2\delta\tau}\nor{u}{m-1}\leq C e^{-c\mu }e^{2\delta\tau}\nor{u}{m-1}.
\ena
On $O_2'$, using Lemma \ref{lmsuppdisjoint} with $f_2= \mathds{1}_{O_2'}$ and $f_1 = \widetilde{\widetilde{b}}$ and using the support of $\chit_{\delta}(\psi)$, we get 
\bna
\nor{e^{\tau \psi} b_3 M^{2\mu}_{\lambda}\d^{\gamma} u}{L^2(O_2')}\leq C e^{-c\lambda }e^{2\delta \tau}\nor{u}{m-1}\leq C e^{-c\mu }e^{2\delta \tau}\nor{u}{m-1}.
\ena

Using \eqref{geomrhobis}, we can find a function $\tilde{\vartheta}$ such that $ \tilde{\vartheta}=1$ on a neighborhood of $O_3'$ and supported in $\left\{\phi > 2\rho \right\}\cap B(x^0 , 3R) $. So, for $\lambda$ large enough, we have $\tilde{\vartheta}_{\lambda}\geq 1/2$ on $O_3'$. Moreover, we have $|e^{\tau\psi}|\leq e^{2\delta \tau}$ on $O_3'$. This yields
\bna
\nor{e^{\tau \psi} b_3 M^{2\mu}_{\lambda}\d^{\gamma} u}{L^2(O_3')} &\leq & e^{2\delta \tau}\nor{b_3 M^{2\mu}_{\lambda}\d^{\gamma} u}{L^2(O_3')}\leq Ce^{2\delta \tau}\nor{M^{2\mu}_{\lambda}\d^{\gamma} u}{L^2(O_3')}\\
& \leq & C e^{2\delta \tau}\nor{\tilde{\vartheta}_{\lambda} M^{2\mu}_{\lambda}\d^{\gamma} u}{L^2(O_3')}.
\ena

We can then finish the estimates for $B_3$ as for $B_+$ to get, combining the above estimates with \eqref{estimation tildeB 1}, 
\bnan
\label{estimfinB3}
\nor{Q_{\e,\tau}^{\psi}B_{3} u}{0}
& \leq &  Ce^{2\delta \tau}\nor{ M^{2\mu}_{\lambda}\vartheta_{\lambda} u}{m-1}
+ C\mu^{1/2}\tau^N\left(e^{-\frac{\eps\mu^2}{8\tau}}+e^{\delta \tau}e^{-c\mu}\right)e^{C\frac{\tau^2}{\mu}}e^{\delta \tau }\nor{u}{m-1}
\enan

\medskip
Combining~\eqref{estimation B-},~\eqref{e:B_2 final}, ~\eqref{estimfinB+} and \eqref {estimfinB3}, this concludes the estimate of the commutator~\eqref{lmeq:estimcommutat} and the proof of Lemma \ref{lemmacommut}.
\enp
\begin{remark}
\label{rklowerunif}
In the special case of terms $p_{\alpha}(x_b)\partial^{\alpha}$, that is some coefficients independent on $x_a$, we can have some better estimates uniform in the size of $p_{\alpha}$
\bna
\nor{Q_{\e,\tau}^{\psi}[\sigma_{2R}\sigma_{R,\lambda}\chit_{\delta}(\psi)\chi_{\delta,\lambda}(\psi),p_{\alpha}(x_b)\partial^{\alpha}]u}{0}&=&\nor{p_{\alpha}(x_b)Q_{\e,\tau}^{\psi}[\sigma_{2R}\sigma_{R,\lambda}\chit_{\delta}(\psi)\chi_{\delta,\lambda}(\psi),\partial^{\alpha}]u}{0}\\
&\leq &\nor{p_{\alpha}}{L^{\infty}}\nor{Q_{\e,\tau}^{\psi}[\sigma_{2R}\sigma_{R,\lambda}\chit_{\delta}(\psi)\chi_{\delta,\lambda}(\psi),\partial^{\alpha}]u}{0}.
\ena
Also, for $\alpha=0$, that is for a potential $V(x_b)$, we have $[\sigma_{2R}\sigma_{R,\lambda}\chi_{\delta,\lambda}(\psi) \chit_{\delta}(\psi),V]=0$, so this term does not give any contribution.

This will be useful in particular for getting estimates uniform to lower order perturbation.

Moreover, if $p_{\alpha}$ is only analytic in $x_a$ and bounded in $x_b$, all estimates of the commutator remain valid. Indeed, we only use Lemma \ref{lemma: analytic Chambertin} for $k=0$ which remains true in that setting. 
\end{remark}
\medskip
Now, we are ready to apply the Carleman estimate~\eqref{Carleman-copie} to obtain the estimate of the following lemma. 
\begin{lemma}
With the previous notations and assumptions, for any $\vartheta\in C^{\infty}_0(\R^n)$  such that $\vartheta(x)=1$ on a neighborhood of $\left\{\phi> 2\rho\right\}\cap B(x^0,3R)$, there exist $\mu_0>0$, $C>0$, $c>0$ and $N>0$ such that we have the estimate
\bnan
\tau \nor{Q_{\e,\tau}^{\psi}\sigma_{2R}\sigma_{R,\lambda}\chi_{\delta,\lambda}(\psi)\chit_{\delta}(\psi)u}{m-1,\tau}
& \leq  &
C \mu^{1/2}e^{C\frac{\tau^2}{\lambda}}e^{\delta\tau}\nor{Pu}{B(x^0,4R)}
+ C e^{2\delta \tau} \nor{ M^{2\mu}_{\lambda}\vartheta_{\lambda} u}{m-1}   \nonumber \\
 & & + C\mu^{1/2}\tau^N\left( e^{-8\delta \tau}+e^{-\frac{\eps\mu^2}{8\tau}}+ e^{\delta \tau -c\mu}\right)e^{C\frac{\tau^2}{\mu}}e^{\delta \tau }\nor{u}{m-1}  
 \label{eq:carlemantronque} .
\enan
for any $u\in C^{\infty}_0(\R^n)$,  $\mu\geq\mu_0$, $\lambda$ such that~\eqref{e:lambda-eq-mu} holds and $\tau \geq \tau_0$. 
\end{lemma}
\bnp
We only need to estimate the last two terms in the RHS of Carleman estimate~\eqref{Carleman-copie} (the first term being estimated in~\eqref{eq:estimP} and Lemma~\ref{lemmacommut}).
Since we have chosen $\delta \leq \frac{\mathsf{d}}{8}$, we have that $\delta \leq \mathsf{d}-7\delta$ so that the support of $\chi_\delta$ gives using again Lemma \ref{lmboundchipsi}, for $\tau \geq \tau_0$, $\frac{1}{\mathsf{C}}\mu\leq \lambda \leq \mathsf{C} \mu$, 
\bnan
\label{eq:estimepsi-d}
\nor{e^{\tau(\psi-\mathsf{d})}\sigma_{2R}\sigma_{R,\lambda}\chi_{\delta,\lambda}(\psi)\chit_{\delta}(\psi) u}{m-1,\tau}
& \leq & C \lambda^{1/2}\tau^{m-1} e^{-7 \delta \tau} e^{\frac{\tau^2}{\lambda}}\nor{u}{m-1} \nonumber\\
& \leq & C \mu^{1/2}\tau^{m-1} e^{-7 \delta \tau} e^{C\frac{\tau^2}{\mu}}\nor{u}{m-1}.
\enan
We also need to estimate the term $e^{\tau(\psi-\mathsf{d})}P \sigma_{2R}\sigma_{R,\lambda}\chi_{\delta,\lambda}(\psi)\chit_{\delta}(\psi)u$: we have
\bnan
\label{estimPuexp}
\nor{e^{\tau(\psi-\mathsf{d})}P \sigma_{2R}\sigma_{R,\lambda}\chi_{\delta,\lambda}(\psi)\chit_{\delta}(\psi)u}{0}&\leq &\nor{e^{\tau(\psi-\mathsf{d})} \sigma_{2R}\sigma_{R,\lambda}\chi_{\delta,\lambda}(\psi)\chit_{\delta}(\psi)Pu}{0} \nonumber\\
&& +\nor{e^{\tau(\psi-\mathsf{d})}[\sigma_{2R}\sigma_{R,\lambda}\chi_{\delta,\lambda}(\psi)\chit_{\delta}(\psi),P]u}{0}\nonumber\\
&\leq & Ce^{-\tau \mathsf{d}} \lambda^{1/2}e^{\delta \tau}e^{\frac{\tau^2}{\lambda}}\left(\nor{Pu}{L^2(B(x^0,4R))}+\nor{u}{m-1}\right)\nonumber\\
&\leq & C\mu^{1/2}e^{-7\delta \tau}e^{C\frac{\tau^2}{\mu}}\left(\nor{Pu}{L^2(B(x^0,4R))}+\nor{u}{m-1}\right)
\enan
where we have used several times Lemma \ref{lmboundchipsi} to $\chi_{\delta,\lambda}(\psi)$ or some of its derivatives of order less than $m-1$.
So, the Carleman estimate~\eqref{Carleman-copie} applied to $\sigma_{2R}\sigma_{R,\lambda}\chi_{\delta,\lambda}(\psi)\chit_{\delta}(\psi)u$, together with~\eqref{eq:estimP},~\eqref{lmeq:estimcommutat},~\eqref{eq:estimepsi-d} and ~\eqref{estimPuexp}  gives for all $\tau_0 \leq \tau$,  $\mu$ large enough, $\lambda$ such that~\eqref{e:lambda-eq-mu} holds, the sought estimate~\eqref{eq:carlemantronque}.
\enp
\subsection{Step 3: A complex analysis argument}
\label{s:complex-analysis-arg}
The purpose of this part is to transfer the information given by the Carleman estimate to some estimates on the low frequencies of the function and conclude the proof of Theorem~\ref{th:alpha-unif}. The presence of the non local regularizing term $e^{-\frac{\e|D_a|^2}{2\tau}}$ makes this task more intricate than in the usual case and imposes to work by duality.
Following~\cite{Tataru:95, Hor:97, Tataru:99, Tatarunotes}, the idea is to proceed with the following three steps:
\begin{enumerate}
\item We make a kind of foliation along the level sets of $\psi$: if we want to measure $u$, we rather define the distribution $h_f=\psi_*(fu)$ by $\langle h_f , w\rangle_{\E'(\R), C^\infty(\R)}=\langle fu , w(\psi) \rangle_{\E'(\R^n), C^\infty(\R^n)}$ and estimate it for any test function $f$. Heuristically, $h_f(s)$ is the integral of $fu$ on the level set $\left\{\psi(x)=s\right\}$.
\item We notice that the Fourier transform of $h_f$ is $\widehat{h_f}(\zeta)= \langle fu , e^{-i\zeta \psi} \rangle$ and can be extended to the complex domain if $u$ is compactly supported. In particular, on the imaginary axis, $\widehat{h_f}(i \tau)= \langle f , ue^{\tau \psi} \rangle$. Since the Carleman estimate gives information on the norm of $e^{\tau \psi}u$ for $\tau$ large, this can be translated in some information on $\widehat{h_f}$ on the upper imaginary axis. A Phragm\'en-Lindel\"of type argument allows to transfer this estimate to the (almost) whole upper plan. 
\item Finally, using a change of contour, this information can be transferred to the real axis where we can estimate the real Fourier transform $\widehat{h_f}$.  
\end{enumerate}

Note that in the problem of (qualitative) unique continuation, the third step is replaced by a Paley-Wiener type argument: a bound of exponential type for $|\widehat{h_f}(\zeta)|$ on $\C$ implies some conditions on the support of $h_f$. 
Roughly speaking, if $\psi(x)=x_1$, the problem is to transfer some information on the Laplace transform (with respect to the $x_1$ variable) $\int_{x_1\geq C}  e^{\tau x_1} fu$ (given by the Carleman estimate) to some information on the Fourier transform using complex analysis. Moreover, since the Carleman estimate only gives some information on $e^{-\frac{\e|D_a|^2}{2\tau}} e^{\tau \psi}u$, we need to add some cutoff in frequency to this reasoning. 

\bigskip

More precisely, let us define 
\bna
\eta \in C^\infty_0((-4, 1)) , \quad \eta=1 \textnormal{ in } [-1/2,1/2]  \quad \text{and }\eta_\delta (s) := \eta(s/\delta) .
\ena 
We first prove the following lemma. We then conclude this section with the end of the proof of Theorem~\ref{th:alpha-unif} by estimating the left hand-side of the estimate of the lemma.
\begin{lemma} 
\label{l:local-lem-complex}
Under the above assumptions, there exists $\tilde \tau_0 = (\|\psi\|_{L^\infty(B(x^0,4R))}+11\delta)^\frac12\tau_0>0$ such that for any $\kappa , c_1 >0$, there exists $ \beta_0 ,C,c>0$ (depending on $\delta, \psi , \mathsf{d}, \tau_0, \kappa ,c_1 , \eps$), such that for any $0<\beta<\beta_0$,  for all $\mu \geq \frac{\tilde \tau_0}{\beta}$ and $u \in C^\infty_0(\R^n)$, we have  
\begin{align*}
\nor{M^{\beta \mu} \sigma_{2R} \sigma_{R,\lambda}\chi_{\delta,\lambda}(\psi)\chit_{\delta}(\psi)\eta_{\delta,\lambda}(\psi) u}{m-1}
 \leq C  e^{-c  \mu }   ( D + \nor{u}{m-1} ) , 
\end{align*}
with 
$$D=e^{\kappa\mu}\left(\nor{ M^{2\mu}_{\lambda}\vartheta_{\lambda} u}{m-1} +\nor{Pu}{B(x^0,4R)}\right) ,  \quad \lambda =2 c_1 \mu .$$ 
\end{lemma}

\bnp
We now follow \cite[proposition~2.1]{Hor:97}. 
For any test function $f \in \S(\R^n)$, we define the following distribution (with $\beta>0$ to be chosen later on)
$$
\langle h_f , w\rangle_{\E'(\R), C^\infty(\R)} := \langle (M^{\beta\mu} f) \sigma_{2R}\sigma_{R,\lambda}\chi_{\delta,\lambda}(\psi)\chit_{\delta}(\psi) u , w(\psi) \rangle_{\E'(\R^n), C^\infty(\R^n)} .
$$
We choose the particular test functions $w = \eta_{\delta , \lambda}$, and want to estimate the quantity
\begin{align*}
 \langle h_f , \eta_{\delta,\lambda} \rangle_{\E'(\R), C^\infty(\R)} & =  \langle (M^{\beta \mu} f) \sigma_{2R}\sigma_{R,\lambda}\chi_{\delta,\lambda}(\psi)\chit_{\delta}(\psi) u , \eta_{\delta,\lambda}(\psi) \rangle_{\E'(\R^n), C^\infty(\R^n)} \\
& =  \langle M^{\beta \mu} \sigma_{2R}\sigma_{R,\lambda}\chi_{\delta,\lambda}(\psi)\chit_{\delta}(\psi)\eta_{\delta,\lambda}(\psi) u, f \rangle_{\S'(\R^n), \S(\R^n)} ,
\end{align*} 
uniformly with respect to $f$ to finally obtain an estimate on $\| M^{\beta\mu} \sigma_{2R}\sigma_{R,\lambda}\chi_{\delta,\lambda}(\psi)\chit_{\delta}(\psi)\eta_{\delta,\lambda}(\psi) u\|_{m-1}.$
As the Fourier transform of a compactly supported distribution, $\hat{h}_f$ is an entire function satisfying 
\begin{align*}
\hat{h}_f (\zeta) 
&= \langle (M^{\beta\mu} f) \sigma_{2R}\sigma_{R,\lambda}\chi_{\delta,\lambda}(\psi)\chit_{\delta}(\psi) u , e^{- i \zeta \psi} \rangle_{\E'(\R^n), C^\infty(\R^n)} \\
&= \langle \sigma_{2R}\sigma_{R,\lambda}\chi_{\delta,\lambda}(\psi)\chit_{\delta}(\psi) u , e^{- i \zeta \psi} (M^{\beta\mu} f) \rangle_{\E'(\R^n), C^\infty(\R^n)} \\
&= \langle e^{- i \zeta \psi} \sigma_{2R} \sigma_{R,\lambda}\chi_{\delta,\lambda}(\psi)\chit_{\delta}(\psi) u ,  (M^{\beta\mu} f) \rangle_{\E'(\R^n), C^\infty(\R^n)}  ,
 \quad \zeta \in \C .
\end{align*}
Using $\supp(\sigma_{2R}) \subset B(x^0, 4R)$, we have the {\em a priori} estimate
\begin{align}
\label{e:estim h sur tout C}
|\hat{h}_f (\zeta) |
 & = |\langle e^{- i \zeta \psi} \sigma_{2R} \sigma_{R,\lambda}\chi_{\delta,\lambda}(\psi)\chit_{\delta}(\psi) u ,  (M^{\beta\mu} f) \rangle_{\E'(\R^n), C^\infty(\R^n)} | \nonumber  \\
 & \leq  \|e^{- i \zeta \psi} \sigma_{2R} \sigma_{R,\lambda}\chi_{\delta,\lambda}(\psi)\chit_{\delta}(\psi) u \|_{m-1}  \|(M^{\beta\mu} f) \|_{1-m} \nonumber \\
& \leq C \langle  |\zeta| \rangle^{m-1} e^{|\Im (\zeta)| \|\psi\|_{L^\infty(B(x^0,4R))}} \|u\|_{m-1}  \|f\|_{1-m}, \quad \zeta \in \C .
\end{align}
Next, for $\zeta \in \R$, we have 
\begin{align}
\label{e:estim h sur R}
|\hat{h}_f (\zeta) |
 = |\langle e^{- i \zeta \psi} \sigma_{2R} \sigma_{R,\lambda}\chi_{\delta,\lambda}(\psi)\chit_{\delta}(\psi) u ,  (M^{\beta\mu} f) \rangle_{\E'(\R^n), C^\infty(\R^n)} | 
\leq C \langle \zeta \rangle^{m-1} \|u\|_{m-1}  \|f\|_{1-m}, \quad \zeta \in \R .
\end{align}

Finally, for $\zeta \in i \R^+$, $\zeta =i \tau$ with $\tau >0$, we have 
\begin{align*}
|\hat{h}_f (i \tau)|
&= |\langle (M^{\beta\mu} f) , e^{\tau \psi}\sigma_{2R}\sigma_{R,\lambda}\chi_{\delta,\lambda}(\psi)\chit_{\delta}(\psi) u \rangle_{C^\infty(\R^n), \E'(\R^n)} |\\
&= |\langle e^{\frac{\eps}{2\tau}|D_a|^2}(M^{\beta\mu} f) , e^{- \frac{\eps}{2\tau}|D_a|^2} e^{\tau \psi} \sigma_{2R}\sigma_{R,\lambda}\chi_{\delta,\lambda}(\psi)\chit_{\delta}(\psi) u \rangle_{\S(\R^n), \S'(\R^n)} | \\
&\leq \| e^{\frac{\eps}{2\tau}|D_a|^2}M^{\beta\mu} f\|_{1-m} \| e^{- \frac{\eps}{2\tau}|D_a|^2} e^{\tau \psi} \sigma_{2R}\sigma_{R,\lambda}\chi_{\delta,\lambda}(\psi)\chit_{\delta}(\psi) u  \|_{m-1} \\
&\leq e^{\frac{\eps}{2\tau}\beta^2\mu^{2}}\| f\|_{1-m} \|Q_{\e,\tau}^{\psi}\sigma_{2R}\sigma_{R,\lambda}\chi_{\delta,\lambda}(\psi)\chit_{\delta}(\psi) u  \|_{m-1} ,
\end{align*}
as $|\xi_a|\leq \beta\mu$ on $\supp (m^{\beta\mu})$. Using \eqref{eq:carlemantronque}, we obtain for all $ \tau \geq \tau_0$, $\mu \geq 1$, $\frac{1}{\mathsf{C}}\mu\leq \lambda \leq \mathsf{C} \mu$,
\bna
|\hat{h}_f (i \tau)|
& \leq  &
Ce^{\frac{\eps}{2\tau}\beta^2\mu^{2}}\| f\|_{1-m}  \Bigg(
 \mu^{1/2}e^{C\frac{\tau^2}{\mu}}e^{\delta\tau}\nor{Pu}{B(x^0,4R)}
+ e^{2\delta \tau} \nor{ M^{2\mu}_{\lambda}\vartheta_{\lambda} u}{m-1}   \nonumber \\
 & & + C\mu^{1/2}\tau^N\left( e^{-8\delta \tau}+e^{-\frac{\eps\mu^2}{8\tau}}+ e^{\delta \tau -c\mu}\right)e^{C\frac{\tau^2}{\mu}}e^{\delta \tau }\nor{u}{m-1}\Bigg).
\ena
Now, we choose 
$$\lambda=2c_1\mu ,$$ and to simplify the notation we write, for $\kappa>0$,
$$D=e^{\kappa\mu}\left(\nor{ M^{2\mu}_{\lambda}\vartheta_{\lambda} u}{m-1} +\nor{Pu}{B(x^0,4R)}\right) .$$ 
With this notation, we have
\bnan
\label{e:grosse-estim-hat-h}
|\hat{h}_f (i \tau)|
& \leq  &
Ce^{\frac{\eps}{2\tau}\beta^2\mu^{2}}\| f\|_{1-m}  \Bigg(
 \mu^{1/2}e^{C\frac{\tau^2}{\mu}}e^{\delta\tau}e^{-\kappa \mu}D
+ e^{2\delta \tau}  e^{-\kappa \mu}D  \nonumber \\
 & & +  \mu^{1/2}\tau^N\left( e^{-8\delta \tau}+e^{-\frac{\eps\mu^2}{8\tau}}+ e^{\delta \tau -c\mu}\right)e^{C\frac{\tau^2}{\mu}}e^{\delta \tau }\nor{u}{m-1}\Bigg)  \nonumber \\ 
& \leq  &
C\mu^{1/2}\tau^N e^{\frac{\eps}{2\tau}\beta^2\mu^{2}}e^{C\frac{\tau^2}{\mu}}e^{2\delta \tau }(D+\nor{u}{m-1})\| f\|_{1-m}  \left(
 e^{-c \mu}
 +e^{-\frac{\eps\mu^2}{8\tau}}+ e^{-9\delta \tau}\right) ,
\enan
where the new constant $c>0$ may depend on $\kappa$.

We now come back to the quantity we want to estimate:
\begin{align*}
\langle M^{\beta\mu} \sigma_{2R} \sigma_{R,\lambda}\chi_{\delta,\lambda}(\psi)\chit_{\delta}(\psi)\eta_{\delta,\lambda}(\psi) u , f \rangle_{\S'(\R^n), \S(\R^n)} = \langle h_f , \eta_{\delta,\lambda} \rangle_{\E'(\R), C^\infty(\R)}
 = \int_\R \hat{h}_f (\zeta) \hat{\eta}_{\delta ,\lambda}(- \zeta) d\zeta .
\end{align*} 
As $\eta_\delta \in C^{\infty}_0(-4 \delta, \delta)$, the function  $\hat{\eta}_\delta$ is holomorphic in the lower complex half-plane together with the estimate
$$
|\hat{\eta}_\delta(\zeta)| \leq C e^{-4 \delta \Im(\zeta)} , \quad \text{for } \Im(\zeta) \leq 0 ,
$$
that is,
\bnan
|\hat{\eta}_\delta(- \zeta)| \leq C e^{4 \delta \Im(\zeta)} , \quad \text{for } \Im(\zeta) \geq 0 ,\\
|\hat{\eta}_{\delta,\lambda}(- \zeta)|=|e^{-\frac{\zeta^2}{\lambda}}\hat{\eta}_{\delta}(- \zeta)| \leq C e^{\frac{\Im(\zeta)^2-\Re(\zeta)^2}{\lambda}}e^{4 \delta \Im(\zeta)} , \quad \text{for } \Im(\zeta) \geq 0 \label{e:eta Im+},
\enan

For a constant $0<d\leq 1$ (beware that this $d$ is not the same as $\mathsf{d}$ appearing in the Carleman estimate) to be chosen later on, we split the integral in three parts according to
$$
\int_\R \hat{h}_f (\zeta) \hat{\eta}_{\delta,\lambda} (- \zeta) d\zeta = I_- + I_0 + I_+ , 
$$
with
$$
I_- := \int_{-\infty}^{-d \mu} \hat{h}_f (\zeta) \hat{\eta}_{\delta,\lambda} (- \zeta) d\zeta , \quad 
I_0 :=  \int_{-d \mu}^{d \mu} \hat{h}_f (\zeta) \hat{\eta}_{\delta,\lambda} (- \zeta) d\zeta , \quad 
I_+ :=  \int_{d \mu}^{+\infty} \hat{h}_f (\zeta) \hat{\eta}_{\delta,\lambda} (- \zeta) d\zeta .
$$
According to~\eqref{e:estim h sur R} and~\eqref{e:eta Im+}, we have, for $\mu \geq 1$, $\lambda = 2c_1 \mu$,
\bnan
\label{e:estimIpm}
|I_\pm|
& \leq & C \int_{d \mu}^{+\infty} e^{-\frac{|\zeta|^2}{\lambda}} \langle \zeta \rangle^{m-1} \|u\|_{m-1}  \|f\|_{1-m}d \zeta \leq C\mu^{2m} e^{-d^2 \frac{\mu^2}{\lambda}} \|u\|_{m-1}  \|f\|_{1-m} \nonumber \\
& \leq & C_d e^{-cd^2\mu } \|u\|_{m-1}  \|f\|_{1-m} .
\enan
So the main problem is to estimate $I_0$.
For this, let us define 
$$
\mathcal{H}(\zeta) = \mu^{-1/2} (\zeta + i)^{-N} e^{ i 2\delta \zeta} \hat{h}_f (\zeta) .
$$
From~\eqref{e:grosse-estim-hat-h}, we have the estimate on the imaginary axis for all $ \tau \geq \tau_0$, for $\mu \geq 1$, $\lambda = 2c_1 \mu$, 
\bna
|\mathcal{H}(i \tau)| & \leq & 
C e^{\frac{\eps}{2\tau}\beta^2\mu^{2}}e^{C\frac{\tau^2}{\mu}} (D+\nor{u}{m-1})\| f\|_{1-m}  \left(
 e^{-c \mu}
 +e^{-\frac{\eps\mu^2}{8\tau}}+ e^{-9\delta \tau}\right).
\ena
Moreover, \eqref{e:estim h sur tout C} implies (we can assume $N\geq m-1$ without loss of generality)
$$
|\mathcal{H}(\zeta) |
 \leq C  e^{|\Im (\zeta)| \left(2 \delta  + \|\psi\|_{L^\infty(B(x^0,4R))}\right)} \|u\|_{m-1}  \|f\|_{1-m}, \quad \zeta \in \C , ~\Im(\zeta)\geq 0.
$$
Next, we define $H:= \frac{\mathcal{H}}{c_0} $, with 
\bnan
\label{e:defk0}
c_0 =  C (D+\nor{u}{m-1})\| f\|_{1-m} , 
\enan
and apply Lemma~\ref{l:holomorphic} below to the function $H$.
 
This Lemma implies the existence of $d_0>0$ (depending only on $\delta,\kappa,  \|\psi\|_{L^\infty(B(x^0,4R))} , \eps$ and the constants $C,c$ appearing in the exponents of the estimates of $\mathcal{H}(i\tau)$) such that for any $d<d_0$, there exists $ \beta_0>0$, (depending on the same parameters, together with $d$)
 such that for any $0<\beta<\beta_0$,  for all $\mu \geq \frac{\tilde \tau_0}{\beta} := \frac{\tau_0 (\|\psi\|_{L^\infty(B(x^0,4R))}+11\delta)^\frac12}{\beta}$, 
we have
$$
|\mathcal{H}(\zeta)|\leq c_0 e^{-8 \delta \Im(\zeta)} ,   \quad \text{on } \overline{Q}_1 \cap \{\frac{d}{4} \mu\leq |\zeta| \leq 2 d \mu \} ,
$$
with $Q_1 = \R_+^* + i \R_+^*$. The same procedure leads to the same estimate if $Q_1$ is replaced by the set $ \R_-^* + i \R_+^*$, and hence, by the whole $\C_+ = \{\zeta \in \C , \Im(\zeta)\geq 0\}$.
Coming back to $\hat{h}_f$, we obtain
\bnan
\label{e:estimhf}
|\hat{h}_f (\zeta) | \leq c_0 \mu^{1/2}  \langle |\zeta | \rangle^{N}  e^{-6\delta \Im(\zeta)}\leq c_0 \mu^{N+1/2} e^{-6\delta \Im(\zeta)} ,   \quad \text{on } \C_+ \cap \{\frac{d}{4} \mu\leq |\zeta| \leq 2 d \mu \} .
\enan
where $c_0$ is defined in~\eqref{e:defk0}. 

We now come back to $I_0$. The function $\hat{h}_f (\zeta) \hat{\eta}_{\delta,\lambda} (- \zeta)$ being holomorphic in $\C_+$, we make the following change of contour in the complex plane:
$$
I_0 =   \int_{\Gamma_+^V} \hat{h}_f (\zeta) \hat{\eta}_{\delta , \lambda} (- \zeta) d\zeta
+ \int_{\Gamma^H} \hat{h}_f (\zeta) \hat{\eta}_{\delta , \lambda}(- \zeta)d\zeta
+  \int_{\Gamma_-^V} \hat{h}_f (\zeta) \hat{\eta}_{\delta , \lambda}(- \zeta)d\zeta ,
$$ 
where the contours (oriented counterclockwise) are defined by
\bna
& \Gamma_\pm^V = \{ \Re(\zeta) = \pm d \mu, 0 \leq \Im(\zeta) \leq d \mu/2 \} , \\
& \Gamma^H = \{ -d \mu \leq \Re(\zeta) \leq d \mu , \Im(\zeta) = d\mu /2 \} .
\ena
with $d \in ]0, d_0[$ still to be chosen later on.

\begin{figure}[h!]
  \begin{center}
    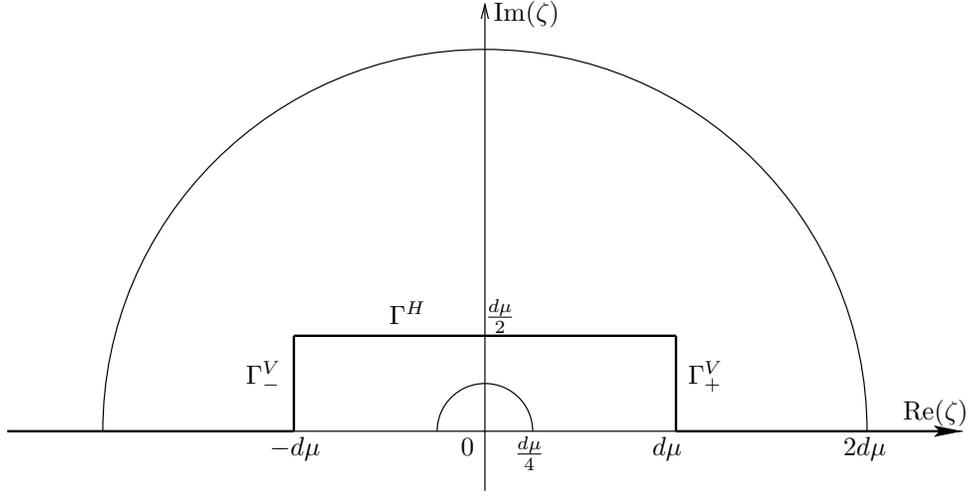 
    \caption{Coutours of integration}
 \end{center}
\end{figure}
 Since $\Gamma_+^V \cup \Gamma^H \cup \Gamma_-^V \subset  \C_+ \cap \{\frac{d}{4} \mu\leq |\zeta| \leq 2 d \mu\}$ and $\lambda = c_1\mu$, estimates~
\eqref{e:eta Im+} and \eqref{e:estimhf} yields the estimate
\bna
|\hat{h}_f (\zeta) \hat{\eta}_{\delta , \lambda} (- \zeta) | & \leq & c_0 \mu^{N+1/2}e^{-6\delta \Im(\zeta)} e^{\frac{\Im(\zeta)^2-\Re(\zeta)^2}{2 c_1\mu}}e^{4 \delta \Im(\zeta)} , \quad \zeta \in \Gamma_+^V \cup \Gamma^H \cup \Gamma_-^V 
\\
 & \leq & c_0 \mu^{N+1/2} e^{-2 \delta \Im(\zeta)}  e^{\frac{\Im(\zeta)^2-\Re(\zeta)^2}{2c_1\mu}}
 , \quad \zeta \in \Gamma_+^V \cup \Gamma^H \cup \Gamma_-^V 
\ena
Using that $\frac{3d^2}{4} \mu^2\leq \Re(\zeta)^2-\Im(\zeta)^2 \leq  d^2 \mu^{2}$ for $\zeta \in \Gamma_+^V \cup \Gamma_-^V$ we now obtain
\bna
|\hat{h}_f (\zeta) \hat{\eta}_{\delta, \lambda} (- \zeta) |  & \leq & c_0 \mu^{N+1/2} e^{-2 \delta \Im(\zeta)} 
e^{-\frac{3d^2\mu }{8 c_1}}
 , \quad \zeta \in \Gamma_+^V  \cup \Gamma_-^V .
\ena
On $\Gamma^H$, we have $\Im(\zeta)= d \mu/2$, so , we can estimate
\bna
|\hat{h}_f (\zeta) \hat{\eta}_{\delta, \lambda} (- \zeta) | 
 & \leq & c_0 \mu^{N+1/2} e^{-  \delta d \mu}  e^{\frac{d^2}{8c_1}\mu}
 , \quad \zeta \in  \Gamma^H  
\ena
Now, we can fix $0< d \leq \min(4 c_1\delta, d_0)$ so that  we have $e^{-  \delta d \mu}  e^{\frac{d^2}{8 c_1}\mu} \leq C e^{-c  \mu}$ (for some $0<c \leq 2c_1 \de^2$ ).
As a consequence, we have 
\bnan
\label{e:estimateI_01}
|I_0| = \left| \int_{\Gamma_+^V \cup \Gamma^H \cup \Gamma_-^V} \hat{h}_f (\zeta) \hat{\eta}_{\delta, \lambda} (- \zeta) d\zeta \right|
&\leq & c_0 \mu^{N+1/2} | \Gamma_+^V \cup \Gamma^H \cup \Gamma_-^V | e^{-c \mu }  \nonumber \\
&\leq & C  e^{-c \mu}
 (D+\nor{u}{m-1})\| f\|_{1-m} ,
\enan
for any $0<\beta<\beta_0$, for all $\mu \geq \max(C,\frac{\tilde \tau_0}{\beta}) $  (as $| \Gamma_+^V \cup \Gamma^H \cup \Gamma_-^V | = C d \mu$).

This, together with \eqref{e:estimIpm} yields, for any $0<\beta<\beta_0$, for all $\mu \geq \frac{\tilde \tau_0}{\beta} $,
\bna
\left| \langle M^{\beta\mu} \sigma_{2R} \sigma_{R,\lambda}\chi_{\delta,\lambda}(\psi)\chit_{\delta}(\psi)\eta_{\delta,\lambda}(\psi) u ,f  \rangle_{\S'(\R^n), \S(\R^n)} \right| 
& =& \left| \int_\R \hat{h}_f (\zeta) \hat{\eta}_{\delta, \lambda} (- \zeta) d\zeta \right| \\
&\leq& C  e^{-c  \mu } 
 (D+\nor{u}{m-1})\| f\|_{1-m} .
\ena
The constants being uniform with respect to $f \in\S(\R^n)$, this provides by duality the estimate
\begin{align*}
\nor{M^{\beta\mu} \sigma_{2R} \sigma_{R,\lambda}\chi_{\delta,\lambda}(\psi)\chit_{\delta}(\psi)\eta_{\delta,\lambda}(\psi) u}{m-1}
 \leq C  e^{-c  \mu }   ( D + \nor{u}{m-1} ) ,
\end{align*}
which concludes the proof of the lemma.
\enp
With Lemma~\ref{l:local-lem-complex}, we can now conclude the proof of the local estimate of Theorem \ref{th:alpha-unif}. Lemma~\ref{l:holomorphic} and its proof are postponed to the end of the section.
\bnp[End of the proof of Theorem \ref{th:alpha-unif}.]
Using Lemma \ref{lmsuppdisjoint} with $m(2 \  \cdot )$ and $1-m(\cdot )$, we get 
$$\nor{M^{\frac{\beta \mu}{2}}_{\lambda}(1-M^{\beta\mu})}{H^{m-1}(\R^n)\to H^{m-1}(\R^n)}\leq Ce^{-c\lambda}.$$
Hence, applying Lemma~\ref{l:local-lem-complex}, we obtain, for any $0<\beta<\beta_0$, for all $\mu \geq \frac{\tilde \tau_0}{\beta} $ and $ \lambda=2c_1 \mu$,
\bnan
\label{estimfinCarl}
\nor{M^{\beta\mu} \sigma_{2R} \sigma_{R,\lambda}\chi_{\delta,\lambda}(\psi)\chit_{\delta}(\psi)\eta_{\delta,\lambda}(\psi) u}{m-1}
 & \leq &
 \nor{M^{\frac{\beta \mu}{2}}_{\lambda} (1-M^{\beta\mu})\sigma_{2R} \sigma_{R,\lambda}\chi_{\delta,\lambda}(\psi)\chit_{\delta}(\psi)\eta_{\delta,\lambda}(\psi) u}{m-1} \nonumber\\
& & \quad + \nor{M^{\frac{\beta \mu}{2}}_{\lambda} M^{\beta\mu}\sigma_{2R} \sigma_{R,\lambda}\chi_{\delta,\lambda}(\psi)\chit_{\delta}(\psi)\eta_{\delta,\lambda}(\psi) u}{m-1} \nonumber\\
& \leq & C  e^{-c  \mu }   ( D + \nor{u}{m-1} ).
\enan
Using Lemma \ref{l:commutateurnon}, estimate \eqref{estimfinCarl} and the definition of $r$ in Corollary~\ref{c:geometric-setting}, we get for any $0<\beta<\beta_0$, for all $\mu \geq \frac{\tilde \tau_0}{\beta} $ and $ \lambda=2c_1 \mu$,
\bnan
\label{estimbfin}
\nor{M^{\frac{\beta \mu}{4}}_{\lambda} \sigma_{r,\lambda} u}{m-1}&\leq& \nor{\sigma_{r,\lambda} M^{\frac{\beta \mu}{2}}_{\lambda}  u}{m-1}  + C e^{-c \mu}\nor{ u }{m-1} \nonumber \\
&\leq &
\nor{\sigma_{r,\lambda} M^{\frac{\beta \mu}{2}}_{\lambda} \sigma_{2R} \sigma_{R,\lambda}\chi_{\delta,\lambda}(\psi)\chit_{\delta}(\psi)\eta_{\delta,\lambda}(\psi) u}{m-1} \nonumber \\
&& +\nor{\sigma_{r,\lambda} M^{\frac{\beta \mu}{2}}_{\lambda} \big(1 - \sigma_{2R} \sigma_{R,\lambda}\chi_{\delta,\lambda}(\psi)\chit_{\delta}(\psi)\eta_{\delta,\lambda}(\psi) \big)u}{m-1} +C e^{-c \mu}\nor{ u }{m-1}\nonumber  \\
& \leq & C  e^{-c  \mu }   ( D + \nor{u}{m-1} ) +\nor{\sigma_{r,\lambda} M^{\frac{\beta \mu}{2}}_{\lambda} \big(1 - \sigma_{2R} \sigma_{R,\lambda}\chi_{\delta,\lambda}(\psi)\eta_{\delta,\lambda}(\psi) \big)u}{m-1} .
\enan
We know that $\sigma_R=\chi_{\delta}(\psi)=\chit_{\delta}(\psi)=\eta_\delta(\psi)=1$ on a neighborhood of $\supp(\sigma_r)$ according to~\eqref{eq:Brsubsetpsi} and the properties of $\chi$, $\chit_{\delta}$ and $\eta$. So, we can select $\Pi\in C^{\infty}_0(\R^n)$ such that $\Pi=1$ on a neighborhood of $\supp(\sigma_r)$ and such that $\sigma_{2R}=\sigma_R=\chi_{\delta}(\psi)=\chit_{\delta}(\psi)=\eta_\delta(\psi)= 1$ on an neighborhood of $\supp(\Pi)$.
Now, we have
\begin{multline}
\nor{\sigma_{r,\lambda} M^{\frac{\beta \mu}{2}}_{\lambda} \big(1 - \sigma_{2R} \sigma_{R,\lambda}\chi_{\delta,\lambda}(\psi)\chit_{\delta}(\psi)\eta_{\delta,\lambda}(\psi) \big)u}{m-1}  \\
\leq \nor{\sigma_{r,\lambda} M^{\frac{\beta \mu}{2}}_{\lambda} \big(1 - \sigma_{2R} \sigma_{R,\lambda}\chi_{\delta,\lambda}(\psi)\chit_{\delta}(\psi)\eta_{\delta,\lambda}(\psi) \big) (1-\Pi)u}{m-1} \\
+ \nor{\sigma_{r,\lambda} M^{\frac{\beta \mu}{2}}_{\lambda} \big(1 - \sigma_{2R} \sigma_{R,\lambda}\chi_{\delta,\lambda}(\psi)\chit_{\delta}(\psi)\eta_{\delta,\lambda}(\psi) \big) \Pi u}{m-1} .
\end{multline}
To estimate the first term, we use Lemma~\ref{l:commgevrey2} to obtain $\nor{\sigma_{r,\lambda} M^{\frac{\beta \mu}{2}}_{\lambda}   (1-\Pi)}{H^{m-1} \to H^{m-1}}\leq Ce^{-c\mu}$.
Concerning the second term, we have 
\begin{multline}
\nor{\sigma_{r,\lambda} M^{\frac{\beta \mu}{2}}_{\lambda} \big(1 - \sigma_{2R} \sigma_{R,\lambda}\chi_{\delta,\lambda}(\psi)\chit_{\delta}(\psi)\eta_{\delta,\lambda}(\psi) \big) \Pi u}{m-1}  \\
\leq
C\nor{\big(1 - \sigma_{2R} \sigma_{R,\lambda}\chi_{\delta,\lambda}(\psi)\chit_{\delta}(\psi)\eta_{\delta,\lambda}(\psi) \big) \Pi u}{m-1}  \leq Ce^{-c \mu}\nor{u}{m-1}
\end{multline}
where we have decomposed in the last inequality
\bna
1 - \sigma_{2R} \sigma_{R,\lambda}\chi_{\delta,\lambda}(\psi)\chit_{\delta}(\psi)\eta_{\delta,\lambda}(\psi)& = &(1-\sigma_{2R})+\sigma_{2R}(1-\sigma_{R,\lambda}) +\sigma_{2R}\sigma_{R,\lambda}(1-\chi_{\delta,\lambda}(\psi))\\
&&+\sigma_{2R}\sigma_{R,\lambda}\chi_{\delta,\lambda}(\psi)  (1 -\chit_{\delta}(\psi))+\sigma_{2R}\sigma_{R,\lambda}\chi_{\delta,\lambda}(\psi) \chit_{\delta}(\psi) (1 -\eta_{\delta,\lambda}(\psi))
\ena
and used Lemmata \ref{lmsuppdisjoint} and \ref{lmsuppdisjointchi}. These two Lemmata can be applied thanks to the geometric fact that 
$$\dist(\supp(\Pi),\left\{x \in \R^n;\sigma_{2R}(x)\neq 1 \right\})>0,  $$ and the same is true with $\sigma_{2R}$ replaced by $\sigma_R$, $\chi_{\delta}(\psi)$, $\chit_{\delta}(\psi)$ or $\eta_\delta(\psi)$. We now have the existence of $\tilde \tau_0>0$ such that for any $\kappa , c_1 >0$, there exist $ \beta_0 ,C,c>0$, such that for any $0<\beta<\beta_0$,  $\mu \geq \frac{\tilde \tau_0}{\beta} $ and $ \lambda=2c_1 \mu$, the following estimate holds:
\bna
\nor{M^{\frac{\beta \mu}{4}}_{\lambda} \sigma_{r,\lambda} u}{m-1}\leq C  e^{-c  \mu }   ( D + \nor{u}{m-1} ) , \quad D=e^{\kappa\mu}\left(\nor{ M^{2\mu}_{\lambda}\vartheta_{\lambda} u}{m-1} +\nor{Pu}{B(x^0,4R)}\right) .
\ena
This concludes the proof of Theorem~\ref{th:alpha-unif} with $\kappa' = c$, when replacing $\mu$ and $\mu_0$ by $\mu/2$ and $\mu_0/2$ respectively.
\enp
It only remains to prove Lemma \ref{l:holomorphic} below.
\begin{lemma}
\label{l:holomorphic}
Let $\delta,\kappa, R_0 , C_1, \eps, \tau_0>0$.
Then, there exists $d_0 = d_0(\delta,\kappa, R_0 , C_1, \eps)$ such that for any $d<d_0$, there exists $ \beta_0(\delta,\kappa, R_0,c_1, \eps , d)$ such that for any $0<\beta<\beta_0$ and for all $\mu \geq \frac{\tau_0 (R_0+9\delta)^\frac12}{\beta}$, we have the following statement:

For every $H$ holomorphic function in $Q_1 = \R_+^* + i \R_+^*$, continuous on $\bar Q_1$ satisfying 
\bnan
\label{hyp h bornee}
|H(i \tau)| \leq  e^{\e\frac{\beta^2}{2\tau}\mu^{2}}  e^{ C_1 \frac{\tau^2}{\mu}} \max(e^{-\kappa\mu} ,   e^{-\frac{\eps\mu^2}{8\tau}} , e^{-9\delta \tau}) \quad \text{for } \tau \in [\tau_0, + \infty)  , 
\enan
\bnan
\label{hyp h bornee bis}
|H(\zeta)| \leq e^{R_0 \Im(\zeta )} \quad \text{on } \bar Q_1  ,
\enan
  
we have
\bnan
\label{conclu h bornee}
|H(\zeta)| \leq e^{-8\delta \Im(\zeta)} \quad \text{on } \bar Q_1 \cap \{\frac{d}{4} \mu\leq |\zeta| \leq 2 d \mu \} .
\enan
\end{lemma}

The proof essentially consists in performing a scaling argument to get rid of the parameter $\mu$ and then applying Lemma~\ref{explicitgreenbis}.
\begin{proof}[Proof of Lemma~\ref{l:holomorphic}]
The function $H$ is holomorphic in $Q_1$ and $z \mapsto \log|z|$ is subharmonic on $\C^*$. As a consequence, the function 
$$
g^\mu : \zeta \mapsto \mu^{-1}\log \left|H(\mu \zeta)\right|
$$
is subharmonic on $Q_1$ (which is invariant by dilations). Assumption~\eqref{hyp h bornee} (used for $\tau \mu\in [\tau_0, + \infty)$) yields 
\bnan
\label{hyp g bornee}
g^\mu (i \tau)\leq  c_1 \tau^2 +  \frac{\e\beta^2}{\tau} +  \max( -\kappa  ,  -9\delta \tau , -\frac{\eps}{8\tau} ) , \quad {  \text{for }  \tau \in [\frac{\tau_0}{\mu}, + \infty)  },
\enan 
and Assumption~\eqref{hyp h bornee bis} yields
\bnan
\label{hyp g bornee bis}
g^\mu  (\zeta) \leq R_0 \Im(\zeta ) , \quad \text{on } \bar Q_1  .
\enan
Now, we set 
\bnan
\label{eq:defoff}
 f_1^\mu (y)= R_0 y  \mathds{1}_{[0, \frac{\tau_0}{\mu})}(y) + \mathds{1}_{[\frac{\tau_0}{\mu}, + \infty)}(y) \min \{ R_0 y , \max(- \kappa , -9\delta y , -\frac{\eps}{8 y} ) +  C_1 y^2+ \frac{\e\beta^2}{y}\} , \quad y \in \R_+.
\enan
According to Lemma \ref{explicitgreenbis}, there exists $d_0 = d_0(\delta, \kappa,R_0,\e,C_1)$ such that for any $d<d_0$, there exists $\beta_0(\delta, \kappa,R_0,d,\e,C_1)$, such that for any $0<\beta<\beta_0$, and any $\mu \geq \frac{\tau_0 (R_0+9\delta)^\frac12}{\beta }$, the function $f_1^\mu$ is continuous and the associated function $f^\mu$ given by Lemma~\ref{l:green} with $f_0 = 0$ and $f_1=f_1^{\mu}$ satisfies 
$$
f^\mu  \in C^0 (\bar Q_1) , 
\quad \Delta f^\mu = 0 \text{ and } |f^\mu (x,y)| \leq C_\mu(1+|(x,y)|)\text{ in } Q_1 , 
\quad f^\mu = f_1^\mu \text{ on } i \R_+ ,
\quad f^\mu = 0 \text{ on } \R_+
$$
together with 
$$
f^\mu (\zeta) \leq -8\delta \Im(\zeta) \quad \text{on } \bar Q_1 \cap \{\frac{d}{4}  \leq |\zeta| \leq 2 d  \} .
$$
This is 
\bnan
\label{eq:estimHmuzeta}
f^\mu (\zeta/ \mu) \leq -8\delta \Im(\zeta)/\mu \quad \text{on } \bar Q_1 \cap \{\frac{d}{4} \mu  \leq |\zeta| \leq 2 d \mu \}.
\enan
Now, as $g^\mu $ is subharmonic and $f^\mu$ harmonic, the function
$$
h^\mu(\zeta) := g^\mu (\zeta) - f^\mu(\zeta)
$$ is subharmonic too. 
As a consequence of~\eqref{hyp g bornee},~\eqref{hyp g bornee bis} and~\eqref{eq:defoff}, we have
$$
h^\mu(\zeta) \leq 0 \quad \text{on } \R_+ \cup i \R_+ .
$$
Moreover, \eqref{hyp g bornee bis} and $|f^\mu (\zeta)| \leq C(1+|\zeta|)$ also yield 
$$
h^\mu(\zeta) \leq C_\mu  + (C_\mu +R_0)|\zeta| .
$$
According to Lemma~\ref{l:phragmen}, this implies
$$
h^\mu(\zeta) \leq 0 \quad \text{on } \bar Q_1,
$$
and hence
\bna
|H(\mu \zeta)| = e^{\mu g^\mu (\zeta)} \leq e^{\mu f^\mu (\zeta)} \quad \text{on } \bar Q_1 .
\ena
Finally, coming back to \eqref{eq:estimHmuzeta}, we obtain
\bna
|H(\zeta)| \leq e^{-8\delta \Im(\zeta)} \quad \text{on } \bar Q_1 \cap \{\frac{d}{4} \mu  \leq |\zeta| \leq 2 d \mu \},
\ena
which concludes the proof of the lemma.
\end{proof}
\section{Semiglobal estimates}
\label{s:semiglobal-estimate}

\subsection{Some tools for propagating the information}
The Local Estimate of Theorem~\ref{th:alpha-unif} only provides information on the low frequency part of the function. Iterating this result alows us to propagate the low frequency information. 
In this section, we define some tools that will be useful for this iterative procedure. 
They are aimed at describing how information on the low frequency part of the solution can be deduced from one subregion to another one.
\begin{definition}
\label{defdependence}
Fix $\Omega $ be an open set of $\R^n=\R^{n_a}\times \R^{n_b}$, $P$ a differential operator of order $m$ defined in $\Omega$, and $(V_j)_{j\in J}$ and $(U_i)_{i\in I}$ two {\em finite collections of bounded open sets} of $\R^n$. We say that $(V_j)_{j\in J}$ is \textbf{under the dependence of }$(U_i)_{i\in I}$, denoted
\bna
(V_j)_{j\in J} \unlhd (U_i)_{i\in I} ,
\ena
 if for any $\vartheta_i\in C^{\infty}_0(\R^n)$ such that $\vartheta_i(x)=1$ on a neighborhood of $\overline{U_i}$, for any $\widetilde{\vartheta}_j\in C^{\infty}_0(V_j)$
and for all $\kappa, \alpha>0$, there exist $C, \kappa', \beta,\mu_0 >0$ such that for all $(\mu,u) \in [\mu_0, + \infty)\times C^{\infty}_0(\R^{n})$, we have 
\bna
\sum_{j\in J}\nor{M^{\beta\mu}_{\mu} \widetilde{\vartheta}_{j,\mu}  u}{m-1}\leq C e^{\kappa \mu}\left(\sum_{i\in I} \nor{M^{\alpha\mu}_{\mu} \vartheta_{i,\mu} u}{m-1} + \nor{Pu}{L^2(\Omega)}\right)+Ce^{-\kappa' \mu}\nor{u}{m-1} .
\ena
If $\sharp I=1$ and $U_1=U$, we write $(V_j)_{j\in J} \unlhd U$, with the same convention for $V$.
\end{definition}
The norm $\nor{\cdot}{m-1}$ being taken in $\R^n$.
\begin{remark}
\label{rkinvariance}
The definition $\unlhd$ actually depends on the splitting $\R^n=\R^{n_a}\times \R^{n_b}$, the set $\Omega$ and the operator $P$. However, in the main part of this work, $\R^n=\R^{n_a}\times \R^{n_b}$, $\Omega$ and  $P$ will be fixed, so it should not lead to confusion (in particular in the applications). The dependence of  $\unlhd$ upon these object will be mentioned when needed.

 For the applications, it is important that the function $u$ is not necessarily supported in $\Omega$. 

In the following, we will only need to use this relation $\unlhd$ in some appropriate coordinate charts. However, it will not be a problem for what we want to prove, even on a compact manifold. Indeed, we will fix some coordinate chart on an open set $\Omega\subset \R^n$ close to a point or close to a trajectory. Then, we will use the relation $\unlhd$ related to $\Omega$ to finally obtain some estimates which will be invariant by change of coordinates. 
\end{remark}
Now, we list some general properties of the relation $\unlhd$, which actually hold without using any asumption on the set $\Omega$ and the operator $P$.
\begin{proposition}
\label{propweak}
We have the following properties
\begin{enumerate}
\item \label{proptoto} If $(V_j)_{j\in J} \unlhd (U_i)_{i\in I}$ with $U_i=U$ for all $i\in I$, then $(V_j)_{j\in J} \unlhd U$.
\item \label{proptoto2} If $(V_j)_{j\in J} \unlhd (U_i)_{i\in I}$ with $U_i\subset W_i$ for all $i\in I$, then $(V_j)_{j\in J} \unlhd (W_i)_{i\in I}$.
\item \label{propinclude}If $V\subset U$ then, $V \unlhd U $. In particular, we always have $U \unlhd U$.
\item \label{propunion} $\bigcup_{i\in I} U_i \unlhd (U_i)_{i\in I}$.
\item \label{propproduit} If for any $i\in I$, $V_i\unlhd U_i$, then $(V_i)_{i\in I}\unlhd (U_i)_{i\in I}$. In particular, we always have $ (U_i)_{i\in I} \unlhd (U_i)_{i\in I}$.
\end{enumerate}
\end{proposition}
\bnp
Property \ref{proptoto} is obvious from the definition. Property \ref{proptoto2} is also immediate since $\vartheta_i(x)=1$ on a neighborhood of $W_i$ implies $\vartheta_i(x)=1$ on a neighborhood of $U_i$ since $U_i\subset W_i$.

Property \ref{propinclude} is a consequence of Lemma \ref{l:commutateurnon} applied with $\alpha\mu/2$ instead of $\mu$, $\lambda= \mu$, $f_1=\vartheta$ and $f=\widetilde{\vartheta}$. The assumptions on $\vartheta$ and $\widetilde{\vartheta}$ ensures that $f_1=1$ on a uniform neighborhood of $\supp(f)$. This gives the result with $\beta=\alpha/2$.

Property \ref{propunion} is a consequence of Lemma \ref{lemma: Gevrey Chambertinsomme} with the same parameters as before for Property \ref{propinclude}, but with $b_i=\vartheta_i$. 

Property \ref{propproduit} is almost a consequence of the definition. Actually, the only difference is that a priori, we have one $\beta_i$ for each $i\in I$. Taking the worst of the constants $C,\kappa',\mu_0$ given by the application of the definition for any $i$, it gives
\bna
\sum_{i\in I}\nor{M^{\beta_i\mu}_{\mu} \widetilde{\vartheta}_{i,\mu}  u}{m-1}\leq C e^{\kappa \mu}\left(\sum_{i\in I} \nor{M^{\alpha\mu}_{\mu} \vartheta_{i,\mu} u}{m-1} + \nor{Pu}{L^2(\Omega)}\right)+Ce^{-\kappa' \mu}\nor{u}{m-1} .
\ena
with $\vartheta_{i}=1$ on $\overline{U_i}$ and $\widetilde{\vartheta}_{i}\in C^{\infty}_0(V_i)$. But taking $2\beta=\inf \left\{\beta_i,i\in I\right\}$, we have
\bna
\nor{M^{\beta\mu}_{\mu} \widetilde{\vartheta}_{i,\mu}  u}{m-1}&\leq & \nor{M^{\beta_i\mu}_{\mu} M^{\beta\mu}_{\mu}\widetilde{\vartheta}_{i,\mu}  u}{m-1}+\nor{M^{\beta\mu}_{\mu} (1-M^{\beta_i\mu}_{\mu})\widetilde{\vartheta}_{i,\mu}  u}{m-1}\\
&\leq & \nor{M^{\beta_i\mu}_{\mu}\widetilde{\vartheta}_{i,\mu}u}{m-1}+Ce^{-c \mu}\nor{u}{m-1} ,
\ena
where we have used Lemma \ref{lmsuppdisjoint} and the properties of support of $m(\frac{\cdot}{\beta})$ and $(1-m(\frac{\cdot}{\beta_i}))$ for the last estimate. The second part comes from the combination with $U_i\unlhd U_i$ for all $i\in I$.
\enp

The relation is not clearly transitive but we have the following weaker but sufficient property:
if $(V_j)_{j\in J} \unlhd  (\widetilde{U}_i)_{i\in I}$ and $\widetilde{U}_i\Subset U_i$ with compact inclusion (that is $\overline{\widetilde{U}}_i\subset U_i$) and $(U_i)_{i\in I} \unlhd (W_k)_{k\in K }$, then, $(V_j)_{j\in J} \unlhd (W_k)_{k\in K}$ (this is proved by introducing functions $f_i\in C^{\infty}_0(U_i)$ equal to $1$ on $\widetilde{U}_i$: see the proof of Item~\ref{proptransstrong} in Proposition~\ref{propstrong} below). 

For this reason, it is convenient to introduce the following stronger property.
\begin{definition}
\label{defdependencestrong}
Given $\Omega $ an open set of $\R^n=\R^{n_a}\times \R^{n_b}$, $P$ a differential operator of order $m$ defined in $\Omega$, and $(V_j)_{j\in J}$ and $(U_i)_{i\in I}$ two {\em finite collections of bounded open sets} of $\R^n$, we say that $(V_j)_{j\in J}$ is \textbf{under the strong dependence} of $(U_i)_{i\in I}$ if there exists $\widetilde{U}_i\Subset U_i$ such that $(V_j)_{j\in J}\unlhd (\widetilde{U}_i)_{i\in I}$. In that case, we write.
\bna
(V_j)_{j\in J} \lhd (U_i)_{i\in I}.
\ena
\end{definition}
This makes the relation transitive, but it becomes more strict in the sense that we do not always have $U \lhd U$.
We sumarize again the properties of this relation.
\begin{proposition}
\label{propstrong}
We have the following properties
\begin{enumerate}
\item \label{propweakstrong} $(V_j)_{j\in J} \lhd (U_i)_{i\in I}$ implies $(V_j)_{j\in J} \unlhd (U_i)_{i\in I}$.
\item \label{proptotostrong} If $(V_j)_{j\in J} \lhd (U_i)_{i\in I}$ with $U_i=U$ for all $i\in I$, then $(V_j)_{j\in J} \lhd U$.
\item \label{propincludestrong}If $V_i\Subset U_i$ for any $i\in I$, then, $(V_i)_{i\in I} \lhd  (U_i)_{i\in I}$.
\item \label{propunionstrong} If $V_i\Subset U_i$ for any $i\in I$, then $\bigcup_{i\in I} V_i \lhd (U_i)_{i\in I}$.
\item \label{propproduitstrong} If for any $i\in I$, $V_i\lhd U_i$, then $(V_i)_{i\in I}\lhd (U_i)_{i\in I}$. In particular, if for any $i\in I$, $U_i\lhd U$, then $(U_i)_{i\in I}\lhd U$. 
\item \label{proptransstrong} The relation is transitive, that is \bna
\left[(V_j)_{j\in J} \lhd (U_i)_{i\in I} \textnormal{ and } (U_i)_{i\in I}\lhd (W_k)_{k\in K } \right]\Longrightarrow (V_j)_{j\in J} \lhd (W_k)_{k\in K }.
\ena
\end{enumerate}
\end{proposition}
\bnp
Property~\ref{propweakstrong} is obvious.
For \ref{proptotostrong}, the assumption gives some $(\widetilde{U}_i)_{i\in I}$ with $(V_j)_{j\in J} \unlhd (\widetilde{U}_i)_{i\in I}$ and $\widetilde{U}_i \Subset U$ for all $i\in I$. Since $\overline{\widetilde{U}_i} \subset U$ for all $i\in I$ and $I$ is finite, we have $\overline{\cup_{i\in I}\widetilde{U}_i}= \cup_{i\in I}\overline{\widetilde{U}_i} \subset U$. Denote $W=\cup_{i\in I}\widetilde{U}_i$. We have $\widetilde{U}_i\subset W$ for all $i\in I$, so Property \ref{proptoto2} and then Property \ref{proptoto} of the previous Lemma give $(V_j)_{j\in J} \unlhd W$ which implies $(V_j)_{j\in J} \lhd U$ since $W\Subset U$. 

For \ref{propincludestrong}, we use $ (V_i)_{i\in I} \unlhd (V_i)_{i\in I}$ from Property \ref{propproduit} of the previous Lemma and $V_i\Subset U_i$.

For \ref{propunionstrong}, we use Property \ref{propunion} of the previous Lemma, which gives $\bigcup_{i\leq I} V_i \unlhd (V_i)_{i\in I}$. This is $\bigcup_{i\leq I} V_i \lhd (U_i)_{i\in I}$ by the definition of $\lhd$. 

For \ref{propproduitstrong}, assume $V_i\unlhd \widetilde{U}_i$ with $\widetilde{U}_i\Subset U_i$. Then, Property \ref{propproduit} of the previous Lemma gives $(V_i)_{i\in I}\unlhd (\widetilde{U}_i)_{i\in I}$ which gives $(V_i)_{i\in I}\lhd (U_i)_{i\in I}$ by definition. The second part is direct by combining with Property \ref{proptotostrong}.

For \ref{proptransstrong}, the assumptions give the existence of $\widetilde{U}_i\Subset U_i$ and $\widetilde{W}_k\Subset W_k$ such that
\bna
(V_j)_{j\in J} \unlhd (\widetilde{U}_i)_{i\in I} \textnormal{ and } (U_i)_{i\in I}\unlhd (\widetilde{W}_k)_{k\in K }
\ena
Since $\widetilde{U}_i\Subset U_i$, we can pick $\chi_i\in C^{\infty}_0(U_i)$ such that $\chi_i=1$ in an neighborhood of $\overline{\widetilde{U}_i}$. Let $\alpha>0$, $\kappa>0$, and take $\vartheta_k\in C^{\infty}_0(\R^n)$ (for all $k\in K$) such that $\vartheta_k(x)=1$ on a neighborhood of $\overline{\widetilde{W}_k}$ and $\widetilde{\vartheta}_j\in C^{\infty}_0(V_j)$ (for all $j\in J$). Since we have $(U_i)_{i\in I}\unlhd (\widetilde{W}_k)_{k\in K }$ and $\chi_i\in C^{\infty}_0(U_i)$, there exist $C, \kappa', \beta,\mu_0 >0$, such that we have
\bna
\sum_{i\in I}\nor{M^{\beta\mu}_{\mu} \chi_{i,\mu}  u}{m-1}\leq C e^{\frac{\kappa}{2} \mu}\left(\sum_{k\in K} \nor{M^{\alpha\mu}_{\mu} \vartheta_{k,\mu} u}{m-1} + \nor{Pu}{L^2(\Omega)}\right)+Ce^{-\kappa' \mu}\nor{u}{m-1} .
\ena
Now, we apply the relation given by $(V_j)_{j\in J} \unlhd (\widetilde{U}_i)_{i\in I}$ with $\alpha$ replaced by the above $\beta$ and $\kappa$ replaced by $\kappa_1=\min(\kappa',\kappa)/2>0$. Since $\chi_i=1$ in an neighborhood of $\widetilde{U}_i$ and $\widetilde{\vartheta}_j\in C^{\infty}_0(V_j)$, there exist $C', \kappa'', \beta',\mu_0' >0$ such that
\bna
\sum_{j\in J}\nor{M^{\beta'\mu}_{\mu} \widetilde{\vartheta}_{j,\mu}   u}{m-1}\leq C' e^{\kappa_1 \mu}\left(\sum_{i\in I} \nor{M^{\beta\mu}_{\mu} \chi_{i,\mu} u}{m-1} + \nor{Pu}{L^2(\Omega)}\right)+C'e^{-\kappa'' \mu}\nor{u}{m-1} .
\ena
Combining the above two estimates now yields
\bna
\sum_{j\in J}\nor{M^{\beta'\mu}_{\mu} \widetilde{\vartheta}_{j,\mu}   u}{m-1}&\leq& CC' e^{(\frac{\kappa}{2}+\kappa_1) \mu}\sum_{k\in K} \nor{M^{\alpha\mu}_{\mu} \vartheta_{k,\mu} u}{m-1} + C' e^{\kappa_1 \mu}\left(1+Ce^{\frac{\kappa}{2} \mu}\right)\nor{Pu}{L^2(\Omega)}\\
&& +\left(C'e^{-\kappa'' \mu}+CC'e^{(\kappa_1-\kappa')\mu}\right)\nor{u}{m-1}.
\ena
Since $\kappa/2+\kappa_1\leq \kappa$ and $\kappa_1-\kappa'<\kappa'/2-\kappa'=-\kappa'/2<0$, it gives $(V_j)_{j\in J} \unlhd (\widetilde{W}_k)_{k\in K}$, which implies the result since $\widetilde{W}_k\Subset W_k$. 

Note that in the proofs above, we have omitted to precise each time the restriction $\mu\geq \mu_0$. Yet, all the estimates have to be taken with that restriction, taking the worst constant $\mu_0$ when several restrictions are involved.
\enp
\begin{corollary}
\label{cordominance}
Under the assumptions of Theorem \ref{th:alpha-unif}, there exists $R_0>0$ such that for any $R\in (0,R_0)$, there exists $r$, $\rho>0$ so that  we have
\bna
B(x^0,r) \unlhd \left[\left\{\phi> 2\rho\right\}\cap B(x^0,3R)\right] , \\
B(x^0,r) \lhd \left[\left\{\phi> \rho\right\}\cap B(x^0,4R)\right] .
\ena
\end{corollary}
\bnp[Proof of Corollary \ref{cordominance}]
First, we restrict $R_0$ so that $B(x^0,4R_0)\subset \Omega$. Theorem \ref{th:alpha-unif} gives some $R$, $r$, $\rho$, $\tilde{\tau}_0>0$. 

Let $\kappa$, $\alpha>0$. We apply the result with $\mu=\alpha \mu'$, $c_1=1/\alpha$ and $\kappa$ replaced by $\kappa/\alpha$ to obtain, uniformly  for $\mu'\geq \tilde{\tau_0}/(\alpha\beta)$,
\bna
\nor{M^{\beta\alpha\mu'}_{\mu'} \sigma_{r,\mu'} u}{m-1}\leq C e^{\kappa \mu'}\left(\nor{M^{\alpha\mu'}_{\mu'} \vartheta_{\mu'} u}{m-1} + \nor{Pu}{L^2(B(x^0,4R))}\right)+Ce^{-\alpha\kappa' \mu'}\left(\nor{u}{m-1}\right) .
\ena
Now, let $\widetilde{\vartheta} \in C^{\infty}_0(B(x^0,r))$. Since $\sigma_r=1$ on $B(x^0,r)$, Lemma \ref{l:commutateurnon} gives
\bna
\nor{M^{\beta\alpha\mu'/2}_{\mu'}\widetilde{\vartheta}_{\mu'} u}{m-1} \leq \nor{M^{\beta\alpha\mu'}_{\mu'} \sigma_{r,\mu'} u}{m-1}+Ce^{-c\mu'} \nor{u}{m-1}
\ena
This gives the result. The second comes from the compact inclusion of $\left[\left\{\phi> 2\rho\right\}\cap B(x^0,3R)\right]$ into
$B(x^0,r) \lhd \left[\left\{\phi> \rho\right\}\cap B(x^0,4R)\right]$.
\enp
\subsection{Semiglobal estimates along foliation by graphs}
This section is devoted to the proof of Theorem \ref{thmsemiglobal}.
Actually, this result is a corollary of the following stronger theorem, stated here in the context of zone of dependence.
\begin{theorem}
\label{thmsemiglobaldep}
Under the assumptions of Theorem \ref{thmsemiglobal}, there exists an open neighborhood $U$ of $K$ such that for any open neighborhood $\hat \omega$ of $S_0$, we have 
\bna
U\lhd \hat{\omega}.
\ena
\end{theorem}
In the present section, we first prove that Theorem \ref{thmsemiglobaldep} implies Theorem \ref{thmsemiglobal}, and then prove Theorem~\ref{thmsemiglobaldep}.
\bnp[Proof that Theorem \ref{thmsemiglobaldep} implies Theorem \ref{thmsemiglobal}]
We first apply Theorem \ref{thmsemiglobaldep} for a neighborhood $\hat{\omega}$ of $S_0$ such that $\hat \omega  \Subset \tilde{\omega}$, where $\tilde{\omega}$ is that in the statement of Theorem~\ref{thmsemiglobal}. We obtain $U\lhd \omega_1$. Take $\chi\in C^{\infty}_0(U)$ such that $\chi=1$ on a neighborhood $U_{\chi}$ of $K$, and $\varphi \in C^{\infty}_0(\tilde{\omega})$ such that $\varphi=1$ on a neighborhood of $\omega_1$. We obtain that for any $\kappa>0$, there exist $C, \beta, \kappa', \mu_0 >0$ such that for $\mu\geq \mu_0$,
\bnan
\label{estimaMmu}
\nor{M^{\beta\mu}_{\mu} \chi_{\mu}  u}{m-1}\leq C e^{\kappa \mu}\left(\nor{M^{\mu}_{\mu} \varphi_{\mu} u}{m-1} + \nor{Pu}{L^2(\Omega)}\right)+Ce^{-\kappa' \mu}\nor{u}{m-1} .
\enan
But since $\varphi \in C^{\infty}_0(\tilde{\omega})$, taking again $\widetilde{\varphi}\in C^{\infty}_0(\tilde{\omega})$, $\widetilde{\varphi}=1$ on a neighborhood of $\supp (\varphi)$, we get thanks to Lemma \ref{lmsuppdisjoint}
\bna
\nor{M^{\mu}_{\mu} \varphi_{\mu} u}{m-1} 
& \leq & \nor{ M^{\mu}_{\mu}  \widetilde{\varphi}\varphi_{\mu} u}{m-1}+\nor{(1-\widetilde{\varphi}) \varphi_{\mu} u}{m-1} \\
& \leq & \sum_{|\alpha| + |\beta| \leq m-1}\nor{ D_a^\alpha M^{\mu}_{\mu} ( D_b^\beta \widetilde{\varphi}\varphi_{\mu} u)}{0}+Ce^{-c \mu}\nor{u}{m-1} .
\ena
Next, we have
\bna
\nor{ D_a^\alpha M^{\mu}_{\mu} f}{0} & \leq & \nor{\xi_a^\alpha m_{\mu}(\xi_a/\mu)}{L^\infty(\R^{n_a})} \nor{f}{0} \\
& \leq & \mu^{|\alpha|}\nor{\xi_a^\alpha m_{\mu}(\xi_a)}{L^\infty(\R^{n_a})} \nor{f}{0} 
\leq C \mu^{|\alpha|} \nor{f}{0} ,
\ena
since the function $\xi_a \mapsto \xi_a^\alpha m_{\mu}(\xi_a)$ is uniformly bounded on $\R^{n_a}$ for $\mu \geq 1$. As a consequence, we now have
\bna
\nor{M^{\mu}_{\mu} \varphi_{\mu} u}{m-1} 
& \leq & C\sum_{|\alpha| + |\beta| \leq m-1}\mu^{|\alpha|}\nor{ D_b^\beta ( \widetilde{\varphi}\varphi_{\mu} u)}{0}+Ce^{-c \mu}\nor{u}{m-1} \\
& \leq & C\mu^{m-1}\sum_{|\beta| \leq m-1}\nor{ D_b^\beta u}{L^2(\tilde\omega)}+Ce^{-c \mu}\nor{u}{m-1}\\
& \leq & C\mu^{m-1}\nor{ u}{H_b^{m-1}(\tilde\omega)}+Ce^{-c \mu}\nor{u}{m-1}.
\ena
In the particular case where $n_a=n$, we change slightly the estimate
\bna
\nor{M^{\mu}_{\mu} \varphi_{\mu} u}{m-1}&\leq& \nor{M^{2\mu} M^{\mu}_{\mu} \varphi_{\mu} u}{m-1}+\nor{(1-M^{2\mu})M^{\mu}_{\mu} \varphi_{\mu} u}{m-1}\\
&\leq & C \mu^{s+m-1}\nor{\varphi_{\mu} u}{-s}+Ce^{-c \mu}\nor{u}{m-1}\\
&\leq &C \mu^{s+m-1}\nor{\widetilde{\varphi}\varphi_{\mu} u}{-s}+C \mu^{s+m-1}\nor{(1-\widetilde{\varphi}) \varphi_{\mu} u}{-s}+Ce^{-c \mu}\nor{u}{m-1}\\
&\leq &C \mu^{s+m-1}\nor{\widetilde{\varphi} u}{H^{-s}}+Ce^{-c \mu}\nor{u}{m-1}.
\ena
In~\eqref{estimaMmu}, the constant $\kappa>0$ is arbitrary (all other constants in that estimate depending on it): imposing $\kappa<c/2$ and noticing that $\mu^{m-1}\leq C_m e^{\kappa \mu}$, we obtain, with $c' := \min (c/2,\kappa')$,
\bnan
\label{e:numero1}
\nor{M^{\beta\mu}_{\mu} \chi_{\mu}  u}{m-1}\leq C e^{2\kappa \mu}\left(\nor{ u}{H^{m-1}_b(\tilde{\omega})} + \nor{Pu}{L^2(\Omega)}\right)+Ce^{-c' \mu}\nor{u}{m-1} .
\enan
In the analytic case, $n_a=n$, using $\mu^{s+m-1}\leq C_s e^{\kappa \mu}$, we have similarly
\bna
\nor{M^{\beta\mu}_{\mu} \chi_{\mu}  u}{m-1}\leq C e^{2\kappa \mu}\left(\nor{\widetilde{\varphi} u}{H^{-s}} + \nor{Pu}{L^2(\Omega)}\right)+Ce^{-c' \mu}\nor{u}{m-1} .
\ena
Now, let $\widetilde{\chi}\in C^{\infty}_0(U_{\chi})$ be such that $\widetilde{\chi}=1$ in a neighborhood of $K$. We have, using again  Lemma \ref{lmsuppdisjoint},
\bnan
\label{e:numero2}
\nor{\widetilde{\chi}u}{0}&\leq &\nor{\widetilde{\chi}\chi_{\mu}u}{0}+\nor{(1-\chi_{\mu})\widetilde{\chi}u}{0} \nonumber\\
&\leq & C\nor{\chi_{\mu}u}{0}+Ce^{-c\mu }\nor{u}{m-1} \nonumber\\
&\leq&  C\nor{M^{\beta\mu}_{\mu} \chi_{\mu}u}{0}+C\nor{(1-M^{\beta\mu}_{\mu}) \chi_{\mu}u}{0}+Ce^{-c\mu }\nor{u}{m-1}.
\enan
Concerning the second term in this estimate, we write
\bna
\nor{(1-M^{\beta\mu}_{\mu}) \chi_{\mu}u}{0}\leq C \sup_{(\xi_a,\xi_b)\in \R^{n_a+n_b}}\left|\frac{ (1-m_{\mu})(\frac{\xi_a}{\beta\mu})}{\left| \xi_a\right|^{m-1}+\left\langle \xi_b\right\rangle^{m-1}}\right| \nor{\chi_{\mu}u}{m-1}.
\ena
Hence, in the range $|\xi_a|\geq \beta\mu/2$ with $\mu\geq \mu_0$, we have the loose estimate
\bnan
\label{e:numero3}
\left|\frac{ (1-m_{\mu})(\frac{\xi_a}{\beta\mu})}{\left| \xi_a\right|^{m-1}+\left\langle \xi_b\right\rangle^{m-1}}\right| \leq \frac{C}{\mu^{m-1}}.
\enan
In the range $|\xi_a|\leq \beta\mu/2$, using $\dist\left( \supp (1-m(\frac{\cdot}{\beta})), \left\{|\xi_a|\leq \beta/2\right\}\right)>0$, we have $\left|(1-m_{\mu})(\frac{\xi_a}{\beta\mu})\right|\leq Ce^{-c\mu}$ according to Lemma \ref{lmsuppdisjoint}. In this range of $\xi_a$, this yields
\bna
\left|\frac{(1-m_{\mu})(\frac{\xi_a}{\beta\mu})}{\left| \xi_a\right|^{m-1}+\left\langle \xi_b\right\rangle^{m-1}}\right| \leq Ce^{-c\mu},
\ena
so that \eqref{e:numero3} holds for all $\xi_a \in \R^{n_a}$, and $\mu \geq \mu_0$.
This yields 
$\nor{(1-M^{\beta\mu}_{\mu}) \chi_{\mu}u}{0} \leq \frac{C}{\mu^{m-1}} \nor{\chi_{\mu}u}{m-1}$, which, combined with~\eqref{e:numero1} and~\eqref{e:numero2}  gives, for $\mu\geq \mu_0$,
\bna
\nor{\widetilde{\chi}u}{0}&\leq &C e^{2\kappa \mu}\left(\nor{ u}{H^{m-1}_b(\tilde{\omega})} + \nor{Pu}{L^2(\Omega)}\right)+\frac{C}{\mu^{m-1}}\nor{u}{m-1}.
\ena
Similarly, in the analytic case, we have
\bna
\nor{\widetilde{\chi}u}{0}&\leq &C e^{2\kappa \mu}\left(\nor{\widetilde{\varphi} u}{H^{-s}} + \nor{Pu}{L^2(\Omega)}\right)+\frac{C}{\mu^{m-1}}\nor{u}{m-1}.
\ena
Finally, the case $n_a=0$ is a direct consequence of \eqref{estimaMmu} since there is no regularization.

Now, we notice that the previous estimates are true for any $\Omega$ neighborhood of $K$. Denoting now by $\Omega$ the neighborhood of $K$ given by the assumptions of the Theorem, we can apply the previous estimates to an open neighborhood $\widetilde{\Omega}$ of $K$ so that $\widetilde{\Omega} \Subset\Omega$. This gives that for any $\widetilde{\omega}\subset \widetilde{\Omega}$ neighborhood of $S_0$, there exists an open set $\widetilde{U}$ neighborhood of $K$ (that we can impose included in $\widetilde{\Omega}$) so that we have the estimates 
\bnan
\label{inegchi0}
\nor{u}{L^2(\widetilde{U})}&\leq &C e^{2\kappa \mu}\left(\nor{ u}{H^{m-1}_b(\tilde{\omega})} + \nor{Pu}{L^2(\widetilde{\Omega})}\right)+\frac{C}{\mu^{m-1}}\nor{u}{m-1}.
\enan
Take $\chi_0$ supported in $\Omega$ and so that $\chi_0=1$ in $\widetilde{\Omega}$. In particular, we have $\nor{P(\chi_0 u)}{L^2(\widetilde{\Omega})}=\nor{Pu}{L^2(\widetilde{\Omega})}\leq \nor{Pu}{L^2(\Omega)}$, $\nor{\chi_0 u}{L^2(\widetilde{U})}=\nor{u}{L^2(\widetilde{U})}$, $\nor{ \chi_0 u}{H^{m-1}_b(\tilde{\omega})}=\nor{ u}{H^{m-1}_b(\tilde{\omega})}$ and $\nor{\chi_0 u}{m-1}\leq C \nor{u}{H^{m-1}(\Omega)}$. Applying inequality \eqref{inegchi0} to $\chi_0 u$ gives
\bna
\nor{u}{L^2(\widetilde{U})}&\leq &C e^{2\kappa \mu}\left(\nor{ u}{H^{m-1}_b(\tilde{\omega})} + \nor{Pu}{L^2(\Omega)}\right)+\frac{C}{\mu^{m-1}}\nor{u}{H^{m-1}(\Omega)}.
\ena
This concludes the proof of Theorem \ref{thmsemiglobal} in the general case. The end of the proof in the cases $n_a=n$ and $n_a=0$ is similar.
\enp

Now, we come to the proof of the main result of this section, namely Theorem~\ref{thmsemiglobaldep}. This proof consists in two main steps: first defining the adapted geometrical context, and second to iterate the local result in this geometric context, using an induction argument.

\bnp[Proof of Theorem~\ref{thmsemiglobaldep}]
To begin with, we choose $\omega_1 \Subset \omega_2  \Subset \hat{\omega}$ where $\omega_1$ is another open neighborhood of $S_0$.
We fix $R$ such that 
\bnan
\label{eq:defR}
2 R < \min (\dist(K , \Omega^c), \dist(\omega_1^c, S_0)) ,
\enan
define the set
$$
K^R = \bigcup_{x \in K} B(x, 2R),
$$
and pick a cutoff function
\bnan
\label{eq:defchiK}
\chi_K \in C^\infty_c(\Omega), \quad \text{such that }\chi_K = 1 \text{ on } K^R , \text{ and } \supp(\chi_K) \cap \{x_n \leq0 \} \subset \omega_1 .
\enan
Given any point $x\in K$, there exists $\eps>0$ such that $x \in S_\eps$. We denote by $R_0>0$ the constant given by Theorem \ref{th:alpha-unif} associated to the point $x$ and the function $\phi_\eps$. Next, we set 
\bnan
\label{e:defRx}
R_x:=\min(R_0/2 , R/4), 
\enan
and then 
$$
r_x := \min(r/2, 3R_x), \qquad \rho_x=\rho ,
$$
where $r, \rho>0$ are the constants given by Theorem \ref{th:alpha-unif} (and Corollary \ref{cordominance}) associated to $x$, $\phi_\eps$ and $R_x$.

For any $\e\in (0,1]$ and $x\in S_{\e}$, we have $\phi_{\e}(x)=0$. So, we can write 
\bna
S_{\e}\subset \bigcup_{x\in S_{\e}}  B(x,r_x)  ,
\ena 
and, since $S_{\e}$ is compact, we can extract a finite covering, i.e. there is a finite set of indices $I_\eps$ and a finite number of points $(x_i^\eps)_{i \in I_\eps}$, such that 
\bna
S_{\e}\subset \bigcup_{i\in I_{\e}} B(x_i^\eps,r_{x_i^\eps}) , \quad x_i^\eps \in S_\eps.
\ena
For $x_i^\eps \in S_\eps$, we rename the associated radii, setting 
$$
R_i^\eps := R_{x_i^\eps} , \qquad r_i^\eps := r_{x_i^\eps} , \qquad \rho_i^\eps := \rho_{x_i^\eps} ,
$$
and define  $$\rho_{\e}:=\min_{i \in I_{\e}} \rho_i^\eps>0.$$ Since $\phi_{\e}=0$ on $S_{\e}$, we still have 
\bna
S_{\e}\subset \left(\bigcup_{i\in I_{\e}} B(x_i^\e,r_{i}^\e) \right) 
\cap \left\{ \phi_{\e} <\rho_{\e}\right\}= : \mathcal{U}_{\e}.
\ena
The definition of $\mathcal{U}_{\e}$ is illustrated on Figure~\ref{f:def-U-eps}.
\begin{figure}[h!]
  \begin{center}
    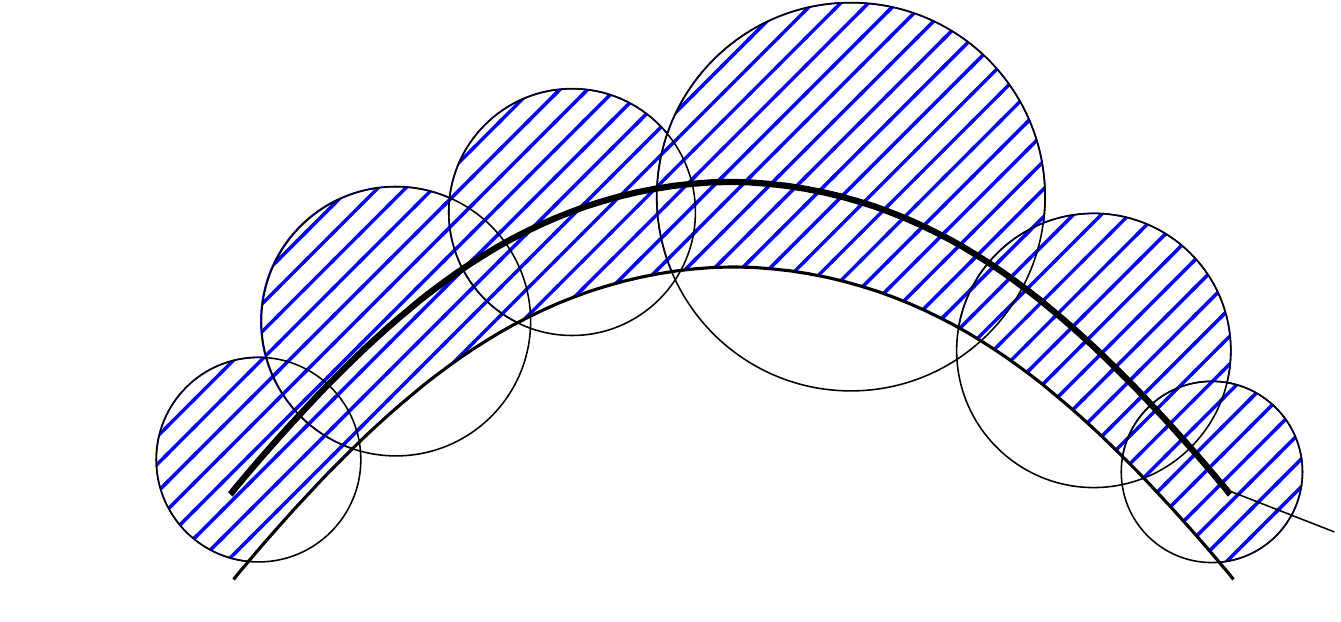 
    \caption{Definition of the set $\mathcal{U_\e}$, striped in blue}
    \label{f:def-U-eps}
 \end{center}
\end{figure}
Therefore, for $\eps \in ]0,1]$, $\mathcal{U}_{\e}$ is an open neighborhood the compact surface $S_{\e}$. Since $G$ is $C^1$, we claim that we can find $g(\e)>0$ so that 
\bnan
\label{defVj}
\mathcal{V}_{\e} := \bigcup_{\e'\in ]\e-g(\e),\e+g(\e)[}S_{\e'} \subset \mathcal{U}_{\e}
\enan
(the definition of $\mathcal{V}_{\e}$ is illustrated on Figure~\ref{f:def-V-eps}).
\begin{figure}[h!]
  \begin{center}
    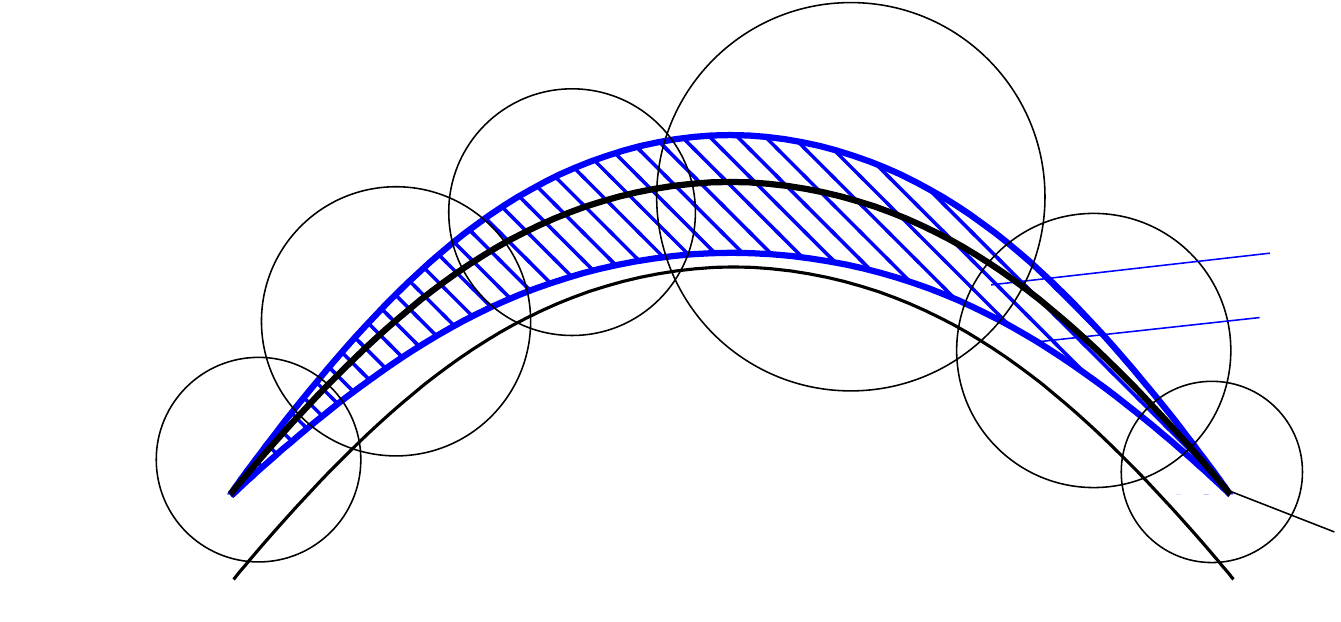 
    \caption{Definition of the set $\mathcal{V_\e}$, striped in blue}
    \label{f:def-V-eps}
 \end{center}
\end{figure}
Indeed, since $G \in C^1(\bar D \times ]0,1])$, we can find $C>0$ so that
\bna
\left|G(x',\e)-G(x',\e')\right|\leq C |\e-\e'| ,
\ena
uniformly for $x'\in\overline{D}$. In particular, if $|\e-\e'|\leq \frac{1}{2C} \dist(S_\e , \mathcal{U}_\e^c)$ with $\dist(S_\e , \mathcal{U}_\e^c) >0$, we have 
\bna
\dist\left[\left(x', G(x',\e')\right),S_{\e}\right]
\leq \dist\left[\left(x', G(x',\e)\right),\left(x', G(x',\e')\right)\right]
\leq \left|G(x',\e)-G(x',\e')\right|
\leq \dist(S_\e , \mathcal{U}_\e^c)/2.
\ena
This holds for any $x' \in \bar{D}$, so that $S_{\e'}$ is contained in a neighborhood of $S_\eps$ of size $\dist(S_\e , \mathcal{U}_\e^c)/2$, and hence contained in $\mathcal{U}_\e$. This proves~\eqref{defVj} with $g(\eps) = \dist(S_\e , \mathcal{U}_\e^c)/2C >0 $.

\medskip
As a consequence of~\eqref{defVj}, we have in particular, for any $\e \in ]0,1]$,
\bnan
\label{emptyinterV}
\mathcal{V}_{\e}\subset \mathcal{U}_{\e}\subset \left\{  \phi_{\e} <\rho_{\e}\right\}.
\enan
Now, we also have
$$
K\subset \bigg( S_0 \cup \bigcup_{\e\in (0,1]}\mathcal{V}_{\e} \bigg)
\subset \bigg( \omega_1\cup \bigcup_{\e\in (0,1]}\mathcal{V}_{\e} \bigg) .
$$
The same argument as above using that $\omega_1$ is a neighborhood of $S_0$ shows that there exists $\eps_0$ such that $$
\mathcal{V}_0 := \bigcup_{\e\in [0, \e_0)}S_{\e} \subset \omega_1.
$$
 As a consequence, we now have
$$
K \subset \bigg(  \mathcal{V}_0 \cup \bigcup_{\e\in [\eps_0,1]}\mathcal{V}_{\e}\bigg) , \quad \mathcal{V}_0 \subset \omega_1. 
$$
From the covering $[\e_0, 1] \subset  \bigcup_{\e\in [\eps_0,1]}]\e-g(\e),\e+g(\e)[$, we now extract a finite covering $[\e_0, 1] \subset \bigcup_{j\in J}]\e_j-g(\e_j),\e_j+g(\e_j)[$, where $J$ is a finite set of indices. In particular, this yields a finite covering
\bnan
\label{e:coveringofzero1}
[0,1] \subset [0, \e_0) \cup \bigcup_{i\in J}]\e_j-g(\e_j),\e_j+g(\e_j)[.
\enan
As a consequence, we now have ($\mathcal{V}_{\e_j}$ being defined in~\eqref{defVj})
\bnan
\label{inclusionK}
K\subset \omega_1 \cup \bigcup_{j\in J}\mathcal{V}_{\e_j}
\quad \left(\subset \ \ \omega_1 \cup \bigcup_{j\in J}\bigcup_{i\in I_{\e_j}} \left(B(x_i^{\e_j},r_{i}^{\e_j})\cap \left\{ \phi_{\e_j} <\rho_{\e_j}\right\}\right)\right).
\enan
Now, we reorder the set $J$ by increasing order of $\e_j-g(\e_j)$, that is 
\bnan
\label{increasinge}
J =  \llbracket0,N\rrbracket , \quad \text{with} \quad \e_j-g(\e_j)\leq\e_{j+1}-g(\e_{j+1}), \quad \text{for all } j \in \llbracket0,N-1\rrbracket .
\enan
 Note that if $\e_j-g(\e_j)=\e_{j+1}-g(\e_{j+1})$, we can suppress that associated to the smaller $\e_j+g(\e_j)$ is smaller, and the covering property remains true. We will also need to check that we have
\bnan
\label{ordree}
\e_{k+1}-g(\e_{k+1})<\max_{1\leq j\leq k}(\e_j+g(\e_j)).
\enan
Indeed, if it is not the case, we have $\e_{k+1}-g(\e_{k+1})\geq \max_{0\leq j\leq k}(\e_j+g(\e_j))$. In particular, for $j\leq k$, we have $ \e_j+g(\e_j)\leq \e_{k+1} - g(\e_{k+1})$ and $\e_{k+1}-g(\e_{k+1})\notin ]\e_j- g(\e_j) ,\e_j+g(\e_j)[$. But for $j\geq k+1$, by increasing choice \eqref{increasinge}, we have $\e_{k+1}-g(\e_{k+1})\leq \e_{j}-g(\e_{j})$, and in particular, $\e_{k+1}-g(\e_{k+1})\notin ]\e_j- g(\e_j) ,\e_j+g(\e_j)[$. Hence $\e_{k+1}-g(\e_{k+1})\notin \bigcup_{j \in J} ]\e_j- g(\e_j) ,\e_j+g(\e_j)[$. Moreover, we have $\e_{k+1}-g(\e_{k+1})\geq \max_{1\leq j\leq k}(\e_j+g(\e_j)) \geq \e_0$ as $\e_j\geq \e_0$ for $j \geq 1$ and hence $\e_{k+1}-g(\e_{k+1}) \notin [0,\e_0[$. This contradicts \eqref{e:coveringofzero1} and proves \eqref{ordree}.
 
\bigskip

This preparatory definitions were made to state the following geometrical Lemma that we prove below.
\begin{lemma}
\label{lem:supp theta ou il faut}
With the notation of the proof of Theorem~\ref{thmsemiglobaldep}, we have for any $k\in \llbracket0,N-1\rrbracket$ and $i\in I_{\e_k}$.
\bna
\left\{\phi_{\e_{k+1}}> \rho_{\e_{k+1}}\right\}\cap B(x_i^{\e_{k+1}},4R_i^{\e_{k+1}}) \Subset \left[\omega_1 \cup\bigcup_{j\in \llbracket1,k\rrbracket}\bigcup_{i\in I_{\e_j}} B(x_i^{\e_j},r_{i}^{\e_j})\right],
\ena
where we consider the union $\bigcup_{j\in \llbracket1,k\rrbracket}$ empty if $k=0$.
\end{lemma}

Now, we are going to use an abstract iteration argument, so we set the following notations for $j\in \llbracket 1,N\rrbracket=J$ and $i\in I_{\e_j}$:
\begin{itemize}
\item $I_J = I_{\eps_j}$,
\item $U_{i,j}=B(x_i^{\e_j},2r_{i}^{\e_j})$,
\item $\omega_{i,j}=B(x_i^{\e_j},r_{i}^{\e_j})$,
\item $V_{i,j}=\left[\left\{\phi_{\eps_{j}}> \rho_{\eps_{j}}\right\}\cap B(x_i^{\eps_{j}},4R_i^{\eps_{j}})\right]$,
\item $V_0=\hat{\omega}$,
\item $U_0=\omega_1$.
\end{itemize}
The choice of the $r_i^{\e_j}$ and $\rho_i^{\e_j}\leq \rho_{\eps_{j}}$ according to Corollary \ref{cordominance} implies
\bna
U_{i,j}\lhd V_{i,j}.
\ena
Moreover, we have $\omega_{i,j}\Subset U_{i,j}$ and  Lemma \ref{lem:supp theta ou il faut} can be written as
$V_{i,k+1} \Subset \left[U_0 \cup\bigcup_{j\in \llbracket 1,k\rrbracket}\bigcup_{i\in I_j} \omega_{i,j}\right]$.
\bigskip
Now, we are in position to apply the following  iteration Proposition, that we prove later on.
\begin{proposition}
\label{propiterationabstrait}
Assume that there exists some open sets $U_0$, $U_{i,j}$, $\omega_{i,j}\Subset U_{i,j}$, with $j\in \llbracket 1,N\rrbracket$ and $i\in I_j$ ($I_j$ finite) such that we have
\bna
& U_{i,j}\lhd V_{i,j} \quad \text{ and } \quad \omega_{i,j}\Subset U_{i,j}, \quad \text{ for all } j\in \llbracket 1,N\rrbracket \text{ and }i\in I_j ; & \\ 
& V_{i,k+1} \Subset \left[U_0 \cup\bigcup_{j\in \llbracket 1,k\rrbracket}\bigcup_{i\in I_{j}} \omega_{i,j}\right], \quad \text{ for } k\in \llbracket 0,N -1\rrbracket,&
\ena 
where we consider the union $\bigcup_{j\in \llbracket 1,k\rrbracket}$ empty if $k=0$.
Then, we have $\left[U_0 \cup\bigcup_{j\in \llbracket 1,N\rrbracket}\bigcup_{i\in I_{j}} \omega_{i,j}\right]\lhd V_0$ for any $U_0\Subset V_0$.
\end{proposition}
Now, we always have $\omega_2\lhd\hat{\omega}$, as a consequence of Properties \ref{propproduitstrong} (second part) and \ref{proptransstrong} of Proposition \ref{propproduitstrong}, 
Hence, denoting $U:=\left[\omega_1 \cup\bigcup_{j\in \llbracket 1,k\rrbracket}\bigcup_{i\in I_{\e_j}} B(x_i^{\e_j},r_{i}^{\e_j})\right]$, the application of Proposition~\ref{propiterationabstrait} yields
\bna
U \lhd \hat{\omega}.
\ena
Since $U$ is a neighborhood of $K$ by the covering property \eqref{inclusionK}, this concludes the proof of Theorem \ref{thmsemiglobaldep}, up to the proofs of Lemma \ref{lem:supp theta ou il faut} and Proposition \ref{propiterationabstrait}.
\enp

The next two sections are devoted to the proofs of Lemma \ref{lem:supp theta ou il faut} and Proposition \ref{propiterationabstrait}, respectively.
\subsubsection{Proof of Lemma \ref{lem:supp theta ou il faut}}
In this section,we give a proof of Lemma \ref{lem:supp theta ou il faut}. 
We first prove, for later use, that for any $x'\in \bar{\Omega}$, $\e>0$, we have
\bnan
\label{inegge}
G(x',\e-g(\e))\geq G(x',\e)-\rho_{\e}
\enan
Indeed, let $x\in \mathcal{V}_{\e}$, so $x\in S_{\e'}$ for one $\e'\in]\e-g(\e),\e+g(\e)[$. That is $x_n=G(x',\e')$. Using \eqref{emptyinterV}, we have $\phi_{\e}(x)<\rho_{\e}$, that is $G(x',\e)-x_n<\rho_{\e}$ and so $G(x',\e)-G(x',\e')<\rho_{\e}$. This is true for any point $x=(x',G(x',\e')$ for $\e'\in]\e-g(\e),\e+g(\e)[$. Letting $\e'$ tending to $\e-g(\e)$ and using the continuity of $G$, we get $G(x',\e)-G(x',\e-g(\e))\leq \rho_{\e}$, which is \eqref{inegge}. 

\medskip
We now come back to the proof of the Lemma. 
Notice that, as a consequence of the definitions of $\mathcal{U}_{\e}$, $\mathcal{V}_{\e}\subset \mathcal{U}_{\e}$ and of~\eqref{inclusionK}, we have for all $k \in \llbracket 0,N\rrbracket$
\bnan
\label{hypiteration2}
\left[\mathcal{V}_0\cup\bigcup_{j\in \llbracket1,k\rrbracket}\mathcal{V}_{\e_j}\right]
\Subset \left[\omega_1 \cup\bigcup_{j\in \llbracket1,k\rrbracket}\bigcup_{i\in I_{\e_j}} B(x_i^{\e_j},r_{i}^{\e_j})\right].
\enan
By \eqref{hypiteration2}, it is sufficient to prove, for any $k\in \llbracket 0,N-1\rrbracket$, the inclusion
\bna
\Big(\left\{\phi_{\e_{k+1}}\geq \rho_{\e_{k+1}}\right\}\cap B(x_i^{\e_{k+1}}, 4R_i^{\e_{k+1}}) \Big)
\subset  \Big(\omega_1 \cup \bigcup_{j\in \llbracket1,k\rrbracket}\mathcal{V}_{\e_j} \Big) ,
\ena
which shall follow from the following two inclusions:
\bnan
\label{weakerphi}
\left(\left\{\phi_{\e_{k+1}}\geq \rho_{\e_{k+1}}\right\}\cap K \right)\subset 
\Big( \omega_1\bigcup_{j\in \llbracket 1,k\rrbracket}\mathcal{V}_{\e_j}  \Big) ,
\enan
and 
\bnan
\label{e:remaining-part}
\left(\left\{\phi_{\e_{k+1}}\geq \rho_{\e_{k+1}}\right\}\cap K^c \right)\cap B(x_i^{\e_{k+1}},4R_i^{\e_{k+1}})\subset \omega_1 .
\enan

Let us first prove~\eqref{weakerphi}. Since $K\subset \left(\omega_1\cup\bigcup_{j\in \llbracket 1,N\rrbracket}\mathcal{V}_{\e_j}\right)$ by \eqref{inclusionK}, we have
\bnan
\label{expressioninclusion}
\left(\left\{\phi_{\e_{k+1}}\geq \rho_{\e_{k+1}}\right\}\cap K \right)\subset
\Big( \omega_1\cup\bigcup_{j\in \llbracket 1,N\rrbracket}\left(\mathcal{V}_{\e_j} \cap \left\{\phi_{\e_{k+1}}\geq \rho_{\e_{k+1}}\right\}\right) \Big) .
\enan
Moreover, using \eqref{emptyinterV}, we get
\bna
\mathcal{V}_{\e_{k+1}} \subset  \left\{\phi_{\e_{k+1}}< \rho_{\e_{k+1}}\right\} .
\ena
Now, we will use the fact that $G$ is increasing in $\e$ to prove that we also have
\bnan
\label{inclusionk}
\mathcal{V}_{\e_{j}} \subset  \left\{\phi_{\e_{k+1}}< \rho_{\e_{k+1}}\right\}\quad \textnormal{ for }j\geq k+1 .
\enan
Actually, for $x\in \mathcal{V}_{\e_{j}}$, with $j\geq k+1$, we have $x_n=G(x',\e)$ for some $\e> \e_j-g(\e_j)\geq \e_{k+1}-g(\e_{k+1})$ (that is here that we use the order of the $\e_j$ defined in \eqref{increasinge}). But since $G$ is strictly increasing in $\e$, this implies $x_n> G(x',\e_{k+1}-g(\e_{k+1}))$. Using the inequality \eqref{inegge}, true for any $\e>0$, we obtain $x_n> G(x',\e_{k+1})-\rho_{\e_{k+1}}$. This gives $\phi_{\e_{k+1}}(x',x_n)< \rho_{\e_{k+1}}  $ and therefore \eqref{inclusionk}.
As a consequence, in the right hand-side of \eqref{expressioninclusion} only the terms for $j\leq k$ are nonempty, and it thus implies precisely~\eqref{weakerphi}.

\medskip We now prove~\eqref{e:remaining-part}.
Since $x_i^{\e_{k+1}}\in K$ and $4R_i^{\e_{k+1}}\leq R$, it is sufficient to prove
\bna
\left(\left\{\phi_{\e_{k+1}}\geq 0\right\}\cap K^c \cap K^{R}\right)\subset \omega_1 .
\ena
We first notice that, according to the definition of $K$, we have
\bna
K^c=\left\{x_n<0\right\}\cup \left\{x_n >G(x',1)\right\} .
\ena
In addition, since $G$ is increasing in $\e$, we have,  
\bna
\left\{\phi_{\e_{k+1}}\geq 0\right\}=\left\{x_n\leq G(x',\e_{k+1})\right\}\subset \left\{x_n\leq G(x',1)\right\}.
\ena
As a consequence, $\left\{\phi_{\e_{k+1}}\geq 0\right\}\cap K^c\subset \left\{x_n<0\right\}$. We are thus left to prove 
\bna
\left(\left\{x_n<0\right\} \cap K^{R}\right)\subset \omega_1,
\ena
which is true thanks to \eqref{eq:defR}. This concludes the proof of~\eqref{e:remaining-part}.

We finally check that the proof works the same way for the degenerate case $k=0$, which corresponds to the same proof with $\emptyset$ instead of  $\bigcup_{j\in \llbracket 1,k\rrbracket}$. This concludes the proof of Lemma \ref{lem:supp theta ou il faut}.
\hfill \qedsymbol 
\begin{remark}
\label{rkbouleloin}
In this process, we can also impose that the points $x_i^{\e_j}$ are far from $\left\{x_n=0\right\}$, by forcing $B(x_i^{\e_j},4R_i^{\e_j})\cap \left\{x_n=0\right\}=\emptyset$. 

Indeed, if $B(x_i^{\e_j},4R_i^{\e_j})\cap \left\{x_n=0\right\}\neq \emptyset$, we have necessarily $\dist(x_i^{\e_j},S_0)<4R_i^{\e_j}$ because\\ $\dist(x_i^{\e_j},\left\{x_n=0\right\})$ is necessarily reached at a point in $S_0=\overline{D}\times \left\{0_{x_n}\right\}$, since $x_i^{\e_j}\in S_{\e_j} \subset \overline{D}\times \R_{x_n}$. But, in the process, see~\eqref{eq:defR} and~\eqref{e:defRx}, we have chosen $R_i^{\e_j}\leq \dist(\omega_1^c, S_0)/8$. This implies $\dist(x_i^{\e_j},\omega_1^c)\geq \dist(\omega_1^c,S_0) -\dist(x_i^{\e_j},S_0)> 8R_i^{\e_j}-4R_i^{\e_j}$ and so $B(x_i^{\e_j},4R_i^{\e_j})\subset \omega_1$. In particular, these points $x_i^{\e_j}$ can be removed without affecting the set 
\bna
\left[\omega_1 \cup\bigcup_{j\in \llbracket1,k\rrbracket}\bigcup_{i\in I_{\e_j}} B(x_i^{\e_j},r_{i}^{\e_j})\right].
\ena
for any $k$.

This fact was not used here but it will be useful later in the presence of boundary. 
\end{remark}

\subsubsection{Semiglobal estimates by iteration: proof of Proposition~\ref{propiterationabstrait}}
We now prove Proposition \ref{propiterationabstrait},
which follows an induction argument on $k\in \llbracket 1,N\rrbracket=J$.
We make the following induction assumption at step $k$: 
\begin{equation}
\label{IAK}
\tag{$IA_k$}
\text{For any } j\in \llbracket 1,k\rrbracket \text{ and } i\in I_{j}, \text{ we have} \quad  
U_{i,j} \lhd V_0.
\end{equation}
Note that using Property \ref{propunionstrong} of Proposition \ref{propproduitstrong} and since we can select $W_0$ with $U_0 \Subset W_0  \Subset V_0$ and $\omega_{i,j}\Subset U_{i,j}$, we have 
\bna
\left[U_0 \cup\bigcup_{j\in \llbracket 1,k\rrbracket}\bigcup_{i\in I_{j}} \omega_{i,j}\right] \lhd \left(W_0,U_{i,j} \right)_{j\in \llbracket 1,k\rrbracket, i\in I_{j}}
\ena
So, since we always have $W_0\lhd V_0$, using Properties \ref{propproduitstrong} (second part) and \ref{proptransstrong} of Proposition \ref{propproduitstrong},~\eqref{IAK} directly implies 
\bnan
\label{dependrecabstrait}
\left[U_0 \cup\bigcup_{j\in \llbracket 1,k\rrbracket}\bigcup_{i\in I_{j}} \omega_{i,j}\right] \lhd V_0.
\enan
In particular, proving $(IA_{N})$ implies \eqref{dependrecabstrait} for $k=N$, which is the result of the proposition:
\bnan
U:=\left[U_0 \cup\bigcup_{j\in \llbracket 1,N\rrbracket}\bigcup_{i\in I_{j}} \omega_{i,j}\right] \lhd V_0.
\enan
\medskip
We now come to the proof of~\eqref{IAK} by induction

For $k =1$, we need to prove $U_{i,1} \lhd V_0$ for $i\in I_1$. But the assumption with $k=0$ gives $V_{i,1} \Subset U_0$, which implies $V_{i,1} \lhd U_0$. Since $U_{i,1} \lhd V_{i,1}$ by assumption, we get by transitivity $U_{i,1} \lhd U_0$. Since, we also have $U_0 \lhd V_0$, we obtain the expected result $U_{i,1} \lhd V_0$.

\medskip
We now prove $(IA_{k}) \Longrightarrow (IA_{k+1})$ for $k \in  \llbracket 1,N-1\rrbracket$.  
The assumption of the proposition gives
\bna
V_{i,k+1} \Subset \left[U_0 \cup\bigcup_{j\in \llbracket 1,k\rrbracket}\bigcup_{i\in I_{j}} \omega_{i,j}\right] .
\ena
Combined with Property \ref{propincludestrong} of Proposition \ref{propproduitstrong}, this yields
\bna
V_{i,k+1}
 \lhd \left[U_0 \cup\bigcup_{j\in \llbracket 1,k\rrbracket}\bigcup_{i\in I_{j}} \omega_{i,j}\right].
\ena
Using \eqref{dependrecabstrait} true for $k$ since $(IA)_{k}$ is true and the transitivity of $\lhd$, we get
\bna
V_{i,k+1} \lhd V_0
\ena
Since $U_{i,j} \lhd V_{i,j}$, the transitivity Property gives again $U_{i,k+1} \lhd V_0$.
This implies $(IA_{k+1})$ and thus proves the induction property for $k \in  \llbracket 1,N-1\rrbracket$.

This concludes the proof of Proposition~\ref{propiterationabstrait}.
\hfill \qedsymbol

\subsection{Semiglobal estimates along foliation by hypersurfaces}
The previous framework, where we define hypersurfaces by graphs may look a bit rigid for the applications. This definition of these hypersurfaces as graphs was mainly convenient to make the foliation more effective and order the hypersurfaces more easily.

Now, we give a slight variant of Theorem \ref{thmsemiglobaldep}, more adapted to some possible changes of variables.
\begin{theorem}
\label{thmsemiglobaldepchgt}
Let $\Omega\subset \R^n = \R^{n_a} \times \R^{n_b}$ and $P$ smooth differential operator of order $m$ on $\Omega$, analytically principally normal in $\{\xi_a= 0\}$. Let $\Phi$ a diffeomorphism of class $C^2$ from $\Omega$ to $\widetilde{\Omega}=\Phi(\Omega)$. Assume that the Geometric Setting of Theorem~\ref{thmsemiglobal} is satisfied for some $D$, $G$, $K$, $\phi_{\e}$ on $\widetilde{\Omega}$ (and not on $\Omega$). 
Assume further that for any $\eps \in [0, 1+\eta)$, the oriented surface $\{\phi_{\e}\circ \Phi = 0\} = \Phi^{-1}(S_{\e})$ (well defined on $\Omega$) is be stricly pseudoconvex with respect to $P$ on $\Phi^{-1}(S_{\e})$.

Then, for any $\omega$ a neighborhood of $\Phi^{-1}(S_{0})$, there exists an open neighborhood $U\subset \Omega$ of $\Phi^{-1}(K)$ such that
\bna
U\lhd \omega.
\ena
where $\lhd=\lhd_{\Omega,P}$ is related to the operator $P$ defined on $\Omega$ (see Remark \ref{rkinvariance}).
\end{theorem}
\bnp
The proof is exactly the same as that of Theorem \ref{thmsemiglobal}/\ref{thmsemiglobaldep} except that the local uniqueness estimates are performed in $\Omega$. So, for any $x\in \Phi^{-1}(S_{\e})$, it furnishes some $r_x$, $R_x$ and $\rho_x$, so that
\bna
B_{\Omega}(x,r_x) \lhd_{\Omega,P} \left[\left\{\phi_{\e}\circ \Phi > \rho_x \right\}\cap B_{\Omega}(x,4R_x)\right].
\ena
But since $\Phi$ is an homeomorphism, it implies the existence of $\widetilde{r_x}$ and $\widetilde{R_x}$ (that can still be chosen small enough) so that $\Phi^{-1}\left[B_{\widetilde{\Omega}}(\Phi(x),\widetilde{r_x})\right] \Subset B_{\Omega}(x,r_x)$ and $B_{\Omega}(x,4R_x)\Subset \Phi^{-1}\left[B_{\widetilde{\Omega}}(\Phi(x),4\widetilde{R_x})\right]$, so that 
\bna
\Phi^{-1}\left[B_{\widetilde{\Omega}}(\Phi(x),\widetilde{r_x}) \right]\lhd_{\Omega,P}
\Big(  \left\{\phi_{\e}\circ \Phi > \rho_x \right\}\cap \Phi^{-1}\left[B_{\widetilde{\Omega}}(\Phi(x),4\widetilde{R_x})\right] \Big).
\ena
where $B_{\Omega}$ (resp. $B_{\widetilde{\Omega}}$) denote balls in $\Omega$ (resp. $\widetilde{\Omega}$).

The geometric part of the proof of Theorem \ref{thmsemiglobal}/\ref{thmsemiglobaldep} is then exactly the same, performed in $\widetilde{\Omega}$, i.e. replacing $r_x$, $R_x$ by $\widetilde{r_x}$ and $\widetilde{R_x}$. Once the geometric part is done, the iteration process, performed in $\Omega$, is exactly the same by replacing each geometric term by the preimage in $\Omega$ (for instance $\Phi^{-1}\left[B_{\widetilde{\Omega}}(\Phi(x_i^{\eps_{k}}),4 \widetilde{R}_{x_i^{\eps_{k}}})\right]$ replaces $B(x_i^{\eps_{k}},4R_{x_i^{\eps_{k}}})$ etc.). 
\enp

\section{The Dirichlet problem for some second order operators}
\label{s:Dirichlet-pb-waves}
In this section, we shall consider a particular class of operators as described in Remark \ref{rknoncaractwave}, that is, with symbols the form $p_2(x,\xi) = Q_x(\xi)$ where $Q_x$ is a smooth family of real quadratic forms. 
Assuming that the variables $x_a$ are tangent to the boundary, and that the functions satisfy Dirichlet boundary conditions, we prove a counterpart of the local estimate of Theorem~\ref{th:alpha-unif} for this boundary value problem. For this, the main goal to achieve is to prove a Carleman estimate adapted to this boundary value problem. All local, semiglobal and global results shall then follow.

This situation is of particular interest for the wave equation for which $x_a$ is the time variable, which is always tangent to the boundary of cylindrical domains.  

For the sake of simplicity, we shall further assume that the operator principal symbol of $P$ is independent of the $x_a$ variable (we would otherwise need to assume the coefficients of $P$ to be analytic with respect to $x_a$). This allows to avoid some additional technicalities in the (already rather technical) proofs.

\subsection{Some notation}
Here, we shall always assume that the analytic variables are tangential to the boundary, that is 
$$
x  = (x_a, x_b) \in  \R^{n_a} \times \R^{n_b}_+, \quad \text{ with } \R^{n_b}_+ = \R^{n_b - 1} \times \R_+, \quad \text{ and } x_b = (x_b', x_b^n) .
$$
When the distinction between analytic and non-analytic variables is not essential, we shall split the variables according to
$$
x  = (x', x_n) \in \R^{n}_+ = \R^{n-1} \times \R_+ , \quad \text{ with } x' = (x_a , x_b') \in \R^{n_a + n_b - 1} , \quad \text{ and } x_n = x_b^n \in \R^+ .
$$ 
We also denote by  $\xi' = (\xi_a,\xi_b')\in \R^{n-1}$ the cotangential variables and $\xi_n = \xi_b^n$ the conormal variable, by $D' = (D_a,  D_{x_b'}) = \frac{1}{i}(\d_{x_a}, \d_{x_b'})$ the associated tangential derivations and $D_n = D_{x_b^n} = \frac{1}{i} \d_{x_n}$ the normal derivation. 

For any $r_0>0$, we define 
\bnan 
\label{e:def-Kr}
K_{r_0}=\left\{x\in \R^n_+;|x_a|\leq r_0, |x_b|\leq r_0\right\}
 =   \overline{B}_{\R^{n_a}}(0, r_0) \times \overline{B}_{\R^{n_b}}(0, r_0) \cap \{x_n \geq 0\} .
\enan
We denote by $C^{\infty}_0(\R^{n}_+)$ the space of restrictions to $\R^{n}_+$ of functions in $C^{\infty}_0(\R^{n})$, and by $C^{\infty}_0(K_{r_0})$ the space of functions $C^{\infty}_0(\R^n_+)$ supported in $K_{r_0}$. the trace of a function $f \in C^{\infty}_0(\R^{n}_+)$ at $x_n=0$ is denoted by $f_{|x_n=0}$.

\medskip 
We denote by $(f,g)= \int_{\R^n_+}f \overline{g}$, $\nor{f}{0, +}^2 = (f,f)$ the $L^2(\R^n_+)$ inner product and norm.
For $k\in \N$, the norm $\nor{\cdot}{k, +}$ will denote the classical Sobolev norm on $\R^n_+$ and $\nor{\cdot}{k, +, \tau}$ the associated weighted norms, that is,
$$
\nor{f}{k, +, \tau}^2 =  \sum_{j+|\alpha|\leq k }  \tau^{2j} \nor{\d^{\alpha}f}{0,+}^2 ,  \quad \tau \geq 1 .
$$
We also define the tangential Sobolev norms, given by
$$
|f|_{k, \tau}^2 = \nor{(|D'|  + \tau )^k f}{0,+}^2 \sim  \sum_{j+|\alpha|\leq k }  \tau^{2j} \nor{\d_{x'}^{\alpha}f}{0,+}^2 ,  \quad \tau \geq 1 .
$$
We shall also use, for $f,g \in C^\infty_0(\R^n_+)$, the notation $(f,g)_0 = \int_{\R^{n-1}}f_{|x_n=0}(x') g_{|x_n=0}(x')dx'$.

\medskip
Finally,  for $j \in \N$, we denote by $\mathcal{D}^k_\tau$, the space of {\em tangential} differential operators, i.e. operators of the form
$$
P(x, D' , \tau) 
= \sum_{j + |\alpha|\leq k}  a_{j, \alpha}(x) \tau^{ j} D'^{\alpha} ,
$$ and by 
$$
\sigma(P) = p(x, \xi' , \tau) = \sum_{j + |\alpha| = k}  a_{j, \alpha}(x) \tau^{ j} \xi'^{\alpha}
$$ their principal symbol.
\begin{remark}
\label{rkrestriction} 
Denote $T$ the restriction operator from $\mathcal{D}'(\R^n)$ to $\mathcal{D}'(\R^n_+)$. We denote $H^k(\R^n_+)=T(H^k(\R^n))$ with the restriction Sobolev norms 
\bna
\nor{u}{k,+}:=\inf \left\{\nor{v}{k}\left|v\in H^k(\R^n); Tv=u \textnormal{ in }\mathcal{D}'(\R^n_+)\right.\right\} "="\inf \left\{\nor{v}{k}\left|v\in H^k(\R^n); v=u \textnormal{ on }\R^n_+\right.\right\}
\ena
We have the property
\bna
\nor{u}{k,+}\approx \sup_{|\alpha|\leq k} \nor{\partial^{\alpha} u}{L^{2}(\R^n_+)} ,
\ena
see Chapter B2. of \cite{Hoermander:V3} and Corollary B.2.5 (with different notations $\overline{H}_{(k,0)}(\R^n_+)$) . 
Moreover, the set $C^{\infty}_0(\R^n_+)= T(C^{\infty}_0(\R^n))$ of restriction of smooth functions is dense in $H^k(\R^n_+)$ (see Theorem B.2.1 of \cite{Hoermander:V3}).  
As a conclusion, if $L$ is a linear operator from $H^k$ to $H^l$ of norm $C$ that sends $\ker(T) \cap H^k$ into $\ker(T) \cap H^l$, then, $L$ extends to a linear operator from $H^k(\R^n_+)$ to $H^l(\R^n_+)$ and we have
\bna
\nor{ L u}{l,+}\leq C\nor{u}{k,+} .
\ena
In particular, this will be the case for all ``tangential'' operators.
\end{remark}

\subsection{The Carleman estimate}
In this section, we state and prove the counterpart of the Carleman estimate~\eqref{Carleman} asociated to the Dirichlet problem for waves. Recall that the operator $Q_{\e,\tau}^{\psi}$ is defined in~\eqref{Qeps} and acts in the variable $x_a$ only, and hence, is tangential to the boundary.
\begin{theorem}[Local Carleman estimate]
\label{th:carleman-bord}
Let $r_0>0$ and $P = D_{x_b^n}^2 +r(x_b ,D_{x_a},D_{x_b'})$ be a differential operator of order two on a neighborhood of $K_{r_0}$, with real principal part, where $r(x_b ,D_{x_a},D_{x_b'})$ does not depend on $x_a$ and is a smooth $x_b^n$ family of second order operators in the (tangential) variable $(x_a, x_b')$.

Let $\psi$ be quadratic polynomial such that $\psi_{x_b^n}' \neq 0$ on $K_{r_0}$ and
\bnan
\label{hypopseudoconvecCarletau0}
 \left\{ p, \{p ,\psi \} \right\} (x,\xi) >0  , 
&\text{ if } p(x,\xi) = 0, \quad x \in K_{r_0}\text{ and } \xi_a=0, \quad \xi \neq 0 ; \\
\label{hypopseudoconvecCarle}
\frac{1}{i\tau}\{ \overline{p}_\psi,p_\psi \} (x ,\xi) >0  , 
&\text{ if } p_\psi(x,\xi) = 0,  \quad x \in K_{r_0} \text{ and } \xi_a = 0 , \quad \tau >0 ,
\enan
where $ p_\psi(x,\xi) = p(x, \xi + i \tau\nabla \psi)$.

Then, there exist $\eps >0$, $\mathsf{d}>0$, $C>0$, $\tau_0>0$ such that for any $\tau >\tau_0$, we have  for all $u \in C^\infty_0(K_{r_0/4})$
\bnan
\label{Carlemanfromboundary}
 \tau \|Q_{\e,\tau}^{\psi}u\|_{1,+, \tau}^2 & \leq & C\left(
 \nor{Q_{\e,\tau}^{\psi}P u}{0,+}^2+ e^{-\mathsf{d}\tau}\nor{e^{\tau\psi}u}{1,+,\tau}^2
 +  \tau^3 |(Q_{\e,\tau}^{\psi} u)_{|x_n=0}|_{0}^2   \right.\nonumber\\
  && \left. + e^{-\mathsf{d}\tau} |e^{\tau\psi}u_{|x_n=0}|_{0}^2
 + \tau |\big(D(Q_{\e,\tau}^{\psi} u)\big)_{|x_n=0}|_{0}^2 + e^{-\mathsf{d} \tau } |e^{\tau\psi}Du_{|x_n=0}|_{0}^2  \right).
\enan
If moreover $\psi_{x_n}'>0$ for $(x',x_n=0) \in K_{r_0}$, then we have for all $u \in C^\infty_0(K_{r_0/4})$ such that $u_{|x_n=0} = 0$,
\bnan
\label{Carlemanuptoboundary}
 \tau \|Q_{\e,\tau}^{\psi}u\|_{1,+, \tau}^2 \leq  C\left(
 \nor{Q_{\e,\tau}^{\psi}P u}{0,+}^2+ e^{-\mathsf{d}\tau}\nor{e^{\tau\psi}u}{1,+,\tau}^2 \right).
\enan
\end{theorem}
The proof of this theorem relies on a Carleman estimate interpolating between the ``boundary elliptic Carleman estimates'' of Lebeau and Robbiano~\cite{LR:95} and the ``partially analytic Carleman estimates'' of Tataru~\cite{Tataru:95} (see also~\cite{Hor:97}). We first state two Corollaries and get to the proof.
\begin{corollary}
\label{cor:carl-potentiel-bord}
Under the assumptions of Theorem~\ref{th:carleman-bord}, there exist $\eps >0$, $\mathsf{d}>0$, $C>0$, $\tau_0>0$ such that for any $V  \in L^\infty (K_{r_0})$, $W \in L^\infty (K_{r_0}; \R^n )$, independent of $x_a$ and any $\tau >\tau_0 \max\{1 , \|V\|_{L^\infty}^{\frac23},\|W\|_{L^\infty}^2\}$, the Carleman estimates \eqref{Carlemanfromboundary} or \eqref{Carlemanuptoboundary} are satisfied with $P$ replaced by $P_{V,W} = P+W \cdot \nabla + V$. 
\end{corollary}
\bnp
Applying the Carleman estimates \eqref{Carlemanfromboundary} or \eqref{Carlemanuptoboundary} for $P = P_{V,W}- i W \cdot D -V$, we need to estimate the term 
$$Q_{\e,\tau}^{\psi}P u = Q_{\e,\tau}^{\psi}P_{V,W} u - i W \cdot  Q_{\e,\tau}^{\psi} (Du) - V Q_{\e,\tau}^{\psi}u
$$
where we used $V = V(x_b)$, $W = W(x_b)$. Notice first that we have
$$
C \nor{V Q_{\e,\tau}^{\psi}u }{0,+}^2\leq C\|V\|_{L^\infty}^2\nor{Q_{\e,\tau}^{\psi}u}{0,+}^2  \leq \frac14 \tau \nor{Q_{\e,\tau}^{\psi}u}{1,+}^2  ,
$$
as soon as $\tau^3/4C \geq \|V\|_{L^\infty}^2$. Next, using~\eqref{e:conjugationpsitaueps}, we write
$$
Q_{\e,\tau}^{\psi} (Du) =  (D - \e \psi_{x, x_a}''D_a + i \tau \psi')Q_{\e,\tau}^{\psi} u ,
$$
and consequently 
$$
C\nor{i W \cdot  Q_{\e,\tau}^{\psi} (Du)}{0,+}^2 \leq C'\|W\|_{L^\infty}^2\nor{Q_{\e,\tau}^{\psi}u}{1,+}^2
\leq \frac14 \tau \nor{Q_{\e,\tau}^{\psi}u}{1,+}^2  ,
$$
as soon as $\tau/4C' \geq \|W\|_{L^\infty}^2$. For such $\tau$, these two terms may hence be absorbed in the left hand-side of the inequality. This concludes the proof of the corollary. 
\enp
\begin{corollary}
\label{cor:carl-potentielanal-bord}
Under the assumptions of Theorem~\ref{th:carleman-bord}, take $R(x,D)$ a differential operator of order $1$, with coefficients which can be extended to a bounded function in $\left\{(z_a,x_b)\in \C^{n_a}\times \R^{n_b}; |z_a|<5r_0,|x_b|<5r_0\right\} $ which are analytic with respect to $z_a$, for fixed $x_b$. 

Then, there exist $\eps >0$, $\mathsf{d}>0$, $C>0$, $\tau_0>0$ such that for any any $\tau>\tau_0$, the Carleman estimates \eqref{Carlemanfromboundary} or \eqref{Carlemanuptoboundary} are satisfied with $P$ replaced by $P_{R} = P+R$. 
\end{corollary}
\bnp
Lemma 4.8 of H\"ormander \cite{Hor:97} yields
\bna
\nor{Q^{\psi}_{\e,\tau}R(x,D)u}{0,+}\leq C \nor{Q^{\psi}_{\e,\tau}u}{1,+,\tau}+C e^{-\tau \mathsf{d}}\nor{e^{\tau \psi}u}{1,+,\tau}
\ena
for all $u \in C^\infty_0(K_{r_0/4})$. 
Actually, it is stated for the interior case, with the norm $\nor{\cdot}{1,+,\tau}$ replaced by the norm $\nor{\cdot}{1,\tau}$. Yet, the estimates used for the proof, (3.13) and (3.14) in \cite{Hor:97}, are actually made first in the variable $x_a$ and then integrated in $x_b$. Since, the variable $x_a$ is tangential, the same proof gives the expected result.  

As in Corollary \ref{cor:carl-potentiel-bord}, we can absorb the term $C \nor{Q^{\psi}_{\e,\tau}u}{1,+,\tau}$ for $\tau$ large enough. The second term has the same form as the right hand side of the Carleman estimate, up to changing $\mathsf{d}$.
\enp

\begin{remark}
This theorem, as well as its consequences may be extended with some modification to the Neumann case following Lebeau-Robbiano~\cite{LR:97}. It could also be generalized to a larger class of operators and boundary condition (satisfying a Lopatinskii condition) following Tataru \cite{Tataru:96} and Bellassoued-Le~Rousseau~\cite{BLR:15}.
\end{remark}

To prove Theorem~\ref{th:carleman-bord}, we define the conjugated operator $P_\psi = e^{\tau\psi} P e^{- \tau\psi}=P(x, D+ i \tau \psi')$, and also $P_{\psi,\eps}$ the conjugate of $P_\psi$ with respect to $e^{- \frac{\e}{2\tau}|D_a|^2}$, that is, such that
\bnan
\label{e:conjugationpsitaueps}
 e^{-\frac{\e}{2\tau}|D_a|^2} P_\psi w = P_{\psi,\eps}  e^{- \frac{\e}{2\tau}|D_a|^2} w  . 
\enan
Since $P$ is independent on $x_a$, we have
\bna
P_{\psi,\eps}  = P(x, D - \e \psi_{x, x_a}''D_a + i \tau \psi') ,
\ena
where $\psi_{x, x_a}''D_a = \psi_{xx}''\big((D_a ,0) \big)$ (with the notation of~\cite{Hor:97}).

When proving the theorem, we shall drop the index $+$ in the norms to lighten the notation; of course, all inner norms and integrals are meant on $\R^n_+$. We first need the following proposition.

\begin{proposition}
\label{l:carleman-bord-compact}
Under the assumptions of Theorem~\ref{th:carleman-bord}, there exist $C>0$, $\tau_0>0$ such that for any $\tau>\tau_0$ and $f \in C^\infty_0(K_{r_0})$, we have
\bnan
\label{e:carleman-bord-general}
\tau \|f\|_{1, \tau}^2 
\leq C
\nor{P_{\psi,\e}f}{0}^2
 + \tau\| D_a f\|_0^2 
 +\tau^3 |f_{|x_n=0}|_{0}^2 
+ \tau |D f_{|x_n=0}|_{0}^2 .
 \enan
 If moreover $\psi_{x_n}'>0$ for $(x',x_n=0) \in K_{r_0}$, then 
 \bnan
 \label{e:carleman-bord-dirichlet}
\tau \|f\|_{1, \tau}^2 
\leq C
\nor{P_{\psi,\e}f}{0}^2
 + \tau\| D_a f\|_0^2 , \quad \text{for all }f \in C^\infty_0(K_{r_0})\text{ such that }f_{|x_n=0} = 0 . 
  \enan
\end{proposition}

\bnp
Defining $\tilde{Q}_2^\eps = \frac{1}{2}(P_{\psi,\eps} + P_{\psi,\eps}^* )$ and $\tilde{Q}_1^\eps = \frac{1}{2i\tau }(P_{\psi,\eps} - P_{\psi,\eps}^* )$, we have 
$$P_{\psi,\eps} = \tilde Q_2^\eps + i \tau \tilde{Q}_1^\eps,$$
and denote by $\tilde{q}_j^\eps$ the principal symbol of $\tilde{Q}_j^\eps$, $j =1,2$.
We have 
\bneqn
\label{formuleQi}
\tilde{Q}_2^\eps &=& D_n^2 - 2\eps \psi_{x_n , x_a}'' (D_n ; D_{a}) + Q_2^\eps \\
\tilde{Q}_1^\eps &=& D_n \psi_{x_n}' + \psi_{x_n}' D_n  + 2  Q_1^\eps , 
\eneqn
where $Q_2^\eps \in \mathcal{D}^2_\tau$ and $Q_1^\eps \in \mathcal{D}^1_\tau$ with principal symbols 
\bna
q_2^\eps & = &\eps^2  \big( \psi_{x_n , x_a}'' \xi_a \big)^2 - \tau^2 (\psi_{x_n}')^2 + r(x, \xi'- \eps\psi_{x' , x_a}'' \xi_{a} ) - \tau^2 r(x,\psi_{x'}' )\\
q_1^\eps & = &\tilde{r}(x_b , \xi'- \eps \psi_{x' , x_a}'' \xi_{a}  ,\psi_{x'}' ) , 
\ena
where $\tilde{r}$ is the bilinear form associated with the quadratic form $r$. Note that, even if it does not appear in the notation, all these operators depend upon the parameter $\tau$.

With this notation, we hence have $p_\psi = \tilde{q}_2^0 + i\tau \tilde{q}_1^0$, so that $\frac{1}{i\tau}\{ \overline{p}_\psi, p_\psi \}=2 \{ \tilde{q}_2^0 , \tilde{q}_1^0 \}$. Assumptions~\eqref{hypopseudoconvecCarletau0} and \eqref{hypopseudoconvecCarle} then translate respectively into 
\bnan
\label{hypopseudoconvecCarletau0bis}
 \{ \tilde{q}_2^0 , \tilde{q}_1^0 \} (x ,\xi) >0  , 
&\text{ if } p(x,\xi) = 0 , \quad x \in K_{r_0} \text{ and } \xi_a = 0 , \tau =0 ; \\
\label{hypopseudoconvecCarlebis}
\{ \tilde{q}_2^0 , \tilde{q}_1^0 \} (x ,\xi) >0  , 
&\text{ if } p_\psi(x,\xi) = 0,  \quad x \in K_{r_0} \text{ and } \xi_a = 0 , \tau >0 ,
\enan
where the second assertion is a direct consequence of~\eqref{hypopseudoconvecCarle}, and the first one follows from~\eqref{hypopseudoconvecCarletau0} together with the fact that, using that $p$ is real, we have
$$
\lim_{\tau \to 0^+} \frac{1}{i\tau}\{ \overline{p}_\psi, p_\psi \}
= \left.\frac{\d}{\d \tau} \frac{1}{i}\{ \overline{p}_\psi, p_\psi \}\right|_{\tau=0} 
=2\left\{ p, \{p ,\psi \} \right\} .
$$ 

Next, we have the integration by parts formul\ae:
\bneqn
\label{IPPformula}
(g , \tilde{Q}_2^\eps f) & = & (\tilde{Q}_2^\eps g ,  f) - i \left[ (g , D_n f)_0 + (D_n g , f)_0 + 2 \eps (g , \psi_{x_n , x_a}'' D_{a}  f)_0\right]  ,\\
(g , \tilde{Q}_1^\eps f) & = & (\tilde{Q}_1^\eps g ,  f) - 2 i \left( \psi_{x_n}' g , f\right)_0  .
\eneqn
So, we have for $f\in C^{\infty}_0(K_{r_0})$
\bnan
\nor{P_{\psi,\e}f}{0}=\nor{\tilde{Q}_2^\eps f}{0}^2 +\tau^2\nor{\tilde{Q}_1^\eps f}{0}^2+i\tau \left[ \left(\tilde{Q}_1^\eps f,\tilde{Q}_2^\eps f\right)-\left(\tilde{Q}_2^\eps f,\tilde{Q}_1^\eps f\right)\right] .
\enan
So, we get, using the integration by parts formul\ae~\eqref{IPPformula}
\bnan
\label{e:IPP-B}
\nor{P_{\psi,\e}f}{0}=\nor{\tilde{Q}_2^\eps f}{0}^2 +\tau^2\nor{\tilde{Q}_1^\eps f}{0}^2+i\tau \left([\tilde{Q}_2^\eps,\tilde{Q}_1^\eps]f,f\right)+\tau \mathcal{B}^\eps(f) ,
\enan
with the boundary term
\bnan
\label{e:boundary-bilinear-term}
\mathcal{B}^\eps(f)&=& \left[ (\tilde{Q}_1^\eps f , D_n f)_0 + (D_n \tilde{Q}_1^\eps f , f)_0 + 2 \eps (\tilde{Q}_1^\eps f , \psi_{x_n , x_a}'' D_{a}  f)_0\right]- 2 \left( \psi_{x_n}' \tilde{Q}_2^\eps f, f\right)_0 \nonumber\\
&=&2(\psi_{x_n}' D_n f , D_n f)_0+(M_1^\eps f , D_n f)_0+(M_1'^{\eps} D_n f , f)_0+(M_2^{\eps} f , f)_0 ,
\enan
for some tangential operator $M_1^\eps$ of order $1$ (in $\xi',\tau$) (note that terms of order two in $D_n$ cancel).

Now that we have made the exact computations, we will make some estimates on the symbols of the interior part of the commutator. The idea is to tranfer the positivity assumption of the full symbol to some positivity of a tangential symbol, which will then allow to apply the tangential G{\aa}rding.

The first step is to perform a factorisation of $[\tilde{Q}_2^\eps,\tilde{Q}_1^\eps]$ with respect to $\tilde{Q}_1^\eps$ and $\tilde{Q}_2^\eps$ to have a tangential reminder. Since $[\tilde{Q}_2^\eps,\tilde{Q}_1^\eps]$ is of order $2$, it can be written $i[\tilde{Q}_2^\eps,\tilde{Q}_1^\eps]=C_2+C_1D_n+C_0D_n^2$ where $C_i\in \mathcal{D}^i_\tau$. But using \eqref{formuleQi}, and $\psi_{x_n}'\neq 0$ on $K_{r_0}$, we can replace $D_n=\frac{1}{2\psi_{x_n}'} \tilde{Q}_1^\eps+ \mathcal{D}^1_\tau$ and $D_n^2 =\tilde{Q}_2^\eps  + 2\eps \psi_{x_n , x_a}'' (D_n ; D_{a}) - Q_2^\eps$. So, in particular, we can write
\bnan
\label{e:crochet Q2 Q1 tilde}
i [\tilde{Q}_2^\eps,\tilde{Q}_1^\eps]=B_0^\eps \tilde{Q}_2^\eps+B_1^\eps \tilde{Q}_1^\eps + B_2^\eps .
\enan
where $B_i^\eps \in \mathcal{D}^i_\tau$ with real symbol $b_i^\eps$. Now, we need to 
\begin{itemize}
\item use the assumption to get some positivity of the symbol $\{ \overline{p}_\psi,p_\psi \} $, this is Lemma \ref{lmposit};
\item transfer this positivity to $\{ \overline{p}_\psi^\eps,p_\psi^\eps \}$ for $\e$ small enough by approximation, this is Lemma \ref {lmpositeps};
\item transfer this information to a tangential information on the symbol, this is Lemma \ref {lmposittangent}.
\end{itemize}
\begin{lemma}
\label{lmposit}
There exist $C_1, C_2>0$ such that for all $(x,\xi)\in K_{r_0}\times \R^n$  and $\tau>0$, we have
\bna
(|\xi|^2+\tau^2)\leq C_1 \{ \tilde{q}_2^0 , \tilde{q}_1^0 \} (x ,\xi) + C_2\left[ \frac{|p_\psi(x,\xi)|^2}{|\xi|^2+\tau^2}+|\xi_a|^2\right] .
\ena
\end{lemma}
\bnp
All the terms are homogeneous of order $2$ in $(\xi, \tau)$ and continuous on the compact $(x, \xi , \tau ) \in K_{r_0} \times \{(\xi, \tau) \in \R^n \times \R^+  ,  |\xi|^2 + \tau^2 =1\}$. Thus, on this set, the result is a consequence of \eqref{hypopseudoconvecCarletau0bis}, \eqref{hypopseudoconvecCarlebis} and Lemma \ref{l:fgh-three-fcts} applied to $f=\frac{|p_\psi(x,\xi)|^2}{|\xi|^2+\tau^2}+|\xi_a|^2\geq 0$, $g=\{ \tilde{q}_2^0 , \tilde{q}_1^0 \}$ and $h=0$. The result on the whole $K_{r_0}\times \R^n \times \R^+$ follows by homogeneity.
\enp
\begin{lemma}
\label{lmpositeps}
There exists $\e_0$ such that for all $\e \in (0,\e_0)$, there exist $C_1, C_2>0$ such that for all $(x,\xi)\in K_{r_0}\times \R^n$  and $\tau>0$, we have 
\bna
(|\xi|^2+\tau^2)\leq C_1  \{ \tilde{q}_2^\eps , \tilde{q}_1^\eps \}(x ,\xi) + C_2\left[\frac{|p_\psi^\e(x,\xi)|^2}{|\xi|^2+\tau^2}+|\xi_a|^2\right] .
\ena
\end{lemma}
\bnp
By the same argument, we may restrict to the compact $(x, \xi , \tau ) \in K_{r_0} \times \{(\xi, \tau) \in \R^n \times \R^+  ,  |\xi|^2 + \tau^2 =1\}$. There, the inequality follows from Lemma~\ref{lmposit} and the continuity of the maps $\eps \mapsto q_j^\eps$, $j=1,2$ from $\R$ to $C^1(V)$, where $V$ is a neighborhood of $K_{r_0} \times \{(\xi, \tau) \in \R^n \times \R^+  ,  |\xi|^2 + \tau^2 =1\}$ in $\R^n \times \R^n \times \R^+$.
\enp
Now, we set 
\bna
\mu^{\eps}(x,\xi')= (q_1^{\e})^2+2\e q_1^{\e} \psi_{x_n , x_a}'' (\psi_{x_n}'; \xi_{a}) +(\psi_{x_n}')^2 q_2^{\e}. 
\ena
The symbol $\mu^\eps(x,\xi')$ satisfies the property that $\mu^\eps(x,\xi')=0$ if and only if there exists $\xi_n$ real such that $p_{\psi}^\eps(x,\xi',\xi_n)=0$. This is easily seen by noticing that the zero of $q_1^{\e}$ can only be with $\xi_n= -\frac{q_1^\eps}{\psi_{x_n}'}$.

Notice also that $\mu^{\eps}(x,\xi')$ is a tangential symbol of order $2$.
\begin{lemma}
\label{lmposittangent}
There exists $\e_0$ such that for all $\e \in (0,\e_0)$, there exist $C_1, C_2>0$ such that for all $(x,\xi')\in K_{r_0}\times \R^{n-1}$ and $\tau>0$, we have 
\bnan
\label{e:tangential-subell}
(|\xi'|^2+\tau^2)\leq C_1 b_2^\eps + C_2\left[\frac{\left[\mu^{\eps}(x,\xi')\right]^2}{|\xi'|^2+\tau^2}+|\xi_a|^2\right] .
\enan
\end{lemma}
\bnp
Note first that for any $(x, \xi', \xi_n)$ with $\xi_n= -\frac{q_1(x, \xi')}{\psi_{x_n}'}$, we have $\tilde{q}_1^\eps(x, \xi',\xi_n) = 0$ and
$$
p_\psi^\eps(x, \xi' , \xi_n) = \tilde{q}_2^\eps(x, \xi' , \xi_n)  = (\psi_{x_n}')^{-2}\mu^{\eps}(x,\xi') .
$$
Now, assume $\mu^\eps(x,\xi')=0$ and $\xi_a=0$. Setting $\xi_n= -\frac{q_1(x, \xi')}{\psi_{x_n}'}$, we have $p_\psi^\eps(x, \xi' , \xi_n) =0$. Using Lemma~\ref{lmpositeps}, we have $\{ \tilde{q}_2^\eps , \tilde{q}_1^\eps \}(x ,\xi', \xi_n)>0$. According to the definition of $B_2^\eps$ in~\eqref{e:crochet Q2 Q1 tilde}, we have $b_2^\eps(x, \xi' , \xi_n)>0$. As a consequence, we have
\bna
\big(\mu^\eps(x,\xi')=0 \quad \text{ and } \quad \xi_a=0  \big)\Longrightarrow b_2^\eps(x, \xi' , \xi_n)>0.
\ena
Moreover, all terms in~\eqref{e:tangential-subell} are homogeneous of order $2$ in the variables $(\xi', \tau)$ and continuous on $(\xi', \tau) \neq (0,0)$. Hence, applying Lemma~\ref{l:fgh-three-fcts} below on the compact set $K_{r_0} \times \{(\xi', \tau) \in \R^{n-1} \times \R^+  ,  |\xi' |^2 + \tau^2 =1 , \xi_a = 0  \}$ yields~\eqref{e:tangential-subell} on that set. The conclusion follows by homogeneity.
\enp
Taking the real part of \eqref{e:IPP-B} and using~\eqref{e:crochet Q2 Q1 tilde}, we obtain 
\bnan
\label{e:ptaupsi-t-B2}
\nor{P_{\psi,\e}f}{0}- \tau \Re \left( \mathcal{B}^\eps(f) \right)
=\nor{\tilde{Q}_2^\eps f}{0}^2 +\tau^2\nor{\tilde{Q}_1^\eps f}{0}^2
+ \tau \Re  \left(B_2^\eps f,f\right)
+ \tau \Re \left((B_0^\eps \tilde{Q}_2^\eps+B_1^\eps \tilde{Q}_1^\eps) f,f\right) .
\enan
Concerning the remainder term, we have
\bnan
\label{e:estim-term-rest}
\tau |\Re \left((B_0^\eps \tilde{Q}_2^\eps+B_1^\eps \tilde{Q}_1^\eps) f,f\right) |
& \leq & \tau \|f\|_0 \|\tilde{Q}_2^\eps f\|_0 + \tau |f|_1 \|\tilde{Q}_1^\eps f\|_0 \nonumber\\
& \leq & \tau^{-1/2} \left( \tau |f|_{1,\tau}^2 + \|\tilde{Q}_2^\eps f\|_0^2 + \tau^2 \|\tilde{Q}_1^\eps f\|_0^2 \right) .
\enan
Defining now $$\Sigma = (Q_1^{\e})^2+2\e Q_1^{\e} \psi_{x_n , x_a}'' (\psi_{x_n}'; D_{a}) +(\psi_{x_n}')^2 Q_2^{\e}, 
$$ with principal symbol $\mu^\eps$, and for an operator $G$ with principal symbol $\frac{\mu^{\eps}(x,\xi')}{|\xi'|^2+\tau^2}$, the tangential G{\aa}rding inequality in the class $S\Big( (|\xi'| + \tau)^2 , |dx'|^2 + \frac{|d\xi'|^2}{(|\xi'| + \tau)^2}\Big)$ (see~\cite[Chapter~XVIII]{Hoermander:V3} or~\cite{Lerner:10}), in which symbols are allowed to depend smoothly upon the variable $x_n$ yields, for $\tau$ sufficiently large,
\bnan
\label{e:garding}
|f|_{1,\tau}^2 \leq C \Re \left(B_2^\eps f,f\right) 
+ \Re \left( \Sigma f , G f \right) + \| D_a f\|_0^2 .
\enan
Writing $\psi_{x_n}'D_n= \frac12 (\tilde{Q}_1^\eps - [D_n , \psi_{x_n}']) - Q_1^\eps$ (where $\psi_{x_n}'$ does not vanish), this allows to estimate the full norm $\|f\|_{1, \tau}$ according to
\bnan
\label{e:double-simple-norm}
\|f\|_{1, \tau} \leq C (\|\tilde{Q}_1^\eps f \|_0 + |f|_{1, \tau}) .
\enan
Recalling the definitions of $\tilde{Q}_i^\eps$ in~\eqref{formuleQi}, we also have
\bnan
\Sigma & = & \left(\frac12( \tilde{Q}_1^\eps - [D_n , \psi_{x_n}']) - \psi_{x_n}' D_n \right)^2
+2\e Q_1^{\e} \psi_{x_n , x_a}'' (\psi_{x_n}'; D_{a}) \nonumber \\
& & + (\psi_{x_n}')^2 \left(\tilde{Q}_2^\eps - D_n^2 +2\eps \psi_{x_n , x_a}'' (D_n ; D_{a})\right) \nonumber \\
& = & \left(\frac12( \tilde{Q}_1^\eps - [D_n , \psi_{x_n}']) - \psi_{x_n}' D_n \right)
\frac12( \tilde{Q}_1^\eps - [D_n , \psi_{x_n}'])
+2\e Q_1^{\e} \psi_{x_n , x_a}'' (\psi_{x_n}'; D_{a})  \nonumber \\
& & + (\psi_{x_n}')^2 \left(\tilde{Q}_2^\eps +2\eps \psi_{x_n , x_a}'' (D_n ; D_{a})\right),
\enan
and hence $$\Sigma  \in (\psi_{x_n}')^2 \tilde{Q}_2^\eps - \frac12 \psi_{x_n}' D_n\tilde{Q}_1^\eps 
+2\e \psi_{x_n , x_a}'' \left( (\psi_{x_n}')^2 D_n +  Q_1^{\e} \psi_{x_n}'; D_{a} \right) + \mathcal{D}^1_\tau \tilde{Q}_1^\eps +  \mathcal{D}^1_\tau + \mathcal{D}^0_\tau D_n .$$

We now want to estimate the term $\Re \left( \Sigma f , G f \right)$ in \eqref{e:garding}. For this, 
integrating by parts in the tangential direction $x_a$, we have
\bna
\left|\left(\psi_{x_n , x_a}'' \left( (\psi_{x_n}')^2 D_n +  Q_1^{\e} \psi_{x_n}'; D_{a} \right) f , Gf \right) \right|  \leq C \|\left<D_a \right> f\| \|f\|_{1,\tau} .
\ena
This yields
\bnan
\label{estimMG}
|\left( \Sigma f , G f \right)| & \leq & C \|\tilde{Q}_2^\eps f \|_0 \| f \|_0
+ \left|\left( \frac{1}{2i}\psi_{x_n}' \tilde{Q}_1^\eps f, Gf \right)_0 \right| \nonumber \\
&& + \|\tilde{Q}_1^\eps f \|_0 \| f \|_{1, \tau} + \| f \|_0 \| f \|_{1, \tau}
+C \|\left<D_a \right> f\| \|f\|_{1,\tau} \nonumber\\
& \leq &\left|\left( \frac{1}{2i}\psi_{x_n}' \tilde{Q}_1^\eps f, Gf \right)_0 \right| +
 C \| f \|_{1, \tau} \left(\tau^{-1}\|\tilde{Q}_2^\eps f \|_0 + 
+ \|\tilde{Q}_1^\eps f \|_0  + \tau^{-1}\| f \|_{1,\tau}
+ \| D_a  f\|_0  \right) 
\enan
According to~\eqref{e:garding} and~\eqref{e:double-simple-norm} and~\eqref{estimMG}, this now implies
\bna
\|f\|_{1, \tau}^2  &\lesssim &\Re \left(B_2^\eps f,f\right) + \|\tilde{Q}_1^\eps f \|_0^2
+ \left|\left( \frac{1}{2i}\psi_{x_n}' \tilde{Q}_1^\eps f, Gf \right)_0 \right| +
 \tau^{-2}\|\tilde{Q}_2^\eps f \|_0^2   + \| D_a f\|_0^2 .
\ena
Coming back to~\eqref{e:ptaupsi-t-B2}, we obtain, for $\tau$ large enough,
\bna
\tau \|f\|_{1, \tau}^2
& \lesssim & 
\nor{P_{\psi,\e}f}{0}^2- \tau \Re \left( \mathcal{B}^\eps(f) \right)
- \nor{\tilde{Q}_2^\eps f}{0}^2 -\tau^2\nor{\tilde{Q}_1^\eps f}{0}^2
 + \tau\| D_a f\|_0^2 
 +\tau \left|\left( \frac{1}{2i}\psi_{x_n}' \tilde{Q}_1^\eps f, Gf \right)_0 \right| \\
 & \lesssim & 
\nor{P_{\psi,\e}f}{0}^2- \tau \Re \left( \mathcal{B}^\eps(f) \right)
 + \tau\| D_a f\|_0^2 
 +\tau \left|\left( \frac{1}{2i}\psi_{x_n}' \tilde{Q}_1^\eps f, Gf \right)_0 \right|  .
 \ena
Recalling te definition of $ \tilde{Q}_1^\eps$, we have $\psi_{x_n}'  \tilde{Q}_1^\eps = D_n + G_1$, where $G_1 \in \mathcal{D}^1_\tau$ is a differential operator of order $1$ (in $(\tau, D')$), we finally have
\bnan
\label{e:LRLem2}
\tau \|f\|_{1, \tau}^2
  \lesssim  
\nor{P_{\psi,\e}f}{0}^2- \tau \Re \left( \mathcal{B}^\eps(f) \right)
 + \tau\| D_a f\|_0^2 
 +\tau \left|\left( D_n f + G_1 f, Gf \right)_0 \right| ,
 \enan
 where $G$ a tangential pseudodifferential operator of order zero, 
 Recalling the form of $\mathcal{B}^\eps(f)$ in~\eqref{e:boundary-bilinear-term} gives the bound $| \mathcal{B}^\eps(f)| \leq  \tau^2 |f_{|x_n=0}|_{0}^2 
+  |D f_{|x_n=0}|_{0}^2$, which concludes the proof of~\eqref{e:carleman-bord-general}.
 
 Now if $f_{|x_n=0} = 0$, all tangential derivatives vanish. With~\eqref{e:LRLem2} and the form of $\mathcal{B}^\eps(f)$ in~\eqref{e:boundary-bilinear-term}, this yields
 \bna
\tau \|f\|_{1, \tau}^2
  \lesssim  
\nor{P_{\psi,\e}f}{0}^2- 2\tau  
 (\psi_{x_n}' D_n f , D_n f)_0
 + \tau\| D_a f\|_0^2 ,
 \ena
 which proves~\eqref{e:carleman-bord-dirichlet} since  $\psi_{x_n}'>0$ for $(x',x_n=0) \in K$.
This concludes the proof of Proposition~\ref{l:carleman-bord-compact}.
\enp
We turn now to the proof of Theorem~\ref{th:carleman-bord}.
\bnp[Proof of Theorem~\ref{th:carleman-bord}]
In the proof, we consider functions $u\in C^\infty_0 (K_{r_0/4})$ where $K_r$ is defined in~\eqref{e:def-Kr}.
Let $\chi \in C^\infty_0 (B_{\R^{n_a}}(0, r_0))$ such that $\chi=1$ on $B_{\R^{n_a}}(0, r_0/2)$. Setting  $v = Q_{\e,\tau}^{\psi} u = e^{-\frac{\e}{2\tau}|D_a|^2}(e^{\tau\psi}u)$ and $f = \chi(x_a)v(x)$, we have $\supp (f) \subset K_{r_0}$ so that we may apply Proposition~\ref{l:carleman-bord-compact} to $f$. We have $v-f = (1-\chi)Q_{\e,\tau}^{\psi} u = (1-\chi) e^{-\frac{\e}{2\tau}|D_a|^2}(\check{\chi}e^{\tau\psi}u)$ for some $\check{\chi} \in C^\infty_c (B_{\R^{n_a}}(0, r_0/3))$ with $\check{\chi} =1$ in a neighborhood of $\overline{B}_{\R^{n_a}}(0, r_0/4)$. As a consequence of Lemma~\ref{l:noyau-chaleur-hormander}, we have, for $\tau \geq \tau_0$
 \bnan
\|v\|_{1,\tau} \leq \|f\|_{1,\tau} + C e^{-C\frac{\tau}{\eps}} \|e^{\tau\psi}u\|_{1, \tau}
\enan
Now, it remains to estimate the terms on the RHS of Proposition~\ref{l:carleman-bord-compact} in terms of $v$. Notice first that the same reasonning with Lemma~\ref{l:noyau-chaleur-hormander} (using that $D_a$ is tangential) allows to estimate the boundary terms as:
 \bnan
 \label{e:term-bord-f-v-0}
|f_{|x_n=0}|_{0} \leq |v_{|x_n=0}|_{0} + C e^{-C\frac{\tau}{\eps}} |e^{\tau\psi}u_{|x_n=0}|_{0},
\enan
and, with $Dv - Df = D \left( (1-\chi) e^{-\frac{\e}{2\tau}|D_a|^2}(\check{\chi}e^{\tau\psi}u) \right)$, 
  \bnan
 \label{e:term-bord-f-v-1}
|Df_{|x_n=0}|_{0} & \leq & |Dv_{|x_n=0}|_{0} + C  e^{-C\frac{\tau}{\eps}} |e^{\tau\psi}u_{|x_n=0}|_{0}
+ C e^{-C\frac{\tau}{\eps}} |e^{\tau\psi} (\tau \psi' + D) u_{|x_n=0}|_{0} \nonumber\\
&&+ C e^{-C\frac{\tau}{\eps}} |e^{\tau\psi}Du_{|x_n=0}|_{0} \nonumber\\
&  \leq & |Dv_{|x_n=0}|_{0} + C \tau e^{-C\frac{\tau}{\eps}} |e^{\tau\psi}u_{|x_n=0}|_{0} 
+ C e^{-C\frac{\tau}{\eps}} |e^{\tau\psi}Du_{|x_n=0}|_{0}
\enan
Second, we estimate $\nor{P_{\psi,\e}f}{0}=\nor{P_{\psi,\e}\chi v}{0}=\nor{\chi P_{\psi,\e}v}{0}+\nor{[P_{\psi,\e},\chi ]v}{0}$. For the commutator, we write $[P_{\psi,\e},\chi ]v=[P_{\psi,\e},\chi ]e^{-\frac{\e}{2\tau}|D_a|^2}\check{\chi}e^{\tau\psi}u$. We notice that $[P_{\psi,\e},\chi ]$ is a differential operator of order $1$ in $(D,\tau)$ with some coefficients supported on $\supp(\chi'_{x_a} )$ that is, away from $\supp (\check{\chi})$. In particular, Lemma \ref{l:noyau-chaleur-hormander} implies $\nor{[P_{\psi,\e},\chi ]v}{0}\leq C e^{-c\frac{\tau}{\eps} }\nor{e^{\tau\psi}u}{1,\tau}$. This yields
\bnan
\label{estimPpsi}
\nor{P_{\psi,\e}f}{0}\leq \nor{P_{\psi,\e}v}{0}+C e^{-c\frac{\tau}{\eps} }\nor{e^{\tau\psi}u}{1,\tau}
\enan

Now, it remains to treat the term $\| D_a f\|_0$. Similarly, we obtain
\bnan
\label{estimDa1}
\| D_a f\|_0=\| D_a (\chi v)\|_0 \leq \| \chi D_a v\|_0+\| \chi'_{x_a} e^{-\frac{\e}{2\tau}|D_a|^2}\check{\chi}e^{\tau\psi}u \|_0\leq \|  D_a v\|_0 +C e^{-c\frac{\tau}{\eps} }\nor{e^{\tau\psi}u}{0}
\enan
where we have used again Lemma \ref{l:noyau-chaleur-hormander}.

Let $\varsigma$ a small constant to be fixed later on. We distinguish between frequencies of size smaller and bigger than $\varsigma \tau$. We get for $\tau \geq \frac{1}{\varsigma^2 \eps}$ large enough
 (so that the function $s\mapsto s e^{-\frac{\e}{2\tau}s^2}$ is decreasing on $s\geq \sqrt{\frac{\tau}{\eps}}$)
\bnan
\label{estimDa2}
\|  D_a v\|_0 =\|  D_a  e^{-\frac{\e}{2\tau}|D_a|^2}e^{\tau\psi}u\|_0&\leq &\|  D_a \mathds{1}_{|D_a|\leq \varsigma \tau}v\|_0+\|  D_a \mathds{1}_{|D_a|\geq \varsigma \tau} e^{-\frac{\e}{2\tau}|D_a|^2}e^{\tau\psi}u\|_0\nonumber\\
&\leq & \varsigma \tau \|  v\|_0+\varsigma \tau e^{-\frac{\tau\varsigma^2 \e}{2}}\| e^{\tau\psi}u\|_0
\enan

We may now apply Proposition~\ref{l:carleman-bord-compact} to $f$.
Combining the Carleman estimate \eqref{e:carleman-bord-general} with \eqref{estimPpsi}, \eqref{estimDa1}, \eqref{estimDa2}, \eqref{e:term-bord-f-v-0}, \eqref{e:term-bord-f-v-1}, we obtain, for some $C_1>0$ and $\tau \geq \tau_0$ with $\tau_0$ (depending also on $\varsigma , \eps)$) sufficiently large, 
\bna
C_1 \tau \|v\|_{1, \tau}^2 & \leq &
 \nor{P_{\psi,\e}v}{0}^2+C e^{-2c\frac{\tau}{\eps} }\nor{e^{\tau\psi}u}{1,\tau}^2+\varsigma^2 \tau^3 \|  v\|_0^2+\varsigma^2 \tau^3 e^{-\tau\varsigma^2 \e}\| e^{\tau\psi}u\|_0^2 \\
 && +  \tau^3 |v_{|x_n=0}|_{0}^2 +  \tau^3 e^{-2c\frac{\tau}{\eps}} |e^{\tau\psi}u_{|x_n=0}|_{0}^2
 + \tau |Dv_{|x_n=0}|_{0}^2 + \tau e^{-2 c\frac{\tau}{\eps}} |e^{\tau\psi}Du_{|x_n=0}|_{0}^2 .
\ena
Fixing $\varsigma \leq C_1/2$, this yields, for some $\mathsf{d}>0$ ($\eps$ is fixed already) and $\tau \geq \tau_0$, 
\bnan
\label{e:carleman-final-1}
\frac{C_1}{2} \tau \|v\|_{1, \tau}^2 & \leq &
 \nor{P_{\psi,\e}v}{0}^2+C e^{-\mathsf{d}\tau}\nor{e^{\tau\psi}u}{1,\tau}^2 \nonumber \\
 && +  \tau^3 |v_{|x_n=0}|_{0}^2 + e^{-\mathsf{d}\tau} |e^{\tau\psi}u_{|x_n=0}|_{0}^2
 + \tau |Dv_{|x_n=0}|_{0}^2 + e^{-\mathsf{d} \tau } |e^{\tau\psi}Du_{|x_n=0}|_{0}^2 .
\enan
Similarly, if moreover $\psi_{x_n}'>0$ for $(x',x_n=0) \in K_{r_0}$, then~\eqref{e:carleman-bord-dirichlet} yields for all $u \in C^\infty_0(K_{r_0/4})$ such that $u_{|x_n=0} = 0$, 
\bna
\tau \|v\|_{1, \tau}^2 \lesssim 
 \nor{P_{\psi,\e}v}{0}^2+ e^{-2c\frac{\tau}{\eps} }\nor{e^{\tau\psi}u}{1,\tau}^2+\varsigma^2 \tau^3 \|  v\|_0^2+\varsigma^2 \tau^3 e^{-\tau\varsigma^2 \e}\| e^{\tau\psi}u\|_0^2 ,
\ena
and hence
\bnan
\label{e:carleman-final-2}
\frac{C_1}{2} \tau \|v\|_{1, \tau}^2 & \leq &
 \nor{P_{\psi,\e}v}{0}^2+ e^{-\mathsf{d}\tau}\nor{e^{\tau\psi}u}{1,\tau}^2 .
\enan
Rewriting~\eqref{e:carleman-final-1}-\eqref{e:carleman-final-2} in terms of $u$ concludes the proof of Theorem~\ref{th:carleman-bord}.
\enp

\subsection{The local quantitative uniqueness result}
The Carleman estimates of the previous section have been proved when $P$ has a very specific form. 
Before proving the local quantitative uniqueness result, we first state them in a more invariant way that can be obtained by change of coordinates in $x_b$. When doing so, we strengthen the assumptions made on the operator $P$, still encompassing the cases of wave and Schr\"odinger operators (or more generally of the form of Remark~\ref{rknoncaractwave})

Up to now, and until the end of the section, $P$ will have the following property:
\begin{hypo}
\label{assumptionP}
$P$ is a differential operator on $\R^{n_a}\times \R^{n_b}_+$ of order two with coefficients analytic in the variable $x_a$. Assume moreover that $P$ has principal symbol independent of $x_a$ of the form $p(x,\xi) = q_{x_b}(\xi_a)+\tilde{q}_{x_b}(\xi_b)$, where $q_{x_b}$, $\tilde{q}_{x_b}$ are smooth $x_b$-families of real quadratic forms on $\R^{n_a}$ and $\R^{n_b}$ respectively.

Moreover, if $V  \in L^\infty (\R^{n_b}_+)$ and $W \in L^\infty (\R^{n_b}_+; \R^n )$, independent of $x_a$, we denote $P_{V,W} = P+W \cdot \nabla + V$.
\end{hypo}

 The proof of the local quantitative uniqueness will then be essentially the same as in the boundaryless case. 
The following Proposition is the counterpart, in the boundary case, of the end of the first step in Section~\ref{sectionlocal} (hence containing the geometrical part of the proof of the local uniqueness result).

\begin{proposition}
\label{p:boundary-geom}
Let $x^0 \in \{x_n = 0\}$ and let $P$ satisfying Assumption \ref{assumptionP}. 

Assume that $\{x_n = 0\}$ is non-characteristic with respect to $P$.  

Let $\phi$ be a function  defined in a neighborhood of $x^0$ in $\R^n$ such that $\phi(x^0)=0$, and $\{\phi = 0\}$ is a $C^2$ strongly pseudoconvex oriented surface at $x^0$ in the sense of Definition~\ref{def: pseudoconvex-surface}.

Then, there exists $R_0>0$ and a smooth function $\psi : B(x^0 , 4R_0) \to \R$ which is a quadratic polynomial with respect to $x_a \in \R^{n_a}$,  such that for any $R\in (0,R_0]$, there exist  $\eps, \delta ,\rho , r ,\mathsf{d} , \tau_0, C >0$, such that we have 

\begin{enumerate} 

\item $\delta \leq \frac{\mathsf{d}}{8}$ and~\eqref{geomrhobis}-\eqref{geomrho}-\eqref{eq:Brsubsetpsi},

\item
 for any $\tau \geq \tau_0 $, the Carleman estimate \eqref{Carlemanfromboundary} holds for $P$, for all $u \in C^\infty_0(\R^n_+)$ with $\supp(u) \subset B(x^0 , 4R)$.
 
 If moreover $\phi_{x_n}'(x^0)>0$, the Carleman estimate \eqref{Carlemanuptoboundary} holds for $P$ for all $u \in C^\infty_0(\R^n_+)$ with $\supp(u) \subset B(x^0 , 4R)$ and $u_{|x_n=0} = 0$.

The estimates can also be made uniform for $\tau >\tau_0 \max\{1 , \|V\|_{L^\infty}^{\frac23},\|W\|_{L^\infty}^2\}$ if $P$ is replaced by $P_{W,V}$, as in Corollary \ref{cor:carl-potentiel-bord}.
\end{enumerate}
\end{proposition}
\bnp
First, according to non-charactericticity assumption, we have $\tilde{q}_{x_b}(\xi_b) \neq 0$ for $x_b = (x_b' ,0)$ and $\xi_b'=0,  \xi_b^n = 1$. We may thus place ourselves in normal geodesic coordinates for $\tilde{q}_{x_b}$ in $\R^{n_b}$, in a sufficiently small neighborhood of $\{x_n = 0\}$. More precisely (see \cite[Appendix~C.5]{Hoermander:V3}) there exists a local diffeomorphism $\Psi_b$ from a neighborhood of $x^0_b$ in $\R^{n_b}_+$ to a neighborhood of $0$ in $\R^{n_b}_+$ such that, setting $\Psi := \id_{\R^{n_a}} \otimes \Psi_b$, the principal part of $\Psi^* P$ takes the form $(\xi_b^n)^2 +r(x_b ,\xi_a,\xi_b')$. From the function $\phi \circ\Psi^{-1}$ (still defining a strictly pseudoconvex surface for $\Psi^* P$ since this property is invariant), we can construct a quadratic polynomial $\tilde\psi$ exactly as in Lemma~\ref{l:phi-psi}/Corollary~\ref{c:geometric-setting} such that the Carleman estimates~\eqref{Carlemanfromboundary}-\eqref{Carlemanuptoboundary} hold for $\Psi^* P$ and $\tilde\psi$. We then use Corollary \ref{cor:carl-potentielanal-bord} and then Corollary \ref{cor:carl-potentiel-bord} to allow, first, lower order terms analytic in $x_a$ and then lower order terms independent on $x_a$ with the right estimates (note that both properties are invariant by our change of coordinates in $x_b$). Applying then the diffeomorphism $\Psi$ to come back to the original setting yields the sought estimate with $\psi = \tilde\psi \circ \Psi$, which remains a quadratic polynomial with respect to the variable $x_a$ (only) since $\Psi := \id_{\R^{n_a}} \otimes \Psi_b$.  This proves Item~2.

Finally, the geometric assertion of Item~1 comes from the application of Lemma~\ref{l:phi-psi} in the geodesic coordinates. There, using the distance $N(x,y)=|\Psi^{-1}(x)-\Psi^{-1}(y)|$ allows to obtain \eqref{geomrhobis}-\eqref{geomrho}-\eqref{eq:Brsubsetpsi} with euclidian balls as claimed by Item 1.
\enp
The aim of this section is now to prove the following two local results, namely local quantitative uniqueness up to and from the boundary.
\begin{theorem}[Local quantitative uniqueness up to the boundary]
\label{thmquantitativeupboundary}

Let $x^0 \in \{x_n = 0\}$ and $P$ satisfying Assumption \ref{assumptionP}. Assume that $\{x_n = 0\}$ is non-characteristic with respect to $P$.

Assume that there is a function $\phi$ defined in a neighborhood of $x^0$ in $\R^n$ such that $\phi(x^0)=0$, and $\{\phi = 0\}$ is a $C^2$ strongly pseudoconvex oriented surface at $x^0$ in the sense of Definition~\ref{def: pseudoconvex-surface} and such that $\phi_{x_n}'(x^0)>0$.

Then there exists $R_0>0$ such that for any $R\in (0,R_0)$, there exist $r>0, \rho>0$ for any $\vartheta\in C^{\infty}_0(\R^n)$  such that $\vartheta(x)=1$ on a neighborhood of $\left\{\phi\geq 2\rho\right\}\cap B(x^0,3R)$,
for all $c_1, \kappa>0$ there exist $C, \kappa', \beta, \tilde{\tau}_0 >0$ such that we have
\bna
\nor{M^{\beta\mu}_{c_1\mu} \sigma_{r,c_1\mu} u}{1,+}\leq C e^{\kappa \mu}\left(\nor{M^{\mu}_{c_1\mu} \vartheta_{c_1\mu} u}{1,+} + \nor{Pu}{L^2(B(x^0,4R) \cap \R^n_+)}\right)+Ce^{-\kappa' \mu}\nor{u}{1,+}.
\ena
for all $\mu \geq \tilde{\tau}_0 $ and $u\in C^{\infty}_0(\R^{n}_+)$ such that $u_{|x_n=0} = 0$.

Moreover, under the same assumptions, there exists $C_0, \kappa', \beta, \tilde{\tau}_0 >0$ such that for all $V\in L^{\infty}(\R^{n_b})$, $W\in  L^{\infty}(\R^{n_b};\R^n)$ the previous estimate is still true with $P$ replaced by $P_{W,V} = P+  W\cdot \nabla + V$ with $C$ replaced by $C_0 \max \left\{1,\nor{W}{L^{\infty}}\right\}$ and uniformly for all $\mu \geq \tilde{\tau}_0 \max\{1,\|V\|_{L^\infty}^{\frac23},\nor{W}{L^{\infty}}^2\}$.
\end{theorem}

This theorem is proved similarly as in the case without boundary. See the details in the proof of the related Theorem \ref{thmquantitativefromboundary} below. 
\begin{theorem}[Local quantitative uniqueness from the boundary]
\label{thmquantitativefromboundary}
Let $x^0$ and $P$ satisfying Assumption \ref{assumptionP}.

Assume that $\{x_n = 0\}$ is non-characteristic with respect to $P$.

 Assume that the function $\phi(x) = -x_n$ satisfies the property of Definition~\ref{def: pseudoconvex-surface} at $x^0$.

Then there exists $R_0>0$ such that for any $R\in (0,R_0)$, there exist $r>0$ 
for all $c_1, \kappa>0$ there exist $C, \kappa', \beta ,\tilde{\tau}_0 >0$ such that we have
\bna
\nor{M^{\beta\mu}_{c_1\mu} \sigma_{r,c_1\mu} u}{1,+}\leq C e^{\kappa \mu}\left(\nor{D_n u}{L^2(B(x^0,4R)\cap \left\{x_n=0\right\}} + \nor{P u}{L^2(B(x^0,4R) \cap \R^n_+)}\right)+Ce^{-\kappa' \mu}\nor{u}{1,+} .
\ena
for all $\mu \geq \tilde{\tau}_0$ and $u\in C^{\infty}_0(\R^{n}_+)$ such that $u_{|x_n=0} = 0$.

The same dependence of the constants holds if $P$ is replaced by $P_{W,V}$ as in Theorem \ref{thmquantitativeupboundary}
\end{theorem}
\bnp
The proof is very similar to the proof of Theorem \ref{th:alpha-unif} in Section \ref{sectionlocal}, using the Carleman estimate~\eqref{Carlemanfromboundary} of Theorem \ref{th:carleman-bord} . We only sketch it and underline the differences with respect to the boundaryless case. We moreover added the potential $V$ with respect to the general case; we need also check that it is painless in the proof.

\medskip 
\noindent
\textbf{Step 1: The geometric setting.} We start by choosing $\phi=-x_n$. The surface $\left\{\phi=0\right\}=\left\{-x_n=0\right\}$ is non characteristic by assumption, and according to Remark \ref{rknoncaractwave},  is hence a strongly pseudoconvex oriented surface for $P$. 
Proposition~\ref{p:boundary-geom} furnishes an appropriate convexified $\psi$, polynomial of degree two in the variable $x_a$, that satisfies the desired geometric conditions, together with the Carleman estimate~\eqref{Carlemanfromboundary}. We now follow the proof of the boundaryless case. 

\medskip 
\noindent
\textbf{Step 2: Using the Carleman estimate.}
The point is to use the Carleman estimate~\eqref{Carlemanfromboundary} with weight $\psi$, applied to the (compactly supported) function $w=\sigma_{2R}\sigma_{R,\lambda}\chi_{\delta,\lambda}(\psi)\chit_{\delta}(\psi)u$.

Similarly, using the same support property $\supp(\chi_\delta) \subset ]-8\delta , \delta[$, and Lemma \ref{lmboundchipsi}, we write
\bna
\nor{Q_{\e,\tau}^{\psi}P_{W,V} w}{0,+}&\leq& \nor{Q_{\e,\tau}^{\psi}\sigma_{2R}\sigma_{R,\lambda}\chi_{\delta,\lambda}(\psi) \chit_{\delta}(\psi) P_{W,V} u}{0,+}+\nor{Q_{\e,\tau}^{\psi}[\sigma_{2R}\sigma_{R,\lambda}\chi_{\delta,\lambda}(\psi) \chit_{\delta}(\psi) ,P_{W,V}]u}{0,+} \nonumber\\
&\leq &e^{\frac{\tau^2}{\lambda}}e^{\delta\tau}\nor{P_{W,V} u}{L^2(B(x^0,4R)\cap \left\{x_n\geq 0 \right\})}+\nor{Q_{\e,\tau}^{\psi}[\sigma_{2R}\sigma_{R,\lambda}\chi_{\delta,\lambda}(\psi)\chit_{\delta}(\psi),P_{W,V}]u}{0,+}.
\ena
Next, Lemma~\ref{lemmacommut} still holds in $\R^n_+$ since $x_a$ is a tangential variable (see Remark~\ref{rkrestriction}). Hence, the commutator term is bounded by 
\bna
\nor{Q_{\e,\tau}^{\psi}[\sigma_{2R}\sigma_{R,\lambda}\chi_{\delta,\lambda}(\psi)\chit_{\delta}(\psi),P]u}{0,+}
& \leq &
 C e^{2\delta \tau} \nor{ M^{2\mu}_{\lambda}\vartheta_{\lambda} u}{1,+} \\
 & & + C\lambda^{1/2} \tau^N \left(e^{-\frac{\eps\mu^2}{4\tau}}+ e^{-8\delta \tau}+ e^{\delta \tau -c\mu}\right)e^{\frac{\tau^2}{\lambda}}e^{\delta \tau }\nor{u}{1,+}  ,
\ena
with some $\vartheta$ (equal to one in a neighborhood of $\{\phi \geq 2 \rho\} \cap B(x^0,3R)$) supported in $\left\{\phi>\rho\right\}=\left\{x_n<-\rho\right\}$. 

Moreover, following Remark \ref{rklowerunif}, we can get uniform estimates for the commutator of $P_{W,V}$ by replacing $C$ by $C_0\max \left\{1,\nor{W}{L^{\infty}(\R^{n_b})}\right\}$. We will not write it any more for sake of clarity but it appears multiplically in all the estimates.

Since the operator $M^{\mu}_{c_1\mu}$ only applies in the tangential variable $x_a$, we have $\nor{M^{\mu}_{c_1\mu} \vartheta_{c_1\mu} u}{1,+}\leq \nor{\vartheta_{c_1\mu} u}{1,+}$. Moreover, since $\vartheta$ is supported in $\left\{x_n<-\rho\right\}$ and $\vartheta_{c_1\mu}=e^{-\frac{|D_a|^2}{c_1 \mu}}\vartheta$ is a regularization in the variable $x_a$, $\vartheta_{c_1\mu}$ is also supported in $\left\{x_n<-\rho\right\}$ and $\vartheta_{c_1\mu}(x)=0$ if $x_n\geq 0$. In particular, $\nor{\vartheta_{c_1\mu} u}{1,+}=0$. That is
\bna
\nor{Q_{\e,\tau}^{\psi}P_{W,V} w}{0,+}
& \leq & Ce^{\frac{\tau^2}{\lambda}}e^{\delta\tau}\nor{P_{W,V} u}{L^2(B(x^0,4R)\cap \R^n_+)}\\
&&+C\lambda^{1/2} \tau^N\left(e^{-\frac{\eps\mu^2}{4\tau}}+ e^{-8\delta \tau}+ e^{\delta \tau -c\mu}\right)e^{\frac{\tau^2}{\lambda}}e^{\delta \tau }\nor{u}{1,+} .
\ena
The other term in the Carleman estimate that we have to check are 
\bnan
\label{termbord}
\tau |\big(D(Q_{\e,\tau}^{\psi} w)\big)_{|x_n=0}|_{0}^2 + e^{-\mathsf{d} \tau } |e^{\tau\psi}Dw_{|x_n=0}|_{0}^2 \leq C \tau|e^{\tau\psi}D_nw_{|x_n=0}|_{0}^2 , 
\enan
where we have used that $u_{|x_n=0}=w_{|x_n=0}=0$. This also implies $$D_nw_{|x_n=0}=(\sigma_{2R}\sigma_{R,\lambda}\chi_{\delta,\lambda}(\psi)\chit_{\delta}(\psi)D_nu)_{|x_n=0} .$$ Since $ \nor{e^{\tau \psi}\chi_{\delta,\lambda}(\psi)}{L^{\infty}}\leq C \lambda^{1/2}e^{\delta\tau}e^{\frac{\tau^2}{\lambda}}$ thanks to Lemma \ref{lmboundchipsi}, the left hand-side of \eqref{termbord} is bounded by $C\lambda e^{2\delta\tau}e^{2\frac{\tau^2}{\lambda}}\tau\left|D_n u\right|_{L^2(B(x^0,4R)\cap \left\{x_n=0\right\})}^2$. 

So, combining the Carleman estimate of Corollary~\ref{cor:carl-potentiel-bord} and the previous bounds, we have proved for all  $\tau \geq \tau_0 \max\{1 , \|V\|_{L^\infty}^{\frac23}\}$, $\mu \geq 1$, $\frac{1}{C}\mu\leq \lambda \leq C\mu$, 
\bna
\tau^{1/2} \nor{Q_{\e,\tau}^{\psi}\sigma_{2R}\sigma_{R,\lambda}\chi_{\delta,\lambda}(\psi)\chit_{\delta}(\psi)u}{1,+,\tau}
& \leq  &
C e^{\frac{\tau^2}{\lambda}}e^{\delta\tau}\nor{P_{W,V} u}{L^2(B(x^0,4R)\cap \R^n_+)}
\nonumber \\
 & & +  
C\lambda^{1/2}\tau^{1/2} e^{\delta\tau}e^{\frac{\tau^2}{\lambda}}\left|D_n u\right|_{L^2(B(x^0,4R)\cap \left\{x_n=0\right\})}   \nonumber \\
 & & + C\lambda^{1/2} \tau^N\left(e^{-\frac{\eps\mu^2}{4\tau}}+ \tau e^{-8\delta \tau}+ e^{\delta \tau -c\mu}\right)e^{\frac{\tau^2}{\lambda}}e^{\delta \tau }\nor{u}{1,+}   .
\ena
So, denoting $D=e^{\kappa\mu}\left(\nor{D_n u}{L^2(B(x^0,4R)\cap \left\{x_n=0\right\})} +\nor{Pu}{L^2(B(x^0,4R)\cap \R^n_+)}\right)$, we can rewrite it as
\bna
\nor{Q_{\e,\tau}^{\psi}\sigma_{2R}\sigma_{R,\lambda}\chi_{\delta,\lambda}(\psi)\chit_{\delta}(\psi)u}{1,+,\tau}
& \leq  &
C\mu^{1/2} e^{\delta\tau}e^{C\frac{\tau^2}{\mu}}e^{-\kappa \mu} D   \nonumber \\
 & & + C\mu^{1/2}\tau^N \left(e^{-\frac{\eps\mu^2}{4\tau}}+ \tau e^{-8\delta \tau}+ e^{\delta \tau -c\mu}\right)e^{C\frac{\tau^2}{\lambda}}e^{\delta \tau }\nor{u}{1,+}   .
\ena

\medskip 
\noindent
\textbf{Step 3: A complex analysis argument.}
We now proceed exactly as in the boundaryless case. For any test function $f \in C^{\infty}_0(\R^n_+)$, we define the distribution $h_f$ (with $\beta>0$ to be chosen later on)
$$
\langle h_f , w\rangle_{\E'(\R), C^\infty(\R)} := \langle  \sigma_{2R}\sigma_{R,\lambda}\chi_{\delta,\lambda}(\psi)\chit_{\delta}(\psi)  w(\psi) u , (M^{\beta\mu} f)\rangle_{H^1_0(\R^n_+), H^{-1}(\R^n_+)} .
$$
We proceed similarly, noticing at the end that $C^{\infty}_0(\R^n_+)$ is dense in the dual space $H^{-1}(\R^n_+)$ and that all operations are tangential.
The analogue of Lemma~\ref{l:local-lem-complex} is proved with the same complex analysis argument (which does not involve the $x$-space, but only  complexifies the Carleman large parameter $\tau$), using Lemma~\ref{l:holomorphic}. This yields the analogous result for $\mu \geq C \tau_0 \max\{1 , \|V\|_{L^\infty}^{\frac23}\}$.

Finally, it remains to transfer the estimate obtained on $\nor{Q_{\e,\tau}^{\psi}\sigma_{2R}\sigma_{R,\lambda}\chi_{\delta,\lambda}(\psi)\chit_{\delta}(\psi)u}{1,+,\tau}$ to an estimate on $\nor{M^{\beta\mu}_{c_1\mu} \sigma_{r,c_1\mu} u}{1,+}$. The computations of the end of Section~\ref{s:complex-analysis-arg} remain valid in the present context for the following two reasons: (a) the operators $M^{\beta\mu}_{c_1\mu}$ are tangential and the associated estimates of Section~\ref{s:prelim-estim} still hold; (b) these computations only rely on the geometric fact that $\sigma_R=\chi_{\delta}(\psi)=\chit_{\delta}(\psi)=\eta_\delta(\psi)=1$ on a neighborhood of $\supp(\sigma_r)$, which now follows from Proposition~\ref{p:boundary-geom}.
\enp

\subsection{The semiglobal estimate with boundary}
In this section, we prove a version of Theorem~\ref{thmsemiglobal}/\ref{thmsemiglobaldep} adapted to the boundary value problem. More precisely, the following result considers, under the assumptions of the above uniqueness results, the Dirichlet boundary condition at the bottom and the top of the graph, with an observation at the bottom.

Recall that in the present context, the analytic variable is supposed to be tangential to the boundary. In the following results (as opposed to the boundaryless case), this translates into the fact that we assume that, in the splittings $x =(x' , x_n) \in \R^{n-1} \times [0,\ell_0]$ and $x = (x_a,x_b) \in \R^{n_a} \times \R^{n_b}$, the variable $x_n = x_b^n$ always belongs to the $x_b$ variables.

In Theorem~\ref{thmsemiglobalwavebord} below, we state the semiglobal estimate with an observation from the boundary (i.e. the first hypersurface $S_0$ is a Dirichlet boundary) and if the last hypersurface $S_1$ touches a (Dirichlet) boundary.
This is the most intricate situation. The proof is the same in the cases where the last hypersurface does not touch the boundary, or if we have an internal observation around the first surface. We do not state these cases for the sake of concision. 
\begin{theorem}
\label{thmsemiglobalwavebord}
Let $D$ be a bounded open subset of $\R^{n-1}$ with smooth boundary. Let $G = G(x',\eps) \in C^1(\overline{D} \times [0,1+\eta) )$, such that 
\begin{itemize}
\item For all $\eps \in (0,1]$, we have $\{x' \in \R^{n-1} , G(x' , \eps) \geq 0\} = \overline{D}$
\item for all $x' \in D$, the function $\eps \mapsto G(x' , \eps)$ is strictly increasing
\item for all $\eps \in (0,1]$, we have $\{x' \in \R^{n-1} , G(x' , \eps) = 0\} = \partial D$
\end{itemize}
We further set 
$$
\ell_0= \max_{x' \in \overline{D}} G(x' , 1), \qquad  G(x', 0) = 0, \qquad S_0=\overline{D}\times \left\{x_n=0\right\} ,
$$
and,  for $\eps \in (0,1]$, 
\bna
&&S_\eps = \{(x',x_n) \in \R^{n} , x_n \geq 0 \text{ and } G(x' , \eps) =x_n \} =( \overline{D} \times \R )\cap\{(x',x_n) \in \R^{n} , G(x' , \eps) =x_n \},\\
&&K = \{x \in \R^n ,0\leq  x_n \leq G(x', 1)\}.
\ena
We let $\Omega$ be a neighborhood of $K$ in $\R^{n-1}\times [0,\ell_0]$ and $\widetilde{D}$ be a neighborhood of $\overline{D}$ in $\R^{n-1}$. 

Let $P$ satisfying Assumption \ref{assumptionP}. Assume that $\{x_n = 0\}$ and $\{x_n = \ell_0\}$ are non-characteristic with respect to $P$.

Assume also that for any $\e\in [0,1]$, the function
$$
\phi_\eps(x',x_n): =  G(x' , \eps) - x_n 
$$
is strictly pseudoconvex with respect to $P$ on the whole $S_\eps$.

Then, there exist a neighborhood $U$ of $K$ and constants $\kappa, C, \mu_0 >0$ such that for all for all $u\in C^{\infty}_0(\R^{n-1}\times [0,\ell_0])$ satisfying
\bna
 u_{|x_n=0}=u_{|x_n=\ell_0} = 0, \textnormal{ on } \widetilde{D}
\ena
 we have, with $P_V = P + V$
$$
\nor{ u}{L^2(U)}\leq C e^{\kappa \mu}\left(\nor{D_n u_{|x_n=0}}{L^2(\widetilde{D})}  + \nor{Pu}{L^2(\Omega)}\right)+\frac{C}{\mu}\nor{u}{H^1(\R^{n-1}\times [0,\ell_0])}
$$
for all $\mu\geq \mu_0$.

Moreover, under the same assumptions, there exists $C_0, \kappa', \beta, \tilde{\tau}_0 >0$ such that for all $V\in L^{\infty}(\R^{n_b})$, $W\in  L^{\infty}(\R^{n_b};\R^n)$ the previous estimate is still true with $P$ replaced by $P_{W,V} = P+  W\cdot \nabla + V$ with $C$ replaced by $C_0 \max \left\{1,\nor{W}{L^{\infty}}\right\}$ and uniformly for all $\mu \geq \tilde{\tau}_0 \max\{1,\|V\|_{L^\infty}^{\frac23},\nor{W}{L^{\infty}}^2\}$.
\end{theorem}
\bnp
For simplicity, we first make the proof for $V=0$ and we will check the dependence in $V$ at the end.

We will use the same scheme of proof as for Theorem \ref{thmsemiglobaldep}. We first note that the notion of $\lhd$ can be extended to the case when there is a boundary and the variables $\xi_a$ are tangential to this boundary. Then, the local uniqueness results of Corollary \ref{cordominance}, and  Theorem \ref{thmquantitativeupboundary}, can be written as 
\bnan
\label{e:niemelhd}
B(x^0,r) \lhd \left[\left\{\phi> \rho\right\}\cap B(x^0,4R)\right]
\enan
as long as $B(x^0,4R)\cap \left\{x_n=0\right\}=\emptyset$. Indeed, in~\eqref{e:niemelhd}, the case where $B(x^0,4R)\cap \left\{x_n=\ell_0\right\} = \emptyset$ follows from the internal quantitative uniqueness result (e.g. Corollary~\ref{cordominance}), whereas the case ``up to the boundary'' $B(x^0,4R)\cap \left\{x_n=\ell_0\right\}\neq \emptyset$ follows from Theorem \ref{thmquantitativeupboundary}. To apply this theorem in this context, one needs to make the change of variables $x_n \mapsto \ell_0-x_n$, which transforms $\left\{x_n\leq \ell_0\right\}$ into $\R^{n}_+$ and $\phi_\eps = G(x',\e) -x_n$ to $\tilde \phi_\eps :=G(x',\e) -(\ell_0 - x_n)$. The condition $\d_{x_n} \tilde \phi_\eps =- \d_{x_n}\phi_\eps =1>0$ is satisfied, the surface $\{x_n=0\}$ (new coordinates) remains noncharacteristic; the pseudoconvexity assumption is invariant as well. 

\medskip
\textbf{Claim:} For any $\widetilde{\omega}$ open neighborhood of $S_0=\overline{D}\times \left\{x_n=0\right\}$, there exists an open neighborhood $U$ of $K$ (for the topology of $\R^{n-1}\times [0,\ell_0]$) such that
\bna
U\lhd \widetilde{\omega}.
\ena
The claim can be proved with almost the same proof as that of Theorem \ref{thmsemiglobaldep}, but using in addition Theorem \ref{thmquantitativeupboundary} instead of only Theorem \ref{th:alpha-unif}. So, we have to ensure that in the proof, we only apply Theorem \ref{thmquantitativeupboundary} for some points $x_i^{\e_j}$ with $B(x_i^{\e_j},4R_i^{\e_j})\cap \left\{x_n=0\right\}=\emptyset$. This is the point of Remark \ref{rkbouleloin}, which then allows to prove the Claim as in Theorem \ref{thmsemiglobaldep}.

\medskip

Now, let $x^0\in \overline{D}\times \left\{x_n=0\right\}$. We apply Theorem \ref{thmquantitativefromboundary} with $R_x$ small enough so that $\R^{n-1}\times \left\{x_n=0\right\}\cap B(x,R_x)\subset \left\{x_n=0\right\}\times \widetilde{D}$ and $B(x,R_x)\subset \Omega $. It gives $r_x$ so that for some $\beta , \kappa , C , \kappa' , \mu_0>0$,
\bna
\nor{M^{\beta\mu}_{c_1\mu} \sigma^{x^0}_{r,c_1\mu} u}{1,+}\leq C e^{\kappa \mu}\left(\nor{D_n u_{|x_n=0}}{L^2(\widetilde{D})}  + \nor{Pu}{L^2(\Omega)}\right)+Ce^{-\kappa' \mu}\nor{u}{1,+} .
\ena
where $\sigma^{x^0}_r$ is centered in $x^0$. By compactness of $\overline{D}$, we can cover it by a finite number of such balls $\left(B(x^i,r^{i})\right)_{i\in I}$. Pick $\vartheta\in C^{\infty}_0(\R^{n-1}\times [0,\ell_0])$ with $\supp(\vartheta)\subset \cup_{i\in I} B(x^i,r^{i})$ so that $\vartheta=1$ in a neighborhood $\widetilde{\omega}$ of $S_0$. Lemma \ref{lemma: Gevrey Chambertinsomme} gives that for functions $\sigma^{x^i}_{r^i}$ equal to one on $B(x^i,r^{i})$, the estimate
\bna
\nor{M^{2\beta\mu}_{\mu} \vartheta_{\mu} u}{m-1}\leq \sum_{i\in I}\nor{M^{\beta\mu}_{c_1\mu} \sigma^{x^i}_{r^i,c_1\mu} u}{1,+}+Ce^{-c \mu}\nor{u}{1,+} 
\ena
Now, apply the Claim with the selected $\widetilde{\omega}$ and for some $\widetilde{\vartheta}\in  C^{\infty}_0(U\cap \R^{n-1}\times [0,\ell_0])$ equal to $1$ in a neighborhood of $K$. For some $\kappa_1<\min(c/2,\kappa')$, there exist $C_1 , \kappa'_1>0$ so that 
\bna
\nor{M^{\alpha\mu}_{\mu} \widetilde{\vartheta}_{\mu}  u}{1,+}\leq C e^{\kappa_1 \mu}\left(\nor{M^{2\beta\mu}_{\mu} \vartheta_{\mu} u}{m-1} + \nor{Pu}{L^2(\Omega)}\right)+Ce^{-\kappa'_1 \mu}\nor{u}{1,+} .
\ena
This implies, for some $\kappa_2 , \kappa_2' , C>0$,
\bna
\nor{M^{\alpha\mu}_{\mu} \widetilde{\vartheta}_{\mu}  u}{1,+}\leq  C e^{\kappa_2 \mu}\left(\nor{D_n u_{|x_n=0}}{L^2(\widetilde{D})}  + \nor{Pu}{L^2(\Omega)}\right)+Ce^{-\kappa'_2 \mu}\nor{u}{1,+}.
\ena
We finish the proof as in Theorem \ref{thmsemiglobal} once Theorem \ref{thmsemiglobaldep} is proved, taking into account Remark \ref{rkrestriction}. 

\medskip
Now, if $P$ is replaced by $P_{W,V}$, we want to obtain the uniformity with respect to the size of $V$ and $W$. It is clear that the proof of the Theorem involves a finite number of applications of Theorem \ref{thmquantitativeupboundary} and \ref{thmquantitativefromboundary}. Indeed, the scheme of proof of Theorem \ref{thmsemiglobaldep} only involves a finite number of applications of the geometric propagation of the property $\lhd$. They can be divided in two categories: the general ones described in Proposition \ref{propstrong} that are completely independent of the operator $P$ (so, the constants will be independent of $V$ and $W$) and those using Theorems \ref{thmquantitativeupboundary} and \ref{thmquantitativefromboundary} where the dependence of the constants $\mu_0$ and $C$ is explicitly described. Note also that in all the properties (propagation, transitivity, simplification...) that we prove about some relations $\lhd$, once $\kappa$ is fixed, the $\mu_0$ corresponding to some relations is always transformed into the some linear combination (with universal constants) of the $\mu_0$ corresponding to the previous ones. This is the same for the constants $C$ involved in $\lhd$. Finally, a finite number of applications of these rules will always conclude with the restriction of the form $\mu \geq \tilde{\tau}_0 \max\{1,\|V\|_{L^\infty}^{\frac23},\nor{W}{L^{\infty}}^2\}$ and $C$ of the form $C_0 \max \left\{1,\nor{W}{L^{\infty}}\right\}$, once $\kappa$ is fixed.
\enp

\section{Applications}
\label{Application}
We now give applications of the above main results, namely Theorem~\ref{thmsemiglobal} and, in the case with boundary Theorem~\ref{thmsemiglobalwavebord} to the wave and Schr\"odinger operators. In these applications, we study an evolution equation in the analytic variable. We thus have $n_a =1$, $n_b = n-1 = \dim (\M)$ and we denote accordingly $t=x_a$ the time variable and $x=x_b$ the space variable. In this section, we prove general versions of Theorems~\ref{thmobserwaveintro} and~\ref{thmobserschrodintro}: we add (complex valued) lower order terms that are analytic in time. We also provide uniform estimates with respect to these lower order terms if they are time independent.

The proof consists each time in the application of the quantitative estimates of Theorem \ref{thmsemiglobalwavebord} and then using energy estimates to relate the energy to the initial data and source term.
Note that the first step, the quantitative unique continuation itself, does not see the lower order terms. For instance, Theorem \ref{thmobserSchro-bis} below is equally valid for the Schr\"odinger equation $i\partial_t+\Delta_g$, the heat equation $\partial_t-\Delta_g$, Ginzburg-Landau $e^{i\theta}\partial_t+\Delta_g$, etc.  

\subsection{The wave equation}
Our result for the wave equation can be formulated as follows.
\begin{theorem}
\label{thmobserwave}
Let $\M$ be a compact Riemannian manifold with (or without) boundary, $\Delta_g$ the Laplace-Beltrami operator on $\M$, and 
$$P = \partial_t^2 -\Delta_g + W_0 \d_t + W_1\cdot \nabla + V$$ with $V,W_0,W_1, \div(W_1)$ bounded and depending analytically on the variable $t \in (-T,T)$.

For any nonempty open subset $\omega$ of $\M$ and any $T> \mathcal{L}(\M ,\omega)$, there exist $C, \kappa ,\mu_0>0$ such that for any $(u_0,u_1)\in H^1_0(\M)\times L^2(\M)$, $f\in L^2((-T,T)\times\M)$ and associated solution $u$ of 
\bneqn
\label{e:free-wavethm}
Pu= f& & \text{ in } (-T,T) \times \M , \\
u_{\left|\partial \M\right.}=0 & &  \text{ in } (-T,T) \times \d \M ,\\
(u,\partial_tu)_{\left|t=0\right.}=(u_0,u_1)& & \text{ in }\M ,
\eneqn
we have, for any $\mu\geq \mu_0$,
\bnan
\label{estimthm1}
\nor{(u_0,u_1)}{L^2\times H^{-1}}
\leq C e^{\kappa \mu}\left(\nor{u}{L^2((-T,T); H^1(\omega))} + 
\nor{f}{L^2((-T,T)\times \M)} \right)
+ \frac{C}{\mu}\nor{(u_0,u_1)}{H^1\times L^2}.
\enan
If moreover all coefficients of $P$ are analytic in $t$ and $x$, and $\partial \M=\emptyset$, there exists $\tilde{\varphi} \in C^\infty_0((-T,T)\times \omega)$ such that for any $s \in \R$, we have
\bna
\nor{(u_0,u_1)}{L^2\times H^{-1}}\leq C e^{\kappa \mu}
\left(\nor{\tilde{\varphi} u}{H^{-s}((-T,T)\times \M)} + 
\nor{f}{L^2((-T,T)\times \M)} \right)
+\frac{C}{\mu}\nor{(u_0,u_1)}{H^1\times L^2}.
\ena 
If $\partial\M\neq \emptyset$ and $\Gamma$ is a non empty open subset of $\partial \M$, for any $T>\mathcal{L}(\M ,\Gamma)$, there exist $C, \kappa,\mu_0 >0$ such that for any $(u_0,u_1)\in H^1_0(\M)\times L^2(\M)$, $f\in L^2((-T,T)\times\M)$ and associated solution $u$ of \eqref{e:free-wavethm}, we have
\bnan
\label{estimthm2}
\nor{(u_0,u_1)}{L^2\times H^{-1}}\leq 
C e^{\kappa \mu}\left(\nor{\partial_{\nu}u}{L^2((-T,T)\times \Gamma)} + 
\nor{f}{L^2((-T,T)\times \M)} \right)
 +\frac{C}{\mu}\nor{(u_0,u_1)}{H^1\times L^2}.
\enan
Finally, if $V$, $W_0$ and $W_1$ are time-independent then we have the following stronger result. There exist $C_0, \kappa ,\mu_0>0$ such that for any $(u_0,u_1)\in H^1_0(\M)\times L^2(\M)$, $f\in L^2((-T,T)\times\M)$ and associated solution $u$ of \eqref{e:free-wavethm}, and for any $V,W_0,W_1, \div(W_1)$ bounded in the $x$-variable (all independent of $t$), estimates \eqref{estimthm1} and \eqref{estimthm2} hold uniformly for all $\mu\geq \mu_0\max\{1 , \|V\|_{L^\infty}^{\frac23},\|W_0\|_{L^\infty}^{2},\|W_1\|_{L^\infty}^{2}\}$ with constant
\bna
C= C_0\exp \left(C_0\max \left\{\nor{V}{L^{\infty}(\M)},\nor{W_0}{L^{\infty}(\M)},\nor{W_1}{L^{\infty}(\M)},\nor{\div(W_1)}{L^{\infty}(\M)}\right\}\right).
\ena
\end{theorem}
\begin{remark}
Using Lemma \ref{lmfonction} and the admissibility $\nor{\partial_{\nu}u}{L^2(]-T,T[\times \Gamma)}\leq C \nor{(u_0,u_1)}{H^1\times L^2}$, the previous estimates can be written as in Corollary \ref{corlogstab} with some constants depending explicitly on the norms of the lower order terms.
\end{remark}
Theorem~\ref{thmobserwave} above is a consequence of the following result, together with basic energy estimates for solutions to the wave equation.
\begin{theorem}
\label{thmobserwave-bis}
Let $\M$ be a compact Riemannian manifold with (or without) boundary, $\Delta_g$ the Laplace-Beltrami operator on $\M$, and $P = \partial_t^2 -\Delta_g + R$ with $R = R(t, x, \d_t, \d_x)$ is a differential operator of order one on $(-T,T) \times \M$, bounded in the $x$-variable and depending analytically on the variable $t \in (-T,T)$ at any $x \in \M$.

For any nonempty open subset $\omega$ of $\M$ and any $T> \mathcal{L}(\M ,\omega)$, there exist $\eps, C, \kappa ,\mu_0>0$ such that for any $u \in H^1((-T,T) \times \M)$ and $f \in L^2((-T,T) \times \M)$ solving 
\bneqn
\label{e:free-wave-ter}
P u= f& & \text{ in } (-T,T) \times \M , \\
u_{\left|\partial \M\right.}=0 & &  \text{ in } (-T,T) \times \d \M ,\\
\eneqn
we have, for any $\mu\geq \mu_0$,
\bna
\nor{u}{L^2((-\eps ,\eps) \times \M)}\leq C e^{\kappa \mu}\left(\nor{u}{L^2((-T,T); H^1(\omega))} + 
\nor{f}{L^2((-T,T)\times \M)} \right)
+\frac{C}{\mu}\nor{u}{H^1((-T,T) \times \M)}.
\ena
If moreover $\M$ and the metric $g$ and lower order terms $R$ are analytic, and $\partial \M=\emptyset$, there exists $\tilde{\varphi} \in C^\infty_0((-T,T)\times \omega)$ such that for any $s \in \R$, we have
\bna
\nor{u}{L^2((-\eps ,\eps) \times \M)}\leq C e^{\kappa \mu}
\left(\nor{\tilde{\varphi} u}{H^{-s}((-T,T)\times \M)} + 
\nor{f}{L^2((-T,T)\times \M)} \right)
+\frac{C}{\mu}\nor{u}{H^1((-T,T) \times \M)}.
\ena
If $\partial\M\neq \emptyset$ and $\Gamma$ is a non empty open subset of $\partial \M$, for any $T>\mathcal{L}(\M ,\Gamma)$, there exist $\eps , C, \kappa,\mu_0 >0$ such that for any $u \in H^1((-T,T) \times \M)$ and $f \in L^2((-T,T) \times \M)$ solving~\eqref{e:free-wave-ter}, we have
\bna
\nor{u}{L^2((-\eps ,\eps) \times \M)}\leq C e^{\kappa \mu}
\left(\nor{\partial_{\nu}u}{L^2((-T,T)\times \Gamma)} + 
\nor{f}{L^2((-T,T)\times \M)} \right)
+\frac{C}{\mu}\nor{u}{H^1((-T,T) \times \M)}.
\ena
Finally, if all lower order terms are time-independent, that is if $R = W_0 \d_t + W_1\cdot \nabla + V$ does not depend on $t$, then we have the following stronger result. There exist $\eps, C_0, \kappa ,\mu_0>0$ such that for any $u \in H^1((-T,T) \times \M)$ and $f \in L^2((-T,T) \times \M)$ solving~\eqref{e:free-wave-ter} and for any $V, W_0 \in L^\infty(\M)$ and $W_1$ a $L^\infty$ vector field on $\M$, all above estimates hold uniformly for all $\mu\geq \mu_0\max\{1 , \|V\|_{L^\infty}^{\frac23},\|W_0\|_{L^\infty}^{2},\|W_1\|_{L^\infty}^{2}\}$ and $C$ replaced by $C_0 \max \left\{1,\nor{W}{L^{\infty}}\right\}$.
\end{theorem}
We first prove Theorem~\ref{thmobserwave-bis} and then conclude with the proof of Theorem~\ref{thmobserwave}.
\bnp[Proof of Theorem~\ref{thmobserwave-bis}]
We only prove here the more complicated case of the boundary observation. The internal observation case is simpler and follows the same proof. To transport the information from one point $x^0$ to another point $x^1$, the idea is to build nice coordinates in a neighborhood of a path between $x^0$ and $x^1$. In these coordinates, we construct an appropriate foliation in which to apply our semi-global estimate. To construct these coordinates, we follow the presentation of Lebeau~\cite[pp~21-22]{Leb:Analytic}. 

We fix a point $x^1\in \overline{\M}$. We can find $x^0\in \Gamma$ and a path $\gamma: [0,1]\rightarrow \overline{\M}$ of length $\ell_0$ with $\mathcal{L}(\M, \Gamma)<\ell_0<T$ (see the definition of $\mathcal{L}(\M, \Gamma)$ in~\eqref{e:def-L}) so that $\gamma(0)=x^0$ and $\gamma(1)=x^1$. Moreover, we can impose that 
\bneq
&&\gamma \textnormal{ does not have self intersection}\\
&&\gamma(s)\in \M \textnormal{ for }s\in ]0,1[\\
&&\dot{\gamma}(0) \textnormal{ and }\dot{\gamma}(1)\textnormal{ are orthogonal to }  \partial \M.
\eneq
According to Lemma~\ref{lmcoord} below, we can find local coordinates $(w, x_n)$ near $\gamma$ in which $\overline{\M}$ is defined by $0\leq x_n \leq \ell_0$, the path $\gamma$ by $\gamma(s)=(0,s\ell_0)$ and the metric is given by the matrix $m(w,x_n) \in M_n(\R)$ with
\bnan
\label{fromdiagm}
m(w,x_n)  = \left( \begin{array}{cc}
 m' (x_n) & 0 \\
 0 & 1  
 \end{array}
\right) + \mathcal{O}_{M_n(\R)}(|w|) , \quad \text{ for } w\in B_{\R^{n-1}}(0,\delta), \delta >0 ,
\enan
with $m' (x_n)\in M_{n-1}(\R)$ (uniformly) definite symmetric. With these coordinates in the space variable, and still using the straight time variable, the symbol of the wave operator is given by
\bnan
\label{e:ptauwxn}
p(t,w,x_n, \tau, \xi_w, \xi_n)=p(w,x_n, \tau, \xi_w, \xi_n) = -\tau^2  + \langle m(w,x_n) \xi  , \xi  \rangle ,  \quad\xi = (\xi_w, \xi_n), 
\enan
where we have used $\tau$ for the dual of the time variable and $\xi_w$, $\xi_n$ for the dual to $w\in B_{\R^{n-1}}(0,\delta)$ and $x_n\in [0,\ell_0]$.

We now aim to apply Theorem \ref{thmsemiglobalwavebord}.
Pick again $t_0$ with $\ell_0<t_0<T$. For $b <\delta$ small, to be fixed later on, we define
 \bna
x_n=l , \quad x' = (t,w), \quad D=\left\{(t,w)\left| \Big(\frac{w}{b}\Big)^2+\Big(\frac{t}{t_0}\Big)^2\leq 1\right.\right\}\\
G(t,w,\e)=  \e \ell_0\psi \left(\sqrt{\Big(\frac{w}{b}\Big)^2+\Big(\frac{t}{t_0}\Big)^2} \right), 
 \quad 
 \phi_\eps(t,w,x_n): =  G(t,w , \eps) - x_n , \quad  \e \in [0,1]
\ena 
where $\psi$ is such that 
 \bna
 \psi \text{ even,}\quad  \psi(\pm1) =0, \quad \psi(0)=1,\\
\psi(s)\geq 0, \quad |\psi'(s)| \leq \alpha,  \textnormal{ for }s\in [-1,1], 
 \ena
with $1<\alpha <\frac{t_0}{\ell_0}$. This is possible since $\frac{t_0}{\ell_0}>1$.

Note also that the point $(t=0,w=0,x_n=\ell_0)$ corresponding in the local coordinates to $x^1$ belongs to $\left\{\phi_{1}=0\right\}$.
 We have
  $$
d \phi_\eps(t,w,x_n) 
=   \e \ell_0  \left(\Big(\frac{w}{b}\Big)^2+\Big(\frac{t}{t_0}\Big)^2\right)^{-1/2} \psi' \left(\sqrt{\Big(\frac{w}{b}\Big)^2+\Big(\frac{t}{t_0}\Big)^2} \right) \left( \frac{tdt}{t_0^2}  + \frac{w dw}{b^2}  \right) - d x_n .
 $$
Given the form of the principal symbol of the wave operator in these coordinates (see \eqref{fromdiagm}-\eqref{e:ptauwxn}), we obtain
\bna
p( w,x_n, d \phi_\eps(t,w,x_n) ) &=&  -\e^2 \ell_0^2\frac{t^2}{t_0^4}    \left(\Big(\frac{w}{b}\Big)^2+\Big(\frac{t}{t_0}\Big)^2 \right)^{-1}|\psi'|^2   \\
&&+ \ell_0^2\frac{\e^2}{ b^4}  \langle m' (x_n) w, w \rangle \left(\Big(\frac{w}{b}\Big)^2+\Big(\frac{t}{t_0}\Big)^2 \right)^{-1} |\psi'|^2  + 1 \\
&&+O(|w|^2) \left( 1 + \frac{\e^2 \ell_0^2}{b^4}  |w|^2 \left(\Big(\frac{w}{b}\Big)^2+\Big(\frac{t}{t_0}\Big)^2 \right)^{-1}|\psi'|^2 \right),
\ena
where $|\psi'|^2 $ is taken at the point $\left(\sqrt{\Big(\frac{w}{b}\Big)^2+\Big(\frac{t}{t_0}\Big)^2} \right)$. Now, since $\alpha <\frac{t_0}{\ell_0}$ and $m' (x_n)$ is uniformly (for $x_n\in [0,\ell_0]$) definite positive, there is $\eta>0$ so that for $|w|\leq b$ small enough, we have
\bna
1+O(|w|^2)&\geq &\alpha^2 \frac{\ell_0^2}{t_0^2}\eta\\
 \langle m' (x_n) w, w \rangle +O(|w|^2)|w|^2&\geq&  \frac{1}{2} \langle m' (x_n) w, w \rangle\geq 0.
\ena
Hence, there is a sufficiently small neighborhood (taking again $b$ small enough) of the path (i.e. of $w = 0$), in which we have (for any $\e\in [0,1]$), and any $(t,w,x_n)\in \overline{D}\times [0,\ell_0]$
\bna
p( w,x_n, d \phi_\eps(t,w,x_n) )& \geq&  - \frac{\e^2}{t_0^2}\ell_0^2   \Big(\frac{t}{t_0}\Big)^2\left(\Big(\frac{w}{b}\Big)^2+\Big(\frac{t}{t_0}\Big)^2 \right)^{-1}  |\psi'|^2  
+  \alpha^2 \frac{\ell_0^2}{t_0^2}+\eta  \\
& \geq & -\frac{\ell_0^2}{t_0^2}  |\psi'|^2 +\alpha^2 \frac{\ell_0^2}{t_0^2}+\eta \geq \eta.
 \ena
So, the surface $\{\phi_\eps =0\}$ is noncharacteristic for any $\eps \in [0,1]$ and, therefore, strictly pseudoconvex with respect to the wave operator, see Remark \ref{rknoncaractwave}. 

Moreover, since $b$ can be chosen arbitrary small and $x^0\in \Gamma$ open, we can select $b$ small enough so that in the chosen coordinates, we have $D\subset [-t_0,t_0]\times \Gamma $. Therefore, applying Theorem \ref{thmsemiglobalwavebord} in the chosen coordinates and writing (with a slight abuse of notation) the final result in an invariant way, we get
$$
\nor{ u}{L^2(U)}\leq C e^{\kappa \mu}\left(\nor{\partial_{\nu} u}{L^2((-T,T)\times \Gamma)}  + \nor{Pu}{L^2((-T,T)\times \M)}\right)+\frac{C}{\mu}\nor{u}{H^1((-T,T)\times \M)} ,
$$
where $U$ is a neigborhood (in the local coordinates) of $\left\{\phi_1=0\right\}$ and in particular a neighborhood of $x^1$ (in the global coordinates).
Note, that we actually apply the Theorem to $\chi u$ with $\chi \in C^{\infty}(]-T,T[\times \M)$ so that in the coordinate charts, $\chi u\in C^{\infty}_0([0,\ell_0]\times \R^{n-1})$ and $\chi=1$ on a neighborhood of the $\Omega$ defined in Theorem \ref{thmsemiglobalwavebord}. We have therefore $ \nor{P \chi u}{L^{2}(\Omega)}=\nor{P u}{L^{2}(\Omega)}\leq C\nor{Pu}{L^2(]-T,T[\times \M)}$ and $\nor{\chi u}{H^1([0,\ell_0]\times \R^{n-1})}\leq \nor{u}{H^1((-T,T)\times \M)}$ (where we have switched from some coordinate set to another with a slight abuse of notation).

Since the previous property is true for any $x^1\in \overline{\M}$, we obtain by compactness (taking the worst of all the constants $\kappa$, $C$, $\mu_0$), using only a finite number of this estimate, that there exists $\e>0$ so that we have
$$
\nor{ u}{L^2((-\e,\e)\times \M)}\leq C e^{\kappa \mu}\left(\nor{\partial_{\nu} u}{L^2((-T,T)\times \Gamma)}  + \nor{Pu}{L^2(]-T,T[\times \M)}\right) +\frac{C}{\mu}\nor{u}{H^1((-T,T)\times \M)}.
$$
This concludes the proof of the theorem in the general (boundary) case. 

For the last analytic case, we apply the same reasoning as before using the case $n_a=n$ of Theorem \ref{thmsemiglobal} and taking care for having some analytic change of coordinates. For instance, we need to have an analytic path. So, it leads to an observation $\nor{\varphi u}{H^{-s}}$ where $\varphi=1$ on all the cutoff functions obtained by the theorem. 

The lower order term depending analytically in time are treated using Corollary \ref{cor:carl-potentielanal-bord} and Remark \ref{rklowerunif}.
 
The uniform dependence with respect to time independent lower order terms follows from the fact that we use only a finite number of times Theorem~\ref{thmsemiglobalwavebord}. 
\enp
With Theorem~\ref{thmobserwave-bis}, we now conclude the proof of Theorem~\ref{thmobserwave}, using energy estimates to relate $\nor{(u_0,u_1)}{H^{1}_0\times L^2(\M)}$ to $\nor{u}{H^1((-T,T)\times \M)}$, and $\nor{(u_0,u_1)}{L^2\times H^{-1}(\M)}$ to $\nor{u}{L^2((-T,T)\times \M)}$. These estimates are very classical in the selfadjoint case (which we omit here) and need a little care in the general case.
\bnp[Proof of Theorem~\ref{thmobserwave}]
We consider a perturbation of order one $R(t,x,\d_t,\d_x)u=V(t,x)u+W_0(t,x)\partial_t u +W_1(t,x)\cdot \nabla u$ and perform the energy estimates. We have the pointwise in time estimate, for $s\in [-T,T]$, 
\bna
\nor{R(s)u(s)}{L^2}\leq C_R\left(\nor{u(s)}{H^1(\M)}+\nor{\partial_t u(s)}{L^2(\M)}\right)
\ena with 
$$C_R=\nor{V}{L^{\infty}([-T,T]\times \M)}+\nor{W_0}{L^{\infty}([-T,T]\times \M)}+\nor{W_1}{L^{\infty}([-T,T]\times \M)}.
$$
Using the Dumamel formula and Gronwall Lemma, it gives
\bna
\nor{(u,\partial_tu)(t)}{H^1\times L^2(\M)}\leq Ce^{CC_R}\left((\nor{(u_0,u_1)}{H^1\times L^2(\M)}+\nor{f}{L^1([-T,T], L^2)}\right),
\ena
and in particular, integrating in time,
\bnan
\label{estimNRJH1}
\nor{u}{H^1(]-T,T[\times \M)}\leq Ce^{CC_R}\left((\nor{(u_0,u_1)}{H^1\times L^2(\M)}+\nor{f}{L^1([-T,T], L^2)}\right).
\enan
Let $R^*(t,x,\partial_t,D_x)u=V(t,x)u-\partial_t(W_0(t,x)u) -div(W_1(t,x)u)$ be the formal (space-time) adjoint of $R$ (we take the real duality for simplicity). 

If $(v_0,v_1)\in H^{1}\times L^2$, let $v$ be the associate solution of $\Box v+R^*v=0$. We have 
$$
\nor{R^*(s)v(s)}{L^2}\leq C_{R^*}\left(\nor{v(s)}{H^1(\M)}+\nor{\partial_t v(s)}{L^2(\M)}\right),
$$
with 
$$
C_{R^*}=\nor{V}{L^{\infty}([0,\e]\times \M)}+\nor{W_0}{W^{1,\infty}([0,\e],L^{\infty}(\M))}+\nor{W_1}{L^{\infty}([0,\e]\times \M)}+\nor{\div W_1}{L^{\infty}([0,\e]\times \M)}.
$$ Similar energy estimate applied to $v$ give
\bna
\nor{v}{H^1((-\e,\e)\times \M)}\leq Ce^{C\e C_{R^*}}\nor{(v_0,v_1)}{H^1\times L^2(\M)}.
\ena

 $\chi\in C^{\infty}([0,\e])$ so that $\chi(0)=1$, $\Dot{\chi}(0)=0$ $\chi(\e)=0$ $\Dot{\chi}(\e)=0$. Then, $w=\chi(t)v$ is solution of 
\bneq
\Box w+R^*w&=&2\dot{\chi}(t)\partial_t v+ \dot{\chi}(t)W_0 v+\ddot{\chi}(t)v:=g\\
w_{\left|\partial \M\right.}&=&0  \\
(w,\partial_t w)_{\left|t=0\right.} &=&(v_0,v_1)
\eneq  
Then, $g$ is a (trivial) control that $(v_0,v_1)$ to zero, i.e. $(w,\partial_t w)_{\left|t=\e\right.} = (0,0)$, with $\nor{g}{L^2(]0,\e[\times \M)}\leq Ce^{CC_{R^*}} \nor{(v_0,v_1)}{H^{1}\times L^2}$. So, the usual computation yields, after integrating by parts
\bna
\int_{]0,\e[\times \M} u g=\int_{]0,\e[\times \M} u (\Box +R^*)w=\int_{\M}u_1v_0-\int_{\M}u_0 v_1-\int_{\M} W_0(0,x)u_0v_0+\int_{]-\e,\e[\times \M} fw
\ena
and in particular
\bna
\left\langle (u_0,u_1),(-v_1,v_0)\right\rangle_{L^2\times H^{-1},L^2\times H^1}&\leq & C \nor{u}{L^2(]0,\e[\times \M)}\nor{g}{L^2(]0,\e[\times \M)}+C\nor{f}{L^2(]0,\e[\times \M)}\nor{w}{L^2(]0,\e[\times \M)}\\
&\leq &Ce^{CR^*}\nor{(v_0,v_1)}{H^{1}\times L^2}\left(\nor{u}{L^2(]0,\e[\times \M)}+ \nor{f}{L^2(]0,\e[\times \M)}\right)
\ena
where $\left\langle\cdot,\cdot\right\rangle $ is the twisted duality $\left\langle (u_0,u_1),(v_1,v_0)\right\rangle_{L^2\times H^{-1},L^2\times H^1}=\int_{\M}u_1v_0-\int_{\M}u_0 v_1-\int_{\M} W_0(0,x)u_0v_0$.

By specifying to $v_0$ and $\nor{v_1}{L^2}=1$, this gives first by duality.
\bna
\nor{u_0}{L^2}=\sup_{\nor{v_1}{L^2}=1}\int_{\M}u_0 v_1 \leq Ce^{C_R^*}\left(\nor{u}{L^2(]0,\e[\times \M)}+ \nor{f}{L^2(]0,\e[\times \M)}\right).
\ena
Then, with $v_1=0$ and $\nor{v_0}{H^{1}}=1$, we get
\bna
\nor{u_1}{H^{-1}}&=&\sup_{\nor{v_0}{H^1}=1}\int_{\M}u_1v_0 \leq \sup_{\nor{v_0}{H^1}=1}\int_{\M}\left(u_1v_0 -\int_{\M} W_0(0,x)u_0v_0+\int_{\M} W_0(0,x)u_0v_0\right) \\
&\leq & \sup_{\nor{v_0}{H^1}=1} \left\langle (u_0,u_1),(0,v_0)\right\rangle_{L^2\times H^{-1},L^2\times H^1}+\sup_{\nor{v_0}{H^1}=1}\int_{\M} W_0(0,x)u_0v_0\\
&\leq &Ce^{C_R^*}\left(\nor{u}{L^2(]0,\e[\times \M)}+ \nor{f}{L^2(]0,\e[\times \M)}\right)+C\nor{W_0}{L^{\infty}}\nor{u_0}{L^2}.
\ena
So, finally, we have
\bnan
\label{observtotale}
\nor{(u_0,u_1)}{L^2\times H^{-1}}\leq Ce^{C_R^*}\left(\nor{u}{L^2(]0,\e[\times \M)}+ \nor{f}{L^2(]0,\e[\times \M)}\right).
\enan
In the particular case where the perturbation is independent on time, we have 
$$
C_R+C_{R^*}\leq C \max \left\{\nor{V}{L^{\infty}(\M)},\nor{W_0}{L^{\infty}(\M)},\nor{W_1}{L^{\infty}(\M)},\nor{\div(W_1)}{L^{\infty}(\M)}\right\}.
$$ The combination of Theorem \ref{thmobserwave-bis}, together with estimates \eqref{estimNRJH1} and \eqref{observtotale} gives the sought result.
\enp

The following Lemma is contained in Lebeau \cite{Leb:Analytic} p22, see also Lemma 11.38 pp 221 of \cite{BarilariAgrachevBoscainBook}. We give the proof for sake of completeness.

\begin{lemma}
\label{lmcoord}
Let $\gamma [0,1]\rightarrow \overline{\M}$ be a smooth path without self intersection of length $\ell_0$ so that
\bneq
&&\gamma(s)\in \M \textnormal{ for }s\in ]0,1[\\
&&\gamma(0)=x_0 \textnormal{ and }\gamma(1)=x_1  \textnormal{ belong to } \partial \M\\
&&\dot{\gamma}(0) \textnormal{ and }\dot{\gamma}(1)\textnormal{ are orthogonal to }  \partial \M
\eneq
Then, there are some coordinates $(w,l)\in B_{\R^{n-1}}(0,\e)\times [0,\ell_0]$ in an open neighborhood $U$ near $\gamma([0,1])$ so that
\begin{itemize}
\item  $\gamma([0,1])=\{w=0 \}\times [0,\ell_0]$,
\item the metric $g$ is of the form $
m(l,w)  = \left( \begin{array}{cc}
 1 & 0 \\
 0 &   m' (l)
 \end{array}
\right) + O_{M_n(\R)}(|w|) ,
$
\item in coordinates, we have $\overline{\M}\cap U=B_{\R^{n-1}}(0,\e)\times [0,\ell_0]$ for some $\e>0$.
\end{itemize}

\end{lemma}
\bnp
The path $\gamma$ is of length $\ell_0$ so, we can reparametrize it by $\gamma: [0,\ell_0]\rightarrow \M$ such that $\gamma$ is unitary (that is $\nor{\dot{\gamma}(s)}{\gamma(s)}=1$) Moreover, since $\gamma$ does not have self intersection, there exist $U$ a neighborhood (in the topology of $\overline{\M}$) of $\gamma$ and a diffeomorphism $\psi$ (in the structure of $\overline{\M}$) such that 
\begin{itemize}
\item $\psi(U)\subset \left\{(x,y)\in \R^{n}\left|x\in [-\e,\ell_0+\e],|y|\leq \e\right.\right\}$,
\item $\psi(\gamma(s))=(s,0)$,
\item $\psi(U)=\left\{(x,y)\in \R^{n}, f_1(y)\leq x\leq f_2(y)\left|x\in [-\e,\ell_0+\e],|y|\leq \e\right. \right\}$ for some smooth functions $f_i$ locally defined 
\end{itemize}
Up to making the change of variable $(x,y)\mapsto (x-f_1(y),y)$, we can moreover impose $f_1=0$ and change $f_2$ by $f_2-f_1$. 

Then, we make some change of variable to diagonalize the metric on $\gamma$.
By unitarity of the coordinates, the metric on $\gamma$ has the form
$$
m(x,0)  = \left( \begin{array}{cc}
 1 & l(x) \\
^t l(x) &  g(x)  
 \end{array}
\right) ,
$$
where $l$ is a line vector and $g$ is a positive definite matrix. We perform the change of variable $\Phi: (x,y)\mapsto (\widetilde{x},\widetilde{y})=\left(x-a_x\cdot y, y\right)$. In $y=0$, we have $D\Phi(x,0)= \left( \begin{array}{cc}
1&  -a_x \\
0 & Id  \end{array} \right)$ with $^tD\Phi(x,0)= \left( \begin{array}{cc}
1& 0\\ -^ta_x & Id  \end{array}
\right) $ (in particular, the change of variable is valid for small $y$) and $D\Phi(x,0)^{-1}= \left( \begin{array}{cc}
1&  a_x \\
0 & Id  \end{array} \right)$ with $^tD\Phi(x,0)^{-1}= \left( \begin{array}{cc}
1& 0\\ ^ta_x & Id  \end{array}
\right) $. Moreover, in the new coordinates, the set  in $\left\{\widetilde{y}=0\right\}$ and the metric there is given by 
\bna
^tD\Phi(x,0)^{-1}m(x,0)D\Phi(x,0)^{-1}=\left( \begin{array}{cc}
 1 & l(x)+a(x) \\
^t l(x)+^t a(x) & *  
 \end{array}
\right)
\ena

So, we choose $a(x)=-l(x)$ so that in this new coordinates $m(x,0)$ is of the form
\bnan
\label{formm}
m(x,0)  = \left( \begin{array}{cc}
 1 & 0 \\
 0&  *  
 \end{array}
\right).
\enan

We notice that since $\dot{\gamma}(0)$ is orthogonal to $\partial M$ which is defined locally by $\left\{x=0\right\}$, we have $l(0)=0$ ($\dot{\gamma}(0)=(1,0)$ so it implies $^t(0,y)m(0,0)\dot{\gamma}(0)=^tl(0)y$ for all $y$). In particular, $\Phi$ restricted to $\left\{x=0\right\}$ is the identity.

 This implies that in this new coordinates, $\overline{\M}$ is still defined near $\gamma$ by $0\leq x\leq f_2(y)$ (now, we still denote $(x,y)$ for $(\tilde{x},\tilde{y})$).
We still have $f_2(0)=\ell_0$. Morever, since $\dot{\gamma}(\ell_0)=(1,0)$ is orthogonal to $\partial \M$ which is defined locally by $\left\{x=f_2(y)\right\}$ and using that $m(x,0)$ is of the form \eqref{formm}, we get $df_2(0)=0$. 

Finally, making the change of variable $(x,y)\mapsto (\frac{\ell_0}{f(y)} x,y)$, which is the identity on $\gamma$, we get that $\overline{\M}$ is given $0\leq x\leq \ell_0$. Moreover, since $df(0)=0$, the metric is not changed on $\gamma$. 

The expected property of $m$ is then obtained by the mean value theorem using the diagonal form \eqref{formm} on $\gamma$.
\enp
\subsection{The Schr\"odinger equation}
Now, we turn to the Schr\"odinger equation. The result are quite similar to the wave equation except for two facts. 

The first one is that there is no minimal time. This is quite natural with the infinite speed of propagation. In the proof, this appears in the fact that the principal symbol is $|\xi|_g^2$. Therefore, a hypersurface $\left\{\varphi(t,x)=0\right\}$ is non characteristic if $\nabla_x \varphi \neq 0$, without assumption on the time derivative. 

The second difference is that the remainder term involving the $H^1((-T,T),\M)$ norm involves some derivative in time and space which do not have the same weight. Hence, since $\partial_t u=i\Delta_g u$, this term will actually count for two derivatives in space.

\begin{theorem}
\label{thmobserschrod}
Let $\M$ be a compact Riemannian manifold with (or without) boundary, $\Delta_g$ the Laplace-Beltrami operator on $\M$, and 
$$P = i\partial_t +\Delta_g + V$$ with $V$ depending analytically on the variable $t$ in a neighborhood of $(-T,T)$. Assume moreover that $V\in L^{\infty}((-T,T);W^{2,\infty}(\M))$.

 For any nonempty open subset $\omega$ of $\M$ and any $T> 0$, there exist $C ,\kappa , \mu_0 >0$ such that for any $u_0 \in H^2\cap H^1_0$, $f\in L^2((-T,T); H^2(\M))$ and associated solution $u$ of 
\bneqn
\label{e:schrod-freethm}
i\partial_t u+\Delta_g u+Vu= f &&\text{ in } (-T,T) \times \M , \\
u_{\left|\partial \M\right.}=0   && \text{ in } (T,T) \times \d \M ,\\
u(0)=u_0&& \text{ in }\M ,
\eneqn
we have, for any $\mu\geq \mu_0$,
\bnan
\label{estimthm1schrod}
\nor{u_0}{L^2}\leq C e^{\kappa \mu}\left(\nor{u}{L^2((-T,T); H^1(\omega))}+\nor{f}{L^2((-T,T); H^2(\M))}\right)+\frac{C}{\mu}\nor{u_0}{H^2}.
\enan
If moreover all coefficients of $P$ are analytic in $t$ and $x$, and $\partial \M=\emptyset$, there exists $\tilde{\varphi} \in C^\infty_0((-T,T)\times \omega)$ such that for any $s \in \R$, we have
\bnan
\label{estimthm2schrod}
\nor{u_0}{L^2}\leq C e^{\kappa \mu}\left(\nor{\varphi u}{H^{-s}((-T,T)\times \M)}+\nor{f}{L^2((-T,T); H^2(\M))}\right)+\frac{C}{\mu}\nor{u_0}{H^2}.
\enan
If $\partial\M\neq \emptyset$ and $\Gamma$ is a non empty open subset of $\partial \M$, then for any $T> 0$, there exist $C ,\kappa , \mu_0 >0$ such that for any $u_0 \in H^2\cap H^1_0$,  and associated solution $u$ of~\eqref{e:schrod-freethm}, we have, for any $\mu\geq \mu_0$,
\bnan
\label{estimthm3schrod}
\nor{u_0}{L^2}\leq C e^{\kappa \mu}\left(\nor{\partial_{\nu}u}{L^2((-T,T)\times \Gamma)}+\nor{f}{L^2((-T,T); H^2(\M))}\right)+\frac{C}{\mu}\nor{u_0}{H^2}.
\enan
Finally, if $V$ is time-independent then we have the following stronger result. There exist $C_0, \kappa ,\mu_0>0$ such that for any $u_0\in H^2\cap H^1_0(\M)$, $f\in L^2((-T,T)\times\M)$ and associated solution $u$ of \eqref{e:schrod-freethm}, and for any $V$ bounded in the $x$-variable, estimates \eqref{estimthm1schrod} and \eqref{estimthm3schrod} hold uniformly for all $\mu\geq \mu_0\max\{1 , \|V\|_{L^\infty}^{\frac23}\}$ with constant
\bna
C= C_0\exp \left(C_0\nor{V}{W^{2,\infty}(\M)}\right).
\ena
\end{theorem}
As in the case of the wave equation, the previous Theorem is a combination of the following Theorem and energy estimates for the Schr\"odinger equation. 
\begin{theorem}
\label{thmobserSchro-bis}
Let $\M$ be a compact Riemannian manifold with (or without) boundary, $\Delta_g$ the Laplace-Beltrami operator on $\M$, and $P = \Delta_g + R$ with $R = R(t, x, \d_t, \d_x)$ is a differential operator of order one on $(-T,T) \times \M$, bounded in the $x$-variable and depending analytically on the variable $t \in (-T,T)$ at any $x \in \M$.

For any nonempty open subset $\omega$ of $\M$ and any $T> 0$, there exist $\eps, C, \kappa ,\mu_0>0$ such that for any $u \in H^1((-T,T) \times \M)$ and $f \in L^2((-T,T) \times \M)$ solving 
\bneqn
\label{e:free-wave-bis}
P u= f& & \text{ in } (-T,T) \times \M , \\
u_{\left|\partial \M\right.}=0 & &  \text{ in } (-T,T) \times \d \M ,\\
\eneqn
the same estimates as Theorem \ref{thmobserwave-bis} hold.

In the case that $R = W_0 \d_t + W_1\cdot \nabla + V$ does not depend on $t$, the dependence on the size of the coefficients of $R$ remains the same as Theorem \ref{thmobserwave-bis}.
\end{theorem}
\bnp[Proof of Theorem \ref{thmobserSchro-bis}]
The proof is quite similar to the one for the wave equation, so we only sketch the main steps of the proof. The main difference will be that $T$ can be chosen arbitrary. Pick $t_0$ arbitrary with $t_0<T$, this time without any relation with $\ell_0$.

We use the same coordinate charts as defined in the proof of Theorem \ref{thmobserwave} for the wave equation. Then, the principal symbol of the Schr\"odinger operator will be 
$$
p(w,x_n, \tau, \xi_w, \xi_n) = - \langle m(w,x_n) \xi  , \xi  \rangle ,  \quad\xi = (\xi_w, \xi_n). 
$$

Therefore, $p$ is a quadratic form with real coefficients that is definite on the set $\left\{\tau=0\right\}$. Remark \ref{rknoncaractwave} allows to get that any non characteristic hypersurface is strictly pseudoconvex. So, with the same definition of $\phi_{\e}$, we get
\bna
p( w,x_n, d \phi_\eps(t,w,x_n) ) = 
- \ell_0^2\frac{\e^2}{ b^4}  \langle m' (x_n) w, w \rangle \left(\Big(\frac{w}{b}\Big)^2+\Big(\frac{t}{t_0}\Big)^2 \right)^{-1} |\psi'|^2  
- 1 \\
+O(|w|^2) \left( 1 + \frac{\e^2 \ell_0^2}{b^4}  |w|^2 \left(\Big(\frac{w}{b}\Big)^2+\Big(\frac{t}{t_0}\Big)^2 \right)^{-1}|\psi'|^2 \right).
\ena
But, for $w$ small enough, we still have
\bna
-1+O(|w|^2)&\leq &-1/2\\
- \langle m' (x_n) w, w \rangle +O(|w|^2)|w|^2&\leq&  0.
\ena
In particular, with the same notations as for the wave equation, there exists $b$ small enough so that for any $\e\in [0,1]$), and any $(t,w,x_n)\in \overline{D}\times [0,\ell_0]$, we have
\bna
p( w,x_n, d \phi_\eps(t,w,x_n) )& \leq&   -\frac{1}{2}.
 \ena
So, applying the same reasoning as for the wave equation, we obtain the existence of some $\kappa$, $C$, $\mu_0$ and $\e>0$ so that we have
$$
\nor{ u}{L^2(]-\e,\e[\times \M)}\leq C e^{\kappa \mu}\nor{\partial_{\nu} u}{L^2(]-T,T[\times \Gamma)} +\frac{C}{\mu}\nor{u}{H^1(]-T,T[\times \M)}
$$
for any $\mu\geq \mu_0$.

The dependence on the lower order term $R$ follows the same way as for the wave equation.
\enp
\bnp[Proof of Theorem \ref{thmobserschrod}]
Since the multiplication by $V$ acts on $H^1_0$ and $H^2$ if $V\in W^{2,\infty}(\M)$, using Duhamel formula and a Gronwall argument allows to obtain, for $s\in [-T,T]$, 
\bna
\nor{u_0}{L^2(\M)}&\leq &Ce^{C\nor{V}{L^{\infty}(\M)}}\left(\nor{u(s)}{L^2(\M)}+\nor{f}{L^2((-T,T)\times \M)}\right)\\
\nor{u(s)}{H^2(\M)}&\leq& Ce^{C\nor{V}{W^{2,\infty}(\M)}}\left(\nor{u_0}{H^2}+\nor{f}{L^2((-T,T);H^2(\M))}\right).
\ena
Integrating in time, it gives
\bna
\nor{u_0}{L^2(\M)}&\leq& Ce^{C\nor{V}{L^{\infty}(\M)}}\left(\nor{u}{L^2((-\e,\e)\times \M)}+\nor{f}{L^2((-T,T)\times \M)}\right)\\
\nor{u}{L^2((-T,T);H^2(\M))}&\leq& Ce^{C\nor{V}{W^{2,\infty}}}\left(\nor{u_0}{H^2}+\nor{f}{L^2((-T,T);H^2(\M))}\right).
\ena
To estimate $\partial_t u$, we notice that $\partial_t u=i(\Delta+V)u-if$. Therefore, we only need to estimate $\nor{\Delta u}{L^2}$.
\bna
\nor{\partial_t u}{L^2((-T,T)\times \M)}&\leq &C\nor{u}{L^2((-T,T);H^2)}+C \nor{V}{L^{\infty}(\M)}\nor{u}{L^2((-T,T)\times \M)}+\nor{f}{L^2(]-T,T[\times \M)}\\
&\leq &Ce^{C\nor{V}{W^{2,\infty}(\M)}}\left(\nor{u_0}{H^2}+\nor{f}{L^2((-T,T);H^2(\M))}\right).\\
\ena
So, this gives 
\bna
\nor{u}{H^1((-T,T)\times \M)}\leq Ce^{C\nor{V}{W^{2,\infty}(\M)}}\left(\nor{u_0}{H^2}+\nor{f}{L^2((-T,T);H^2(\M))}\right).
\ena
This gives the estimates of the Theorem when combined with Theorem \ref{thmobserSchro-bis}.
\enp
\appendix
\section{Two elementary technical lemmata}
\label{sectionlemmata}
In the above proof, we used the following elementary lemma (see e.g.~\cite{LeLe:09}).
\begin{lemma}
\label{l:fgh-three-fcts}
Let $K$ be a compact set and $f,g,h$ three continuous real valued functions on $K$. Assume that $f \geq 0$ on $K$, and $g>0$ on $\{f=0\}$. Then, there exists $A_0, C>0$ such that for all $A\geq A_0$, we have $g + A f- \frac{1}{A}h \geq C$ on $K$. 
\end{lemma}
Lemma \ref{l:fgh-three-fcts} is a consequence of the following variant.
\begin{lemma}
\label{l:fgh-three-fctsbis}
Let $K$ be a compact set and $f$ a continuous real valued function on $K$. Let $g$ and $h$ be two bounded function defined on $K$. Assume that $f \geq 0$ on $K$, and there exists $V$ an open neighborhood  of $\{f=0\}$ in $K$ so that $g>c$ on $V$ for one constant $c>0$. Then, there exists $A_0, C>0$ such that for all $A\geq A_0$, we have $g + A f- \frac{1}{A}h \geq C$ on $K$. 
\end{lemma}
We also used the following result.
\begin{lemma}
\label{lmfonction}
Let $C_1$, $C_2$, and $\alpha$ be positive. Then, there exists $K>0$ such that for all $\mu_0>0$, for any $a,b,c>0$ such that there the following estimates hold
\bna
b\leq C_2 c , \quad a\leq  c, \quad \text{and} \quad a\leq e^{C_1 \mu}b+\frac{1}{\mu^{\alpha}}c, \quad \text{for all }\mu\geq \mu_0 ,
\ena
we have 
\bna
a\leq \frac{D_1}{\log\left(\frac{c}{b}+1\right)^{\alpha}}c , \quad \text{ with } D_1 = (2C_1)^\alpha \max \left\{ K ,\mu_0^\alpha\right\} ,
\ena
and 
\bna
c\leq e^{D_2 \left(\frac{c}{a}\right)^{1/\alpha}}b , \quad \text{ with } D_2 = D_1^{1/\alpha} = 2C_1  \max \left\{ K^{1/\alpha} ,\mu_0 \right\} .
\ena
\end{lemma}
\bnp
Dividing all inequalities by $c$, setting $y=a/c$ and $x=b/c$, it suffices to prove
$$
\Big(x\leq C_2 , \quad y \leq 1, \quad y \leq e^{C_1 \mu} x + \mu^{-\alpha}\text{ for all }\mu\geq \mu_0  \Big)
\Longrightarrow y \leq \frac{D_1}{\log\left(\frac{1}{x}+1\right)^{\alpha}} \Longrightarrow \frac{1}{x} \leq e^{ \left(\frac{D_1}{y}\right)^{1/\alpha}} .
$$ 
Note that the second implication is straightforward since the second assertion is equivalent to $\frac{1}{x} \leq e^{ \left(\frac{D_1}{y}\right)^{1/\alpha}}- 1$.
To prove the first implication, we set $$\mu(x):= \frac{1}{2C_1}\log\left(\frac{1}{x}+1\right) ,$$ so that $e^{C_1 \mu(x)}x=\left(\frac{1}{x}+1\right)^{1/2} x = (1+x)^{1/2}x^{1/2}$. Denoting now $C_3=  C_3(C_1,C_2,\alpha)=\sup_{x\leq C_2}  (1+x)^{1/2}x^{1/2}\mu(x)^{\alpha} < + \infty$, we have $e^{C_1 \mu(x)}x\leq \frac{C_3}{\mu(x)^{\alpha}} $.  As a consequence, if $\mu(x)\geq \mu_0$, then we have $ y \leq \frac{\left(C_3+1\right)}{\mu(x)^{\alpha}}$, which is the sought estimate.

If now $\mu(x)\leq \mu_0$, that is $\frac{1}{2C_1}\log\left(\frac{1}{x}+1\right)\leq \mu_0$, we have $1\leq \left(\frac{2C_1\mu_0}{\log\left(\frac{1}{x}+1\right)}\right)^{\alpha}$ Then, the assumption $y\leq 1$ directly implies $y \leq \left(\frac{2C_1\mu_0}{\log\left(\frac{1}{x}+1\right)}\right)^{\alpha}$. This concludes the proof of the lemma for $D_1 =(2C_1)^\alpha \max \left\{ C_3+1 ,\mu_0^\alpha\right\}$.
\enp
\section{Elementary complex analysis}

We recall that we identify $\C$ and $\R^2$  with $z= x+iy = (x,y)$ and denote 
$$
Q_1 = \{z \in \C, \Re(z) >0 , \Im(z) > 0\} .
$$

\begin{lemma}
\label{l:green}
Let $f_0 , f_1 \in  W^{1,\infty}(\R^+)$ such that $|f_0'(x)|, |f_1'(x)|\leq C$ 
for some $C>0$ and almost all $x \in \R^+$. Then, the function defined for $(x,y) \in Q_1$ by
\bnan
\label{explicitgreen}
f(x,y)= \frac{4xy}{\pi}\int_0^{\infty}\frac{\xi f_0(\xi)}{((x-\xi)^2+y^2)((x+\xi)^2+y^2)}d\xi+\frac{4xy}{\pi}\int_0^{\infty}\frac{\eta f_1(\eta)}{(x^2+(y+\eta)^2)(x^2+(y-\eta)^2)}d\eta
\enan
satisfies $|f(z)|\leq 2C(1+|z|)$ in $\overline{Q}_1\setminus (0,0)$ together with
$$
\Delta f = 0 \quad \text{in }Q_1, \qquad f(x,0) =f_0(x) , \quad f(0,y) =f_1(y) , \quad x, y \in \R^+ .
$$
If moreover, $f_0(0)=f_1(0)$, then $f$ is continuous on $\overline{Q}_1$.
\end{lemma}
Remark that this theorem provides an existence result for the Poisson Problem on $Q_1$ associated to Lipschitz boundary conditions. The Phragm\'en-Lindel\"of theorem~\ref{l:phragmen} below provides an associated uniqueness result in the class of functions having a sub-quadratic growth at infinity.

The next lemma is a key point in the proof of the local estimate.
\begin{lemma}
\label{explicitgreenbis}
{
Let $R>0$, $\delta>0$, $\kappa >0$, $\eps>0$ and $c_1>0$. Then, there exists $d_0 = d_0(\delta, \kappa,R,\eps, c_1)$ such that for any $d<d_0$, there exists $\beta_0(\delta, \kappa,R,\eps, c_1, d)$, such that for any $0<\beta<\beta_0$, the following two assertions hold:
\begin{itemize}
\item the function 
\bna
f_1(y)= Ry  \mathds{1}_{[0, \gamma)}(y) + \mathds{1}_{[\gamma, + \infty)}(y) \min \big\{ Ry , \max(- \kappa , - 9\delta y ,-\frac{\e}{y}) + c_1y^2+ \frac{\beta^2}{y} \big\} .
\ena
is continuous for all $\gamma \leq \frac{\beta}{(R+9\delta)^\frac12}$ (in the application $\gamma = \frac{\tau_0}{\mu}$).
\item the function $f$ then given by Lemma~\ref{l:green} associated to $f_1$ and $f_0=0$ satisfies 
\bna
f(x,y)\leq -8\delta y , \quad \text{ for } \frac{d}{4}\leq |(x,y)|\leq 2 d .
\ena
\end{itemize}
}
\end{lemma}

\bnp[Proof of Lemma~\ref{l:green}]
Let us first justify the form~\eqref{explicitgreen} of the solution.
From the green function $G_\C(z,z') = (2\pi)^{-1}\ln|z'-z|$ in $\C$, we first construct a Green function in $Q_1$ using the so-called ``image points'' $\bar{z}$, $-z$ and $-\bar{z}$.  This yields
$$
G_{Q_1}(z,z') := \frac{1}{2\pi} \ln|z'-z| - \frac{1}{2\pi} \ln|z'-\bar{z}|- \frac{1}{2\pi} \ln|z' +\bar{z}| + \frac{1}{2\pi} \ln|z'+z| ,
$$
that is, with $z= (x,y)$ and $z' = (\xi,\eta)$,
\bna
G_{Q_1}((x,y),(\xi,\eta)) & := & \frac{1}{4\pi} \ln\left((\xi-x)^2 + (\eta -y)^2\right) 
- \frac{1}{4\pi} \ln\left((\xi-x)^2 + (\eta +y)^2\right) \\
&& - \frac{1}{4\pi} \ln\left((\xi+x)^2 + (\eta -y)^2\right)
+ \frac{1}{4\pi} \ln\left((\xi+x)^2 + (\eta +y)^2\right) .
\ena
For fixed $z \in Q_1$, the last three terms are smooth in $z' \in Q_1$ so that $-\Delta_{z'}G_{Q_1}(z,z') = \delta_{z' = z}$. Moreover, for $z' = (\xi, \eta) \in \d Q_1$, either $\xi = 0$ or $\eta = 0$ so that $G_{Q_1} = 0$ for $z' \in Q_1$.

Now we compute
\bna
\frac{\d G_{Q_1}}{\d \xi}\big|_{\xi =0}  &=&- \frac{4xy}{\pi}\frac{\eta}{(x^2+(y+\eta)^2)(x^2+(y-\eta)^2)} ,
\\
\frac{\d G_{Q_1}}{\d \eta}\big|_{\eta =0}  &=&-  \frac{4xy}{\pi} \frac{\xi }{((x-\xi)^2+y^2)((x+\xi)^2+y^2)} .
\ena
The representation formula for solutions of $\Delta f = 0$ in $Q_1$ and $f|_{\d Q_1} = \tilde f$ writes 
$$
f(z) = \int_{\d Q_1} \frac{\d G_{Q_1}}{\d \nu_{\d Q_1}}(z,z')|_{z'\in \d Q_1}\tilde f (z') dz' ,
$$
which justifies~\eqref{explicitgreen}.

\medskip
Let us now estimate for $(x,y) \in Q_1$ the term
\bna
\left| \frac{4xy}{\pi}\int_0^{\infty}\frac{\eta f_1(\eta)}{(x^2+(y+\eta)^2)(x^2+(y-\eta)^2)}d\eta \right| 
&\leq & \frac{4xy}{\pi}\int_0^{\infty}\frac{\eta C(1 +\eta)}{(x^2+(y+\eta)^2)(x^2+(y-\eta)^2)}d\eta \\
&\leq & C\left( 2/\pi \arctan(y/x) + y \right) \\
&\leq & C\left( 1+ y \right) ,
\ena
where we used Lemma~\ref{integralespourries} in the second inequality.
The other term containing $f_0$ can be estimated as well in $Q_1$ by $C\left( 1+ x\right)$ so that 
$$
|f(z)| \leq  C\left( 2+x+ y \right) \leq 2 C\left( 1+|z| \right) , \quad z=(x,y) \in Q_1 .
$$

That $\Delta f = 0$ follows from the definition of $G_{Q_1}$ as a Green function, and it only remains to check the boundary values of $f$. For this, according to the symmetry, it suffices to prove that for all $x_0 , y_0 >0$, we have
\bnan
\label{eq:limits to be proved}
\lim_{(x,y)\to (x_0,0)} (Tf_1)(x,y) = 0 , \quad \lim_{(x,y)\to (0,y_0)} (Tf_1)(x,y) = f_1(y_0)   .
\enan
with 
\bna
(Tf_1)(x,y) : = \frac{4xy}{\pi}\int_0^{\infty}\frac{\eta f_1(\eta)}{(x^2+(y+\eta)^2)(x^2+(y-\eta)^2)}d\eta .
\ena 
Since $f_1' \in L^\infty(\R^+)$, we have
$$
|f_1(\eta)| \leq |f_1(0)| + \eta \|f_1'\|_{L^\infty} .
$$
Hence, according to the definition of  $T$, we obtain 
\bnan
\label{estimTf1}
|(Tf_1) | \leq |f_1(0)| T(1) +  \|f_1'\|_{L^\infty} T(\eta) .
\enan
Using Lemma~\ref{integralespourries}, this implies
$$
|(Tf_1) (x,y)| \leq |f_1(0)|  2/\pi \arctan(y/x)  +  \|f_1'\|_{L^\infty} y ,
$$
and thus $(Tf_1) (x,y)\to 0 $ as ($x,y)\to (x_0,0)$, which yields the first part of~\eqref{eq:limits to be proved}.

To prove the second part of~\eqref{eq:limits to be proved}, we write
$$
|f_1(\eta) - f_1(y_0) | \leq |\eta - y_0|  \|f_1'\|_{L^\infty} .
$$
This implies
\bnan
\label{eq:presque conclu T}
|Tf_1 (x,y) - 2/\pi \arctan(y/x) f_1(y_0) | =  |Tf_1 - T (f_1(y_0)) | \leq   \|f_1'\|_{L^\infty} T(|\eta - y_0|).
\enan
We now study the term 
\bna
T(|\eta - y_0|)(x,y) & = & \frac{4xy}{\pi}\int_0^{\infty}\frac{\eta |\eta - y_0|}{(x^2+(y+\eta)^2)(x^2+(y-\eta)^2)}d\eta  \\
 & = & \frac{4xy}{\pi}\int_0^{y_0}\frac{\eta (y_0 - \eta)}{(x^2+(y+\eta)^2)(x^2+(y-\eta)^2)}d\eta \\
 && \quad + \frac{4xy}{\pi}\int_{y_0}^{\infty}\frac{\eta (\eta - y_0)}{(x^2+(y+\eta)^2)(x^2+(y-\eta)^2)}d\eta  \\
 & = & 2 \frac{4xy}{\pi}\int_0^{y_0}\frac{\eta (y_0 - \eta)}{(x^2+(y+\eta)^2)(x^2+(y-\eta)^2)}d\eta \\
 && \quad + \frac{4xy}{\pi}\int_{0}^{\infty}\frac{\eta (\eta - y_0)}{(x^2+(y+\eta)^2)(x^2+(y-\eta)^2)}d\eta  \\
  & = & 2 \frac{4xy}{\pi}\int_0^{y_0}\frac{\eta (y_0 - \eta)}{(x^2+(y+\eta)^2)(x^2+(y-\eta)^2)}d\eta  + T(\eta - y_0)(x,y) .
  \ena
  With Lemma~\ref{integralespourries}, we have $T(\eta - y_0)(x,y) = y - 2/\pi \arctan(y/x)y_0 \to 0$ as $(x,y)\to (0,y_0)$. Moreover, we have
  \bna
  \frac{4xy}{\pi}\int_0^{y_0}\frac{\eta (y_0 - \eta)}{(x^2+(y+\eta)^2)(x^2+(y-\eta)^2)}d\eta = \frac{1}{\pi}\int_0^{y_0}-\frac{x (y_0 - \eta)}{x^2+(y+\eta)^2}+\frac{x (y_0 - \eta)}{x^2+(y-\eta)^2}d\eta
  \ena
(see the proof of Lemma~\ref{integralespourries}). The term $\int_0^{y_0}\frac{x (y_0 - \eta)}{x^2+(y+\eta)^2} d\eta$ vanishes when $(x,y)\to (0,y_0)$. Concerning the second term, we have
  \bna
  \frac{1}{\pi}\int_0^{y_0}\frac{x (y_0 - \eta)}{x^2+(y-\eta)^2}d\eta 
  & = &  \frac{1}{\pi}\int_{-y/x}^{(y_0-y)/x}(y_0 - y - xs)\frac{ds}{1+s^2} \\
  & = &  \frac{y_0 - y}{\pi}\left( \arctan \left(\frac{y_0 - y}{x}\right) +  \arctan \left(\frac{y}{x}\right)  \right) 
     - \frac{x}{2\pi} \ln \left( \frac{x^2 + (y_0 - y)^2}{x^2 + y^2} \right) ,
  \ena
which vanishes when $(x,y)\to (0,y_0)$. The last three estimate prove $T(|\eta - y_0|)(x,y) \to 0$ as $(x,y)\to (0,y_0)$.
In view of~\eqref{eq:presque conclu T}, this implies 
$$
\lim_{(x,y)\to (0,y_0)}  |Tf_1 (x,y) - 2/\pi \arctan(y/x) f_1(y_0) | = 0
$$
which is the second part of~\eqref{eq:limits to be proved}. 

For the continuity, by symmetry and translation by a constant, it is sufficient to prove that if $f_1(0)=0$, then $Tf_1(x,y)$ converges to zero as $(x,y)$ converges to zero. This is implied by \eqref{estimTf1}.
This concludes the proof of the Lemma.
\enp

\bnp[Proof of Lemma~\ref{explicitgreenbis}]
Let us define 
\bna
I_{\beta} : = \left[\beta \sqrt{2/\delta} , \min(\frac{\delta}{2 c_1} ,  \frac{\kappa}{9\delta} , \frac{\sqrt{\eps}}{3 \sqrt{\delta}}) \right]  ,
\ena
and notice that $I_\beta \neq \emptyset$ for $\beta \leq \beta_0$ with $\beta_0 = \beta_0(\delta, \kappa, c_1 , \eps)$ sufficiently small.
We first prove that for all $\gamma \leq \beta \sqrt{4/\delta}$, we have
\bnan
\label{f1Ibeta}
f_1(y) = -9 \delta y + c_1 y^2  +\frac{\beta^2}{y} \textnormal{ on } I_\beta, 
\enan
and
\bnan
\label{eqn:Ibeta2}
I_{\beta}  \subset \{f_1 \leq -8 \delta y \}, 
\enan
and, \textit{a fortiori} for $\gamma \leq \frac{\beta}{(R+9\delta)^\frac12} \leq \beta \sqrt{4/\delta}$.

For this, notice that $y \in I_\beta$ implies $y \leq \delta/(4 c_1)$ and $y \geq \beta \sqrt{4/\delta}$ which yields 
$$
-\frac{\delta y^2}{2} + c_1 y^3 \leq -\frac{\delta}{4}y^2 \leq -\beta^2 .
$$
As a consequence, \bnan
\label{inegRy}
-\frac{\delta}{2} y + c_1 y^2  +\frac{\beta^2}{y} \leq 0, \quad \text{and hence} 
\quad-9 \delta y + c_1 y^2  +\frac{\beta^2}{y} \leq -8.5 \delta y \leq 0 \leq Ry .
\enan
In particular, \eqref{f1Ibeta} implies \eqref{eqn:Ibeta2}. Moreover, for $y \in I_\beta$, we have $-\kappa \leq -9\delta y$ together with $- \frac{\eps}{y} \leq - 9 \delta y$, so that $\max(- \kappa , - 9\delta y ,-\frac{\e}{y})=-9\delta y$.  This proves \eqref{f1Ibeta} with the help of \eqref{inegRy}. 

\medskip
Let us now check the continuity of $f_1$. First remark that both $Ry$ and  $\min \big\{ Ry , \max(- \kappa , - 9\delta y , - \frac{\eps}{4 y}) + c_1 y^2+ \frac{\beta^2}{y} \big\}$ are continuous. Second, we prove that both functions coincide  for $y \leq \gamma$ which provides the continuity of $f_1$. For $0\leq y \leq \gamma \leq \frac{\beta}{(R+9\delta)^\frac12}$, we have $(9 \delta + R - c_1 y)y^2 \leq \beta^2$ and we obtain $Ry \leq  -9 \delta y + c_1 y^2  +\frac{\beta^2}{y}$. For $\beta \leq \beta_0$ we have $I_\beta \neq \emptyset$ so that $y \leq \beta \sqrt{4/\delta} \leq \min(  \frac{\kappa}{9\delta} , \frac{\sqrt{\eps}}{3 \sqrt{\delta}}) $, and $ \max(- \kappa ,-9 \delta y, -\frac{\eps}{y})= -9 \delta y$ for $y \leq \gamma$. As a consequence, we have 
$$ Ry  = \min \big\{ Ry , \max(- \kappa , - 9\delta y , -\frac{\eps}{y}) + c_1 y^2+ \frac{\beta^2}{y} \big\} , \quad \text{ for } 0\leq y \leq \gamma ,
$$
and $f_1$ is continuous for all $\beta \leq \beta_0$ and $\gamma \leq \frac{\beta}{(R+9\delta)^\frac12}$.

\medskip
Since $f_1$ is continuous and globally Lipschitz, it satisfies all assumptions of lemma~\eqref{l:green} (and $f_0 = 0$), so that we can define $f$ by
\bna
f(x,y)= \frac{4xy}{\pi}\int_0^{\infty}\frac{\eta f_1(\eta)}{(x^2+(y+\eta)^2)(x^2+(y-\eta)^2)}d\eta
\ena
Setting $\widetilde{f}=f+8.5\delta y$, we now prove an upper bound for $\widetilde{f}$. Using the second formula of Lemma~\ref{integralespourries}, we have
\bna
\widetilde{f}(x,y) 
& = & \frac{4xy}{\pi}\int_0^{\infty}\frac{\eta(f_1(\eta)+8.5 \delta\eta)}{(x^2+(y+\eta)^2)(x^2+(y-\eta)^2)}d\eta \\
& = &  \frac{4xy}{\pi}\int_{\R_+ \setminus I_\beta} \cdots d\eta  
+ \frac{4xy}{\pi}\int_{I_\beta} \cdots d\eta .
\ena
According to~\eqref{eqn:Ibeta2}, we have
\bnan
\label{eq:Ibetanegatif}
\frac{4xy}{\pi}\int_{I_\beta}\frac{\eta(f_1(\eta)+8.5 \delta\eta)}{(x^2+(y+\eta)^2)(x^2+(y-\eta)^2)}d\eta \leq 0 .
\enan
Next, for small $\beta$, we have $\R_+ \setminus I_\beta = [0, D_\beta] \cup [D, +\infty]$, with $D_\beta := \beta \sqrt{4/\delta}< D: = \min(\frac{\delta}{2 c_1} ,  \frac{\kappa}{9\delta} , \frac{\sqrt{\eps}}{6 \sqrt{\delta}})$. Since $f_1(y) \leq Ry$, we have
\bna
\frac{4xy}{\pi}\int_D^{\infty} \cdots \leq 
 \frac{4xy}{\pi}\int_D^{\infty}\frac{(R+8.5\delta)\eta^2}{(x^2+(y+\eta)^2)(x^2+(y-\eta)^2)}d\eta .
\ena
If $0\leq y\leq D/2$ and $\eta\geq D$, we have $(y-\eta)^2\geq (\eta-D/2)^2$ and $(y+\eta)^2\geq \eta^2$, so we estimate
\bna
\frac{4xy}{\pi}\int_D^{\infty} \cdots \leq \frac{16xy}{\pi}\int_D^{\infty}\frac{ (R+8.5\delta)\eta^2}{\eta^2(\eta-D/2)^2}d\eta = C(\delta,\kappa, R, \eps , c_1)xy
\ena
So, if $x\leq \nu D$ and $y\leq D/2$, this implies 
\bnan
\label{estimeepourrieD}
\frac{4xy}{\pi}\int_D^{\infty} \cdots
\leq \nu C(\delta,\kappa, R,\eps , c_1)D(\delta,\kappa,\eps , c_1)y 
\leq \de y /4
\enan
as soon as $\nu \leq \frac{\delta}{4C D }$. Now we fix $2d_0 :=2 d_0(\delta,\kappa, R,\eps , c_1) = \min \{ \nu D ,  D/2\}$. For any $d \leq d_0$, we have~\eqref{estimeepourrieD} for all $(x,y)$ such that $|(x,y)| \leq 2 d$.

\medskip
Finally, we study the term $\frac{4xy}{\pi}\int_0^{D_\beta} \cdots d\eta$. 
For $\beta$ sufficiently small (namely $\beta \leq \frac{d\sqrt{\delta}}{16}$), we have $\frac{d}{4} - D_\beta \geq \frac{d}{8}$ (recall that $D_\beta = \beta \sqrt{4/\delta}$). As a consequence, for $(x,y)$ such that $\frac{d}{4} \leq |(x,y)| \leq 2 d$, and for all $\eta \in [0, D_\beta]$, the triangle inequality yields 
$$
(x^2+(y+\eta)^2) \geq (\frac{d}{4} - D_\beta)^2 \geq \frac{d^2}{8^2} , 
\qquad (x^2+(y-\eta)^2) \geq (\frac{d}{4} - D_\beta)^2 \geq \frac{d^2}{8^2} .
$$
Still using that $f_1(y) \leq Ry$, we have
\bna
\frac{4xy}{\pi}\int_0^{D_\beta} \cdots & \leq & 
 \frac{4xy}{\pi}\int_0^{D_\beta}\frac{(R+8.5\delta)\eta^2}{(x^2+(y+\eta)^2)(x^2+(y-\eta)^2)}d\eta \\
 & \leq & \frac{4xy}{\pi}\left( \frac{8}{d}\right)^4 (R+8.5\delta) \int_0^{D_\beta}\eta^2 d\eta \
 \\
 & \leq & \frac{4xy}{\pi}\left( \frac{8}{d}\right)^4 (R+8.5\delta) \frac{D_\beta^3}{3}\
 \\
 & \leq &  C'(R, \delta,   d) \beta^3 y\
\ena
Now, for all $\beta \leq \left(\frac{\delta}{4C'(R, \delta,   d) } \right)^{1/3}$ this is less than $\delta y/4$.

This together with \eqref{eq:Ibetanegatif} and \eqref{estimeepourrieD} implies that $\tilde f(x,y) \leq \delta y/2$ for $(x,y)$ such that $\frac{d}{4} \leq |(x,y)| \leq 2 d$, that is 
\bna
f (x,y )\leq  -8\delta y , \quad \text{ for }\frac{d}{4} \leq |(x,y)| \leq 2 d .
\ena
This concludes the proof of the lemma.
\enp
\begin{lemma}
\label{integralespourries}
For all $x,y >0$, we have
\bna
&&\frac{4xy}{\pi}\int_0^{\infty}\frac{\eta}{(x^2+(y+\eta)^2)(x^2+(y-\eta)^2)}d\eta = 2/\pi \arctan(y/x)
\\
&&\frac{4xy}{\pi}\int_0^{\infty}\frac{\eta^2}{(x^2+(y+\eta)^2)(x^2+(y-\eta)^2)}d\eta =y.
\ena
\end{lemma}
\bnp[Proof of Lemma~\ref{integralespourries}]
First notice that
\bna
\frac{4xy\eta}{(x^2+(y+\eta)^2)(x^2+(y-\eta)^2)}=-\frac{x}{x^2+(y+\eta)^2}+\frac{x}{x^2+(y-\eta)^2} .
\ena
Hence, we obtain
\bna
4xy\int_0^{N}\frac{\eta}{(x^2+(y+\eta)^2)(x^2+(y-\eta)^2)}d\eta & = & \int_0^{N}-\frac{x}{x^2+(y+\eta)^2}+\frac{x}{x^2+(y-\eta)^2}d\eta\\
& = & - \int_{y/x}^{(N+y)/x}\frac{1}{1+s^2} ds +\int_{(y-N)/x}^{y/x} \frac{1}{1+s^2} ds\\
& = & - \arctan((N+y)/x)) + \arctan(y/x) \\
& & + \arctan(y/x) -  \arctan((y-N)/x))\\
& \to & 2 \arctan(y/x), \quad \text{ as }N \to \infty,
\ena
since $x,y>0$.

\medskip
Concerning the second equation, we have
\bna
\int_0^{N}\frac{4xy\eta^2}{(x^2+(y+\eta)^2)(x^2+(y-\eta)^2)}d\eta&=&\int_0^{N}-\frac{x\eta}{x^2+(y+\eta)^2}+\frac{x\eta}{x^2+(y-\eta)^2}d\eta\\
& = &-\int_{-N}^{N}\frac{x\eta}{x^2+(y+\eta)^2}d\eta=-\int_{-N+y}^{N+y}\frac{x(s-y)}{x^2+s^2}ds \\
&=&-\int_{-N+y}^{N+y}\frac{xs}{x^2+s^2}ds+\int_{-N+y}^{N+y}\frac{xy}{x^2+s^2}ds .
\ena
Since the integrand is an odd function, we have $\int_{-N+y}^{N+y}\frac{xs}{x^2+s^2}ds=\int_{-N-y}^{-N+y}\frac{xs}{x^2+s^2}ds$ which converges to zero as $N\to \infty$. Moreover, we have
\bna
\int_{-N-y}^{N-y}\frac{xy}{x^2+s^2}ds=y\int_{(-N-y)/x}^{(N-y)/x}\frac{1}{1+s^2}ds \to \pi y , \quad \text{ as }N\to \infty ,
\ena
which concludes the proof of the lemma. 
\enp

The following is a version of the Phragm\'en Lindel\"of principle for subharmonic functions in a sector of the complex plane. We prove it as a consequence of the maximum principle for subharmonic functions in bounded domains. Note that the usual Phragm\'en Lindel\"of theorem (see~\cite{PL:08} or~\cite[Theorem~3.4]{SS:03}) can be deduced from this one.

\begin{lemma}
\label{l:phragmen}
Let $\phi$ be a subharmonic function in $Q_1$, continuous in $\bar{Q}_1$.
Assume that there exist $\eps>0$ and $C>0$ such that
\bnan
\phi(z) \leq C(1 + |z|^{2-\eps})  , \quad z \in Q_1 ,
\enan
\bnan
\phi(z) \leq 0 , \quad z \in \d Q_1 = \R_+ \cup i \R_+ .
\enan
Then $\phi(z) \leq 0$ for all $z \in Q_1$.
\end{lemma}
Note that the power $2-\eps$ with $\eps >0$ is sharp: the result is false for $\eps = 0$, as showed by the harmonic function $(x, y) \mapsto xy$.
\bnp 
First note that the sector $Q_1$ can be rotated, say to quadrant
$$
Q = \{z \in \C , \arg(z) \in [-\frac{\pi}{4},\frac{\pi}{4}]\} .
$$
We set $v := \Re(z^{2-\frac{\eps}{2}})$ (with the principal determination of the logarithm) which is harmonic in $Q$. We have $v(r, \theta) := r^{2-\frac{\eps}{2}} \cos((2-\eps/2) \theta) \geq r^{2-\frac{\eps}{2}} \cos((2-\eps/2)\pi/4 )$ with $ \cos((2-\eps/2)\pi/4 )>0$. Let 
$$
u_\delta(z) = \phi(z) - \delta v(z) ,
$$ 
which is also subharmonic in $Q$. We have $\limsup_{z \in Q , |z|\to \infty} u(z) = - \infty$. As a consequence, there exists $R>0$ such that $ u_\delta(z) <0$ on $\{|z| \geq R\} \cap Q$. Now, on the bounded  set $Q^R = Q \cap \{|z| \leq R\}$, we apply the maximum principle to the function $u_\delta$,
satisfying $u_\delta \leq 0$ on $\d Q^R$. This yields $u_\delta \leq 0$ on $Q^R$ and hence $u_\delta \leq 0$ on $Q$. Finally letting $\delta$ tend to zero, we obtain the sought result.
\enp
\small
\bibliographystyle{alpha} 
\bibliography{bibli}
\end{document}

%% file: mymacros.tex
\newtheorem{lemma}{Lemma}[section]
\newtheorem{theorem}[lemma]{Theorem}
\newtheorem{proposition}[lemma]{Proposition}
\newtheorem{corollary}[lemma]{Corollary}

\theoremstyle{definition}

\newtheorem{definition}[lemma]{Definition}
\newtheorem{remark}[lemma]{Remark}

\makeatletter
\def\keywords{
    \vspace{1ex}
    \noindent
    \if@twocolumn
      \small{\bf  Keywords}\/---$\!$    \else
      \begin{center}\small\ {\bf Keywords}\end{center}\quotation\small
    \fi}
\def\endkeywords{\vspace{0.6em}\par\if@twocolumn\else\endquotation\fi
    \normalsize\rm}
\makeatother

\newcommand{\mb}[1]{\ensuremath{\mathbb{#1}}}
\newcommand{\N}{{\mb{N}}}

\newcommand{\R}{{\mb{R}}}
\newcommand{\C}{{\mb{C}}}

\newcommand{\de}{\delta}
\newcommand{\eps}{\varepsilon}

\newcommand{\M}{\ensuremath{\mathcal M}}

\let \Re \relax
\DeclareMathOperator{\Re}{Re}
\let \Im \relax
\DeclareMathOperator{\Im}{Im}

\renewcommand{\S}{\ensuremath{\mathscr S}}

\DeclareMathOperator{\supp}{supp}

\DeclareMathOperator{\dist}{dist}
\DeclareMathOperator{\vois}{Vois}

\DeclareMathOperator{\id}{Id}

\renewcommand{\d}{\ensuremath{\partial}}

\let \div \relax
\DeclareMathOperator{\div}{div}

\DeclareMathOperator{\length}{length}

\def\x{x}

\newcommand{\E}{\mathscr E}



%% file: geometric-setting.pdf_t
\begin{picture}(0,0)%
\includegraphics{geometric-setting.pdf}%
\end{picture}%
\setlength{\unitlength}{3947sp}%
\begingroup\makeatletter\ifx\SetFigFont\undefined%
\gdef\SetFigFont#1#2#3#4#5{%
  \reset@font\fontsize{#1}{#2pt}%
  \fontfamily{#3}\fontseries{#4}\fontshape{#5}%
  \selectfont}%
\fi\endgroup%
\begin{picture}(4864,2607)(2086,-1471)
\put(6920,381){\makebox(0,0)[lb]{\smash{{\SetFigFont{10}{12.0}{\familydefault}{\mddefault}{\updefault}{\color[rgb]{0,0,1}$K$}%
}}}}
\put(3453,599){\makebox(0,0)[lb]{\smash{{\SetFigFont{10}{12.0}{\familydefault}{\mddefault}{\updefault}{\color[rgb]{0,0,0}$x_a$}%
}}}}
\put(2688,857){\makebox(0,0)[lb]{\smash{{\SetFigFont{10}{12.0}{\familydefault}{\mddefault}{\updefault}{\color[rgb]{0,0,0}$x_b$}%
}}}}
\put(6872,-1145){\makebox(0,0)[lb]{\smash{{\SetFigFont{10}{12.0}{\familydefault}{\mddefault}{\updefault}{\color[rgb]{0,0,0}$x'$}%
}}}}
\put(2101,989){\makebox(0,0)[lb]{\smash{{\SetFigFont{10}{12.0}{\familydefault}{\mddefault}{\updefault}{\color[rgb]{0,0,0}$x_n$}%
}}}}
\put(6927, 59){\makebox(0,0)[lb]{\smash{{\SetFigFont{10}{12.0}{\familydefault}{\mddefault}{\updefault}{\color[rgb]{0,0,0}$S_1$}%
}}}}
\put(6935,-371){\makebox(0,0)[lb]{\smash{{\SetFigFont{10}{12.0}{\familydefault}{\mddefault}{\updefault}{\color[rgb]{0,0,0}$S_0$}%
}}}}
\put(6919,668){\makebox(0,0)[lb]{\smash{{\SetFigFont{10}{12.0}{\familydefault}{\mddefault}{\updefault}{\color[rgb]{0,0,0}$\Omega$}%
}}}}
\put(6935,-665){\makebox(0,0)[lb]{\smash{{\SetFigFont{10}{12.0}{\familydefault}{\mddefault}{\updefault}{\color[rgb]{1,0,0}$\tilde\omega$}%
}}}}
\end{picture}%

%% file: countours-gamma.pdf_t
\begin{picture}(0,0)%
\includegraphics{countours-gamma.pdf}%
\end{picture}%
\setlength{\unitlength}{3947sp}%
\begingroup\makeatletter\ifx\SetFigFont\undefined%
\gdef\SetFigFont#1#2#3#4#5{%
  \reset@font\fontsize{#1}{#2pt}%
  \fontfamily{#3}\fontseries{#4}\fontshape{#5}%
  \selectfont}%
\fi\endgroup%
\begin{picture}(4694,1599)(2454,-2098)
\put(3230,-1877){\makebox(0,0)[lb]{\smash{{\SetFigFont{10}{12.0}{\familydefault}{\mddefault}{\updefault}{\color[rgb]{0,0,0}$\alpha_k^1- i \e R_f$}%
}}}}
\put(6186,-1441){\makebox(0,0)[lb]{\smash{{\SetFigFont{10}{12.0}{\familydefault}{\mddefault}{\updefault}{\color[rgb]{0,0,0}$\gamma_R $}%
}}}}
\put(4723,-705){\makebox(0,0)[lb]{\smash{{\SetFigFont{10}{12.0}{\familydefault}{\mddefault}{\updefault}{\color[rgb]{0,0,0}$\Im(z_1)$ }%
}}}}
\put(6138,-1878){\makebox(0,0)[lb]{\smash{{\SetFigFont{10}{12.0}{\familydefault}{\mddefault}{\updefault}{\color[rgb]{0,0,0}$\alpha_k^2- i \e R_f$}%
}}}}
\put(6751,-1261){\makebox(0,0)[lb]{\smash{{\SetFigFont{10}{12.0}{\familydefault}{\mddefault}{\updefault}{\color[rgb]{0,0,0}$\Re(z_1)$ }%
}}}}
\put(3301,-1036){\makebox(0,0)[lb]{\smash{{\SetFigFont{10}{12.0}{\familydefault}{\mddefault}{\updefault}{\color[rgb]{0,0,0}$\alpha_k^1$ }%
}}}}
\put(5026,-1636){\makebox(0,0)[lb]{\smash{{\SetFigFont{10}{12.0}{\familydefault}{\mddefault}{\updefault}{\color[rgb]{0,0,0}$\gamma_T$}%
}}}}
\put(6301,-1036){\makebox(0,0)[lb]{\smash{{\SetFigFont{10}{12.0}{\familydefault}{\mddefault}{\updefault}{\color[rgb]{0,0,0}$\alpha_k^2$}%
}}}}
\put(3483,-1441){\makebox(0,0)[lb]{\smash{{\SetFigFont{10}{12.0}{\familydefault}{\mddefault}{\updefault}{\color[rgb]{0,0,0}$\gamma_L$ }%
}}}}
\put(4695,-1066){\makebox(0,0)[lb]{\smash{{\SetFigFont{10}{12.0}{\familydefault}{\mddefault}{\updefault}{\color[rgb]{0,0,0}$0$ }%
}}}}
\end{picture}%

%% file: geometry-local-result.pdf_t
\begin{picture}(0,0)%
\includegraphics{geometry-local-result.pdf}%
\end{picture}%
\setlength{\unitlength}{3947sp}%
\begingroup\makeatletter\ifx\SetFigFont\undefined%
\gdef\SetFigFont#1#2#3#4#5{%
  \reset@font\fontsize{#1}{#2pt}%
  \fontfamily{#3}\fontseries{#4}\fontshape{#5}%
  \selectfont}%
\fi\endgroup%
\begin{picture}(3678,4953)(2362,-2721)
\put(3954,-536){\makebox(0,0)[lb]{\smash{{\SetFigFont{10}{12.0}{\familydefault}{\mddefault}{\updefault}{\color[rgb]{0,0,0}$x^0$}%
}}}}
\put(3058,1941){\makebox(0,0)[lb]{\smash{{\SetFigFont{10}{12.0}{\familydefault}{\mddefault}{\updefault}{\color[rgb]{0,0,0}$\{\phi  = 0\}$}%
}}}}
\put(4426,1664){\makebox(0,0)[lb]{\smash{{\SetFigFont{10}{12.0}{\familydefault}{\mddefault}{\updefault}{\color[rgb]{0,0,0}$\nabla \phi$}%
}}}}
\put(3790,2085){\makebox(0,0)[lb]{\smash{{\SetFigFont{10}{12.0}{\familydefault}{\mddefault}{\updefault}{\color[rgb]{0,0,0}$\{\phi  = 2\rho\}$}%
}}}}
\put(2931, 96){\makebox(0,0)[lb]{\smash{{\SetFigFont{10}{12.0}{\familydefault}{\mddefault}{\updefault}{\color[rgb]{0,0,0}$B(x^0,3R)$}%
}}}}
\put(2935,914){\makebox(0,0)[lb]{\smash{{\SetFigFont{10}{12.0}{\familydefault}{\mddefault}{\updefault}{\color[rgb]{0,0,0}$B(x^0,4R)$}%
}}}}
\put(5476,-286){\makebox(0,0)[lb]{\smash{{\SetFigFont{10}{12.0}{\familydefault}{\mddefault}{\updefault}{\color[rgb]{0,0,0}$\nabla \phi(x^0) $}%
}}}}
\put(3160,-659){\makebox(0,0)[lb]{\smash{{\SetFigFont{10}{12.0}{\familydefault}{\mddefault}{\updefault}{\color[rgb]{0,0,0}$B(x^0,r)$}%
}}}}
\end{picture}%

%% file: level-phi-psi.pdf_t
\begin{picture}(0,0)%
\includegraphics{level-phi-psi.pdf}%
\end{picture}%
\setlength{\unitlength}{3947sp}%
\begingroup\makeatletter\ifx\SetFigFont\undefined%
\gdef\SetFigFont#1#2#3#4#5{%
  \reset@font\fontsize{#1}{#2pt}%
  \fontfamily{#3}\fontseries{#4}\fontshape{#5}%
  \selectfont}%
\fi\endgroup%
\begin{picture}(3981,4908)(2237,-2676)
\put(6203,1367){\makebox(0,0)[lb]{\smash{{\SetFigFont{10}{12.0}{\familydefault}{\mddefault}{\updefault}{\color[rgb]{0,0,1}$\{\psi = \eta_1\}$}%
}}}}
\put(3954,-536){\makebox(0,0)[lb]{\smash{{\SetFigFont{10}{12.0}{\familydefault}{\mddefault}{\updefault}{\color[rgb]{0,0,0}$x^0$}%
}}}}
\put(2457,453){\makebox(0,0)[lb]{\smash{{\SetFigFont{10}{12.0}{\familydefault}{\mddefault}{\updefault}{\color[rgb]{0,0,0}$B(x^0,\frac{R}{8})$}%
}}}}
\put(3790,2085){\makebox(0,0)[lb]{\smash{{\SetFigFont{10}{12.0}{\familydefault}{\mddefault}{\updefault}{\color[rgb]{0,0,0}$\{\phi  = \rho\}$}%
}}}}
\put(5019,1980){\makebox(0,0)[lb]{\smash{{\SetFigFont{10}{12.0}{\familydefault}{\mddefault}{\updefault}{\color[rgb]{0,0,1}$\{\psi =- \eta\}$}%
}}}}
\put(3058,1941){\makebox(0,0)[lb]{\smash{{\SetFigFont{10}{12.0}{\familydefault}{\mddefault}{\updefault}{\color[rgb]{0,0,0}$\{\phi  = 0\}$}%
}}}}
\put(3160,-659){\makebox(0,0)[lb]{\smash{{\SetFigFont{10}{12.0}{\familydefault}{\mddefault}{\updefault}{\color[rgb]{0,0,0}$B(x^0,r)$}%
}}}}
\put(5476,-286){\makebox(0,0)[lb]{\smash{{\SetFigFont{10}{12.0}{\familydefault}{\mddefault}{\updefault}{\color[rgb]{0,0,0}$\nabla \phi(x^0) =\nabla \psi(x^0)$}%
}}}}
\put(4426,1664){\makebox(0,0)[lb]{\smash{{\SetFigFont{10}{12.0}{\familydefault}{\mddefault}{\updefault}{\color[rgb]{0,0,0}$\nabla \phi$}%
}}}}
\put(5612,1782){\makebox(0,0)[lb]{\smash{{\SetFigFont{10}{12.0}{\familydefault}{\mddefault}{\updefault}{\color[rgb]{0,0,1}$\{\psi  = 0\}$}%
}}}}
\end{picture}%

%% file: contours-dmu.pdf_t
\begin{picture}(0,0)%
\includegraphics{contours-dmu.pdf}%
\end{picture}%
\setlength{\unitlength}{3947sp}%
\begingroup\makeatletter\ifx\SetFigFont\undefined%
\gdef\SetFigFont#1#2#3#4#5{%
  \reset@font\fontsize{#1}{#2pt}%
  \fontfamily{#3}\fontseries{#4}\fontshape{#5}%
  \selectfont}%
\fi\endgroup%
\begin{picture}(6044,3102)(1779,-1348)
\put(5851,-1111){\makebox(0,0)[lb]{\smash{{\SetFigFont{10}{12.0}{\familydefault}{\mddefault}{\updefault}{\color[rgb]{0,0,0}$d\mu$}%
}}}}
\put(7051,-1111){\makebox(0,0)[lb]{\smash{{\SetFigFont{10}{12.0}{\familydefault}{\mddefault}{\updefault}{\color[rgb]{0,0,0}$2d\mu$}%
}}}}
\put(7426,-886){\makebox(0,0)[lb]{\smash{{\SetFigFont{10}{12.0}{\familydefault}{\mddefault}{\updefault}{\color[rgb]{0,0,0}$\Re(\zeta)$ }%
}}}}
\put(4801,-286){\makebox(0,0)[lb]{\smash{{\SetFigFont{10}{12.0}{\familydefault}{\mddefault}{\updefault}{\color[rgb]{0,0,0}$\frac{d\mu}{2}$}%
}}}}
\put(4985,-1146){\makebox(0,0)[lb]{\smash{{\SetFigFont{10}{12.0}{\familydefault}{\mddefault}{\updefault}{\color[rgb]{0,0,0}$\frac{d\mu}{4}$}%
}}}}
\put(4857,1619){\makebox(0,0)[lb]{\smash{{\SetFigFont{10}{12.0}{\familydefault}{\mddefault}{\updefault}{\color[rgb]{0,0,0}$\Im(\zeta)$ }%
}}}}
\put(3451,-1111){\makebox(0,0)[lb]{\smash{{\SetFigFont{10}{12.0}{\familydefault}{\mddefault}{\updefault}{\color[rgb]{0,0,0}$-d\mu$}%
}}}}
\put(4201,-286){\makebox(0,0)[lb]{\smash{{\SetFigFont{10}{12.0}{\familydefault}{\mddefault}{\updefault}{\color[rgb]{0,0,0}$\Gamma^H$ }%
}}}}
\put(6076,-661){\makebox(0,0)[lb]{\smash{{\SetFigFont{10}{12.0}{\familydefault}{\mddefault}{\updefault}{\color[rgb]{0,0,0}$\Gamma_+^V$ }%
}}}}
\put(4651,-1111){\makebox(0,0)[lb]{\smash{{\SetFigFont{10}{12.0}{\familydefault}{\mddefault}{\updefault}{\color[rgb]{0,0,0}$0$ }%
}}}}
\put(3301,-661){\makebox(0,0)[lb]{\smash{{\SetFigFont{10}{12.0}{\familydefault}{\mddefault}{\updefault}{\color[rgb]{0,0,0}$\Gamma_-^V$}%
}}}}
\end{picture}%

%% file: def-U-eps.pdf_t
\begin{picture}(0,0)%
\includegraphics{def-U-eps.pdf}%
\end{picture}%
\setlength{\unitlength}{3947sp}%
\begingroup\makeatletter\ifx\SetFigFont\undefined%
\gdef\SetFigFont#1#2#3#4#5{%
  \reset@font\fontsize{#1}{#2pt}%
  \fontfamily{#3}\fontseries{#4}\fontshape{#5}%
  \selectfont}%
\fi\endgroup%
\begin{picture}(6433,2978)(1896,-1271)
\put(6947,1110){\makebox(0,0)[lb]{\smash{{\SetFigFont{10}{12.0}{\familydefault}{\mddefault}{\updefault}{\color[rgb]{0,0,0}$B(x_4^\e ,r_4^\e)$}%
}}}}
\put(8314,-887){\makebox(0,0)[lb]{\smash{{\SetFigFont{10}{12.0}{\familydefault}{\mddefault}{\updefault}{\color[rgb]{0,0,0}$S_\e = \{\phi _\e = 0\}$}%
}}}}
\put(7858,-1207){\makebox(0,0)[lb]{\smash{{\SetFigFont{10}{12.0}{\familydefault}{\mddefault}{\updefault}{\color[rgb]{0,0,0}$\{\phi _\e = \rho_\e\}$}%
}}}}
\put(2539,574){\makebox(0,0)[lb]{\smash{{\SetFigFont{10}{12.0}{\familydefault}{\mddefault}{\updefault}{\color[rgb]{0,0,0}$B(x_2^\e ,r_2^\e)$}%
}}}}
\put(3755,1306){\makebox(0,0)[lb]{\smash{{\SetFigFont{10}{12.0}{\familydefault}{\mddefault}{\updefault}{\color[rgb]{0,0,0}$B(x_3^\e ,r_3^\e)$}%
}}}}
\put(1911,-282){\makebox(0,0)[lb]{\smash{{\SetFigFont{10}{12.0}{\familydefault}{\mddefault}{\updefault}{\color[rgb]{0,0,0}$B(x_1^\e ,r_1^\e)$}%
}}}}
\put(8163,-418){\makebox(0,0)[lb]{\smash{{\SetFigFont{10}{12.0}{\familydefault}{\mddefault}{\updefault}{\color[rgb]{0,0,0}$B(x_6^\e ,r_6^\e)$}%
}}}}
\put(7595,626){\makebox(0,0)[lb]{\smash{{\SetFigFont{10}{12.0}{\familydefault}{\mddefault}{\updefault}{\color[rgb]{0,0,0}$B(x_5^\e ,r_5^\e)$}%
}}}}
\end{picture}%

%% file: def-V-eps.pdf_t
\begin{picture}(0,0)%
\includegraphics{def-V-eps.pdf}%
\end{picture}%
\setlength{\unitlength}{3947sp}%
\begingroup\makeatletter\ifx\SetFigFont\undefined%
\gdef\SetFigFont#1#2#3#4#5{%
  \reset@font\fontsize{#1}{#2pt}%
  \fontfamily{#3}\fontseries{#4}\fontshape{#5}%
  \selectfont}%
\fi\endgroup%
\begin{picture}(6433,2978)(1896,-1271)
\put(2539,574){\makebox(0,0)[lb]{\smash{{\SetFigFont{10}{12.0}{\familydefault}{\mddefault}{\updefault}{\color[rgb]{0,0,0}$B(x_2^\e ,r_2^\e)$}%
}}}}
\put(3755,1306){\makebox(0,0)[lb]{\smash{{\SetFigFont{10}{12.0}{\familydefault}{\mddefault}{\updefault}{\color[rgb]{0,0,0}$B(x_3^\e ,r_3^\e)$}%
}}}}
\put(1911,-282){\makebox(0,0)[lb]{\smash{{\SetFigFont{10}{12.0}{\familydefault}{\mddefault}{\updefault}{\color[rgb]{0,0,0}$B(x_1^\e ,r_1^\e)$}%
}}}}
\put(8163,-418){\makebox(0,0)[lb]{\smash{{\SetFigFont{10}{12.0}{\familydefault}{\mddefault}{\updefault}{\color[rgb]{0,0,0}$B(x_6^\e ,r_6^\e)$}%
}}}}
\put(7595,626){\makebox(0,0)[lb]{\smash{{\SetFigFont{10}{12.0}{\familydefault}{\mddefault}{\updefault}{\color[rgb]{0,0,0}$B(x_5^\e ,r_5^\e)$}%
}}}}
\put(6947,1110){\makebox(0,0)[lb]{\smash{{\SetFigFont{10}{12.0}{\familydefault}{\mddefault}{\updefault}{\color[rgb]{0,0,0}$B(x_4^\e ,r_4^\e)$}%
}}}}
\put(7858,-1207){\makebox(0,0)[lb]{\smash{{\SetFigFont{10}{12.0}{\familydefault}{\mddefault}{\updefault}{\color[rgb]{0,0,0}$\{\phi _\e = \rho_\e\}$}%
}}}}
\put(8314,-947){\makebox(0,0)[lb]{\smash{{\SetFigFont{10}{12.0}{\familydefault}{\mddefault}{\updefault}{\color[rgb]{0,0,0}$S_\e = \{\phi _\e = 0\}$}%
}}}}
\put(5790,1064){\makebox(0,0)[lb]{\smash{{\SetFigFont{10}{12.0}{\familydefault}{\mddefault}{\updefault}{\color[rgb]{0,0,1}$S_{\e + g(\e)}$}%
}}}}
\put(7963,146){\makebox(0,0)[lb]{\smash{{\SetFigFont{10}{12.0}{\familydefault}{\mddefault}{\updefault}{\color[rgb]{0,0,1}$S_{\e - g(\e)}$}%
}}}}
\put(8026,417){\makebox(0,0)[lb]{\smash{{\SetFigFont{10}{12.0}{\familydefault}{\mddefault}{\updefault}{\color[rgb]{0,0,1}$\mathcal{V}_\e$}%
}}}}
\end{picture}%